\documentclass[11pt]{article}
\usepackage{graphicx}
\usepackage{amssymb,amsfonts,amsthm}
\usepackage{enumerate}
\usepackage[T2B]{fontenc}
\usepackage[cp1251]{inputenc}
\usepackage{enumitem}

\usepackage{amsmath,amsthm,amssymb}
\usepackage{amscd}
\usepackage{amsfonts}
\usepackage{color}

\newcommand{\pr}{\pageref}

\newcommand{\N}{{\mathbb{N}}}

\newcommand{\cost}{{\mathrm{time}}}

\newtheorem{theorem}{Theorem}[section]
\newtheorem{lemma}[theorem]{Lemma}

\newtheorem{cy}[theorem]{Corollary}
\newtheorem{prop}[theorem]{Proposition}
\newtheorem{property}[theorem]{Property}
\theoremstyle{definition}
\newtheorem{df}[theorem]{Definition}

\newtheorem{rk}[theorem]{Remark}

\newcommand{\card}{\mathrm{card}}
\newcommand{\area}{\mathrm{Area}}

\newcommand{\me}{\medskip}

\newcommand{\Lab}{{\mathrm{Lab}}}

\newcommand{\tool}{\stackrel{\ell}{\too} }

\newcommand{\ttt}{{\cal T}}

\newcommand{\aaa}{{\cal A}}

\newcommand{\bb}{{\cal B}}
\newcommand{\topp}{{\bf top}}
\newcommand{\ttopp}{{\bf ttop}}
\newcommand{\tbott}{{\bf tbot}}
\newcommand{\bott}{{\bf bot}}
\newcommand{\vk}{van Kampen }

\newcommand{\Sym}{\hbox{Sym}}

\newcommand{\ccc}{{\cal C}}

\newcommand{\cee}{\alpha}
\newcommand{\dol}{\omega}
\newcommand{\iv}{^{-1}}

\newcommand{\too}{\to }

\newcommand{\sss}{{\cal S} }

\newcommand{\base}{\mathrm{base}}
\begin{document}
\renewcommand{\theequation}{\thesection.\arabic{equation}}

\title{Groups with undecidable word problem and almost quadratic
Dehn function}

 \author{A.Yu. Ol'shanskii\thanks{The 
author was supported in part by the NSF grant DMS 0700811 and by 
the Russian Fund for Basic Research
grant 08-01-00573}}
\date{}
\maketitle
\begin{center}
\large{\textit{with an Appendix by M.V.Sapir}}\footnote{The author of the Appendix was supported in
part by the NSF Grant  DMS-0700811.}
\end{center}
\maketitle

\begin{abstract} We construct a finitely presented group with undecidable word problem
and with Dehn function bounded by a quadratic function on an infinite set of positive integers.
\end{abstract}


\medskip

\medskip


\tableofcontents

\section{Introduction}

\subsection{Formulation of results}

The minimal non-decreasing function $f(n)\colon \mathbb{N}\to \mathbb{N}$ such that
every word $w$ vanishing in a group $G=\langle A\mid R\rangle$  and having length \label{lengthw||} $||w||\le n,$ 
freely equal to a product of at most $f(n)$ conjugates of relators from $R$
is called the \label{Dehnf}{\em Dehn function} of the presentation 
$G=\langle A\mid R\rangle$ \cite{GrHyp}. By van Kampen's Lemma, $f(n)$ is equal to the maximal
area of minimal diagrams $\Delta$ with perimeter $\le n.$ (See Subsection \ref{md}
for the definitions.) For {\it finitely presented} groups (i.e., both sets $A$ and $R$ are finite)
Dehn functions are usually taken up to equivalence to get rid of the dependence  on
a finite presentation for $G$ (see \cite{MO}). To introduce this \label{equivf}{\em equivalence} $\sim,$ we write $f\preceq g$
if there is a positive integer $c$ such that
$f(n)\le cg(cn)+cn\;\;\; for \;\; any \;\;n\in \mathbb{N}.$ 
Two non-decreasing functions $f$ and $g$ on $\mathbb{N}$ are called equivalent if
$f\preceq g$ and $g\preceq f.$

A function $f\colon \mathbb{N}\to \mathbb{N}$
is called \label{almostqf}{\em almost quadratic} if there exists a constant $C>0$ and an infinite set of integers $B$,
such that
$f(b)<Cb^2$ for all $b\in B$.

It is well known that a finitely presented group has undecidable word problem if and
only if its Dehn function $d(n)$ is not bounded by a recursive function (and if and only if $d(n)$ is not  recursive itself; see \cite{Ger}, \cite{BRS}), whence
for every recursive function $f(n)$, $d(n)>f(n)$ for infinitely many
values of $n$. The main result of this paper shows that a
non-recursive Dehn function can be almost quadratic at the same
time.

\begin{theorem}\label{mainth} There exists a finitely presented group $G$ with
undecidable word problem and almost quadratic Dehn function.
\end{theorem}

Note that ``almost quadratic'' is the smallest Dehn function one can get, because if the Dehn function of a finitely presented group is $o(n^2)$ on some infinite set of integers, then the group is hyperbolic and its Dehn function is linear (this follows from Gromov \cite[6.8.M]{GrHyp} or Bowditch \cite{Bow}).

By Theorem \ref{mainth},  for some infinite set $B$ of natural numbers $b,$ the Dehn
function of $G$ satisfies the condition $f(b)<Cb^2$ for some constant $C$. Notice that
the set $B$ is not recursive or even recursively enumerable although its complement is recursively enumerable. Indeed, if $f(b)\ge Cb^2,$ then there exists a word $w$ of length $\le b$ which is equal to 1 in the group, but which is not the boundary label of any van Kampen diagram with less than $Cb^2$ cells; all diagrams with this number of cells and boundary length at most $b$ can be enumerated; and all words that are equal to 1 in the group can be enumerated too.
Moreover, $B$ cannot contain any infinite recursively enumerable
subset (i.e. it is {\em immune} in the terminology of \cite{Mal}).
Indeed if $B$ contains an infinite recursively enumerable set
enumerated by a Turing machine $M$, then in order to check if a word
$w$ is 1 in $G$ (and solve the word problem in $G$) we would do the
following: wait till $M$ produces a word $w'$ longer than $w$. Then
the area of the minimal van Kampen diagram for $w$ cannot exceed $C||w'||^2$ (here and below $||w||$ denotes the length of the word $w$), and
it would remain to check all diagrams of that area. Thus although
$B$ exists and is infinite, there is no algorithm to find any
infinite part of it. 

As a corollary of Theorem  \ref{mainth} and the results of  \cite{OS2} and
\cite{Gr1} we get

\begin{cy} The group $G$ from Theorem \ref{mainth} has a simply connected and a non-simply connected asymptotic cone.
\end{cy}

Indeed the asymptotic cone corresponding to the sequence
$B$ discussed in the previous paragraph is simply connected by
\cite{OS2}. On the other hand all asymptotic cones of $G$ cannot be
simply connected because that would imply decidability of the word
problem in $G$ by \cite{Gr1}.

Undecidability of conjugacy problem is easier to achieve than
undecidability of the word problem.

\begin{theorem}\label{thmain1} There exists a finitely presented
(multiple) HNN extension $M$ of a free group with finitely generated
associated subgroups  and with Dehn  function $f(n)$ such that:
\begin{enumerate}
\item The conjugacy problem is undecidable in $M$;
\item There is an infinite set $N_1\subseteq \mathbb{N},$
such that for some constant $C$ we have $f(n)<Cn^2$ for every $n\in N_1$;
\item For every $n$, $f(n)\le C n^3$.
\end{enumerate}
\end{theorem}

\begin{rk} Probably the first example of an almost quadratic but not quadratic
Dehn function of a finitely presented group was constructed in \cite{O}. However
that function is $O(n^2 \log n /\log\log n)$ which is not much bigger
than a quadratic function. A slight modification of the proofs of the
present paper provides us with a {\it recursive} almost quadratic Dehn function $d(n)$ 
rapidly increasing on some infinite subset $N_2$ of $\mathbb N.$ (For example, almost
quadratic $d(n)$ is at least exponential on $N_2$ and at most exponential on the entire $\mathbb N$;
see Theorem \ref{exp} and Remark \ref{other} for details.)
The difference is that the proof of Theorem \ref{mainth} uses Sapir's Theorem \ref{re}, 
but the recursive examples are independent of it.
\end{rk}

\begin{rk} \label{almpol}
Using \cite{SBR} and Theorem \ref{re} one can easily obtain a weaker version of Theorem \ref{mainth} replacing ``almost quadratic'' by ``almost polynomial''. Recall that a function 
 $f\colon \mathbb{N}\to \mathbb{N}$ is superadditive if $f(m+n)\ge f(m)+f(n)$ for any $m,n\in \mathbb{N}.$ The superadditive closure $\bar f(n)$ of a function $f(n)$ is given by the
formula $\bar f(n) = \max(f(n_1)+\dots+f(n_k))$ over all non-negative partitions $n=n_1+\dots+n_k$. If a Turing machine $M$ accepts a language $L$
with at least linear time function $T(n),$ and $T(n)^4$ is equivalent to a superadditive 
function, then by Theorem 1.3 \cite{SBR}, there is a finitely
presented group $G(M)$ with Dehn function $d(n)$ equivalent to $T(n)^4.$  But in fact, it is proved
in \cite{SBR} that omitting the assumption that $T(n)^4$ is superadditive, we have
inequalities $T(n)^4 \preceq d(n)\preceq \overline{ T(n)^4}.$  Thus it suffices to construct
a Turing machine $M$ with non-recursive but ``almost linear'' time function $T(n).$ The existence
of such a machine follows from Theorem \ref{re}. (Moreover, one can derive from
Theorem \ref{re} that $\overline{ T(n)^4}$ is ``almost $n^4$''.) 
\end{rk}


Reducing to ``almost quadratic" (as in Theorem \ref{mainth}) requires a new approach.
The $S$-machine we are going to use will be different from \cite{SBR}, and the analysis of diagrams will be much more delicate. The main reason for the difficulties arising here
is that the ``almost quadratic'' property is unimprovable. For example, the {\it cubic} upper bound
of the Dehn function of the group $M$ is obvious in \cite{SBR} (see also Step 1 in the proof of Lemma \ref{grubo} below), but the main contents of our paper focus on $M,$ starting with
the properties of the machine defining $M$ and ending with new quadratic invariants
of the diagrams called mixture(s) on their boundaries. In the next subsection of the introduction, we discuss the outline of the proof of Theorem \ref{mainth}, and some ideas needed in its proof.

\subsection{A short description of the proof of Theorem \ref{mainth}}

Relations of a finitely presented group with undecidable word problem 
simulate the commands of a Turing machine $M_0$ with undecidable halting problem,
and as in the works of P.Novikov, W.Boone and many other authors (see \cite{R}, \cite{SICM}), one has to 
properly code the work of a Turing machine in terms of group relations. To
obtain an almost quadratic Dehn function of a group $G$, we must start with a machine having almost linear time function (but which is not bounded from above by any recursive function). 
Thus we can just demand that the lengths of words accepted by $M_0$  form a very sparse subset of positive integers $B\subset\mathbb N.$  As a measure of how sparse  $B$ is, M.V. Sapir
gives the following exact definition.

Let $X$ be a recursively enumerable (r.e.) language in the binary
alphabet recognized by a Turing machine $M$. If
$w\in X$ then the {\em time} of $w$ \label{cost} (denoted $\cost(w)$ or $\cost_M(w)$) is, by
definition, the minimal time of an accepting computation of $M$ with
input $w$. For an increasing function $h\colon \N\to \N$, a number $m\in \mathbf N$ is called {\em $h$-good}\label{good} for $M$ if for
every word $w\in X$ of length $<m,$ we have $h(\cost(w))<m$. 

For our estimates, it suffices to start with a Turing machine $M_0$ recognizing a r.e. non-recursive
set $X$ such that the set of all $f$-good numbers for $M_0$ is infinite, where $f$ is arbitrary
double exponential function. Such a machine is constructed in the Appendix written by M.V.Sapir (see
Theorem \ref{re}).

Group relations always interpret the symmetrization of a machine. Thus as a preliminary step,  one has to add the inverse
commands, in spite of the fact that the machine $M_0$ is replaced by a non-deterministic
machine $M_1.$
(Of course, one should be concerned that the symmetrization preserves some basic characteristics
of the machine.) 
However the interpretation problem for groups remains much harder than for semigroups even after
modifying the machine because the group theoretic simulation can execute unforeseen computations with non-positive
words. Boone and Novikov secured the positiveness of admissible configurations  with the help of an
additional `quadratic letter' (see \cite{R}, Ch.12). However this old trick implies that the constructed
group $G$ contains Baumslag - Solitar groups $B_{1,2}$ and has at least exponential Dehn function. Since we want to obtain almost quadratic Dehn function, we use a new approach  suggested in \cite{SBR}. Invented by
Sapir, S-machines can work with non-positive words on the tapes. Here we use an 
a composition $M_2$ of the symmetric machine $M_1$ with an `adding machine' $Z(A)$
introduced in \cite{OS}. This  S-machine is equivalent to $M_1.$

Here we have to guarantee at least two important properties of the machine  $M_2$
(since the violation of either of them makes the Dehn function of the group $M_2$ non-almost
quadratic): a reduced computation of $M_2$ does not repeat the same configuration twice,
and every {\it accepting} computation of $M_2$ is uniquely determined by the initial configuration (though $M_2$ is highly non-deterministic).

Every $S$-machine is, on the one hand, a rewriting system and, on the other hand,
it can be treated as a multiple HNN-extension of a free group (see \cite{OS} or \cite{SICM}).
But when one takes an $S$-machine as an HNN-extension, then the number of working heads 
can be arbitrarily large and their order on the common tape can be non-standard.
Therefore, as in \cite{SBR} or \cite{OS}, we have to extend the set of admissible words for
the machine treated as a rewriting system. This makes  the control of
arbitrary computation difficult. (There was no need for such accurate control in \cite{SBR} or \cite{OS}.) Hence  we are forced to introduce auxiliary control heads
which are called upon to examine the order of heads after and/or before the
application of every rule of the machine $M_2.$ The obtained machine $M_3$ is
better than $M_2$ because it is able to accomplish only `simple' computations
with non-standard disposition of the heads. 

When we introduce a new
machine, then clearly, we should check that it inherits the important properties
of the machines studied earlier. In particular, $M_3$ inherits the language accepted
by $M_1.$ The next modification is the machine $M_4$ which has two additional
tapes with histories of what $M_3$ computes. For any  computation with the standard
order of heads of $M_4$ (the notion of standard base of $M_4$ is given in subsection \ref{M4}), we prove that either the time $T$ of this computation is `close' to the time $T_i$ of some  computation accepting a word from the sparse set provided by
Theorem \ref{re}, or  the space on the `historical' tapes at the beginning or at the end
of the computation is bounded from below by a linear function of $T$. In the latter case we have a quadratic
upper estimate for the area of the trapezium corresponding to the computation. Here the definition of trapezium as a special van Kampen diagram is borrowed from \cite{OS}, and Lemma \ref{simul} 
translates the machine language to the diagram language. 

Finally, the machine $M$ is a union of many copies and mirror copies of $M_4$
working in a parallel way. The corresponding HNN-extension $M$ is the group from
Theorem \ref{thmain1}. The accept word of the machine $M$ is called the hub. There is
no algorithm deciding if a given word in the generators of $M$ is conjugate to the hub.
The usual adding of the hub relation to the list of defining relations of $M$ 
(as in \cite{R}, and many papers)
provides us with the group $G$ for Theorem \ref{mainth}. As in \cite{SBR} or \cite{O1}
the hub relation has many copies of the accept words of $M_4.$ This makes the hub graph 
(with vertices in hub cells; see Subsection \ref{dwh}) associated with a van Kampen diagram, hyperbolic, and this is used in
Lemmas \ref{extdisc} and \ref{mnogospits}. The mirror symmetry of the hub is
used for the surgery removing a hub (see Subsection \ref{srh}).

Unsolvability of the halting problem for the S-machine $M$  immediately implies
that the Dehn function $d(n)$ of the group $G$ is not bounded from above by any recursive function.
Other precautions used in the construction of the machines $M_0,\dots, M_4, M$ eliminate
a number of visual obstacles standing in the way of the almost quadratic property for the Dehn functions of $M$ and $G.$
(For instance, if one uses only one historic tape for $M_4$ or the arrangement 
of the  historic heads is different, then the almost quadratic estimate is not
achieved by our model.)  But how  can one {\it prove} this property ?

The areas of diagrams whose perimeters are close to some numbers $T_i$ mentioned above
can be non-recursively high in comparison with their perimeters. Therefore one
has to consider a diagram $\Delta$ whose perimeter $n$ is far from the infinite increasing
sequence $\{T_1,T_2,\dots\},$ for example, $\exp T_{i-1} < n < T_i$ for some $i$.
Unlike a trapezium, an arbitrary diagram has irregular structure. Therefore we want
to find some more regular pieces to cut them off and then use an induction on the
perimeter $n.$  

Indeed assume that there is a simple path $\bf y$ in $\Delta$ cutting up $\Delta$
into two subdiagrams $\Delta_1$ and $\Delta_2$ with boundary paths ${\bf yz}$ and ${\bf y}^{-1}{\bf z'},$ resp., where ${\bf zz'}$ is the boundary of $\Delta.$    Assume that
$|{\bf z}|>|{\bf y}|$ (where $|\bf x|$ is the length of a path $\bf x$) and moreover,
$\area(\Delta_1)\le C|{\bf y}|(|{\bf z}|-|{\bf y}|),$ where the positive constant $C$
does not depend on the diagrams. Then it is easy to see that the quadratic estimate
$\area(\Delta_2)\le C|{\bf yz}|^2$ for the subdiagram $\Delta_2$ with perimeter $<n$
together with the estimate for $\area(\Delta_1)$ give 
$$\area(\Delta)\le \area(\Delta_1)+\area(\Delta_2)\le C|{\bf zz'}|^2=Cn^2$$ as required.
Thus we are looking for pieces whose area can be estimated as for $\Delta_1.$ 

First of all, among such `good' pieces, we have so called rim $\theta$-bands with a restriction on the
length (but here we have to change the usual combinatorial metric by the metric,
where the generators from the tape alphabet of the machine $M$ are much shorter than other generators of the groups $M$ and $G$).
The `good' pieces of second type (again, under some restrictions) are combs defined 
in Section \ref{cmb} (and introduced earlier in \cite{OS}).

The upper bound of the form $C|{\bf y}|(|{\bf z}|-|{\bf y}|)$ works for many types
of combs but unfortunately, it is false for other combs whose areas must also be estimated.
We have found another quadratic invariant of the boundaries  of the diagrams, called
mixture. In Section \ref{mix}, we associate a two-colored necklace with the boundary $\bf p$ of $\Delta$. The black and white beads of this necklace correspond to different types of edges in $\bf p.$  
To obtain the mixture $\mu(\Delta)$ one calculates the number of pairs of white beads
separated in $\bf p$ by black ones.  (Another quadratic invariant, called dispersion,
was introduced and applied earlier in \cite{OS}, but the dispersion depends on the whole diagram and works for  hub free diagrams while the mixture depends on the boundary label only and works for arbitrary diagram over $G$.)  

The important observation is that for many types of subcombs $\Delta_1,$ we have
inequalities $\area(\Delta_1)\le C|{\bf y}|(|{\bf z}|-|{\bf y}|)+ \mu(\Delta_1),$
and $\mu(\Delta_2)\le \mu(\Delta)-\mu(\Delta_1).$ 
This was a breakthrough which inspired the confidence that the whole project would be completed.
However the original mixture cannot help in case of some special combs. Therefore we have to
consider boundary necklaces  of 3 different types. The different mixtures help to estimate the areas of different combs. But one of these  mixtures helps in some cases and can be negative in some other cases, which causes a problem for our induction. Therefore we use
a weighted linear combination of 3 mixtures in the Lemma \ref{itog} summarizing our
estimates of comb areas. Hence we have to estimate the behavior of these mixtures in
different situations, which makes a number of comb lemmas complicated, and the comb part 
of the paper is the hardest one.

Then we consider a diagram $\Delta$ with hubs. Due to  hyperbolicity of the hub structure
mentioned above,
there is a hub $\Pi$ such that almost all `spokes' starting on $\Pi$ end on the
boundary $\partial\Delta,$ and they bound (together with $\partial\Delta$ and $\partial\Pi$)
a subdiagram $\Psi$ without hubs. Now we are able to remove redundant combs and rim bands
from $\Psi.$ The remaining {\it crescent} $\tilde\Psi$ together with $\Pi$ can be cut off by a relatively short cutting path. (Thus, one can also induct on the number of hubs in $\Delta$.) As in \cite{SBR}, our  surgery uses the mirror symmetry of the hub relation,
but our inequalities are more delicate here than those used for the `snowball decomposition' in \cite{SBR}
since we aim for almost quadratic bounds.
Again we estimate the area of the removed part in terms of the reduction of the perimeter, of the mixtures, and more. To complete the
proof, we take into account that the auxiliary parameters are quadratically bounded with
respect to the perimeter of a diagram.  

The author is aware that such a long proof can be arduous to the reader. Making our apology we collect all the definitions and terms at the end of the paper (see Subject Index) and insert 
many pictures and brief comments throughout the text. Besides, Lemma \ref{summary} reformulates
all machine properties we need in terms of van Kampen diagrams so that the machine
constructions can be forgotten after one has read that lemma.
\medskip

{\bf Acknowledgment.} The author is grateful to M.V. Sapir who was involved in the joint 
work when the project started in March 2005. Although he has retired from the absorbing and
exhausting struggle against the details, his contribution to the proof is certainly
bigger than what is written by him in Appendix.

\section{A Turing machine}

\subsection{Definitions and notations related to Turing machines}

In this section, we collect all information about Turing machines
that we need in the proof of our main results.

As usual we consider words as sequences of symbols from some alphabet $X$. 

We shall use the following standard notation for Turing machines. A
(multi-tape) \label{Turingm}{\em Turing machine} has $k$ tapes and $k$ heads observing
the tapes. One can view it as a structure  $$M= \langle I, Y, Q,\Omega,
\Theta\rangle,$$ where $I$ is the input alphabet, $Y$ is the tape
alphabet ($I\subseteq Y$), $Q=\sqcup Q_i,
i=1,...,k$
is the set of states of the heads
of the machine (and $\sqcup$ denotes disjoint union),\\ $\Omega=\{\alpha_1,\omega_1,\dots,\alpha_k,\omega_k\}$ is the
set of left and right markers of the tapes, and
$\Theta$ is a set of
commands.

The leftmost (the rightmost) square on the $i$-th tape is always marked by $\cee_i$ (by $\omega_i$).
The head is placed between two consecutive squares on the tape. A 
{\em configuration} of the $i$-th tape of a Turing machine is a word $\cee_i u
q v \dol_i,$ where $q\in Q_i$ is the current state of the head of that tape,
$u$ is the word in $Y$ to the left of the head and $v$ is the word
in $Y$  to the right of the head, and so the word written on the
entire tape is $uv;$ so we do not include $\alpha_i,$ $\omega_i$ and
the state letter when we talk about the word written on the tape. 

At every moment the head of each tape observes two letters on that
tape: the last letter of $u$ (or $\alpha_i$) and the first letter of
$v$ (or $\omega_i$).

A \label{config}{\em configuration} $U$ of a
Turing machine is the word
$$U_1U_2\dots U_k,$$
where $U_i$ is the configuration of tape $i$. 
We shall omit the indices $i$ of $\alpha_i$ and $\omega_i$ for the sake of brevity.

Assuming that the Turing machine is recognizing, we can define input configurations and accepted (stop) configurations.
An \label{inputc}{\it input configuration} is a configuration where the word written
on the first tape is in $I$, all other tapes are empty, the head on the first tape observes
the right marker $\omega$, and the states of all tapes form a special \label{startv} {\em start} $k$-vector $\vec s_1$.
An \label{acceptc}{\em accept (or stop) configuration} is any configuration
where the state vector for a special $k$-vector  $\vec s_0$, the \label{acceptv}{\em accept vector} of the machine. We shall always assume (as can be easily achieved) that in the accept configuration of a Turing machine every tape is empty.

A transition (\label{command}{\it command}) of a Turing machine is given by the states of the heads and some of the $2k$ letters
observed by the heads. As a result of a
transition we replace some of these $2k$ letters by other letters,
insert new squares in some of the tapes and may move the heads
one square to the
left (right) with respect to the corresponding tapes.

For example in a one-tape machine, every transition is of the
following form: $$uqv\to u'q'v',$$ where $u, v, u', v'$ are letters
(could be end markers) or empty words. The only constraint  is that
the result of applying the substitution $uqv\to u'q'v'$ to a
configuration word must be a configuration word again, in particular
the end markers cannot be deleted or inserted.
This command means that if the state of the
head is $q$, $u$ is written to the left of $q$ and $v$ is written to
the right of $q,$ then the machine must replace $u$ by $u'$, $q$ by
$q'$ and $v$ by $v'$.

For a general  $k$-tape machine, a command is a vector
$$[U_1\to V_1,...,U_k\to V_k],$$
where $U_i\to V_i$ is a command of a 1-tape machine, the elementary commands (also called {\em parts of the command})
$U_i\to V_i$ are listed in the order of tape numbers. In order to execute this
command, the machine checks if $U_i$ is a subword of the configuration
of tape $i$ ($i=1,...,k$), and then replaces $U_i$ by $V_i$.

Notice that for every command
$[U_1\to V_1,...,U_k\to V_k]$, the vector $[V_1\to U_1,...,V_k\to U_k]$
is also a command of a Turing
machine. These two commands are called \label{inversec}{\em mutually inverse}. A Turing machine is called \label{symmetricm}{\em symmetric} if for every command of that machine, the inverse is also a command of the machine. If a Turing machine is symmetric, we shall always consider a division of the set of its commands $\Theta$ into two disjoint subsets, \label{positivec} positive and negative commands: $\Theta=\Theta^+\sqcup \Theta^-$, so that the inverses of commands in \label{T+-} $\Theta^+$ are in $\Theta^-$ and vice versa.

We will assume that only input configurations of a Turing machine
involve the state letters from $\vec s_1$ and only one (positive) command \label{tstart} $\theta_{start}$ is applicable to the
input configurations. 
Similarly, we assume that  there is a unique
accept configuration $\vec s_0$ and a unique (positive) accepting command \label{taccept}$\theta_{accept}.$  

A \label{computation}{\em computation} is a sequence of configurations $C_0\to \dots\to C_n$ such that for
every $0=1,..., n-1$ the machine passes from $C_i$ to $C_{i+1}$ by applying one
of the commands from $\Theta$.  A configuration $C$ is said to be {\em
accepted} by a machine $M$ if there exists at least one computation which starts
with $C$
and ends with the accept configuration.

A word $u$ over $I$ is said to be \label{acceptediw}{\em accepted} by the machine if the
corresponding input configuration is accepted.  The set of all accepted words
over the input alphabet $I$ is called the \label{languageam}{\em language accepted (recognized) by the machine}.

Let  $C = C_0\to\dots\to C_n$
be a computation of a machine $M$ such that for
every $j=0,...,n-1$ the configuration $C_{j+1}$ is obtained from $C_j$ by a
command $\theta_{j+1}$ from $\Theta$.  Then we call the word $\theta_1 \ldots \theta_{n}$
the \label{history}{\em
history} of this computation.  The number $n$ will be called the \label{lengthtime}{\em time} (or {\em length})
of the computation.

\begin{rk}\label{rk1} Note that we can (and will)  assume that in
every command $[u_1q_1v_1\to u_1'q_1'v_1',...,u_kq_kv_k\to
u_k'q_k'v_k']$ the sum of numbers of letters from the tape alphabet $Y$
in all
$u_i, u_i', v_i, v_i'$, $i=1,...,k,$  
is at most 1. 
Indeed, this can
be achieved by subdividing a command in the standard way. For
example, a command $[aq\to bq']$ is replaced by two commands $[aq\to
q'']$, $[q''\to bq']$, where $q''$ is a new state letter.
\end{rk}

It is
convenient to consider {\em empty computations} consisting of one
word $W$.
The history of an empty computation is the empty word, the
start and end words of this computation are equal to $W$.
We do not only consider {\em deterministic} Turing machines, for example,
we allow several transitions with the same left side. For example, most symmetric Turing machines are not deterministic.

\subsection{A conversion of a deterministic Turing machine into a symmetric Turing machine}

At first, let us add some useful properties to a machine.


\begin{lemma}\label{m0m1}
 Let $M_0$ be a deterministic Turing machine recognizing a set of words $X$. Then there exists a
deterministic Turing machine $M_1$ which recognizes $X$ and such that

\begin{enumerate}
\item[(a)] If $W\equiv W'$ (i.e. these two words are \label{equivgraph} letter-for-letter equal) for a computation $W\to\dots\to W'$, then this computation is empty.
\item [(b)] The property from Remark \ref{rk1} (``at most one tape letter'') holds for every command of $M_1.$ 
\item [(c)] The state letters from the start vector $\overrightarrow s_1$ (from the accept vector $\overrightarrow s_0$)
occur in the left-hand side of a unique command $U\to V$ of $M_1$ and do not occur in the right-hand side of any command 
(resp., occur in the right-hand side of a unique command $U'\to V'$ and do not occur in the left-hand side of any command) 
\item[ (d) ] The letters used on different tapes are from disjoint alphabets.
The letters to the left and to the right of the head of any tape are from disjoint alphabets.

\end{enumerate}
\end{lemma}

\proof Let $W$ be a configuration of $M_0$. The general form of a configuration of $M_1$ will be $W\alpha \tau^lq_{k+1}\omega$, that is the machine $M_1$ has one more tapes than $M_0$.  The 
last tape contains a (non-negative) power of a special tape letter $\tau$.
The set of state letters is then increased by one component $\{q_{k+1}\}$.
At the beginning the last tape is empty.  The machine will
execute $M_0$ on its tapes, adding new $\tau$ on the last tape after every step of computation. 
After $M_0$ accepts, $M_1$ erases the last tape  and stops. 
The last tape  guarantees Property (a)).
In order to get Property (b) of $M_1$, we apply the trick from Remark \ref{rk1} since it
does not violate Property (a).

Property (c) of $M_1$ for the start command 
follows from the same property of $M_0,$ and we can define the accept command of $M_1$
so that Property (c) also holds for it. 

In order to obtain Property (d), we use different copies of the tape alphabet
for different tapes, and moreover, we use different copies from the left and from the right
of each head. \endproof

If $M_1$
is a deterministic Turing machine, satisfying the properties of Lemma \ref{m0m1}, with the set of commands $\Theta$ such that $\Theta\cap \Theta\iv=\emptyset$, then let \label{symM1} $\Sym(M_1)$ be the Turing machine with the set of commands $\Theta\sqcup \Theta\iv$ and the same sets of state and tape letters. The division of the commands of $\Sym(M_1)$ into positive and negative is natural: the commands of $M_1$ are positive, their inverses are negative. The computation of $\Sym(M_1)$ is called \label{reducedc} {\em reduced} if its history is a reduced word. Clearly, every computation can be made reduced (without changing the start or end configurations of the computation) by removing consecutive mutually inverse commands.

\begin{lemma} \label{mach23} 
The Turing machine $\Sym(M_1)$ satisfies the following properties.
\begin{enumerate}

\item [(a)] Every command of $\Sym(M_1)$ 
satisfies Property (b) of Lemma \ref{m0m1}.
\item [(b)] Every reduced history of computation of $\Sym(M_1)$ has the form $H_1H_2\iv,$ where $H_1, H_2$ consist of positive commands.
\item[(c)] $\Sym(M_1)$ satisfies Properties (a), (c) (for positive commands), and (d) of Lemma \ref{m0m1}
\item [(d)] The language recognized by $\Sym(M_1)$ is $X$.
\item [(e)] For every $W\in X$ there exists only one accepting computation of $\Sym(M_1).$ 
It is equal to the computation accepting $W$ by $M_1,$
and if $M_1$ is given by Lemma \ref{m0m1}, then the length of this computation is
big-O of the length of the accepting computation of $M_0$ with input $W$.
\end{enumerate}
\end{lemma}

\proof Property (a) is obvious. Property (b) follows immediately from the fact that in a reduced computation, a command from $\Theta\iv$ cannot be followed by a command from $\Theta$ (since $M_1$ is deterministic). Properties (c), (d), (e) follow from (b).
\endproof

\section{$S$-machines}
\label{smach}

\subsection{$S$-machines as rewriting systems}

There are several interpretations of $S$-machines in groups, the most complicated so far is in \cite{OScol}. Another interpretation is given in \cite{OS}, and in fact it is probably the easiest way to view an $S$-machine 
as a group that is a multiple HNN extension of a free group. Here we  use a definition which is close to the original
one \cite{SBR} and  define $S$-machines as rewriting systems, similar to Turing machines.
 $S$-machines work with words in group alphabets
and they are almost ''blind'', i.e., the heads do not observe the tape letters. But the heads can ''see''
each other if there are no tape letters between them. We will use the following precise definition
of an \label{Smachine}$S$-machine $\cal S$.

A hardware of an $S$-machine $\cal S$ is a pair $(Y,Q),$ where $Q=\sqcup_{i=0}^NQ_i$ and $Y= \sqcup_{i=1}^N Y_i$ (for convenience we always set $Y_0=Y_{N+1}=\emptyset$).
The elements from $Q$ are called \label{statel}{\it state letters}, the elements from $Y$ are \label{tapel}{\it tape letters}. The sets $Q_i$ (resp.
$Y_i$) are called {\em parts} of $Q$ (resp. $Y$).

The {\it language of \label{admissiblew} admissible words}
consists of reduced words of the form 
\begin{equation}\label{admiss}
W\equiv q_1^{\pm 1}u_1q_2^{\pm 1}\dots u_k q_{k+1}^{\pm 1}, 
\end{equation}
where  every subword $q_i^{\pm 1}u_iq_{i+1}^{\pm 1}$ either \\(1) belongs to $(Q_jF(Y_{j+1})Q_{j+1})^{\pm 1}$ for some $j$ and $u_i\in F(Y_{j+1}),$ where $F(Y_i)$ is the set of reduced group words in the alphabet $Y_i^{\pm 1},$
or \\(2) has the form $quq^{-1}$ for some $q\in Q_j$ and $u\in   F(Y_{j+1}),$
or \\(3) is of the form $q^{-1}uq$ for $q\in Q_j$ and $u\in   F(Y_{j}).$

We shall follow the tradition of calling state letters \label{qletter}{\em
$q$-letters} and tape letters \label{aletter}{\em $a$-letters}, even though we
shall use  other letters as state or tape letters.
Usually parts of the set $Q$ of state letters are denoted by capital letters. For example, a set $K$ would consist of
letters $k$ with various indices. Then we shall say that letters in $K$ are $k$-letters or $K$-letters.

If a group word $W$ over $Q\cup Y$
has the form $q_1u_1q_2u_2...q_s,$
and $q_i\in Q_{j(i)}^{\pm 1},$
$i=1,...,s$, $u_i$  are
group words in $Y$, then we shall say that the \label{basew}{\em base} of $W$ is 
$\base(W)\equiv Q_{j(1)}^{\pm 1}Q_{j(2)}^{\pm 1}...Q_{j(s)}^{\pm 1}$ Here $Q_i$ are just letters, denoting the parts of the set of state letters. Note that the base is not necessarily a reduced word, and the last equality is in the free semigroup.
The subword of $W$ between the $Q_{j(i)}^{\pm 1}$-letter and the $Q_{j(i+1)}^{\pm 1}$-letter
will be called a \label{sectorw}
$Q_{j(i)}^{\pm 1}Q_{j(i+1)}^{\pm 1}$-{\em sector} of $W$. A word can certainly have many $Q_{j(i)}^{\pm 1}Q_{j(i+1)}^{\pm 1}$-sectors.

For aesthetic reasons,  we shall substitute the capital names of parts of $Q$ by the corresponding small letters. For
example, if $t\in T, k\in K,...,$
we shall say that the base is $tk...$, that is the state letters in $W$ start
with a $t$-letter, followed by a $k$-letter, and so on. Usually instead of specifying the names of the parts of $Q$
and their order as in $Q=Q_0\sqcup Q_2\sqcup ... \sqcup Q_N$, we say that \label{standardb}{\em the standard} base
of the $S$-machine
is $Q_0...Q_N$ or $q_0...q_N$.

The $S$-machine also has a set of \label{rule}{\em rules} $\Theta$.  Every $\theta\in \Theta$
is assigned two sequences of
reduced words:
$[U_0,...,U_N]$,
$[V_0,...,V_N]$, and a subset $Y(\theta)=\cup
Y_i(\theta)$ of $Y$, where \label{Yit} $Y_i(\theta)\subseteq Y_i$.

The words $U_i, V_i$ satisfy the following restriction:

\begin{itemize}
\item[(**)]
For every $i=0,...,N$, the words $U_i$ and $V_i$ have the form
$$U_i\equiv v_iq_iu_{i+1}, \quad V_i\equiv v_{i}'q_{i}'u_{i+1}',$$ where $q_{i}, q_{i}'$ are state letters in
$Q_{i},$
$u_{i+1}$ and $u_{i+1}'$ are words in the alphabet $Y_{i+1}(\theta)^{\pm
1}$, $v_{i}$ and $v_{i}'$ are words in the alphabet
$Y_{i}(\theta)^{\pm 1}$.
\end{itemize}

The pair of words $U_i,V_i$ is called a \label{partr}{\em part} of the rule, and is denoted $U_i\to V_i$.

We will denote the rule $\theta$ by $[U_1\to
V_1,...,U_N\to V_N]$. This notation contains all the necessary
information about the rule except for the sets $Y_i(\theta)$. In
most cases it will be clear what these sets are. In the $S$-machines
used in this paper, the sets $Y_i(\theta)$ will be equal to
either $Y_i$ or $\emptyset$. By default $Y_i(\theta)=Y_i$. If $Y_{i+1}(\theta)=\emptyset$, then the corresponding
component $U_i\to V_i$ will be denoted \label{tool}$U_i\tool V_i$ and we shall say that the rule \label{locks}{\em locks} the
$Q_{i}Q_{i+1}$-sectors.
In that case we always assume that $U_i, V_i$ do not have tape letters to the right of the state letters, i.e.,
it has the form $v_iq_i\tool v_i'q_i'$. Similarly, these words have no tape letters to the
left of the state letters if the $Q_{i-1}Q_i$-sector is locked by the rule.

Every $S$-rule $\theta=[U_1\to V_1,...,U_s\to V_s]$ has an inverse
$\theta\iv=[V_1\to U_1,...,V_s\to U_s]$ which is also a rule of $\sss$; we set
$Y_i(\theta\iv)=Y_i(\theta)$. We always divide the set of rules
$\Theta$ of an $S$-machine into two disjoint parts, $\Theta^+$ and
$\Theta^-$ such that for every $\theta\in \Theta^+$, $\theta\iv\in
\Theta^-$ and for every $\theta\in\Theta^-$, $\theta\iv\in\Theta^+$ (in particular $\Theta\iv=\Theta$, that is any $S$-machine is symmetric).
The rules from $\Theta^+$ (resp. $\Theta^-$) are called \label{positiver}{\em
positive} (resp. {\em negative}).

 To \label{applicationr} apply an $S$-rule $\theta$ to an admissible word (\ref{admiss}) $W$
means to check if all tape letters of $W$ belong to the alphabet $Y(\theta)$ 
and then, if $W$ satisfies
 this condition, to replace simultaneously subwords $U_i^{\pm 1}$ by subwords $V_i^{\pm 1}$ ($i=1,\dots,k+1$) and to  trim a few first and last $a$-letters (to obtain an admissible
 word starting and ending with $q$-letters).
 This replacement is allowed to perform in the form $q_i^{\pm 1} \to
 (v'_{i-1}v_i^{-1}q'_iu_i^{-1}u'_{i})^{\pm 1}$ followed by the reducing of
 the resulted word. 
 The following convention is important in the definition of
$S$-machine: {\it After every application of a rewriting rule, the word is automatically 
reduced. The reducing is not considered a separate step of an $S$-machine.}

 If a rule $\theta$ is applicable to an admissible word $W$ (i.e., $W$ belongs to the \label{domain} domain
 of $\theta$) then the word $W$ is called
 $\theta$-admissible. The definitions of computation, its history, input admissible words,
 are similar to those for Turing machines. 
Similarly, we sometimes choose a distinguished {\em stop word} \label{W0} $W_0$
from the free group $F(Q)$. It will always have the standard base (and no $a$-letters).
We say that a word $W\in F(Q\cup Y)$ is \label{acceptedws}{\em accepted} if there exists a
computation connecting this word and $W_0$.
 
 \subsection{Modifying the rules of $S$-machines}\label{modif}

We shall need the following properties of our $S$-machines. All S-machines
that appears in this paper,
except for $M_2$, satisfy Property \ref{one} (1) below, and $\tilde M_2$ satisfies Property \ref{one} (2).

\begin{property} \label{one} (1) In every part
$v_{i}q_iu_{i+1}\to v_{i}'q_i'u_{i+1}'$  we have $||v_{i-1}||+||v'_{i-1}||\le 1$ and $||u_i||+||u'_i||\le 1$ (see the notation in (**) above).

(2) For every rule, we have $\sum_i (||v_i||+||v'_i||+||u_i||+||u'_i||) \le 1.$ 
\end{property}

Suppose that Property \ref{one} (1) is not satisfied. For example, suppose that a positive rule $\theta$ of an $S$-machine $\sss$ has the $i$-th part of the form $v_{i-1}aq_iu_i\to v_{i-1}'q_i'u_i',$ where $u_{i-1}, v_i, u_{i-1}', v_i'$ are words in the appropriate parts of the alphabet of $a$-letters, $v_{i-1}$ is not empty, $a$ is a $a$-letter, $q_i,q_i'$ are $q$-letters (a very similar procedure can be done in all other cases). Then we create a new $S$-machine $\tilde\sss$ with the same standard base and the same $a$-letters as $\sss$. In order to make $\tilde\sss$, we add one new ({\em auxiliary}) $q$-letter $\tilde q_i$ to each part of the set of $q$-letters, and replace the rule $\theta$ by two rules $\theta'$ and $\theta''$. The rule $\theta'$ is obtained from $\theta$ by replacing the part $v_{i-1}aq_iu_i\to v_{i-1}'q_i'u_i'$ by $aq_iu_i\to \tilde q_iu_i'$, and all other parts $U_j\to V_j$ by $U_j\to \tilde q_j$ (here $\tilde q_j$ is the auxiliary $q$-letter in the corresponding part of the set of $q$-letters). The rule $\theta''$ is obtained from $\theta$ by replacing the part $v_{i-1}aq_iu_i\to v_{i-1}'q_i'u_i'$ by $v_{i-1}\tilde q_i\to v'_{i-1}q_i'$, and all other parts $U_j\to V_j$ by $\tilde q_j\to V_j$.

The key property of the new $S$-machine is in the following obvious lemma.

\begin{lemma}\label{tilde} There is a one-to-one correspondence between computations $w_0\to...\to w_n$ of $\tilde\sss$ (with any base) such that $w_0$, $w_n$ do not have auxiliary $q$-letters
and computations of $\sss$ connecting the same words. For every history $H$ of such computation of $\sss$, the corresponding history of computation of $\tilde\sss$ is obtained from $H$ by replacing every occurrence of the rule $\theta$ by the 2-letter word $\theta'\theta''$.
\end{lemma}

Note that the sum of lengths of words in all parts of $\theta'$ (resp. $\theta''$) in $\tilde\sss$ is smaller than the similar sum for $\theta$. Clearly, applying this transformation to an $\sss$-machine $\sss$ several times, we obtain a new $\sss$-machine satisfying Property \ref{one} (1). Similarly,
one can obtain Property \ref{one} (2). Thus Lemma \ref{tilde} implies the following

\begin{lemma}\label{tilde1} For every $S$-machine $\sss$ there exists an $S$-machine $\tilde S$ with the same standard base, the same set of $a$-letters, and some new, {\em auxiliary}, $q$-letters such that $\tilde\sss$ satisfies Property \ref{one} (2), and there exists a one-to-one correspondence between computations $w_0\to...\to w_n$ of $\tilde\sss$ (with any base) such that $w_0$, $w_n$ do not have auxiliary $q$-letters
and computations of $\sss$ connecting the same words. For every history $H$ of $\sss$, the corresponding history of computation of $\tilde\sss$ is obtained from $H$ by replacing every occurrence of the rule $\theta$ by the word $\phi(\theta)$ such that all rules in $\phi(\theta)$ are different, and $\phi(\theta), \phi(\theta')$ do not have common rules provided $\theta\ne \theta'$.
\end{lemma}

\subsection{Some general properties of $S$-machines}

\label{gpsm}
Note that the base of an admissible word is not always a reduced word. But we have the following immediate corollary of the definition of admissible word.

\begin{lemma}\label{qqiv}
If the $i$-th component of the rule $\theta$ has the form $q_i\tool q_i',$
i.e. $Y_{i+1}(\theta)=\emptyset$, then the
base of any admissible for $\theta$ word cannot have subwords $Q_iQ_i\iv$ or $Q_{i+1}^{-1}Q_{i+1}.$
\end{lemma}

In this paper we are often using copies of words.
If $W$ is a word and $A$ is an alphabet, then to obtain a \label{copyw}{\em copy}
of $W$ in the alphabet $A$ we substitute letters from $A$ for letters in $W$ so that different letters from $A$
substitute for different letters. Note that if $U'$ and $V'$ are copies of $U$ and $V$ respectively corresponding to
the same substitution, and $U'\equiv V'$, then $U\equiv V.$

The following lemma is obvious.

\begin{lemma}\label{gen1}Suppose that the base of an admissible word $W$ is $Q_{i}Q_{i+1}$. Suppose that each rule of
a reduced computation starting with $W\equiv q_iuq_{i+1}$ and ending with $W'\equiv q_i'u'q_{i+1}'$ multiplies the
$Q_iQ_{i+1}$-sector by a letter on the left (resp. right).
And suppose that different rules multiply that sector by
different letters.
Then the history of computation is a copy
of the reduced form of the word $u'u\iv$ read from right to left
(resp. of the word $u\iv u'$ read from left to right). In particular,
 if $u\equiv u'$, then the computation is empty.
\end{lemma}

\begin{lemma}\label{gen1.5} Let $W_0\to\dots\to W_t$ be a sequence
of transformations of reduced words, where $W_i$
is a conjugate of $W_{i-1}$ ($i=1,...,t$) by  a letter, and $H$ - a
product of these letters, i.e. $W_t=H^{-1}W_0H.$ Then $H$ is equal
to a reduced product $H_1H_2^k H_3$, where $k\ge0, ||H_2||\le \min(||W_0||,
||W_t||)$,  $||H_1||\le ||W_0||/2,$ and $||H_3||\le ||W_t||/2.$
\end{lemma}

\proof One may assume that the consecutive transformations
$W_{i-1}\to W_i\to W_{i+1}$ are not mutual inverse. If
$||W_{i-1}||<||W_i||,$ then $||W_i||-||W_{i-1}||=2,$ and moreover,
$||W_{i'}||-||W_{i'-1}||=2,$ for every $i'\ge i.$ Similar observation is
true for inverse transformations $W_t\to\dots\to W_0.$ It follows
that there exist subscripts $i,j$ such that $0\le i\le j\le t,$ and
each of the transformations $W_0\to\dots\to W_i$ decreases the
length by $2,$ while $||W_i||=||W_{i+1}||=\dots=||W_j||,$  and each of the
transformations $W_j\to\dots\to W_t$ increases the length by $2.$
Thus, we have $W_i = (H'_1)^{-1}W_0H'_1,$ where
$||W_0||-||W_i||=2||H'_1||$ and $W_t = (H'_3)^{-1}W_jH'_3,$  where
$||W_t||-||W_j||= 2||H'_3||.$

We also have $j-i$ one-letter cyclic shifts $W_i\to\dots\to W_j,$
and this procedure is periodic with period $||W_i||,$ whence the
middle part of the conjugating word $H$ must be of the form $\bar
H_2^k\bar H,$ with $k\ge 0$ and $||\bar H||<||\bar H_2||=
||W_i||\le\min(||W_0||,||W_t||)$, where $\bar H$ - is a prefix of the word
$\bar H_2$. Replacing $\bar H_2$ by a cyclic permutation $H_2$ one
rewrites the same middle part as $H'H_2^k H''$, where $||H_2||=||\bar
H_2||$ and $||H'||, ||H''||\le\frac12 (||\bar H||+1) \le ||W_i||/2.$ Finally,
we set $H_1\equiv H'_1H'$ and $H_3\equiv H''H'_3,$ to obtain the required
factorization of $H$ with $||H_1||\le \frac12 (||W_0||-||W_i||)+||W_i||/2 =
||W_0||/2$ and also $||H_3||\le ||W_t||/2.$\endproof

Lemma \ref{gen1.5} immediately implies

\begin{lemma} \label{gen2}
Suppose that the base of an admissible word $W$ is $Q_{i}Q_{i}\iv$
(resp., $Q_i\iv Q_i$). Suppose that each rule $\theta$ of a reduced
computation starting with $W\equiv q_iuq_i\iv$ (resp., $q_i\iv uq_i$), where
$u\ne 1$, and ending with $W'\equiv q_i'u'(q_i')\iv$ (resp., $W'\equiv (q_i')\iv
u'q_i')$
has a component $q_i\to a_\theta
q_i'b_\theta,$ where $b_{\theta}$ (resp., $a_{\theta}$) is a letter, and  for different $\theta$-s the $b_\theta$-s
(resp., $a_\theta$-s) are different. Then the history of the computation
has the form $H_1H_2^kH_3,$ where $k\ge0$, $||H_2||\le \min(||u||, ||u'||),$ $||H_1||\le ||u||/2,$ 
and $||H_3||\le ||u'||/2.$
\end{lemma}

\subsection{Turing machines as $S$-machines}

Every symmetric Turing machine $M$ 
satisfying Condition (d) of Lemma \ref{m0m1} 
can be viewed as an $S$-machine $SM$ \cite[Page 372]{SBR}, such that the positive (negative) commands of $M$ are the positive (negative) rules of $SM$. More precisely, we consider the $\cee$ and $\dol$ symbols as state letters (hence the set of state letters has 2 more parts for each tape of the Turing machine). A part of a command 
of the form $uq_iv\to u'_iq'_iv'_i,$ where
$u$ and $v$ are tape letters or empty words, is replaced by
$$ [\alpha_i\to\alpha_i, u_iq_iv_i \to u'_iq'_iv'_i, \omega_i\tool \omega_i],$$
a part  of the form $\alpha_iq_iv_i\to \alpha_iq'_iv'_i$ 
is replaced by

$$ [\alpha_i\tool\alpha_i, q_iv_i \to q'_iv'_i, \omega_i\tool \omega_i],$$
a part of the form $u_iq_i\omega_i\to u'_iq'_i\omega_i$ is replaced by

$$ [\alpha_i\to\alpha_i, u_iq_i \tool u'_iq'_i, \omega_i\tool \omega_i],$$
and a part of the form $\alpha_iq_i\omega_i\to \alpha_iq'_i\omega_i$ is replaced by

$$ [\alpha_i\tool\alpha_i, q_i \tool q'_i, \omega_i\tool \omega_i].$$

The language recognized by $SM$ is in general much bigger than the 
   language recognized by $M$ since $M$ works with a {\it positive} tape alphabet only. 
   Nevertheless the following property statement holds:

\begin{lemma}(compare with \cite[Proposition 4.1]{SBR})
\label{ss} Let $M$ be a symmetric Turing machine satisfying the conditions of Lemma \ref{mach23} (i.e. the symmetrization of some deterministic Turing machine satisfying conditions of Lemma \ref{m0m1}). 
Let $W_1\to W_2\to\dots\to W_k$ be a computation of the S-machine $SM$ with the standard base consisting of positive words. Then it is a computation of the Turing machine $M$ (with the same history).
\end{lemma}

\proof Indeed, every positive admissible word $W$ of $SM$ with the standard base is a configuration of the Turing machine $M$. If a rule  $\theta$ of $SM$ satisfying Property (c)
of Lemma \ref{m0m1} or
its inverse applies
to this $W$ and the word $W\cdot \theta$ is positive, then, obviously, the command $\theta$
of $M$ applies to $W$ and the result of the application is the same (here we essentially use the fact that the rule $\theta$
or its inverse inserts (deletes) at most one letter).
This immediately implies the statement of the lemma.
\endproof

\section{The $S$-machine}

We turn to the proof of Theorem \ref{mainth}. From now \label{M0} $M_0$ is the deterministic
Turing machine recognizing language $X$ from Theorem \ref{re}, the machine \label{M1} $M_1$ is constructed as in 
Lemma \ref{m0m1} and recognizes the  language \label{X1} $X_1,$ where $X_1=X.$  We keep the same
notation $M_1$ for the symmetrization of $M_1$ given by Lemma \ref{mach23}. 
Note that by claim (e) of that lemma, the machine $M_1$ has infinitely many
$h_{\alpha}$-good numbers for every $\alpha>0,$ where the functions $h_{\alpha}$ are defined in
Theorem \ref{re}. 
First, we need to construct a new $S$-machine which inherits  important
properties of the Turing machine $M_1.$

As in Section \ref{gpsm},  we can view $M_1$ as an $S$-machine.
We shall denote that $S$-machine by the same letter \label{M1S} $M_1$.

Let $Q_0...Q_N$ be the standard base of $M_1$, let the components of the alphabet of $a$-letters be $Y_1,...,Y_N$ (letters from $Y_i$ are in the $Q_{i-1}Q_i$-sectors of admissible words with the standard base).

\subsection{The machine $M_1\circ Z$}

Let $A$ be a finite set of letters. Let the sets $A_1, A_2$ be copies of
$A$. It will be convenient to denote $A$ by $A_0$. For every
letter $a_0\in A_0$ let $a_1, a_2$ denote its copies in $A_1, A_2$.

As in \cite{OS}, consider
the following auxiliary ``adding" $S$-machine \label{ZA} $Z(A)$.

Its set of state letters is $P_1\cup P_2\cup P_3,$ where $$P_1=\{L\},
P_2=\{p(1),p(2), p(3)\}, P_3=\{R\}.$$ The set of tape letters is
$Y_1\cup Y_2,$ where $Y_1=A_0\cup A_1$ and $Y_2=A_2$.

The machine $Z(A)$ has the following positive rules
(there $a$ is an
arbitrary letter from $A$). The comments explain the meanings of
these rules.

\begin{itemize}

\item $r_1(a)=[L\to L, p(1)\to a_1\iv p(1)a_2, R\to R]$.
\me

{\em Comment.} The state letter $p(1)$ moves left searching for a
letter from $A_0$ and replacing letters from $A_1$ by their copies
in $A_2$.

\me

\item $r_{12}(a)=[L\to L, p(1)\to a_0\iv a_1p(2), R\to R]$.

{\em Comment.} When the first letter $a_0$ of $A_0$ is found, it is
replaced by $a_1$, and $p(1)$ turns into $p(2)$.

\me

\item $r_2(a)=[L\to L, p(2)\to a_0p(2)a_2\iv, R\to R]$.

\me

{\em Comment.} The state letter $p(2)$ moves toward $R$.

\me

\item $r_{21}=[L\to L, p(2)\tool p(1), R\to R]$, $Y_1(r_{21})=Y_1,
Y_2(r_{21})=\emptyset$.

\me

{\em Comment.} $p(2)$ and $R$ meet, the cycle starts again.

\me

\item $r_{13}=[L\tool L, p(1)\to p(3), R\to R]$, $Y_1(r_{13})=\emptyset,
Y_2(r_{13})=A_2$.

\me

{\em Comment.} If $p(1)$ never finds a letter from $A_0$, the cycle
ends, $p(1)$ turns into $p(3)$; $p$ and $L$ must stay next to each
other in order for this rule to be executable.

\item $r_{3}(a)=[L\to L, p(3)\to a_0p(3)a_2\iv, R\to R]$,
$Y_1(r_3(a))=A_0, Y_2(r_3(a))=A_2$

 \me

{\em Comment.} The letter $p(3)$ returns to $R$.

\end{itemize}

For every letter $a\in A$ we set $r_i(a\iv)=r_i(a)\iv$ ($i=1,2,3$).

The following Lemmas from \cite{OS} contain the main properties of $Z(A)$
used later.

If $u\equiv a_1...a_m$ is a word, $a_i$ are letters, then we set $r_3(u)\equiv r_3(a_1)r_3(a_2)...r_3(a_m)$, $r_2(u)\equiv r_2(a_1)r_2(a_2)...r_2(a_m)$, $r_1(u)\equiv r_1(a_m)...r_1(a_2)r_1(a_1).$

\begin{lemma} \label{lm569}(\cite[Lemma 3.18]{OS}) Suppose that an admissible word $W$ of $Z(A)$ has the form
$LupvR,$ where $u,v$ are words in $(A_0\cup
A_1\cup A_2)^{\pm 1}$. Let $ W\cdot \theta\equiv Lu'p'v'R$. Then the projections of $uv$ and $u'v'$ onto $A$ are freely equal.
\end{lemma}

\begin{lemma}\label{lm90}(Follows from the proof of \cite[Lemma 3.25]{OS}) Let $W$ be an admissible word of $Z(A)$ with $\base(W)\equiv LpR$.
Then for every reduced
computation $W\equiv W_0\to W_1\to \dots\to W_t\equiv W\cdot H$
of the $S$-machine $Z(A)$:

\begin{enumerate}
\item $||W_i||\le \max(||W_0||, ||W\cdot H||)$, $i=0,...,t$,

\item If $W\equiv LupR,$ where $p=p(1)$ (resp. $p=p(3)$), $u$ is positive, then there exists a computation starting with $W$ and ending with $Lup(3)R$
(resp. $Lup(1)R$). Moreover if $u$ is any word in $a$-letters, and for some history of computation $H$,  $W\cdot H$ contains $p(3)R$ (resp. $p(1)R$) and all $a$-letters in
$W, W\cdot H$ are from $A_0^{\pm 1}$, then the length of $H$ is
between $2^{||u||}$ and $6\cdot 2^{||u||}$, $u$ and all words in the computation $W\to...\to W\cdot H$ are positive,
all words in that computation have the same length, and $H$ is uniquely determined by $u$. That computation (resp. its inverse) has the history of the following form
$$
D(u)\equiv E(u)r_{13}r_3(u),
$$
where $E(u)$ is defined by induction: $E(\emptyset)\equiv \emptyset$ and if $u\equiv a_1u'$, then
$$
 E(u)\equiv E(u')r_{12}(a_1)r_2(u')E(u')r_1(a_1)
$$
\end{enumerate}
\end{lemma}

\begin{lemma}\label{lm89}(The first part of the lemma is \cite[Lemma 3.21]{OS}) For every
admissible word $W$ of $Z(A)$ with $\base(W)\equiv LpR$,
every rule $\theta$ applicable to $W$, and every natural number
$t>1$, there is at most one reduced computation
$W\to_\theta W_1\to ...\to
W_t$ of length $t,$ where the lengths of the words are all the same. (In fact from the proof of \cite[Lemma 3.21]{OS}, it immediately follows that the history of that computation is a subword of
$D(u)^{\pm 1}$ for some $u$).
\end{lemma}

\begin{lemma} \label{lm901}(\cite[Lemma 3.27]{OS}) Suppose $W$ is an admissible word of $Z(A)$ with $\base(W)\equiv LpR$. Suppose that $W\cdot H$ exists for some reduced history $H$. Suppose that
both $W$ and
$W\cdot H$ contain $p(1)R$ (resp. $p(3)R$) and all $a$-letters in
$W$, $W\cdot H$ are from $A_0$. Then $H$ is empty.
\end{lemma}

\begin{lemma} (\cite[Lemma 3.24]{OS}) \label{lm93} Let $W\equiv LvpuR$ and $\base(W)\equiv LpR$. Suppose that
$||W\cdot \theta||>||W||$. Then for every reduced computation $W\to_\theta
W_1\to W_2\to...\to W\cdot H,$
we have $||W_i||>||W||$ for every $i\ge 1$.
\end{lemma}

We define the S-machine $M_2$ as the composition
$M_1\circ Z$ (this operation is defined in \cite{OS}, see also below).

Essentially we insert a $p$-letter between any two
consecutive $q$-letters in admissible words of $M_1$ with standard base, and treat any
subword $q_i...p_{i+1}...q_{i+1}$ as an admissible word for $Z(A)$ (that is $q_i$ plays the role of $L$ and $q_{i+1}$ plays
the role of $R$). The only differences with the construction in \cite{OS} are that, for every state letter $q$, we keep
the sets of $a$-letters that can appear to the left  and to the right of $q$ disjoint, and after application of a main rule, not only the state letters of copies of $Z(A)$ remember the main rule but also the state letters coming from $M_1$ remember that rule. These changes do not affect the proofs of statements in \cite{OS} that we are going to use.

Let us describe \label{M2} $M_2$ in details. First, for every $i=1,...,N$, we make three copies of the alphabet
$Y_i$ of $M_1$ ($i=1,...,N$): $Y_{i,0}=Y_i$,  $Y_{i,1}$, $Y_{i,2}$. Let $\Theta$
be the set of positive commands of $M_1$ (viewed as rules of an $S$-machine). The
set of state letters of the new machine is $$Q_0'\cup P_1\cup
Q_1'\cup P_2\cup...\cup P_{N}\cup Q_N',$$
where $P_i=\{p^{(i)}, p^{(1,i)}, p^{(0,i)},
p_1^{(\theta,i)}, p_2^{(\theta,i)},  p_3^{(\theta,i)}\mid
\theta\in\Theta\}$
$i=1,...,N$, $Q_i'=Q_i\sqcup (Q_i\times \Theta),$
where $Q_0\sqcup Q_1\sqcup ... Q_N$ is the set of state letters of
$M_1$.
We shall denote a pair $(q, \theta)$ from $Q_i\times\Theta$ in $Q_i'$
by $q^{(\theta,i)}.$
Thus every ``old" state letter of $M_1$ gets ``multiple" copies
indexed by positive rules of $M_1$, and the state letters of various
copies of $Z$ have upper indices corresponding to the positive rules
of $M_1$ and the number of sector where the machine is inserted. The
$Q_0'P_1$-sector of an admissible word with the standard base will
be called the {\em input} sector of that word.

The set of tape letters is $$\bar Y=(Y_{1,0}\sqcup Y_{1,1})\sqcup
Y_{1,2}\sqcup (Y_{2,0}\sqcup Y_{2,1})\sqcup Y_{2,2}\sqcup...\sqcup
(Y_{N,0}\sqcup Y_{N,1})\sqcup Y_{N,2};$$
the components of this union
will be denoted by $\bar Y_1,...,\bar Y_{2N}$
We shall call the new (second) indices of tape letters
the \label{M2ind} $M_2$-{\em indices} of these letters.

The set of positive rules $\bar\Theta$ of $M_1\circ Z$ is a union
of the set of suitably modified positive rules of $M_1$  and
$|\Theta|N$
copies $Z^{(\theta,i)}$ ($\theta\in\Theta,
i=1,...,N$) of positive rules of the machine $Z(Y_i)$ (also suitably
modified).

More precisely, every rule $\theta\in \Theta$
of the form
$$[q_0u_1\to q_0'u_1', v_1q_1u_2\to
v_1'q_1'u_2',...,v_{N}q_N\to v_{N}'q_{N}'],$$
where $q_i, q_i'\in
Q_i$, $u_i$ and $v_i$ are words in $Y$, is replaced
by $$\bar\theta=\left[\begin{array}{l}q_0u_1\to (q')^{(\theta,0)} u_1', v_1p^{(1)}\tool
v_1'p_1^{(\theta,1)}, q_1u_2\to (q')^{(\theta,1)}u_2', ...,\\
 v_{N}p^{(N)}\tool v_{N}'p_1^{(\theta,N)}, q_{N}\to
 (q')^{(\theta,N)}\end{array}\right]$$
 with $\bar Y_{2i-1}(\bar\theta)=Y_{i,0}(\theta)$ and $\bar Y_{2i}(\bar\theta)=\emptyset.$
 If $\theta=\theta_{start}$ is the unique start rule of $M_1$ then all $p^{(i)}$-s in the above definition of $\bar\theta$ must be replaced by the special start letters $p^{(1,i)}$-s.

Thus each modified rule from $\Theta$
turns on $N$ copies of the
machine $Z(A)$ (for different $A$'s). The rule $\bar\theta$ will be called \label{bart}{\em the rule of $M_2$ corresponding to the
rule $\theta$ of $M_1$.}

Each\label{Zti} machine $Z^{(\theta,i)}$ is a copy of the machine $Z(Y_{i}(\theta)),$ where
every rule $\tau=[U_1\to V_1, U_2\to V_2, U_3\to V_3]$ is replaced
by the rule of the form
$$\bar\tau_i(\theta)=\left[
\begin{array}{l}\bar U_1\to \bar V_1, \bar U_2\to \bar V_2, \bar U_3\to \bar V_3,\\
(q')^{(\theta,j)}\to (q')^{(\theta,j)}, p_3^{(\theta,j)}\tool  p_3^{(\theta, j)},
j=1,...,i-1,\\
p_1^{(\theta,s)}\tool p_1^{(\theta,s)}, (q')^{(\theta,s+1)}\to (q')^{(\theta,s+1)},
s=i+1,...,N-1\end{array}\right],$$ where $\bar U_1, \bar U_2, \bar
U_3, \bar V_1, \bar V_2, \bar V_3$ are obtained from $U_1, U_2, U_3,
V_1, V_2, V_3$, respectively, by replacing $p(j)$ with
$p_j^{(\theta,i)}$, $L$ with $(q')^{(\theta,i)},$ and $R$ with $(q')^{(\theta,i+1)}$, and for
$s\ne i$, $\bar Y_{2s-1}(\bar\tau_i(\theta))=Y_{s,0}(\theta).$
Thus while the machine $Z^{(\theta,i)}$ works all other machines
$Z^{(\theta,j)}$, $j\ne i$
must stay idle (their state letters do not
change and do not move away from the corresponding $q$-letters).
After the machine $Z^{(\theta,i)}$ finishes (i.e. the state letter
$p_{3}^{(\theta,i)}$ appears next to $(q')^{(\theta,i)}$), the next machine
$Z^{(\theta,i+1)}$ starts working.

In addition, we need the following {\em transition}
rule \label{zetat} $\zeta(\theta)$ that removes $\theta$ from all state letters, and turns all $p_3^{(\theta,j)}$ back into $p^{(j)}$:

$$[(q')^{(\theta,i)}\to q'_i, p_3^{(\theta,j)}\tool
p^{(j)}, i=0,...,N, j=1,...,N].$$

 If $\theta$ is the unique accept rule of $M_1$ then all $p^{(j)}$-s in the above definition of $\zeta(\theta)$ must be replaced by the special (accept) letters $p^{(0,i)}$-s.

\begin{lemma}\label{lmb} Let $H$ be  the history of a reduced computation $W\to...\to W\cdot H$ of $M_2$ with the standard base, of the form $\bar\theta H'\zeta(\theta)$, where $\theta$ is a positive rule of the $S$-machine $M_1$,  $H'$ does not contain rules corresponding to rules of $M_1$ and occurrences of $\zeta(\theta)^{\pm 1}$. Let  $W\cdot \bar \theta= q^{(\theta,0)}u_1p^{(\theta,1)}q^{(\theta,1)}u_2...u_Np^{(\theta,N)}q^{(\theta,N)}$
for some words $u_i$ in $Y_{i,0}$
Then $H'\equiv H_1H_2...H_N,$ where each $H_s$ is the computation
of the machine $Z^{(\theta,s)}$ whose history is a copy of $D(u_s)$
(described in Lemma \ref{lm90}, part 2), and all words in the
computation $W\cdot \bar\theta\to...\to W\cdot H$ are positive. We
shall denote $H$ by \label{Pi12} $\Pi_{1,2}(\theta, W)$. That computation is completely determined by its first (last) word and $\theta$.
\end{lemma}

\proof By assumption, $H'$ consists of the rules of various $S$-machines $Z^{(\theta,i)}$ since only rules $\bar\theta^{-1}$ and $\zeta(\theta)$ can remove the $\theta$-index of state letters.
Since the word $W'=W\cdot (\bar\theta H')$ is in the domain of $\zeta(\theta)$, all $p$-letters in $W'$ have the form $p_3^{(\theta,i)}$. Since rules from $Z^{(\theta,i)}$ can apply to an admissible word with the standard base only if $p_j$-letters ($j\ne i$) stay next to the left of the corresponding (copies of) state letters of $M_1$, and have the form $p_3^{(\theta,j)}$, if $j<i$, and the form $p_1^{(\theta,j)}$ if $j>i$, we can conclude that $H'\equiv H_1H_2...H_m,$ where each $H_s$ is a non-empty history of the computation of some $Z^{(\theta,j(s))}$ such that $j(s+1)=j(s)\pm 1$, $j(1)=1$, and each $m$ between 1 and $N$ occurs as $j(s)$ for some $s$. Note that if $j(s+1)=j(s)-1$, then there must be $s'>1$ such that $j(s'-1)=j(s'+1)=j(s')-1$. But then the subcomputation of the $S$-machine $Z^{(\theta,j(s))}$ with history $H_s$ starts and ends with the $p$-letter of the form $p_{1}^{(\theta,j(s))}$. By Lemma \ref{lm901}, then $H_s$ is empty, a contradiction. Hence $j(s+1)=j(s)+1$ for every $s$, which implies that $m=N$, and $j(s)=s$ for every $s$. By Lemma \ref{lm90}, part 2, each $H_s$ is uniquely determined by the word $u_s$ (and $\theta$) and is equal to a copy of $D(u_s)$ defined in Lemma \ref{lm90}, part 2. This implies the uniqueness of $H'.$ The fact that all words in the computation $W\cdot \bar\theta \to...\to W\cdot H$ are positive
follows from Lemma \ref{lm90}, part 2.
\endproof

\begin{lemma}\label{lma}
Let $H$ be  the history of a reduced computation $W\to...\to W\cdot H$ of $M_2$ with the standard base, $H\equiv \zeta(\theta_1)^{\epsilon_1} H'\zeta(\theta_2)^{\epsilon_2}$, where $\theta_1,\theta_2$ are positive rules of the $S$-machine $M_1$,  $H'$ does not contain rules corresponding to rules of $M_1$ and occurrences of $\zeta(\theta)^{\pm 1}$, $\epsilon_1,\epsilon_2=\pm1, i=1,2$. Then $H'$ is empty, $\epsilon_1=1,$ and $\epsilon_2=-1.$
\end{lemma}

\proof At first, let us  prove that $H'$ is empty. If $\epsilon_1=1$, then the $p$-letters in $W\cdot \zeta(\theta_1)$ have the form $p^{(j)}$, and no rules from any $Z^{(\theta,s)}$ apply to words with such $p$-letters, hence $H'$ is empty (because by assumption it can contain only rules of various $Z^{(\theta,j)}$). Thus we can assume that $\epsilon_1=-1$.
Similarly $\epsilon_2=1$. As in the proof of Lemma \ref{lmb}, $H'\equiv H_1H_2...H_m,$ where each $H_i$ is a non-empty history of computation of some $Z^{(\theta_1, j(s))}$ such that $j(s+1)=j(s)\pm 1$, $j(1)=N.$
 Note that the $p$-letters in $W\cdot \zeta(\theta_1)\iv$ and in $W\cdot \zeta(\theta_1)\iv H'$ are of the form $p_3^{(\theta_1,j)}$. Therefore for some $s$, we must have $j(s-1)=j(s+1)=j(s)+1.$
As in the proof of Lemma \ref{lmb}, this implies that $H_s$ is empty, a contradiction.

Since $H'$ is empty, the $p$-letters in $W\cdot \zeta(\theta_1)$ have no $\theta$-indices
because otherwise they have to be equal to both $\theta_1$ and $\theta_2$, and $H\equiv
\zeta(\theta_1)^{-1} \zeta(\theta_1)$ would not be reduced. Therefore $\epsilon_1=1,$ and $\epsilon_2=-1.$
\endproof

\begin{lemma}\label{lme}
Let $H$ be  the history of a reduced computation $W\to...\to W\cdot H$ of $M_2$ with the standard base, $H\equiv \bar\theta_1\iv H'\bar\theta_2$, where $\theta_1,\theta_2$ are positive rules of the $S$-machine $M_1$,  $H'$ does not contain rules corresponding to the rules of $M_1$. Then $H'$ is empty.
\end{lemma}

\proof
Indeed if $H'$ is not empty, it must start with $\zeta(\theta_3)\iv$, and end with $\zeta(\theta_4)$,
for some positive $\theta_3, \theta_4$, and then $H'$ would  have a subword of the form $\zeta(\theta')\iv H''\zeta(\theta'')$ satisfying to the assumptions of Lemma \ref{lma},
which contradicts with the statement of Lemma \ref{lma}.
\endproof

\begin{lemma} \label{lmc}
Let $H$ be  the history of a reduced non-empty computation $W\to...\to W\cdot H$ of $M_2$ with the standard base, $H\equiv \bar\theta_1 H'\bar\theta_2\iv$, where $\theta_1,\theta_2$ are positive rules of the $S$-machine $M_1$, and let $H'$  contain no rules corresponding to rules of $M_1$. Then $H\equiv \Pi_{1,2}(\theta_1,W)\Pi_{1,2}(\theta_2,W\cdot H)\iv$, all words in that computation except possibly the first and the last ones are positive,  and $\theta_1\ne\theta_2$.
\end{lemma}

\proof If $H$ does not contains the rule $\zeta(\theta_1)$, we get that $H'$ is empty as in the proof of Lemma \ref{lma}. 
Note that the $p$-letters of the admissible words in the domain of $\bar\theta_1\iv$ (of
$\bar\theta_2\iv$) have the form $p_1^{(\theta_1,i)}$ (resp., $p_1^{(\theta_2,i)}$). It follows that $\theta_1\equiv\theta_2,$ and so the history $H$ is not reduced, a contradiction.

Hence we can assume that $H$ contains $\zeta(\theta_1)$. Then the next rule in $H$ must be $\zeta(\theta_3)\iv$  for some $\theta_3$ because only rules of the form $\bar\theta$ and $\zeta(\theta)\iv$ for positive $\theta$  are applicable to admissible words in the range
 of  $\zeta(\theta_1)$ and $H'$ does not contain rules corresponding to rules of $M_1$. After application of $\zeta(\theta_3)\iv$, all $p$-letters in the admissible word have the form $p_3^{(\theta_3,i)}$. Recall that the $p$-letters of the admissible words in the domain of $\bar\theta_2\iv$ have the form $p_1^{(\theta_2,i)}$.
Thus  if $\theta_3\ne\theta_2$, the word $H'$ must contain a subword $\zeta(\theta_3)\iv H''\zeta(\theta_3),$ where $H''$ does not contain rules of the form $\zeta(\theta)$ and their inverses. By Lemma \ref{lma}, $H''$ is empty, and the computation is not reduced, a contradiction.  Hence $H'$ has the form $H''\zeta(\theta_1)\zeta(\theta_2)\iv H''',$ where $H'', H'''$ do not contain rules of the form $\bar\theta^{\pm 1}$ and rules of the form $\zeta(\theta)^{\pm 1}$ by Lemma \ref{lma}. Applying now Lemma \ref{lmb} to the  computation with history $\bar\theta_1H''\zeta(\theta_1)$ and the first word $W$, and to the computation with history $\bar\theta_2(H''')\iv\zeta(\theta_2)$ and the first word $W\cdot H$, we obtain the desired equality $H\equiv \Pi_{1,2}(\theta_1,W)\Pi_{1,2}(\theta_2,W\cdot H)\iv$
and the fact that all words in that computation are positive except possibly the first and the last words.
Now if $\theta_1=\theta_2$, then $H'$ must be empty (since the computation is reduced and $H'$ is a product of two mutually inverse
words in that case by the uniqueness statement of Lemma \ref{lmb}) (b), and so the computation is empty, a contradiction.
\endproof

For every admissible word $W$ of $M_2$ with the standard base,
let \label{pi21} $\pi_{2,1}(W)$ be the word obtained by removing state $p$-letters, removing $\theta$-indices 
(if any exist) of other state letters, and removing the $M_2$-indices
of $a$-letters, and reducing the resulting word. We
obtain a word in the alphabet of state and tape letters of $M_1$.

With every admissible word $W$ of the $S$-machine $M_1$ with the standard base we associate  the admissible word \label{pi12} $\pi_{1,2}(W)$ of $M_2$ by inserting the
state letters $p^{(j)}$-s next to the left
of $q_j$-s, and replacing every $a$-letter $a$ by $a_0$. (We insert the special letters
$p^{(1,j)}$-s instead of $p^{(j)}$-s if the word $W$ is admissible for the unique start command of $M_1.$) Let $W_0$ be the stop word of $M_1$ (considered as an $S$-machine). It exists because the Turing machine $M_1$ is recognizing. We call  the word $\pi_{1,2}(W_0)$ the {\em stop word} of $M_2$.

Note that we have

\begin{equation}\label{p21}
\pi_{2,1}(\pi_{1,2}(W))\equiv W.
\end{equation}

For every input configuration
$W$ of the $S$-machine $M_1,$ we call $\pi_{1,2}(W)$ an {\em input word} of $M_2$. Note that an input word of
$M_2$ has the standard base and all sectors except the $q_0p_1$-sector are empty.

For every rule $\theta'$ of $M_2$, if $(\theta')^{\pm 1}$ corresponds to a positive rule $\theta$ (i.e. if $\theta'\equiv\bar\theta^{\pm 1}$)
of $M_1$
we denote $\theta^{\pm 1}=\Pi_{2,1}(\theta')$. If
$(\theta')^{\pm 1}$ does not correspond to a rule of $M_1$, we denote by $\Pi_{2,1}(\theta')$ the empty rule. The map \label{Pi21} $\Pi_{2,1}$
extends to histories of computations in the natural way.

\begin{lemma}\label{lmP} If $H$ is the reduced history of a computation of $M_2$ with the standard base and $W\cdot H=W'$,
then $\Pi_{2,1}(H)$ is
a reduced history of computation of the $S$-machine $M_1.$
If $\pi_{2,1}(W)$ is an admissible word for the $S$-machine $M_1$, then

\begin{equation}\label{P21}
\pi_{2,1}(W)\cdot \Pi_{2,1}(H) =\pi_{2,1}(W').
\end{equation}
\end{lemma}

\proof The fact that $\pi_{2,1}(H)$ is a history of computation of the $S$-machine $M_1$, and (\ref{P21}), immediately follows from the definition of $M_2$ and the fact that application of rules from $Z^{(\theta,i)}$ does not change
the value of $\pi_{2,1}$  by Lemma \ref{lm569}. The fact that $\pi_{2,1}(H)$ is reduced follows from Lemmas \ref{lmc} and \ref{lme}
\endproof

The next lemma-definition
gives in a sense an inverse function of $\Pi_{2,1}$.

\begin{lemma}\label{Pi12}
For every positive $\bar\theta$-admissible word
$W$ of $M_2$ with the standard base such that there
exists a 
computation 
$w\to w_1\to\dots \to w\cdot H$ 
of the Turing machine
$M_1$ with positive history $H$, starting with $w\equiv \pi_{2,1}(W)$ and having the first rule $\theta$, and all admissible words
positive
there exists
a unique
reduced
computation of $M_2$ starting with $W\to W\cdot\bar\theta$, whose history is $H'$ such that $\Pi_{2,1}(H')\equiv H$ and
the last rule is of the form $\zeta(\theta).$
That history $H'$ will be denoted by $\Pi_{1,2}(H, W)$. This definition agrees with the notation $\Pi_{1,2}(\theta,W)$ of Lemma \ref{lmb}.
\end{lemma}

\proof
Indeed, if $H\equiv \theta\theta_1...\theta_n$, where all $\theta_i$ are positive rules of $M_1$, then we can define $\Pi_{1,2}(H,w)$ as $\Pi_{1,2}(\theta,W)\Pi_{1,2}(\theta_1, W_1)...\Pi_{1,2}(\theta_n,W_n),$ where $W_i=\pi_{1,2}(w_i)$ ($i=1,\dots,n$)
. 

For the uniqueness of $H'$, we note that every rule of the form $\bar\theta_i$ 
can follow only after a rule of the form $\zeta(*)$ since $H$ is positive.
It follows from Lemma \ref{lma} that there is only one $\zeta(*)^{\pm 1}$-rule
between  $\bar\theta_i$ and a preceding rule of this form.  Now the uniqueness
of $H'$ follows 
from Lemmas \ref{lmb}. 
\endproof

Every time we are using the notation $\Pi_{1,2}(H,W)$ below, the conditions of Lemma \ref{Pi12} will be assumed or clearly satisfied.

\begin{rk}\label{rkP} Note that if $\pi_{2,1}(W)\cdot H=W_0$, that is the computation of the Turing machine $M_1$
is accepting, then the
corresponding computation of $M_2$ with history $\Pi_{1,2}(H, W)$ is also accepting.
\end{rk}

\begin{lemma}\label{tt} There are no reduced computations $W\to W\cdot\bar\theta\iv\to W\cdot \bar\theta\iv\bar\theta'$ 
with the standard base,
where the first and the third words are positive and $\theta, \theta'$ are positive rules of $M_1$. 
\end{lemma}
\proof 
Assume that such a computation exists and $W\cdot\bar\theta\iv$ is positive too. Then by Lemma \ref{ss} this computation is a reduced computation of the symmetric Turing machine $M_1$ with
history $\theta\iv\theta',$ contrary to the Property (b) of $M_1$ given by Lemma \ref{mach23}.

Now assume that the word $W\cdot\bar\theta\iv$ is not positive. By Property (b) from Lemma \ref{m0m1}, each $\theta, \theta'$ inserts/deletes at most one tape letter.
The only non-trivial case is when $\theta, \theta'$ insert an $a$-letter:
in other cases the second word in the computation is obviously positive.  But then $\bar\theta\iv$ must insert a letter $a\iv$ which is then removed by $\bar\theta'$ (in the same sector). 
Since both rules $\bar\theta,\bar\theta'$ have word $W\cdot \bar\theta\iv$ in their domains, the left hand sides in all parts of the rules of $\theta, \theta'$ coincide. Since  $M_1$ is the symmetrization of a deterministic Turing machine by construction, $\theta\equiv\theta'$ and our computation is not reduced, a contradiction.
\endproof

\begin{lemma}\label{lmd}
Let $H$ be  the history of a reduced computation $W\to...\to W\cdot H$ of $M_2$ with the standard base.

(1) If $H\equiv \bar\theta_1H'\zeta(\theta_2)$, where $\theta_1,\theta_2$ are positive rules of the $S$-machine $M_1$. Then the word $H$ and all words in that computation except possibly the first one are positive.

(2) If $H\equiv \bar\theta_1 H_1\bar\theta_2^{\epsilon_2}\dots \bar\theta_n^{\epsilon_n}H_n\bar\theta_{n+1}^{\epsilon_{n+1}},$
where $\theta_1,\dots \theta_{n+1},$ are positive rules of the $S$-machine $M_1,$ $\epsilon_i=\pm 1,$ 
$(\epsilon_n,\epsilon_{n+1})\ne (-1,1),$ 
and $H_1,\dots,H_n$ have no rules corresponding to the rules
of $M_1,$ then all words in this computation except possibly the first one and the last one are positive.
\end{lemma}
\proof (1) Induction on the length of $\Pi_{2,1}(H)$. Suppose  $\Pi_{2,1}(H)\equiv\theta_1$. Then
by Lemma \ref{lma},
$H'$ does not contain rules of the form $\zeta(\theta)^{\pm 1}$.
Hence we can apply Lemma \ref{lmb} and conclude that all words in the computation except possibly the first one are positive.

Suppose that the length of $\Pi_{2,1}(H)$ is at least 2. Suppose further that  the second letter of $\Pi_{2,1}(H)$ is positive, that is $H\equiv \bar\theta_1H_1\bar\theta_3H_2$ for some positive $\theta_3$
and $H_1$ not containing rules corresponding to the positive rules of $M_1$. Then $H_1$ must end with $\zeta(\theta')$ for some $\theta'.$ 
Then we can apply the induction assumption to the computations with histories $\bar\theta_1H_1$ and $\bar\theta_3H_2$  
and conclude that $H$ and all words, except for the first one, in the computation with history $H$  are positive as desired.

Now suppose that the second letter in $\Pi_{2,1}(H)$ is $\theta_3\iv$ for some positive $\theta_3.$
Since $H$ ends with $\zeta(\theta_2)$
, the last rule in $\Pi_{2,1}(H)$ is positive. Indeed  the rule used in any reduced computation of $M_2$ immediately after
a rule of the form $\bar\theta\iv$ for some positive $\theta$ is either $\zeta(\theta')\iv$ for some positive $\theta'$ or $\bar\theta'$ for some positive $\theta'$ (this can be seen by looking at the indices of $q$-letters of the admissible words). The first option is  impossible by Lemma \ref{lma}
, the second option is impossible since we consider the last rule in $\Pi_{2,1}(H)$.
Hence $H\equiv \bar\theta_1H_1\bar\theta_2\iv H_2...H_{m-1}\bar\theta_m\iv H_m\bar\theta_{m+1}H''$
for some positive rules $\theta_2,...,\theta_{m+1},$ where $H_2,...,H_m$ do not contain rules corresponding to rules of $M_1$
or their inverses, $\theta_{m+1}$ is the second positive rule in $\Pi_{2,1}(H)$. By Lemma \ref{lme}, $H_m$ is empty. 
By Lemma \ref{lmc}, $\bar\theta_1H_1\bar\theta_2\iv\equiv \Pi_{1,2}(\theta_1,W)\Pi_{1,2}(\theta_2,W\cdot \bar\theta_1H_1\bar\theta_2\iv)\iv.$
Now, consider the computation of $M_2$ started with the admissible word $W'=W\cdot \bar\theta_1H_1\bar\theta_2\iv H_2...H_{m-1}\bar\theta_m\iv$ and having the history $\bar\theta_mH_{m-1}\iv...H_2\iv\Pi_{1,2}(\theta_2,W\cdot \bar\theta_1H_1\bar\theta_2\iv).$
By the inductive hypothesis, 
all words in this computation, starting with the second one,  and all words in the computation $W\cdot\bar\theta_1\to...\to W\cdot \Pi_{1,2}(\theta_1,W)$
are positive.  
By the induction assumption,
the computation $W'\cdot\bar\theta_{m+1}\to...\to W\cdot H$  also consists of positive words. Therefore the first and the third words in the subcomputation with history $\bar\theta_m\iv\bar\theta_{m+1}$ are positive contrary to Lemma \ref{tt}.

(2) It is nothing to prove if $n=0.$ The case $\epsilon_2=1$  was
considered in the proof of claim (1). For the case $H\equiv \bar\theta_1H_1\bar\theta_2\iv,$
we also proved that the computation with subhistory $H\equiv \bar\theta_1H_1$ has
all words positive except possibly the first one. Hence we may assume that $n\ge 2.$
If $H\equiv \bar\theta_1H_1\bar\theta_2\iv H_2...H_{n-1}\bar\theta_n\iv H_n\bar\theta_{n+1}\iv,$
then we consider the computation with history $H^{-1}$ and again come to the case $\epsilon_2=1.$ Therefore we assume that $H\equiv \bar\theta_1H_1\bar\theta_2\iv H_2...H_{m-1}\bar\theta_m\iv H_m\bar\theta_{m+1}H'',$ where $m<n$ by the condition on
$(\epsilon_n,\epsilon_{n+1}).$ By Lemma \ref{lme}, we have that $H_m$ is empty. By the
inductive hypothesis, the computations with histories $\bar\theta_1H_1\bar\theta_2\iv H_2...H_{m-1}\bar\theta_m\iv$ and $\bar\theta_m H''$ have all words positive except possibly the
first one and the last one. Therefore we can apply Lemma \ref{tt} to the
computation with history $\bar\theta_m\iv\theta_{m+1}$, a contradiction.

\endproof

\begin{lemma}\label{norep} For every reduced computation $w_0\to\dots\to w_t$ of $M_2$ with the
standard base and a non-empty history $H$, we have $w_t\ne w_0.$
\end{lemma}

\proof Assume that $w_t=w_0$, and $t>0$ is minimal. Then the computation $w_1\to\dots\to w_{t-1}$
is not a counter-example, and so $H$ is a cyclically reduced word.

If $H'\equiv\Pi_{21}(H)$ is empty, we consider the computation $w_0\to\dots\to w_0\circ H=w_0\to\dots\to w_0\circ H^2\to\dots$
with history $H^k$, where $k$ as large as we want. As in Lemmas \ref{lmb} and \ref{lmc}, we have a decomposition
$H\equiv H_1\dots H_s,$ where $H_i$ corresponds to the work of some $Z^{(\theta,i)}$ or equal to some $\zeta(\theta)^{\pm 1}$ and $s$ is bounded by a constant independent of $k.$ It follows that $H$ corresponds to only one $Z^{(\theta,i)}$, and
then the equality $w_t=w_0$ and Lemma \ref{lm93} imply that $||w_0||=||w_1||=\dots=||w_t||.$ Now $||H^k||$ is uniformly bounded
for all $k$-s, contrary to Lemma \ref{lm89}.

If $||H'||\ge 1,$ then Lemma \ref{lmP} gives a reduced computation
$\pi_{21}(w_0)\to\dots\to\pi_{21}(w_t)$ with history $H'.$ As above we can obtain 
reduced computations $\pi_{21}(w_0)\to\dots$ of the $S$-machine $M_1$ 
with histories $H'^k.$ For $k\ge 3,$ Lemma \ref{lmd} (2) implies that all words
in the computation with history $H$ are positive. Then the same property  must be 
true for the computation with history $H',$ and by Lemma \ref{ss}, it is also
a computation of the Turing machine $M_1$ with the same history $H'$, contrary to Lemma \ref{mach23} (c). Thus the lemma is proved by contradiction. \endproof

\begin{lemma}\label{lmM21} (a) Let \label{X2} $X_2$ be the set of all words $\pi_{1,2}(W)$ accepted by $M_2$,
where $W$ is an
input word of the Turing machine $M_1$. Then a word $W'$ belongs to $X_2$ if and only if $W'\equiv\pi_{1,2}(W)$, and $W$ is an input word of $M_1$
accepted by the Turing machine $M_1$.
Hence the set of words accepted by $M_2$ is not recursive.

(b) For every $W'\equiv\pi_{1,2}(W)\in X_2$ there exists only one reduced computation of $M_2$ accepting $W'$, the length
of that computation is between the length $T$ of the reduced computation of $M_1$ accepting $W$ (this computation is unique by Lemma \ref{mach23} (e)) and $\exp(O(T))$.
\end{lemma}

\proof
(a) Let $W'\equiv\pi_{1,2}(W),$ where $W$ is an input word of the Turing machine $M_1$. Suppose that $W$ is accepted by the (symmetric) Turing machine $M_1$. By part (b) of Lemma \ref{mach23}, the history $H$ of  the accepting computation consists of positive commands only. Then the computation of $M_2$ with history $\Pi_{1,2}(H,W)$
accepts $W'$ by Remark \ref{rkP}.

Suppose that $W'$ is accepted by $M_2$, and $H'$ is the history of an accepting computation $C'$. By Lemma \ref{lmP}, $\Pi_{2,1}(H')$ is a reduced history of an accepting computation of the $S$-machine $M_1$ starting with the input admissible word $\pi_{2,1}(W').$ Therefore by 
part (c) of Lemma \ref{m0m1}, 
the first rule in $H'$
is $\bar\theta$ for some positive rule $\theta$ of $M_1$. Again by Lemma \ref{m0m1} (c),
the last rule of $\Pi_{2,1}(H')$ is positive. It follows that the last rule of $H'$ must be $\zeta(\theta')$ for some positive rule $\theta'$ since the accepted admissible word of $M_2$
has state letters having no $\theta$-indices. By Lemma \ref{lmd}, then all words in computation $C'$ are positive because both $W$ and $W'\equiv\pi_{1,2}(W)$ are positive too. Therefore all words in the accepting computation $W\to...\to \pi_{2,1}(W'\cdot H')$ of the $S$-machine $M_1$ are positive. By Lemma \ref{ss}, the latter computation is an accepting computation of the Turing machine $M_1$, whence $W\in X_1$.

(b) The existence and the uniqueness 
follow from part (a) of this lemma and Lemmas \ref{lmb} and \ref{Pi12}. 
The part about the length of computation follows from Lemma \ref{lm90} (2). 
\endproof

Similarly to the case of Turing machines, for every
function $f(n)$, we
define \label{fgoodn} $f$-good numbers
for $M_2$. We call a number $b$ $f$-good provided
for every
input word $W$ from $X_2$, if the length of the input sector (that is the $Q_0P_1$-sector) of $W$ is $<b$, then
$f(T)\le b,$ where $T$ is the time of accepting $W$ by $M_2$.  Now Theorem \ref{re} and Lemma \ref{lmM21} (b) imply

\begin{lemma}\label{lmM22} For every  $\alpha>0$,
the set of $\exp(\alpha n)$-good numbers of $M_2$ is infinite.
\end{lemma}

\begin{rk} \label{tildeM2}
The machine $Z(A)$ and the copies of it $Z^{(\theta,i)}$ do not satisfy Property
\ref{one} since two $a$-letters are involved in the (copies of) rules $r_{12}(a).$ Therefore
further we will use the machine \label{tildeM2} $\tilde M_2$ obtained from $M_2$ by the application of Lemma
\ref{tilde1}. Note that the claim of Lemma \ref{norep}
is correct for $\tilde M_2$ as well. Indeed if the words $w_0\equiv w_t$ in a computation
$w_0\to\dots\to w_t$ of
$\tilde M_2$ with non-empty reduced history $H$ involve auxiliary state letters, then there is a
computation of $\tilde M_2$ with a reduced history $H',$ where $H'$ is a freely conjugate of $H,$ which starts
and ends with the words $w'_0\equiv w'_{t'},$ having no special state letters. Then Lemma \ref{tilde1} and
Lemma \ref{norep} for $M_2$ lead to a contradiction. Since the modification of $M_2$ does
not touch the rules $\bar\theta$ corresponding to the rules $\theta$ of $M_1,$ the statements
of Lemmas \ref{lmM21} and \ref{lmM22} also remain valid for $\tilde M_2.$ {\it Thus
Lemmas \ref{norep}, \ref{lmM21}, and \ref{lmM22} will be applied to the modified machine
$\tilde M_2.$} Moreover, it follows from the definitions of $M_2$ and $\tilde M_2$ that
these S-machines inherit the Property (c) from Lemma \ref{m0m1} of the Turing machine $M_1$ (for positive rules). 
\end{rk}

If the sum from Property \ref{one} (2) is positive for some rule of $\tilde M_2,$ then this sum is
$1,$ and we have $||v_i||+||v'_i||=1$ or $||u_{i+1}||+||u'_{i+1}||=1$ for a unique $i$.
In the first case (in the second case) we say that the rule is \label{leftrr}{\it left} (is {\it right}).

\begin{rk}\label{rk2}
Note that similarly, we can define the {\em $\circ$-product} $\sss\circ \sss'$ of any $S$-machine $\sss$ and S-machine
$\sss'=\sss'(A)$ depending on the tape alphabet $A$. Furthermore, one can replace
the auxiliary machine $\sss'$ by several S-machines $\sss_1,\dots,\sss_d.$  Namely, one inserts a $p$-letter between
two consecutive state letters in the standard base of $\sss$ and treats any subword
$q_i\dots p\dots q_{i+1}$ as an admissible subword for $\sss_i$-s. For each rule $\theta$ of $\sss,$
one has a modified rule \label{bart} $\bar\theta$ of the composition. The application of the rule
$\bar\theta$ is normally framed by alternated works of the auxiliary machines
$\sss_j^{(\theta,i)}$, and the priorities of the work of these machines may
depend on $\theta.$   We are not going to define this construction formally, leaving it
to the reader. In the next subsection, we shall introduce the $\circ$-product of $\tilde M_2$ and two
primitive $S$-machines.
\end{rk}

\subsection{The machine $M_3$}

Let $M_2$ be the $S$-machine $M_1\circ Z$ and $\tilde M_2$ the modification from Remark \ref{tildeM2}.
For every set of letters
$A,$ let $A'$ and $A''$ be  disjoint  copies of $A$, the maps $a\mapsto a'$ and $a\mapsto a''$ identify $A$ with $A'$ and $A'',$ resp. Let \label{overZ} $\overrightarrow Z=\overrightarrow Z(A)$ and $\overleftarrow Z=\overleftarrow Z(A)$ be the $S$-machines
with tape alphabet $A'\sqcup A''$, state alphabet $\{L\}\cup P\cup \{R\},$ where
$P=\{p(1),p(2),p(3)\}$ and the following positive $S$-rules. For $\overrightarrow Z$ we have the rules

$$\xi_1(a)=[L\to L, p(1)\to a'p(1)(a'')\iv, R\to R], a\in A$$
{\em Comment:} The head moves from left to right, replacing the word on the tape by its copy in the alphabet $A'$.
$$\xi_2:=[L\to L,  p(1)\tool p(2), R\to R]$$
{\em Comment:} When the head meets $R$, it turns into $p(2)$.

For $\overleftarrow Z$, we define the rules

$$\xi_3(a)=[L\to L, p(2)\to (a')\iv p(2)a'', R\to R]$$
{\em Comment:} The head $p(2)$ moves from right to left, replacing the word in $A'$ by its copy in $A''$.
$$\xi_4=[L\tool L, p(2)\to p(3), R\to R]$$
{\em Comment:} When the head reaches the left end of the tape, it turns into $p(3)$.

\begin{rk} \label{aiv} For every $a\in A$, $i=1,3$, it will be convenient to denote $\xi_i(a)\iv$ by $\xi_i(a\iv)$. It is clear from the definition $\xi_i(a)$ that this does not lead to a confusion.\end{rk}

\begin{rk}\label{Z'} Note that if the machine $\overrightarrow Z$ (the machine $\overleftarrow Z$) starts with the word $Lp(1)u''R$ (resp., with $Lu'p(2)R$) and ends with the word $Lu'p(2)R$ (with
$Lp(3)u''R$),
where $u''$ is the word  in $A''$-letters, then
the history $H$ of the only reduced computation
such that $Lp(1)u''R\cdot H=Lu'p(2)R$ (such that $Lu'p(2)R\cdot H=Lp(3)u''R$) is

$\xi_1(a_1)...\xi_1(a_m)\xi_2$ (is $\xi_3(a_m)\dots \xi_3(a_1)\xi_4$)
and its length is $||u||+1$. Here $u\equiv a_1\dots a_m$ is the copy of $u'$ (of $u''$)
in the alphabet $A.$

Similarly, any reduced computation of $\overrightarrow Z$ (of $\overleftarrow Z$)
ending with $Lu'p(2)R$ (resp., with $Lp(3)u''R$) is uniquely determined by its initial
admissible word and has length $\le ||u||+1.$ 
\end{rk}

Remark \ref{Z'} implies, in particular

\begin{lemma}\label{M332} Suppose that $W\to ...\to W\cdot H$ is a reduced computation of $\overrightarrow Z$  (of $\overleftarrow Z$) with history $H$ and
the standard base. Suppose that both $W$ and $W\cdot H$ contain $Lp(1)$ or both contain $p(2)R$ (resp., $p(2)R$ or $Lp(3)$). Then $H$ is empty.
\end{lemma}

Below we define \label{M3} $M_3$ as $\tilde M_2\circ\{\overrightarrow Z,\overleftarrow Z \} $ (see Remark \ref{rk2}), that is we insert copies of  $\overleftarrow Z$ and $\overleftarrow Z$ between every two consecutive state letters of $\tilde M_2$.
We simplify and unify the notation by changing the value of $N$
and renaming the parts of the state alphabet of $\tilde M_2$.
In this section we assume that $\tilde M_2$ has the standard base $s_0s_1...s_N$
(and forget more detailed earlier notations).

For every $i=1,...,N$, we make  copies $Y'_i$ and $Y''_i$ of the alphabet
$Y_i$ of $\tilde M_2$ ($i=1,...,N$).
Let $\Theta$ be the set of positive commands of $\tilde M_2.$ 
The set of state letters of $M_3$ is $$S_0\cup P_1\cup
S_1\cup P_2\cup...\cup P_{N}\cup S_N$$
where $P_i=\{p^{(i)}, p^{(i,1)}, p^{(i,0)},
p^{(\theta,i)}(1), p^{(\theta,i)}(2),  p^{(\theta,i)}(3)\mid
\theta\in\Theta\}$
$i=1,...,N$, $S_i=Q_i\sqcup Q_i\times \Theta$
where $Q_0\sqcup Q_1\sqcup ... Q_N$ is the set of state letters of
$\tilde M_2$.
Thus the state letters $L$ and $R$ of the copies of machines $\overrightarrow Z$ and $\overleftarrow Z$ are identified with the corresponding $S$-letters as in the case of $M_2=M_1\circ Z$. We shall call the state letters from $P_i$-s the \label{controlsl}{\em control state letters} or \label{pletter} $p$-{\it letters},
and the other state letters (i.e. the copies of the state letters of $\tilde M_2$), the {\em basic state letters}.

The set of tape letters of $M_3$ is $Y=Y_1\sqcup \dots \sqcup Y_{2N}= Y'_{1}\sqcup Y''_{1}\sqcup
Y'_{2}\sqcup Y''_{2}\sqcup...\sqcup
Y'_{N}\sqcup Y''_{N}.$

Let $\theta$ be a positive $\tilde M_2$-rule which is not a right rule. Assume $\theta$ 
is of the form
$$[s_0\to s_0', v_1s_1\to
v_1's_1',...,v_{N}s_N\to v_{N}'s_{N}'],$$
where $s_i, s_i'\in
S_i$, and $v_i$-s are words in $Y.$ Then this rule is replaced in $M_3$ 
by positive $$\bar\theta=\left[\begin{array}{l}s^{(\theta,0)}\tool (s')^{(\theta,0)}, p^{(\theta,1)}(1)
\to p^{(\theta,1)}(1), v_1s^{(\theta,1)}\tool v'_1(s')^{(\theta,1)},\\ p^{(\theta,2)}(1)
\to p^{(\theta,2)}(1), ...,
 v_{N}s^{(\theta,N)}\to v'_{N}
 (s')^{(\theta,N)}\end{array}\right]$$
 with $Y_{2i-1}(\bar\theta)=\emptyset $ and $Y_{2i}(\bar\theta)=Y_{i}''(\theta).$
 As an exception
 , the left-hand sides of the parts of $\bar\theta$ are of the
 form $s_0\tool,$ $ p^{(1,1)}\to,$ $ s_1\tool,$ $p^{(2,1)}\to,$ $\dots, s_N\to $ if $\theta$ is the unique
 start rule, 
 i.e., they do not depend on the index $\theta.$

 A right positive rule $\theta$ of the form $$[s_0u_1\to s_0'u'_1, s_1u_2\to
s_1'u_2',...,s_N\to s_{N}']$$ is replaced by the right rule \label{bart} $\bar\theta$ of $M_3$:

$$\bar\theta=\left[\begin{array}{l}s^{(\theta,0)}u_1\to (s')^{(\theta,0)}u'_1, p^{(\theta,1)}(2)
\tool p^{(\theta,1)}(2), s^{(\theta,1)}u_2\to (s')^{(\theta,1)}u'_2,\\ p^{(\theta,2)}(2)
\tool p^{(\theta,2)}(2), ...,
 s^{(\theta,N)}\to 
 (s')^{(\theta,N)}\end{array}\right]$$
 with $Y_{2i-1}(\bar\theta)=Y'_i(\theta) $ and $Y_{2i}(\bar\theta)=\emptyset.$

 Now we want to describe the alternating work of the auxiliary machines \label{Ztilr}
 $\overleftarrow Z^{(\theta,i)}$ and $\overrightarrow Z^{(\theta,i)}$. Normally
 each of them is switched on exactly once in the frame of the rule $\theta,$
 but the sequence of their turning on depends on $\theta.$   
 
 First, we need the following {\em transition}
\label{zmt} rule $\zeta_-(\theta).$ 
This rule adds $\theta$ to all state letters 
and turns all $p^{(j)}$ into $p^{(\theta,j)}(1)$:

$$[s_i\tool s^{(\theta,i)}, p^{(j)}\to
p^{(\theta,j)}(1), i=0,...,N, j=1,...,N]$$ 
so that the rule $\bar\theta$ becomes applicable if $\theta$ is not a right rule. Again,
as an exception, we do not introduce $\zeta_-(\theta)$ for the  start rule $\theta=\theta_{start}$
of $\tilde M_2.$

If $\theta$ is a right rule, then the rule $\zeta_-(\theta)$ successively switches on the
machines $\overrightarrow Z^{(\theta,1)},\dots,$ $ \overrightarrow Z^{(\theta,N)}.$
(We will not present formulas for the rules $\bar\tau_i(\theta)$ as in the
definition of $M_2,$ since the explicit form of these rules are not necessary.)
Then the state letters $p^{(\theta,j)}(1)$ ($j=1,\dots, N$) successively turn 
into $p^{(\theta,j)}(2),$ find themselves just before $s_{i}$-letters, and the rule $\bar\theta$ can be applicable.

After an application of a non-right rule $\bar\theta,$ the machines $\overrightarrow Z^{(\theta,j)}$
move the $p$-letters to the write, change $p^{(\theta,j)}(1)$  by $p^{(\theta,j)}(2),$
and then the machines $\overleftarrow Z^{(\theta,j)}$ move the $p$-letters to
the left and change $p^{(\theta,j)}(2)$  by $p^{(\theta,j)}(3).$ After an application
of a right rule $\bar\theta,$ only machines $\overleftarrow Z^{(\theta,j)}$ work.

Finally, the transition
rule \label{zmt} $\zeta_+(\theta)$ removes index $\theta$ from all state letters, and turns all  $p^{(\theta,j)}(3)$ into $p^{(j)}$:

$$[s^{(\theta,i)}\tool s_i, p^{(\theta,j)}(3)\to
p^{(j)}, i=0,...,N, j=1,...,N].$$ 
If $\theta$ is the unique accept rule of $\tilde M_2$ then all $p^{(j)}$-s in the above definition of $\zeta_+(\theta)$ must be replaced by special letters $p^{(j,0)}$-s.

An important specification is the following. If $\theta$ is a right rule and
$||u_{i+1}||+||u'_{i+1}||=1$ (see (**)), then the application of $\bar\theta$ always
switches on the auxiliary machines in the order $\overleftarrow Z^{(\theta,i)},$
$\overleftarrow Z^{(\theta,i-1)},\dots,$ $\overleftarrow Z^{(\theta,1)},$ 
$\overleftarrow Z^{(\theta,N)},\dots,$ $\overleftarrow Z^{(\theta,i+1)}.$ If $\theta$ is a left
rule and $||v_i||+||v'_i||=1,$ then
$\bar\theta$ must successively start up $\overrightarrow Z^{(\theta,i+1)},\dots,$
$\overrightarrow Z^{(\theta,N)},$ $\overrightarrow Z^{(\theta,1)},$ $\overrightarrow Z^{(\theta,i)}.$
If $\theta$ is neither right nor left, then the order for the first $N$ machines is $\overrightarrow Z^{(\theta,1)},\dots,$ $\overrightarrow Z^{(\theta,N)}.$

 For every admissible word $w$ of $\tilde M_2$ with standard base, let \label{pi23w}
$\pi_{2,3}(w)$ be the admissible word of $M_3$ obtained by inserting control state letters $p^{(i)}$ ($p^{(i,1)}$ or $p^{(i,0)}$ if the word $w$ has state letters from the start vector $\overrightarrow s_1$, resp., from the accept vector $\overrightarrow s_0$    of $\tilde M_2$) next to the
right of each $s_{i-1},$ $i\le N$.
The stop word of $M_3$ is $\pi_{2,3}(\pi_{1,2}(W_0)),$ where $W_0$ is the stop word of $M_1$. For every input word $w$ of $\tilde M_2$  we call $\pi_{2,3}(w)$ an {\em input} word of $M_3$.

\begin{rk}\label{d} It follows from Remark \ref{tildeM2} and the definition of $M_3$
that the S-machine $M_3$ 
inherits the Property (c) from Lemma \ref{m0m1} (for positive rules). 
\end{rk}

The $p^{(1)}s_1$-sector
of an admissible word of $M_3$ is called the {\em input sector} of that word.
\medskip

Assume that $w\to w\cdot \theta$ is a computation of the machine $\tilde M_2$ 
with standard base and a positive rule $\theta$. Then, by the definition of $M_3,$
we have the canonically defined reduced computation $\dots \to\pi_{2,3}(w)\to \pi_{2,3}(w)\cdot\bar\theta\to\dots$ starting and ending with words whose state
letters have no $\theta$-indices and all other words do have $\theta$-indices.
The computation of $M_3$ with
these properties is unique since the base is standard. Indeed Remark \ref{Z'} 
and the definition of $M_3$ uniquely 
determine the order of
rules for each of the auxiliary machines $\overrightarrow Z^{(\theta,j)}$ and
$\overleftarrow Z^{(\theta,j)}$. (For example, a machine $\overrightarrow Z^{(\theta,j)}$
can start working only if the state letter $p^{(\theta,j)}(1)$ is the right neighbor 
of a letter $s^{(\theta, j-1)}$ since the $p^{(\theta,j)}(1)s^{(\theta, j-1)}$-sector is locked before
the start, and $\overrightarrow Z^{(\theta,j)}$ cannot finish its work until the $p$-letter
becomes the left neighbor of $s^{(\theta,j)}$ and turns into $p^{(\theta,j)}(2),$ etc.) Thus the following claim is true.

\begin{lemma}\label{pi23} For every computation $w\to w\cdot \theta$ of the machine 
$\tilde M_2$  with standard base and a positive rule $\theta, $ there is a unique
reduced $M_3$-computation\\
$\dots \to\pi_{2,3}(w)\to \pi_{2,3}(w)\cdot\bar\theta\to\dots$ such that it
starts and ends with words whose state
letters have no $\theta$-indices and all other words have $\theta$-indices. The history
of this computation starts with $\zeta_-(\theta)$ and ends with $\zeta_+(\theta).$
\end{lemma}

We denote the history of this computation by \label{Pi23}
$\Pi_{2,3}(\theta,w).$ It follows from Remark \ref{Z'} that the length
of this history is $1+2(|w|_a+N),$ where \label{wa} $|w|_a$ is the number of $a$-letters in the word $w.$ If $\theta$ is a negative rule of $\tilde M_2,$ then
we invert the computation constructed for $\theta^{-1}$, and so $\Pi_{2,3}(\theta,w)\equiv\Pi_{2,3}(\theta^{-1}, w\cdot\theta)^{-1}.$ 

    Similarly, with arbitrary reduced computation $w\to\dots\to w\cdot H$ with the standard
    base of $\tilde M_2$ and
    having a history $H\equiv \theta_1\dots\theta_t$ we associate the reduced computation
    of $M_3$ with history \label{Pi23} $$\Pi_{2,3}(H,w)\equiv \Pi_{2,3}(\theta_1,w) 
    \Pi_{2,3}(\theta_2,w\cdot \theta_1)\dots \Pi_{2,3}(\theta_t,w\cdot \theta_1\dots\theta_{t-1})$$
    It follows from the previous paragraph that 
    $||H||\le ||\Pi_{2,3}(H,w)||=O(||H||^2)$ for every accepted computation.
    Indeed $|w|_a=O(||H||)$ since the stop word has no tape letters.
    
    Recall that only rules of the form $\zeta_{\pm}(\theta)$ and the start rule
    involve state letters without $\theta$-indices. Therefore it follows from Lemma \ref{pi23} that
    every reduced computation with the standard base of $M_3$ starting and ending with the admissible words
    without $\theta$-indices in their state letters, has history of the
    form $\Pi_{2,3}(H,w)$ for some reduced computation $w\to\dots\to w\cdot H$ of $\tilde M_2.$
    In particular, our discussion and Lemma \ref{lmM21} (b) imply the following

\begin{lemma}\label{M31} (a) Let \label{X3} $X_3$ be the set of all words $\pi_{2,3}(w)$ accepted by $M_3$, where $w$ is an input
word of $\tilde M_2$ (or $M_2$). Then a word
$W$ belongs to $X_3$ if and only if $W=\pi_{2,3}(w)$, and $w\in X_2$. Hence the set of words
accepted by $M_3$ is not recursive.

(b) For every $W\equiv\pi_{2,3}(w)\in X_3$ there exists only one
reduced computation of $M_3$ accepting $W$, the length
of that computation is between the length $T$ of the reduced computation of $\tilde M_2$
accepting $w$ and $O(T^2)$.
\end{lemma}

We can define \label{fgoodn} $f$-good numbers of $M_3$ in a similar way as for $M_2$. Lemmas \ref{lmM22} and \ref{M31} imply

\begin{lemma}\label{M32} For every constant $c>0$,
the set of $\exp(c n)$-good numbers of $M_3$ is infinite.
\end{lemma}

As in the previous section, we need to define more maps
between $S$-machines $\tilde M_2$ and $M_3$.

For every admissible word $W$ of $M_3$
with
the standard base, 
let \label{pi32} $\pi_{3,2}(W)$ be the word obtained by removing
state $p$-letters, $\theta$-indices of state letters, and the indices that
distinguishes $a$-letters from the left and from the right of $p$-letters. 
We obtain an admissible
word of $\tilde M_2$. Note that we have

\begin{equation}\label{p32}
\pi_{3,2}(\pi_{2,3}(w))\equiv w.
\end{equation}

For every rule $\bar\theta$ of $M_3$ corresponding to a rule $\theta$ of $\tilde M_2$ we denote $\theta=\Pi_{3,2}(\bar\theta)$. If
$\bar\theta$ does not correspond to a rule of $\tilde M_2$, we denote by $\Pi_{3,2}(\theta)$ the empty rule. The map \label{Pi32} $\Pi_{3,2}$
extends to histories of computations in the natural way.

\begin{rk} \label{*I}It can be proved similarly to Lemma \ref{lmP}, that if $H$ is a history of a computation of $M_3$ with standard base and $W\cdot H=W'$, then
$\Pi_{3,2}(H)$ is reduced and

\begin{equation}\label{P32}
\pi_{3,2}(W)\cdot \Pi_{3,2}(H)\equiv\pi_{3,2}(W').
\end{equation}
\end{rk}
\begin{lemma} \label{ppm} Suppose a commutation $W\to\dots$ of $M_3$ with a base $B$
has a reduced history $H\equiv \dots\bar\theta_1H'\bar\theta_2^{\eta}\dots ,$ where $\Pi_{3,2}(H)\equiv\theta_1\theta_2^{\eta}$ for some positive $\theta_1$ and $\theta_2,$ and $ \eta =\pm 1.$

(1) if $B$ is standard,
then the word
$W\cdot\bar\theta_1$ is completely determined by $H';$ 

(2) if $\theta_2^{\eta}\ne \theta_1^{-1}$, then  $B$
or $B^{-1}$ is a subword of the standard base of the machine $M_3;$

(3) 
let $||u_{j+1}||+||u'_{j+1}||=1$ ($||v_j||+||v'_j||=1$) for the rule $\theta_1,$  
and $B$ is not a subword of the standard base or of its inverse.
Then $\theta_2^{\eta}\equiv\theta_1^{-1},$ and no rule
from $H'$ locks the $s^{(\theta,j)}p^{(\theta,j+1)}$-sector (resp., the
$p^{(\theta,j)}s^{(\theta,j)}$-sector).

\end{lemma}

\proof (1) The argument used for Lemma \ref{pi23} shows that since the base is standard, each of the machines
$ \overrightarrow Z^{(\theta,j)}$ must accomplish its standard work after the application of the rule $\bar\theta_1.$
Therefore the history of the work of $\overrightarrow Z^{(\theta,j)}$
completely determines the $p^{(j)}s_j$-sector subword of the word $W\cdot\bar\theta_1.$
The $\theta$-indices of the state letters of this word are obviously
determined by the histories of $ \overrightarrow Z^{(\theta,j)},$ and Statement (1) is proved.

(2) The assumptions  implies that the $\theta$-indices that the state letters have after
the application of $\bar\theta_1,$ must disappear earlier than one applies $\bar\theta_2^{\eta}.$ Again by Remark \ref{Z'}, it follows that each of the machines
$ \overrightarrow Z^{(\theta_1,j)}$ must perform its standard work. Therefore for every
$j=1,\dots,N,$ the history $H'$ has rules locking $s^{(\theta,j-1)}p^{(\theta,j)}(1)$-sectors and it has
rules locking $p^{(\theta,j)}(2)s^{(\theta,j)}$-sectors. Hence, by Lemma \ref{qqiv}, the base $B$ has no subwords
of the form $q^{\pm 1}q^{\mp 1},$ and so $B^{\pm 1}$ is a subword of the standard base by
the definition of admissible word.

(3) First of all, we have $\theta_2^{\eta}\equiv\theta_1^{-1}$ by Property (2). 
Then we assume that a right rule $\tau$ from $H'$ locks the $s^{(\theta,j)}p^{(\theta,j+1)}$-sector. 

The locking rule $\tau$ must belong to the machine $\overleftarrow Z^{(\theta_1,j+1)}$ since
other auxiliary machines working after the application of the right rule $\bar\theta_1$ do
not lock this sector.  Taking into account the order of the work of auxiliary machines
after an application of a right rule, we conclude that the machines $\overleftarrow Z^{(\theta_1,j)},\dots,$
$\overleftarrow Z^{(\theta_1,1)},$ $\overleftarrow Z^{(\theta_1,N)},\dots,$ $\overleftarrow Z^{(\theta_1,j+2)}$ works before the machine $\overleftarrow Z^{(\theta_1,j+1)}$ starts
working. 
Since the last rule of $\overleftarrow Z^{(\theta_1,j+2)}$ locks the $p^{(\theta,j+1)}s^{(\theta,j+1)}$-sector, 
$H'$ has a rule locking $p^{(\theta,j+1)}s^{(\theta,j+1)}$-sector. Proceeding in this manner, we then consider the work of the preceding machine $\overleftarrow Z^{(\theta_1,j+2)}$ and conclude that the $s^{(\theta,j+1)}p^{(\theta,j+2)}$- 
sectors was locked by by some rules from $H'.$ Finally, we
see that every sector except for $p^{(\theta,j)}s^{(\theta,j)}$ was locked by some rule from $H'.$ The $p^{(\theta,j)}s^{(\theta,j)}$-sector was locked by $\bar\theta_1$ since 
$\theta_1$ is a right rule. By Lemma  \ref{qqiv},  $B$ is a subword of the standard base or of its inverse, a contradiction.

Similar argument works if $\theta_1$ is a left rule. In this case if $p^{(\theta,j)}s^{(\theta,j)}$-sector is locked
by a rule from $H',$ then $\bar\theta_1$ 
switches on the machines $\overrightarrow Z^{(\theta_1,j+1)},\dots, $ 
$\overrightarrow Z^{(\theta_1,N)},$ $\overrightarrow Z^{(\theta_1,1)},\dots $ 
$\overrightarrow Z^{(\theta_1,j)},$ and we again come to a contradiction.

\endproof

\begin{lemma}\label{*V} Suppose that the admissible word $W$ of $M_3$ has the standard base. Suppose that a reduced computation applicable to $W$ has history of the form $H^3$. Then $H$ does not contain rules $\bar\theta$ corresponding to 
the rules $\theta$ of the $S$-machine $\tilde M_2.$ 
\end{lemma}

\proof Suppose that $H$ contains a rule $\bar\theta_0^{\pm 1}.$ 
Then  $\bar\theta_0$ occurs 
in a subword  $H'$
of the form $$H'\equiv(\bar\theta_0^{\epsilon_0} H_1\bar\theta_1^{\epsilon_1}
H_2\bar\theta_2^{\epsilon_2}...H_s)^2\bar\theta_0^{\epsilon_0}$$
where all 
$\theta_i$ are positive rules from
$\tilde M_2$,
$\epsilon_i\in \{-1,1\}$,
and $H_i$ consists of rules of various copies of the $S$-machines $\overleftarrow Z$ and $\overrightarrow Z$. 
We can assume that $\epsilon_0=1$ (if not, we can replace $H$ by
$H\iv$).

Let $W$ be the initial word of the computation with history $H_1$. Then the
word $W\cdot\bar\theta_0$ is completely determined by $H_1$ by Lemma \ref{ppm}.
Similarly, the word $$W''\equiv W\cdot (\bar\theta_0 H_1\bar\theta_1^{\epsilon_1}
H_2\bar\theta_2^{\epsilon_2}...H_s)\bar\theta_0$$ is determined by the same $H_1$. Thus $W'\equiv W'',$
 but then by Remark \ref{*I}, the equal words $\pi_{32}(W')$ and $\pi_{32}(W'')$ are connected by a non-empty
 reduced computation of the machine $\tilde M_2$. This contradicts  Lemma \ref{norep}, and the lemma is proved.
\endproof

\subsection{The machine \label{M4S} $M_4$}\label{M4}

Recall that the set of state letters of $M_3$ is $S_0\sqcup P_1 \sqcup S_1 \sqcup P_2 ...\sqcup S_N$
where $\cup P_i$ contains the control state letters. 
Let \label{T3} $\Theta(3)$ be the set of positive rules of $M_3$ except for the start and the accept rules. We introduce two copies \label{Y3pr} $Y(3)$
and $Y'(3)$ 
of $\Theta(3)$ which will be parts of the tape alphabet of $M_4.$ 
Let maps $\theta\mapsto y_\theta$ and
$\theta\mapsto y'_\theta$ identify $\Theta(3)$ with $Y(3)$ and $Y'(3)$.

The standard base of $M_4$ is $tks_0p_1s_1...p_Ns_Nk't',$ where $s_0p_1s_1...p_Ns_N$ is the standard base of $M_3.$ (Every state letter is called a $q$-letter as earlier, but from now on, we also can 
use $t$-, $k$-, $s$-, $p$-letters or $s_0$-, $p_1$-letters, and so on.)
 As for $M_3$, the $p_1s_1$-sector of an admissible word is
called the \label{inputs}{\em input sector} of that word.

The new parts of the set of state letters are $T=\{t\}$, $K=\{k(1),k(2), k(3)\}$, $K'=\{k'(1),k'(2),k'(3)\}$, $T'=\{t'\}$.

The new sets of state letters are now denoted by $T, K, S_0, P_1,...,P_N, S_N, K', T'$.
The set of tape letters in the $tk$-sector is $Y(3)$, the sets of tape letters in $ks_0$-sector and in the $s_Nk'$-sector
are empty, and the set of tape letters in the $k't'$-sector is $Y'(3)$,  $k\in \{k(1), k(2), k(3)\}$, $k'\in \{k'(1),k'(2),k'(3)\}.$
The tape letters in the other sectors are as
in $M_3$.

The positive rules of the machine $M_4$ are divided into three \label{Steps} {\em Steps}. Each rule below contains subrules $s_0\to
s_0$, $t\to t$ and $t'\to t'$, so we sometimes omit these subrules.

{\bf Step 1.} $$\rho_1(y)=\left[\begin{array}{l}k(1)\tool  yk(1), p^{(1,1)}\to p^{(1,1)}, s_i\tool s_i 
 (1\le i\le N),\\ 
 p^{(1,j)}\tool p^{(1,j)} (2\le j\le N), k'(1)\tool k'(1)\end{array}\right], y\in Y(3)$$
where the $s$- and $p$-letters  form the start vector $\overrightarrow s_1$ for the machine $M_3.$

{\em Comment:} The machine writes the $Y(3)$-copy of a (positive) history word  in the $tk$-sector to the left of $k_1$.
The word between $k$ and $k'$ is an input word of $M_3$. All sectors except for the $tk$-sector and the
$p^{(1,1)}s_1$-sector are locked by the rules of Step 1.

\label{Tr12}{\bf Transition rule (12)} from Step 1 to Step 2 is the `extension' $\theta(M_4)$ of the unique start rule $\theta=\theta_{start} $ of the machine $M_3$:
$$(12)=\left[\begin{array}{l} k(1)\tool k(2), \dots,
s_N\tool s'_N, k'(1)\tool k'(2)\end{array}\right],$$
where the parts of the rule (12) between $k$- and $k'$-letters are the parts
of $\theta_{start}.$

{\em Comment:} After that rule is applied, the machine is ready to execute copies of the machines $\overrightarrow Z^{(\theta_{start},j)}$ and $\overleftarrow Z^{(\theta_{start},j)},$ on tapes 1
through $N$. All sectors except the $tk$-sector and the $p^{(1,1)}s_1$-sector are locked by this rule.

{\bf Step 2.} For every $\theta\in \Theta(3)$:
\label{tM4}
$$\theta(M_4)=[k(2)\tool y_\theta\iv k(2), \theta, k'(2)\to k'(2)y_\theta']$$
{\em Comment:} On tapes $1$ through $N$, the machine executes (backwards) the history written in the $tk$-sector,
erases the word in that sector, and copies it to the $k't'$-sector. 

\label{Tr23}{\bf Transition
rule (23)}=$\theta(M_4)$ from Step 2 to Step 3 `extends' the accept  rule $\theta=\theta_{accept}$ of $M_3$:

$$(23)=[t\tool t, k(2)\tool k(3), \dots s_N\tool s'_N,  k'(2)\to k'(3), t\to t']$$
where the parts of the rule (23) between $k$- and $k'$-letters are the parts of $\theta_{accept}.$

{\em Comment.} All sectors except for the $k't'$-sector are locked by this rule.

{\bf Step 3.}  $$\rho_3(\theta)=[t\tool t, k(3)\tool k(3),\dots, s'_N\tool s'_N,  k'(3)\to k'(3)(y'_\theta)\iv, t'\to
t'],$$
where the state letters between $k$- and $k'$-letters form the accept vector  $\overrightarrow s_0$ of $M_3$ 

{\em Comment:} The machine erases the history from the $k't'$-sector. All other sectors are locked by the rules of
this Step.

For every admissible input word $W\in X_3$ of $M_3$ let \label{pi34} $\pi_{3,4}(W)\in X_4$ be the admissible word of $M_4$ obtained
by adding  state letters $k(1), k'(1), t, t',$ 
hence
$\pi_{3,4}(W)\equiv k(1)tWt'k'(1).$ 
For every input
word $W$ of $M_3$ we call the word $\pi_{3,4}(W)$ an {\em input} word of $M_4$. The stop word of $M_4$, $W_{M_4}$,
 is obtained from the stop word $W_{M_3}$
of $M_3$ by adding state letters $k(3),k'(3), t, t',$ 
i.e., $W_{M_4}\equiv k(3)tW_{M_3}t'k'(3).$

\begin{rk} From now on, we do
not show the indices $(i)$ ($i=1,2,3)$) of the letters $ k,$ and $k'$ assuming that the indices are appropriate for an admissible word.
\end{rk}

\unitlength 1mm 
\linethickness{0.4pt}
\ifx\plotpoint\undefined\newsavebox{\plotpoint}\fi 


For every accepting computation $W\to W\cdot \theta_1\to W\cdot \theta_1\theta_2\to\dots\to W\cdot\theta_1...\theta_n\equiv W_{M_3}$
(where $\theta_i$ are rules and $W$ is an input word for $M_3$) of $M_3$ with history $H\equiv \theta_1\theta_2...\theta_n$, $W\in X_3$, one canonically
constructs a computation of $M_4$:
$\pi_{3,4}(W)\to...\to W_{M_4}$. 
The history of that computation is denoted by \label{Pi34} $\Pi_{3,4}(H)$. That
computation first uses rules of Step 1
and writes a mirror copy of $H'$ (i.e. $H$ without the start and the accept rules) in the alphabet $Y(3)$ in the $tk$-sector,
then executes
rule (12), then executes the computation with history $H'$ on the subword between $k$ and $k'$ while erasing the word
in the $tk$-sector and moving it onto the $k't'$-sector (written in $Y'(3)$). After $H'$ is completed, 
rule (23) is executed, then the $k't'$-sector is erased using rules of Step 3. Let \label{X4} $X_4=\pi_{3,4}(X_3)$. Every word from this set 
of input configurations is accepted by $M_4$.
To simplify the notation, we can include the rules (12) 
and  (23) to Step
2.

Suppose that a history of computation of $M_4$ has the form $H\equiv H_1H_2...,$ where all rules of each $H_i$ belong to the
same Step $j_i$, and $H_i$ is a maximal subword of $H$ with this property. Then we say that the \label{steph}{\em step history} of that computation is $(j_1)(j_2)...$ (or that $H$ is of \label{typeh} type $(j_1)(j_2)...$).
The following lemma is a straightforward consequence of the definition of $M_4$ and will be used without reference throughout the paper.

\begin{lemma}\label{2w} Every 2-letter subword of any step history of a computation of $M_4$ (with any base) is one of the following words: $(1)(2), (2)(1), (2)(3), (3)(2).$ Two consecutive steps are separated by $(12)^{\pm 1}$ or by $(23)^{\pm 1}$, resp., and the letters
of the history neighboring any $(12)$ (or $(23)$) from the left and from the right belong
to different Steps. 

\end{lemma}

\begin{lemma}\label{M400} An admissible word of $M_4$ 
is not in
the domain of the reduced histories of types:

(a) $(1)(2)(1)$ if the base of $W$ has subword $k't'$;

(b) $(3)(2)(3)$ if the base of $W$ has subword $tk$;

(c) $(3)(2)(1)(2)(3)$ if the base of $W$ is standard.
\end{lemma}

\proof
Cases (a) and (b) are almost identical, so suppose that the history $H$ contains a subword $(12)H(12)\iv,$ where $H$ is of type $(2).$
The word
$V\equiv W\cdot (12)$ from the computation that is in the domain of $H$ must have the subword between $k'$ and $t'$ empty
(since it is in the domain of $(12)\iv$).
Similarly, the word $V\cdot H$ must have the subword between $k'$ and $t'$
empty. If $H\equiv \theta_1(M_4)^{\pm 1}...\theta_s(M_4)^{\pm 1}$ 
where $\theta_i$ are positive rules of $M_3$, then the
subword between $k'$ and $t'$ in $V\cdot H$ is  equal to $(y'_{\theta_s})^{\pm 1}...(y'_{\theta_1})^{\pm 1}$. Since
this word is empty, we conclude that $H$ is not reduced, a contradiction.

Suppose that the step history is of the form (c). Then the history $H$ has the form $H_3(23)\iv H_2 (12)\iv H_1 (12)
H_2'(23)H_3',$ where $H_i, H_i'$ contain rules from Step $i$ only. Restricting the computation to the subwords between $k$ and $k'$
of the admissible words, we obtain two reduced accepting computations of $M_3$ with the same initial word from $X_3$ and
histories $H_2\iv, H_2'$ (this follows from the definitions of the rules of Step 2). By Lemma 
\ref{M31}
(b) $H_2\iv\equiv H_2'$. Since every rule of Step 1 multiplies the $tk$-sector of the admissible word by an $a$-letter
uniquely determined by the rule, the $tk$-sectors $A_{tk}, B_{tk}$ in the words
$W\cdot H_3(23)\iv H_2$ 
and $W\cdot H_3(23)\iv H_2 (12)\iv H_1 (12)$ respectively are the same. Since every
rule of Step 1 multiplies that sector by a letter uniquely determined by that rule, we deduce that a copy of the word
$H_1$ multiplied by $A_{tk}$ is $A_{tk}$. Hence $H_1$ is empty, which contradicts the assumption that the computation
is reduced.
\endproof

\begin{lemma}\label{M40} Suppose that $W\equiv tW_1kW_2k'W_3t'$ is an admissible word of $M_4$ with the standard base. 
Suppose that $W$ is in the domain of a reduced history of the form $(12)H(23)$. Then

(1) $H$ contains only rules from Step 2, $(12)H(23)\equiv\theta_1(M_4)...\theta_n(M_4)$ for some rules $\theta_1,...,\theta_n$ of
$M_3$.

(2) The word $W_2$ is from $X_3$ and $\theta_1...\theta_n$ is a computation of $M_3$ accepting $W_2$.
\end{lemma}

\proof Suppose that $H$ contains rules from Step 1 or 3.
Then it contains a subword of one of two forms $(23)\iv
H_1(23)$ or $(12)H_1(12)\iv $ with $H_1$ consisting of rules of Step 2 which contradicts Lemma \ref{M400}.
This implies part (1) of the lemma.

Since $W$ is in the domain of $(12)$, the
subword $W_2$ 
an admissible input word of $ M_3$.
Since $W\cdot (12)H(23)$ is in the domain of $(23)^{-1},$
the subword between $k$ and
$k'$ is the stop word $W_{M_3}$ of $M_3$. This implies part (2) of the lemma. 
\endproof

\begin{lemma}\label{M401} Suppose that $W$ is an admissible word of $M_4$ with the standard base. Then

(a) The step history of any reduced computation starting with $W$ is a subword of $(2)(1)(2)(3)(2)(1)(2)$.

(b) The step history of any accepting reduced computations starting with $W$ is a suffix of the word
$(2)(1)(2)(3)$.
\end{lemma}

\proof Indeed, in every step history $(i_1)(i_2)...(i_s)$ of a reduced computation of $M_4$, after (1) we should have
(2), after (2) we should have (1) or (3), after (3) we should have (2). The statement then follows immediately from
Lemma \ref{M400}. \endproof

\begin{lemma} \label{*IVb} 
Suppose that a history $H$ of a reduced computation of $M_4$ with standard base contains
both $(12)^{\pm 1}$ and $(23)^{\pm 1}$.

(a) The number of occurrences of $(12)^{\pm 1}$ or $(23)^{\pm 1}$ in $H$ is at most 6.

(b) Suppose that the computation is accepting. Then the number of occurrences of $(12)^{\pm 1}$ or $(23)^{\pm 1}$ in $H$ is at most 3.

\end{lemma}

\proof Immediately follows from Lemmas \ref{M401} and \ref{2w}.  \endproof

\begin{lemma}\label{M41} Recall that $X_4$ is 
the set of all words of the form $\pi_{3,4}(W)$, $W\in X_3$.

An input word $W'\equiv \pi_{3,4}(W)$ is accepted by $M_4$ if and only if $W\in X_3$. 
Hence the language accepted by $M_4$ is not
recursive.

\end{lemma}

\proof If $W\in X_3$ then $\pi_{3,4}(W)\in X_4$ since the corresponding computation was constructed together with the definition of $M_4$. 
Let $W'\equiv \pi_{3,4}(W)$ for some admissible input 
word $W$ of $M_3$ and $H$ be the history of an accepting computation
for $W'$. By Lemma \ref{M401} (b), 
$H\equiv H_1(12)H_2(23)H_3,$ where $H_j$ contains only rules of Step $j$ ($j=1,2,3$).
By Lemma \ref{M40}, $W$ is in $X_3$, and $H_2$ corresponds to a computation of $M_3$ accepting $W$. 
\endproof

\label{Tis} \begin{df} \label{Ti} Let $T_1<T_2<...$ be all the times of acceptance of acceptable  input 
words of $M_3$.\end{df}

We will call a computation of $M_4$ \label{standardcM4}
{\it standard}
if
it has the standard base and history of the form $(12)(2)(23).$
The following lemma gives (almost) linear upper bounds for the lengths
of many computations with standard base.

\begin{lemma} \label{*XVII} \label{*IIIa} (a) Suppose that an admissible word $W$ of $M_4$ is accepted by $M_4$.
Suppose that the length of a reduced  accepting computation of $W$ is not in  $\cup_{i=1}^{\infty} ( T_i, 9T_i).$ Then the length of this accepting
computation of $W$ is at most $6|W|_a.$ 

(b) Let $b$ be an integer such that any {\em standard} computation starting with 
a word $W$ with $|W|_a\le b$ has the history of length $<\log b.$
Suppose $W$ is any accepted admissible word for $M_4$ with $||W|| < b.$
Then the time of accepting $W$ by any reduced computation of $M_4$ is at most $4|W|_a+3\log b$.
\end{lemma}

\proof Let $H$ be the reduced history of an accepting computation of $M_4$ with the first word $W$.
By Lemma \ref{M401}, the step history of $H$ is a suffix of $(2)(1)(2)(3)$. 
Hence the
possible Step histories are $(2)(1)(2)(3)$, $(1)(2)(3)$, $(2)(3)$ or $(3)$. We shall prove (a) and (b) in each of these
cases.

Suppose that the step history of $H$ is $(2)(1)(2)(3)$.
Then $$H\equiv H_2(12)\iv H_1(12) H_2'(23)H_3,$$ where $H_2, H_2'$ consist of rules of Step 2, $H_1$ (resp.
$H_3$) consists of rules of Step 1 (resp. Step 3). By Lemma \ref{M40},
the length of $(12)H_2'(23)$ is one of the $T_i$.
Since every rule of Step 2 multiplies the $tk$-sector by a letter uniquely determined by that rule, and in any word in
the domain of (23), the $tk$-sector is empty, we conclude that the $tk$-sector of the word $W\cdot H_2(12)\iv H_1(12)$
is a copy of $H_2'$, hence its length is $T_i-2$. The $k't'$-sector of that word is empty and every rule from $H_2'$
multiplies that sector by a letter uniquely determined by the rule. Hence the $k't'$-sector of $W\cdot H_2(12)\iv
H_1(12)H'_2$ has length $T_i-2$. Since every rule of Step 3 multiplies that sector by a letter, and in the stop
word of $M_4$ that sector is empty, we conclude that $||H_3||=T_i-2$. Hence $||(12) H_2'(23)H_3||\le 2T_i-2$. 

Note that since every rule of Step 2 multiplies the $k't'$-sector by a letter uniquely
determined by that rule, and in a word in the domain of $(12)\iv$ that sector is empty, we can conclude that $||H_2||\le |W|_a$. Similarly since the rules of Step 1 multiply the $tk$-sector by letters uniquely determined by these rules, we
conclude that $$||H_1||\le ||H_2||+|W|_a+||H_2'||\le 2|W|_a+T_i-2.$$ 
(We use that if a group word $U$ of length $l$ is obtained from a word $V$ of lengths $k$ after a series of one-side multiplications by one letter, and successive multiplications are not mutual inverse, then the number of multiplications does not exceed $k+l$.)
Therefore $||H||\le 2T_i-2+||H_2||+||H_1||+1< 3T_i+3|W|_a,$  and
since in the case under consideration, we have  $||H||\ge 9T_i$ by the condition of the lemma,
it follows that   $|W|_a\ge (||H||-3T_i)/3
>||H||/6$, as required for the part (a).

Now assume that the assumption of $(b)$ holds. Note that every rule of $H_2$ multiplies the $k't'$-sector by a letter
and the input $p_1s_1$-sector also by at most one letter,
the rules of $H_1$ do not touch the input sector. Therefore
the input sectors or $W\cdot H_2$ and $W\cdot H_2(12)\iv H_1(12)$ are the same and their lengths are at most the sum
of lengths of the input sector of $W$ and the $k't'$-sector of $W$. Hence the length of the input sector of
$W\cdot H_2(12)\iv H_1(12)$ does not exceed $|W|_a\le b$. By the condition of the lemma, we have that
$||H_2'||=T_i-2\le \log b-2$ for some $i$. As before $||H||\le 3T_i+3|W|_a\le 3|W|_a +3\log b.$
Suppose that the step history of $H$ is $(1)(2)(3)$, that is
$H=H_1(12)H_2(23)H_3,$ where $H_i$ contains only rules of Step $i$
($i=1,2,3$). Then again by Lemma \ref{M40} $||H_2||=T_i-2$ for some $i$,
and the length of the $tk$-sector in $W\cdot H_1$ is $T_i-2$. As in
the previous paragraph, $||H_3||=T_i-2$.

Under the assumptions of (a) then $||H_1||>9T_i-2T_i+2>7T_i$.
Since every rule of $H_1$ multiplies the $tk$-sector by a
letter, we also have that $||H_1||$ does not exceed the sum of lengths of $tk$-sectors in
$W$ and in $W\cdot H_1$, whence $||H_1||\le |W|_a+T_i-2.$  
Therefore $|W|_a>6T_i$ and
$$||H||=||H_1||+2T_i-2 <  2|W|_a.$$

Suppose that the assumptions of (b) hold. Then since the input sectors of $W$ and $W\cdot H_1(12)$ are the same,
and their length is $\le |W|_a<b$, we conclude that $T_i\le \log b$, and $$||H||\le ||H_1||+||H_2||+||H_3||+2\le
|W|_a+T_i+T_i+T_i+2\le 2|W|_a+3\log b.$$

Suppose that the step history is $(2)(3)$, that is $H\equiv H_2(23)H_3$, and, again, $H_i$ has rules only from Step $i$,
$i=2,3$. Note that every rule of $H_2$ multiplies the $tk$-sector by a letter,
and that sector in any word which is a
domain of $(23)$ must be empty. Hence $||H_2||\le |W|_a$. Every rule in $H_2, H_3$ multiplies the $k't'$-sector by a
letter, hence $||H_3||\le |W|_a+||H_2||\le 2|W|_a$. Therefore $H\le ||H_2||+||H_3||+1\le 3|W|_a+1\le 4|W|_a$. This implies
both (a) and (b).

Finally suppose that the step history is $(3)$. Then clearly $||H||\le |W|_a$, and both (a) and (b) follow.

We conclude that in every case both (a) and (b) hold. 
\begin{lemma}\label{*IIIb} Suppose that an admissible word of $M_4$ is accepted by $M_4$, $H$ is a history of an
accepting computation. Then $||W||_a\le 4||H||$.
\end{lemma}

\proof Indeed, $W\cdot H$ does not contain $a$-letters, and each rule of $H$ decreases the number of $a$-letters in
the admissible word by at most $4$
(every rule of $M_4$  affects at most four $a$-letters: two letters in the subword between $k$ and
$k'$, one letter in the subword between $t$ and $k$ and one letter in the subword between $k'$ and $t'$). \endproof

\label{activel} \begin{df} \label{active}  Let $Q_i$ be a base letter. (Recall that usually
we take a representative $q_i\in Q_i.$) We say that
$Q_i$ (or $q_i$)
is {\em active from
the left (resp., from the right)} for a  rule $\theta$
if in the corresponding component
$v_{i-1}q_iu_i\to v_{i-1}'q_i'u'_i$ of $\theta$,
the word $v_{i-1}\iv v_{i-1}$ (resp. $u'_iu_i\iv$) is not trivial (and so equal to a letter
the free group by Property \ref{one} (1) of $M_4$). 
If $q_i$ is active from the left (right) for $\theta$, then we say that $q_i\iv$ is active from the right (left) for $\theta$. We also say that $q_i$ active for $\theta$ if it is 
active from the left or active from the right. Otherwise $q_i$ is \label{passive} {\it passive} for
$\theta.$
\end{df}

\begin{lemma}\label{*IVa} Let a reduced computation of $M_4$ have history $(12)H$ 
and have the base $s_0p_1.$
Suppose that the letter $p_1$ is active
in every rule $\theta$ of step 2 from $H.$  Then every rule of $H$ is of Step 2.
\end{lemma}

\proof Recall, that the rule  $(12)$ extends the start rule $\theta=\theta_{start}$
of $M_3.$ Therefore the first rule of $H$ has $p^{\theta,1}(1)$ in the left-hand side. If
no rule of $H$ changes $p^{\theta,1}(1),$ then every rule is (the extension of) a rule of the
machine $\overrightarrow Z^{\theta,1}$ with the $p_1$-part of the form $p^{\theta,1}(1)\to a'p^{\theta,1}(1)(a'')^{-1},$ where $p_1$ is active from the both sides. Otherwise 
the history has a subword of
type either $(12)H'(12)^{-1}$ 
where  $p_1$-part of every rule of $H'$ is of form  $p^{\theta,1}(1)\to a'p^{\theta,1}(1)(a'')^{-1}$ because the (the copy of the ) rule $\xi_2$ of  $\overrightarrow Z^{\theta,1}$
belongs to Step 2 but it is passive.
However this case is impossible since then every rule of $H'$
inserts (or deletes) one letter $a'$ in the $s_0p_1$-sector from the right, different rules insertes
different letters, and the $s_0p_1$-sector is empty when the rule $(12)$ or
$(12)^{-1}$ is applicable.
\endproof

Every rule either makes a control state letter $p_i$ active from both sides or 
locks a neibor sector. This property is useful for computations with non-standard bases
as in the following

\begin{lemma} \label{*IX} If the base of a reduced computation of $M_4$ contains a subword $(pp\iv p)^{\pm 1},$ where $p$
is a control state letter, then all rules of the computation correspond to the copy of the $S$-machine $\overrightarrow Z$ or of $\overleftarrow Z$ containing
that state letter, and either every rule is a copy  of some $\xi_1(a)^{\pm 1}$ ($a$ depends on the rule) or every rule is a copy of some $\xi_3(a)^{\pm 1}.$
\end{lemma}

\proof Indeed, suppose that $p\in P_i$. Then every rule not from the copy of $\overrightarrow Z$ or of $\overleftarrow Z$ containing that state letter, locks
the sector $s_{i-1}p_i$ or the sector $p_is_i$, and the copies of $\xi_2$ and of $\xi_4$ lock 
either sector $s_{i-1}p_i$ or sector $p_is_i.$ Now we can apply Lemma \ref{qqiv}.
\endproof

Lemma \ref{*XVII} and the following lemma show the role of the 'historical' $tk$- and $k't'$-sectors.

\begin{lemma} \label{*XV} Suppose that a reduced computation of $M_4$ with the standard base has the history of the form
$(12)H_2(23)H_3(23)\iv H_2'(12)\iv,$ where $H_2, H_2'$ contain rules from Step 2, and $H_3$ has rules of step 3. Then
$||H_3||\le ||H_2||+||H_2'||.$
\end{lemma}

\proof Let $W$ be the initial word of the computation. Since every rule of $H_2, H_2'$ multiplies the $k't'$-sector by
one letter which determines the rule, and every word in the domain of $(12)^{\pm 1}$ has that sector empty, we
conclude that 
$||H_2||$ is equal to the length of the $k't'$-sector $U$ of $W\circ (12)H_2$,
and 
$||H_2'||$  is equal to the
length of the $k't'$-sector $V$ in $W\circ (12)H_2(23)H_3$. Similarly, every rule from $H_3$ multiples the
$k't'$-sector by a letter that determines the rule. 
Hence $||H_3||\le ||U||+||V||=||H_2||+||H_2'||$ as required.  \endproof

\label{MS} \subsection{The machine $M$}

Consider now $2L\gg 1$ copies of the machine $M_4$, denote them by $M_{4}(i)$, $i=1,\dots,2L$,
 We denote the state and
tape letter of \label{M4iS} $M_{4}(i)$ accordingly, by adding index $i$
to all letters, and all rules. Let \label{TM4} $\Theta (M_4)$
be the set
of positive rules of $M_4$. Let \label{Bst} $B$ be the standard base of $M_3$,
\label{Bi} $B(i)$ be the copy word $B$ with new extra index $i$
added to all letters, and $B_i=k(i)B(i)k'(i)$. We now consider the $S$-machine $M$ with the rules $$\theta(M)=[\theta(1),..,\theta(2L)] ,\;\;
\theta\in \Theta (M_4)$$
(we shall denote $\theta(M)$ by $\theta$ also) and the standard base
\begin{equation}\label{baza}
t_1B_1t_2' B_2\iv t_3 B_3 t_4'B_4\iv ...t'_{2L}B_{2L}^{-1}t_{2L+1},
\end{equation}
where we identify the state $t$-letters $t(1)$ and $t'(1)$ of $M_4(1)$
with $t_1$ and $t_2,$ resp., the $t$-letters $t(2)$ and $t'(2)$ of $M_4(2)$
with $t_3^{-1}$ and $(t'_2)^{-1},$ resp, an so on. Moreover we identify $t_{2L+1}$ with $t_1$ and consider the standard base of $M$ up to
cyclic permutations which may start with any $t$-letter and end with the same $t$-letter.
The stop word is defined accordingly (every
letter in the standard base is replaced by the corresponding letter in the stop word of $M_4(i)$). The stop word without the last letter $t_{2L+1}$ is
called the \label{hub} {\em hub}. We also may take the hub up to cyclic permutations. 

That construction is similar to the construction in \cite{SBR} and \cite{OScol}, though the application of mirror copies of machines goes back to the
works of Boon and P.S. Novikov (see \cite{R}). The condition $L>>1$ makes hub graph
hyperbolic (see Lemmas \ref{extdisc} and \ref{mnogospits}), and the mirror
symmetry of the word (\ref{baza}) is used for the surgery we define in Subsection \ref{srh}.

For every admissible word $W$ of $M_4$ with the standard base we denote by \label{WM} $W(M)$ the corresponding admissible word $t_1k(1)W(1)k'(1)t_2'k'(2)^{-1}W(2)^{-1}k(2)^{-1}t_2\dots$   of
$M$ with the standard base (of $M$). By definition, $W(M)$ is an input (the accept) word
of $M$ if $W$ is an input (the accept) word of $M_4.$

 The letters in the copy \label{Wi} $W(i)$ of the word $W$ are
equipped with the extra index $(i).$ Thus every $a$-letters 
and every $q$-letter (except for $t$ and $t'$-letters), and every letter $\theta(i)$ of the alphabets of $M$ has
this extra index. We call it  the \label{Mind} $M$-{\it index} of the letter and take it modulo $2L.$

\begin{rk}\label{rkM} (1) Notice that for every rule $\theta$ of $M_4$ and every admissible word $W$ of $M_4$ with the
standard base of $M_4$, we have $W\cdot \theta=W'$ if and only if $W(M)\cdot \theta=W'(M).$ 

(2) Also notice that
$||W(M)||<2L||W||$
for every $W$.

(3) Both machines $M_4$ and $M$ enjoy Property \ref{one} (1) but not Property \ref{one} (2).

(4) The unique start and accept rules of the machines $M_1,$ $M_2,$ $M_3$ are converted
to the transition rules $(12)$ and $(34)$ of $M_4$ and $M.$ So there are no specific
start and accept rules of $M_4$ and $M.$ In particular $M$ accepts if it reaches the
hub.  

\end{rk}

Remark \ref{rkM} (1) and Lemma \ref{M41}
immediately imply

\begin{lemma} \label{abc} Let \label{X5} $X_5$ be the set of all words of the form $W(M)$, $W\in X_4$. Then for every input
word $W$ of $M_4,$
$W$ is accepted by $M_4$ if and only if 
$W(M)$ is accepted by $M$ and if and only if $W\in X_4$. Hence the set $X_5$ is not recursive.
\end{lemma}

\begin{rk}\label{rkhubsym}\label{*XVI} Considered as a cyclic word, the hub has the following symmetries: it does not
change if we reflect it about any $t$-letter or any $t'$-letter (with indices changing appropriately, and state
letters, except for $t$- and $t'$-letters, replaced by their inverses).
From Remark \ref{rkM} (1), it follows, that every admissible accepted word of $M$ has
similar symmetries.
\end{rk}

\section{Groups and diagrams}\label{gd}

Every S-machine can be considered as a finitely presented group (see \cite{SBR} and
also \cite{OS}, \cite{OS3}). Here we apply the construction to the machine $M.$
To simplify formulas, it is convenient to redefine $N$ once again. From now on we shall denote by $N+1$ the length of
the smallest subword of the hub
containing two $t$-letters.
Thus the length of the hub
is $LN$, $Q=\sqcup_{i=0}^{LN}Q_i$ (where $Q_{LN}=Q_0$) $Y= \sqcup_{i=1}^{LN} Y_i,$ and $\Theta$ is the set of rules of the S-machine $M.$ (But we will remember
that, as for the machine $M_4$, the state letter of $M$ are partitioned into the subsets
of \label{tl} $t$-letters, \label{t'l} $t'$-letters, \label{kl} $k$-letters, \label{k'l} $k'$-letters, \label{sl} $s$-letters, \label{pletter} 
and $p$-letters.)

The finite set of generators of the \label{groupM} group $M$ (the same letter as for the machine) consists of \label{qletter}{\em $q$-letters} corresponding to the states
$Q$, \label{aletter}{\em $a$-letters} corresponding to the tape letters from $Y,$ and
\label{thetal}{\em $\theta$-letters} corresponding to the rules from the positive part $\Theta^+$ of $\Theta.$

The \label{relations} relations of the group $M$ correspond to the rules of the machine $M$;
for every $\theta=[U_0\to V_0,\dots U_{LN}\to V_{LN}]\in \Theta^+$, we have
\begin{equation}\label{rel1}
U_i\theta_{i+1}=\theta_i V_i,\,\,\,\, \qquad \theta_j a=a\theta_j, \,\,\,\, i,j=0,...,LN
\end{equation}
for all $a\in \bar Y_j(\theta)$. (Here $\theta_{LN}\equiv\theta_0.$)
The first type of relations will be
called \label{thetaqr} $(\theta,q)$-{\em relations}, the second type - \label{thetaar}
$(\theta,a)$-{\em relations}. 

Finally, the required \label{groupG} group $G$ is given by the generators and
relations of the group $M$ and by one more additional
relation, namely the {\it hub}-relation  
\begin{equation}\label{rel3}
W_M=1,
\end{equation}
where
$W_M$
is the hub, i.e., the accept word (of length $LN$) of the machine $M.$  

\begin{rk}\label{relsym} The word $W_M$ has the symmetries
mentioned in Remark \ref{rkhubsym}. Since the machine $M$ is built of
$2L$ copies $M_4(i),$ the set of relations is also symmetric in the
following sense. Every relation $\theta_ja=a\theta_j$ from (\ref{rel1}) has $2L$ copies
(including itself) corresponding to different $M_4(i)$-s. If
the relation $U_i\theta_{i+1}=\theta_i V_i$ from (\ref{rel1}) 
involves neither $t$- nor $t'$-letters then it has $L$ copies (including
itself) and $L$ mirror copies. Every relation containing a $t$- or a a
$t'$-letter (denote this letter by $\tilde t$) has form $\tilde t\theta_{i+1}=\theta_i \tilde t,$
i.e., it contains no $a$-letters.

\end{rk}

\subsection{Minimal diagrams}\label{md}

Recall that a van Kampen \label{diagram} {\it diagram} $\Delta $ over a presentation
$P=\langle A\; | \; \mathcal R\rangle$ (or just over the group $P$)
is a finite oriented connected and simply--connected planar 2--complex endowed with a
\label{Lab} labeling function $\Lab : E(\Delta )\to A^{\pm 1}$, where $E(\Delta
) $ denotes the set of oriented edges of $\Delta $, such that $\Lab
(e^{-1})\equiv \Lab (e)^{-1}$. Given a \label{cell} cell (that is a 2-cell) $\Pi $ of $\Delta $,
we denote by $\partial \Pi$ the boundary of $\Pi $; similarly, \label{partial}
$\partial \Delta $ denotes the boundary of $\Delta $. The labels of
$\partial \Pi $ and $\partial \Delta $ are defined up to cyclic
permutations. An additional requirement is that the label of any
cell $\Pi $ of $\Delta $ is equal to (a cyclic permutation of) a
word $R^{\pm 1}$, where $R\in \mathcal R$. The label and the \label{clength} combinatorial length $||\bf p||$ of
a path $\bf p$ are defined as for Cayley graphs.

The van Kampen Lemma states that a word $W$ over the alphabet $A^{\pm 1}$
represents the identity in the group $P$ if and only
if there exists a diagram $\Delta
$ over $P$ such that 
$\Lab (\partial \Delta )\equiv W,$ in particular, the combinatorial perimeter $||\partial\Delta||$ of $\Delta$ equals $||W||.$
(\cite{LS}, Ch. 5, Theorem 1.1). The word $W$ representing $1$ in $P$ is freely equal
to a product of conjugates to the words from $R^{\pm 1}$. The minimal number
of factors in such products is called the \label{areaw} {\em area} of the word $W.$ The \label{aread}{\it area}
of a diagram $\Delta$ is the number of cells in it. 
A diagram having
the smallest number of cells among all diagrams with the same boundary label
is called \label{minimald} {\it minimal}.
By van Kampen's Lemma, $\area(W)$ is equal
to the 
area of a minimal diagram $\Delta$ over $P$ with $\Lab (\partial \Delta )\equiv W.$
This definitions imply

\begin{lemma}\label{2diagr} Assume that a diagram $\Delta_0$ is divided
into two subdiagrams $\Delta_1$ and $\Delta_2$ by a simple path $p.$ Let
a minimal
diagram $\Delta$ have  the same boundary label as $\Delta_0.$ Then 
$\area(\Delta)\le \area(\Delta_0)=
\area(\Delta_1)+\area(\Delta_2).$
\end{lemma}

We will study diagrams over the groups $M$ and $G$. The edges labeled by state
letters ( = $q$-{\it letters}) will be called \label{qedge} $q$-{\it edges}, the edges labeled by tape
letters (= $a$-{\it letters}) will be called \label{aedge} $a$-{\it edges}, and the edges labeled by 
$\theta$-letters are \label{thedge} $\theta$-{\it edges}. 

\begin{rk} \label{diasym} The symmetries of relations observed in Remark \ref{relsym} makes possible the following construction for given $i\le 2L$ and a diagram $\Delta$ over $M.$
Let $\nabla$ be a mirror copy of the map $\Delta.$ For every edge $e$ of $\Delta$ whose
label is equipped with an $M$-index $(j)$ (i.e., if $e$ is neither $t$- nor $t'$-edge), the mirror copy of $e$ in $\nabla$ is marked
by the same letter but with $M$-index equal to $(2i-j-1).$ The label  $t_j$ of $e$
(the label $t'_j$) should be replaced for the mirror image by $t_{2i-j}^{-1}$ (resp.,
by $(t'_{2i-j})^{-1}$). It is easy to see that $\nabla$ is also is a diagram over $M.$
We say that $\nabla$ is obtained by  \label{tlrefl} $t_i$-{\it reflection} from $\Delta.$  
Similarly one can speak on $t_i$-reflections for paths of $\Delta.$
\end{rk}

We denote by $|\bf p|_a$ (by $|\bf p|_{\theta}$, by
$|\bf p|_q$)
the \label{alength} $a$-{\it length} (resp., the \label{thlength} $\theta$-{\it length}, the \label{qlength} $q$-length) of a path/word $\bf p,$ i.e., the number of
$a$-edges/letters (the number of $\theta$-edges/letters, the number of $q$-edges/letters) in $\bf p.$

 The cells corresponding
to Relation (\ref{rel3})  are called \label{hubs} {\it hubs}, the cells corresponding
to $(\theta,q)$-relations are called \label{tq} $(\theta,q)$-{\it cells},
and they are called \label{ta} $(\theta,a)$-{\it cells} if they correspond to $(\theta,a)$-relations.




Every minimal van Kampen diagram is \label{reducedd}{\em reduced}, that is 
it does not contain two cells (= closed $2$-cells) that have a
common edge and are mirror images of each other (if such pairs of cells exist, they can be removed to obtain a  diagram of smaller area and with the same boundary label). 
To study (van Kampen) diagrams
over the group $G$ we shall use their simpler subdiagrams such as bands and trapezia, as in \cite{O1},  \cite{SBR}, \cite{BORS}, etc.
 Here we repeat one more necessary definition.

\label{band} \begin{df}Let $\cal Z$ be a subset of the set of generators ${\cal X}$ of the group $M$. A
$\cal Z$-band $\bb$ is a sequence of cells $\pi_1,...,\pi_n$ in a reduced \vk
diagram $\Delta$ such that

\begin{itemize}
\item Every two consecutive cells $\pi_i$ and $\pi_{i+1}$ in this
sequence have a common edge $e_i$ labeled by a letter from $\cal Z$.
\item Each cell $\pi_i$, $i=1,...,n$ has exactly two $\cal Z$-edges,
$e_{i-1}$ and $e_i$ (i.e. edges labeled by a letter from $\cal Z$).

\item If $n=0$, then $\bb$ is just a $\cal Z$-edge.
\end{itemize}
\end{df}

The counterclockwise boundary of the subdiagram formed by the
cells $\pi_1,...,\pi_n$ of $\bb$ has the factorization $e\iv {\bf q}_1f {\bf q}_2\iv$
where $e=e_0$ is a $\cal Z$-edge of $\pi_1$ and $f=e_n$ is an $\cal Z$-edge of
$\pi_n$. We call ${\bf q}_1$ the \label{bottomb}{\em bottom} of $\bb$ and ${\bf q}_2$ the
\label{topb}{\em top} of $\bb$, denoted \label{bott} $\bott(\bb)$ and \label{topp} $\topp(\bb)$.
Top/bottom paths and their inverses are also called the \label{sideb}{\em
sides} of the band. The $\cal Z$-edges $e$
and $f$ are called the \label{seedgesb}{\em start} and {\em end} edges of the
band. If $n\ge 1$ but $e=f,$ then the $\cal Z$-band is called a \label{annulus} $\cal Z$-{\it annulus}.

We will consider \label{qband} $q$-{\it bands}, where $\cal Z$ is one of the sets $Q_i$ of state letters
for the machine $M$, \label{thband}
$\theta$-{\it bands} for every $\theta\in\Theta$, and \label{aband} $a$-{\it bands}, where
${\cal Z}=\{a\}\subseteq Y$. 
 The convention is that $a$-bands do not
contain $(\theta,q)$-cells, and so they consist of $(\theta,a)$-cells  only.

\begin{rk} \label{tb} To construct the top (or bottom) path of a band $\cal B$, at the beginning
one can just form a product ${\bf x}_1\dots {\bf x}_n$ of the top paths ${\bf x}_i$-s of the cells $\pi_1,\dots,\pi_n$ (where each $\pi_i$ is a $\cal Z$-bands of length $1$).
No  $\theta$-letter is being canceled in the word
$W\equiv\Lab(x_1)\dots\Lab(x_n)$ if $\cal B$ is  a $q$- or $a$-band since
otherwise two neighgbor cells of the band would be mirror copies of each other
which is impossible in a reduced diagram. 

Also there are no cancellations of $a$-letters if $\cal B$ is a $q$-band. Indeed
if both $\pi_i$ and $\pi_{i+1}$ have $a$-edges on their top then the corresponding
rules of $M$ must belong to the same Step since every cell is passive for $(12)$-
and $(23)$-rules. Similarly they correspond to the rule of the same machine
$\overrightarrow Z^{(\theta,i)}$ or $\overleftarrow Z^{(\theta,i)}$ if 
$\cal B$ is a $q$-band for some control letter $p_i$ since the rules $\xi_2$ and
$\xi_4$ provide no active cells. Then the rules are determined by the $a$-letters,
and the cells should be mirror copies as in the previous paragraph. Similar
argument works if $q$ corresponds to any other letter of the standard base except
for $s_i.$ But active $s_i$-cell cannot have a common edge too since this
edge has a $\theta$-index in the label, and so the diagram is not reduced again.

Thus, if $\cal B$ is a $q$-band (or an $a$-band), then the top/bottom label is a product  ${\bf x}_1\dots {\bf x}_n.$
If $\cal B$ is a $\theta$-band then a few cancellations of $a$-letters (but not $\theta$-letters) are possible in $W.$ (This can happen if one of $\pi_i, \pi_{i+1}$
is a $(\theta,q)$-cell and another one is a $(\theta,a)$-cell.) We will always assume
that the top/bottom label of a $\theta$-band is a reduced form of the word $W$.
This property  is easy to achieve: by folding edges
with the same labels having the same initial vertex, one can make
the boundary label of a subdiagram in a \vk diagram reduced (e.g., see
\cite{SBR}).

\end{rk}

If the path $(e\iv {\bf q}_1f)^{\pm 1} $ or the path $(f {\bf q}_2\iv e\iv)^{\pm 1}$
 is the subpath of the boundary path of $\Delta$ then the band is called
 a \label{rimb}{\it rim} band of $\Delta.$ We shall call a $\cal Z$-band \label{maxb}{\em maximal} if it is not contained in
any other $\cal Z$-band.
Counting the number of maximal $\cal Z$-bands
 in a diagram we will not distinguish the bands with boundaries
 $e\iv {\bf q}_1f {\bf q}_2\iv$ and $f {\bf q}_2\iv e\iv {\bf q}_1,$ and
 so every cell having two $\cal Z$-edges belongs to a unique maximal $\cal Z$-band.

We say that a ${\cal Z}_1$-band and a ${\cal Z}_2$-band \label{cross}{\em cross} if 
they have a common cell and ${\cal Z}_1\cap {\cal Z}_2=\emptyset.$

Sometimes we specify the types of bands as follows. A $\theta$-band
corresponding to the transition rule $(12)$ (to $(23)$) is called a
\label{12band}(12)-{\it band} ((23)-{\it band}), and it consists of \label{12cell}(12)-cells (of (23)-cells). A $q$-band corresponding to one
of the letters $t_i$ of the base (\ref{baza}) (resp., to 
$t'_i,$  $k(i),$ $k'(i)$) is called a \label{tband} $t$-{\it band} ( \label{t'band} $t'$-{\it band},
  \label{kband} $k$-{\it band}, \label{k'band} $k'$-{\it band}) since the $M$-index $i$ is, generally,
  not important for further considerations (but we may keep it if it is essential). 
  Similarly, we can omit the
  $M$-index speaking on \label{sband} $s$- and \label{pband} $p$-bands,
  but we distinguish different letters of each
  particular $B_i$ in the standard base, e.g., the $s_0$-letter follows after the $k$-letter
  in each subword $B_i$ hence the standard base (\ref{baza})
  has $2L$ different $s_0$-letters (one in each subword $B_i$), $2L$ different 
  $p_1$-letters, and so on. Also this agreement allows to speak on
  \label{12letter} $(12)$-letters and $(12)$-edges, \dots, \label{piletter} $p_i$-letters, $p_i$-edges,
  and $p_i$- (or \label {siband} $s_i$-) bands.

The papers \cite{O}, \cite{BORS}, \cite{OS2} contain the proof of the
following lemma in a more general setting. (In contrast to Lemmas 6.1 \cite{O} and
 3.11 \cite{OS2}, we have no $x$-cells here.)

\begin{lemma}\label{NoAnnul}
A reduced van Kampen diagram $\Delta$ over $M$ has no
$q$-annuli, no $\theta$-annuli, and no  $a$-annuli.
Every $\theta$-band of $\Delta$ shares at most one cell with any
$q$-band and with any $a$-band. 

\end{lemma}
\medskip

If $W\equiv x_1...x_n$ is a word in an alphabet $X$, $X'$ is another
alphabet, and $\phi\colon X\to X'\cup\{1\}$ (where $1$ is the empty
word) is a map, then $\phi(W)\equiv\phi(x_1)...\phi(x_n)$ is called the
\label{projectw}{\em projection} of $W$ onto $X'$. We shall consider the
projections of words in the generators of $M$ onto
$\Theta$ (all $\theta$-letters map to the
corresponding element of $\Theta$,
all other letters map to $1$), and the projection onto the
alphabet $\{Q_0\sqcup \dots \sqcup Q_{LN-1}\}$ (every
$q$-letter maps to the corresponding $Q_i$, all other
letters map to $1$).

\begin{df}\label{dfsides}
{\rm  The projection of the label
of a side of a $q$-band onto the alphabet $\Theta$ is
called the \label{historyb}{\em history} of the band. The Step history of this projection
is the \label{stephb}{\it Step history} of the $q$-band. The projection of the label
of a side of a $\theta$-band onto the alphabet $\{Q_0,...,Q_{LN-1}\}$
is called the \label{baseb} {\em base} of the band, i.e., the base of a $\theta$-band
is equal to the base of the label of its top or bottom.}
\end{df}
As for words, we will  use representatives of
$Q_j$-s in base words. (If $p\in Q_4$, $s\in Q_5$, we shall
say that the word $pas$ has base $ks$ instead of
$Q_4Q_5$, and so on.)

\begin{df}\label{dftrap}
{\rm Let $\Delta$ be a reduced  diagram over $M$
which has  boundary path of the form ${\bf p}_1\iv {\bf q}_1{\bf p}_2{\bf q}_2\iv,$ where
${\bf p}_1$ and ${\bf p}_2$ are sides of $q$-bands, and
${\bf q}_1$, ${\bf q}_2$ are maximal parts of the sides of
$\theta$-bands such that $\Lab({\bf q}_1)$, $\Lab({\bf q}_2)$ start and end
with $q$-letters.

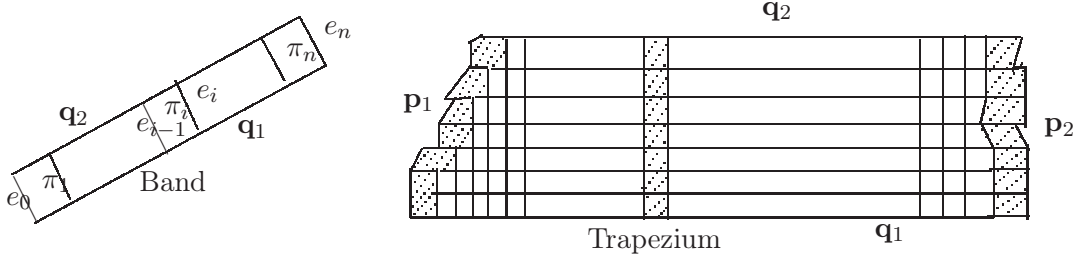
\begin{figure}[h!]
\unitlength 1mm 
\linethickness{0.4pt}
\ifx\plotpoint\undefined\newsavebox{\plotpoint}\fi 
\begin{picture}(148.25,40)(7,115)
\put(76.5,148){\line(1,0){64.25}}
\put(74.75,143.75){\line(1,0){66.25}}
\put(72.75,140){\line(1,0){68.25}}
\put(71,136.5){\line(1,0){69.75}}
\put(69.5,133.25){\line(1,0){71.75}}
\put(68,130.25){\line(1,0){73.25}}
\put(66.75,127.25){\line(1,0){74.5}}
\put(65.5,124){\line(1,0){75.75}}
\put(76.75,147.75){\line(0,-1){23.5}}
\put(74.25,143.75){\line(0,-1){19.75}}
\put(72.25,140){\line(0,-1){16}}
\put(70,133.5){\line(0,-1){9.5}}
\put(67.5,130.5){\line(0,-1){6.5}}
\put(73.5,148){\line(1,0){3.25}}
\multiput(73.5,148.25)(-.0583333,-.0333333){30}{\line(-1,0){.0583333}}
\put(72,147.25){\line(0,-1){2.75}}
\put(72,144.5){\line(0,1){.25}}
\put(72,144.75){\line(0,-1){.75}}
\put(72,144){\line(1,0){2.25}}
\multiput(72,144.25)(-.033653846,-.043269231){104}{\line(0,-1){.043269231}}
\put(68.5,139.75){\line(1,0){1.75}}
\multiput(70,140)(-.03358209,-.05223881){67}{\line(0,-1){.05223881}}
\put(67.75,136.5){\line(1,0){3.25}}
\multiput(69.5,139.75)(.3125,.03125){8}{\line(1,0){.3125}}
\put(67.75,136.5){\line(0,-1){3.25}}
\put(67.75,133.25){\line(1,0){2}}
\put(65.75,133.25){\line(1,0){2}}
\multiput(65.5,133.25)(-.03333333,-.06666667){45}{\line(0,-1){.06666667}}
\put(64,130.25){\line(1,0){3.25}}
\put(64,130.5){\line(0,-1){6.5}}
\put(64,124){\line(0,1){0}}
\put(64,124){\line(1,0){1.75}}
\put(79.25,147.75){\line(0,-1){23.75}}
\multiput(73.68,147.68)(-.375,-.5){5}{{\rule{.4pt}{.4pt}}}
\multiput(74.93,147.93)(-.45,-.75){11}{{\rule{.4pt}{.4pt}}}
\multiput(76.18,147.68)(-.5,-.77941){18}{{\rule{.4pt}{.4pt}}}
\multiput(76.43,145.93)(-.5,-.5833){4}{{\rule{.4pt}{.4pt}}}
\multiput(73.68,141.43)(-.54167,-.75){19}{{\rule{.4pt}{.4pt}}}
\multiput(66.68,133.18)(-.45,-.5){6}{{\rule{.4pt}{.4pt}}}
\multiput(71.93,137.18)(-.55357,-.75){15}{{\rule{.4pt}{.4pt}}}
\multiput(71.93,135.68)(-.375,-.5625){5}{{\rule{.4pt}{.4pt}}}
\multiput(69.68,132.18)(-.4167,-.5833){4}{{\rule{.4pt}{.4pt}}}
\multiput(67.43,128.93)(-.54167,-.58333){7}{{\rule{.4pt}{.4pt}}}
\multiput(66.93,127.18)(-.5,-.55){6}{{\rule{.4pt}{.4pt}}}
\multiput(67.18,125.68)(-.5833,-.5833){4}{{\rule{.4pt}{.4pt}}}
\put(95,148){\line(0,-1){24.25}}
\put(98.25,148){\line(0,-1){24.25}}
\multiput(96.93,147.93)(-.4375,-.625){5}{{\rule{.4pt}{.4pt}}}
\multiput(97.93,146.93)(-.45833,-.66667){7}{{\rule{.4pt}{.4pt}}}
\multiput(97.68,144.18)(-.5,-.7){6}{{\rule{.4pt}{.4pt}}}
\multiput(97.93,142.43)(-.42857,-.57143){8}{{\rule{.4pt}{.4pt}}}
\multiput(97.68,140.43)(-.5,-.7){6}{{\rule{.4pt}{.4pt}}}
\multiput(97.93,138.43)(-.55,-.75){6}{{\rule{.4pt}{.4pt}}}
\multiput(97.68,136.18)(-.5,-.7){6}{{\rule{.4pt}{.4pt}}}
\multiput(97.68,133.93)(-.5,-.75){6}{{\rule{.4pt}{.4pt}}}
\multiput(97.68,131.18)(-.5,-.65){6}{{\rule{.4pt}{.4pt}}}
\multiput(95.18,127.93)(0,0){3}{{\rule{.4pt}{.4pt}}}
\multiput(95.18,127.93)(-.25,0){3}{{\rule{.4pt}{.4pt}}}
\multiput(97.93,129.18)(-.5,-.5){7}{{\rule{.4pt}{.4pt}}}
\multiput(98.18,126.93)(-.6,-.5){6}{{\rule{.4pt}{.4pt}}}
\multiput(97.93,125.18)(-.3333,-.3333){4}{{\rule{.4pt}{.4pt}}}
\put(140.5,148.25){\line(0,-1){8.5}}
\put(140.5,139.75){\line(0,1){0}}
\multiput(140.5,140)(-.0326087,-.1521739){23}{\line(0,-1){.1521739}}
\put(139.75,136.5){\line(0,1){.25}}
\multiput(139.75,136.75)(.03365385,-.06730769){52}{\line(0,-1){.06730769}}
\put(141.5,133.25){\line(0,1){0}}
\put(141.5,133.25){\line(0,-1){9.25}}
\put(141.5,124){\line(-1,0){.25}}
\put(140.25,148){\line(1,0){4.75}}
\multiput(145.25,148)(-.03289474,-.11842105){38}{\line(0,-1){.11842105}}
\put(144.25,144){\line(1,0){1.5}}
\put(145.75,144){\line(0,-1){8}}
\put(145.75,136.5){\line(-1,0){1.5}}
\multiput(144.5,136.25)(.03365385,-.0625){52}{\line(0,-1){.0625}}
\put(146,133.75){\line(0,-1){10}}
\put(146,124.25){\line(1,0){.25}}
\multiput(140.75,124.25)(.0625,-.03125){8}{\line(1,0){.0625}}
\put(141.25,124.25){\line(1,0){4.75}}
\multiput(142.93,147.93)(-.45,-.5){6}{{\rule{.4pt}{.4pt}}}
\multiput(144.43,147.68)(-.5,-.60714){8}{{\rule{.4pt}{.4pt}}}
\multiput(143.68,144.43)(-.5,-.625){7}{{\rule{.4pt}{.4pt}}}
\multiput(145.18,142.93)(.25,.5){3}{{\rule{.4pt}{.4pt}}}
\multiput(144.93,143.93)(.125,.25){3}{{\rule{.4pt}{.4pt}}}
\multiput(145.18,144.43)(-.5,-.65){11}{{\rule{.4pt}{.4pt}}}
\multiput(145.43,142.18)(-.55,-.65){11}{{\rule{.4pt}{.4pt}}}
\multiput(145.18,139.93)(-.5625,-.625){9}{{\rule{.4pt}{.4pt}}}
\multiput(140.68,134.93)(-.125,-.125){3}{{\rule{.4pt}{.4pt}}}
\multiput(145.43,137.43)(-.5,-.53125){9}{{\rule{.4pt}{.4pt}}}
\multiput(144.93,134.68)(-.58333,-.58333){7}{{\rule{.4pt}{.4pt}}}
\multiput(145.68,133.18)(-.60714,-.5){8}{{\rule{.4pt}{.4pt}}}
\multiput(145.68,131.68)(-.57143,-.46429){8}{{\rule{.4pt}{.4pt}}}
\multiput(145.68,129.93)(-.60714,-.53571){8}{{\rule{.4pt}{.4pt}}}
\multiput(145.68,127.68)(-.75,-.5){6}{{\rule{.4pt}{.4pt}}}
\multiput(145.68,126.43)(-.5625,-.4375){5}{{\rule{.4pt}{.4pt}}}
\put(140.25,143.75){\line(1,0){4.5}}
\put(140.25,140){\line(1,0){5.5}}
\put(140,136.5){\line(1,0){4.75}}
\put(141.5,133.25){\line(1,0){4.5}}
\put(141,130.5){\line(1,0){5.25}}
\put(141,127.25){\line(1,0){4.75}}
\put(137.75,148.25){\line(0,-1){24}}
\put(134.75,148){\line(0,-1){23.75}}
\put(131.75,147.75){\line(0,-1){23.5}}
\put(63,138.75){${\bf p}_1$}
\put(125.75,121.5){${\bf q}_1$}
\put(148.25,135.5){${\bf p}_2$}
\put(110.75,151){${\bf q}_2$}
\put(87.5,120){Trapezium}
\multiput(11,129.75)(.0613964687,.0337078652){623}{\line(1,0){.0613964687}}
\multiput(49.25,150.75)(.033653846,-.0625){104}{\line(0,-1){.0625}}
\multiput(52.75,144.25)(-.0621990369,-.0337078652){623}{\line(-1,0){.0621990369}}
\put(11.25,129.5){\line(1,-2){3}}
\multiput(16,132.5)(.03353659,-.07317073){82}{\line(0,-1){.07317073}}
\multiput(44.25,147.75)(.03370787,-.06179775){89}{\line(0,-1){.06179775}}
\put(28.5,139){\line(1,-2){3.25}}
\put(31.75,132.5){\line(0,1){0}}
\multiput(32.75,142.25)(.03370787,-.07303371){89}{\line(0,-1){.07303371}}
\put(15,128){$\pi_1$}
\put(47.5,145.5){$\pi_n$}
\put(31.25,137.75){$\pi_i$}
\put(10.25,125.75){$e_0$}
\put(27.5,135.5){$e_{i-1}$}
\put(35.5,140){$e_i$}
\put(41,135.25){${\bf q}_1$}
\put(17.25,137){${\bf q}_2$}
\put(28,127.5){Band}
\put(52.5,148.5){$e_n$}
\end{picture}
\end{figure}

Then $\Delta$ is called a \label{trapez}{\em trapezium}. The path ${\bf q}_1$ is
called the \label{bottomt}{\em bottom}, the path ${\bf q}_2$ is called the \label{topt}{\em top} of
the trapezium, the paths ${\bf p}_1$ and ${\bf p}_2$ are called the \label{lrsidest}{\em left
and right sides} of the trapezium. The history (Step history) of the $q$-band
whose side is ${\bf p}_2$ is called the \label{historyt}{\em history} (resp., \label{stepht} Step history) of the trapezium;
the length of the history is called the\label{heightt}{\em height}  of the
trapezium. The base of $\Lab ({\bf q}_1)$ is called the \label{baset}{\em base} of the
trapezium.}
\end{df}

\begin{rk} Notice that the top (bottom) side of a
$\theta$-band $\ttt$ does not necessarily coincides with the top
(bottom) side ${\bf q}_2$ (side ${\bf q}_1$) of the corresponding trapezium of height $1$, and ${\bf q}_2$
(${\bf q}_1$) is
obtained from $\topp(\ttt)$ (resp. $\bott(\ttt)$) by trimming the
first and the last $a$-edges
if these paths start and/or end with $a$-edges.
We shall denote the
\label{trim}{\it trimmed} top and bottom sides of $\ttt$ by \label{ttopp} $\ttopp(\ttt)$ and
\label{tbott} $\tbott(\ttt)$. By definition, for arbitrary $\theta$-band $\cal T,$ $\ttopp(\cal T)$
is obtained by such a trimming only if $\cal T$ starts and/or ends with a
$(\theta,q)$-cell; otherwise $\ttopp(\cal T)=\topp(\cal T).$
The definition of $\tbott(\cal T)$ is similar. 
\end{rk}

By Lemma \ref{NoAnnul}, any trapezium $\Delta$ of height $h\ge 1$
can be decomposed into $\theta$-bands $\ttt_1,...,\ttt_h$ connecting
the left and the right sides of the trapezium. The word written on
the trimmed top side of one of the bands $\ttt_i$ is the same as the
word written on the trimmed bottom side of $\ttt_{i+1}$,
$i=1,...,h$.
Moreover, the following lemma claims that every trapezium
simulates the work of $M.$ It summarizes the assertions  of Lemmas 
6.1, 6.3, 
6.9, and 6.16 from \cite{OScol}. For the formulation (1) below, it is important
that $M$  is an $S$-machine. The analog of 
this statement is false for Turing machines. (See \cite{OS3} for a discussion.) 

\begin{lemma}\label{simul} (1) Let $\Delta$ be a trapezium
with history $\theta_1\dots\theta_d$ ($d\ge 1$). 
Assume that $\Delta$ 
has consecutive maximal $\theta$-bands  ${\cal T}_1,\dots
{\cal T}_d$, and the words
$U_j$ 
and $V_j$ 
are the  trimmed bottom and the 
trimmed top labels of ${\cal T}_j,$ ($j=1,\dots,d$). 
Then $U_j$, $V_j$ are admissible
words for $M,$ and
$$V_1\equiv U_1\cdot \theta_1, U_2\equiv V_1,\dots, U_d \equiv V_{d-1}, V_d\equiv U_d\cdot \theta_d$$

(2) For every reduced computation $U\to\dots\to U\cdot H \equiv V$ of $M$ 
with $||H||\ge 1$
there exists 
a trapezium $\Delta$ with bottom label $U$, top label $V$, and with history $H.$ 
\end{lemma}

If $H'\equiv \theta_i\dots\theta_j$ is a subword of the history $\theta_1\dots\theta_d$
from Lemma \ref{simul} (1), then the bands ${\cal T}_i,\dots, {\cal T}_j$ form a subtrapezium
$\Delta'$ of the trapezium $\Delta.$ This subtrapezium is uniquely defined by the
subword $H'$ (more precisely, by the occurrence of $H'$ in the word $\theta_1\dots\theta_d$), and $\Delta'$ is called the \label{H'partt} $H'$-{\it part} of $\Delta.$
\medskip

\subsection{Properties of the group $M$.} 

In this subsection, we want to translate the  properties of the machine $M$ in
the language of diagrams over the group $M.$

Recall that every $(\theta,q)$-cell $\pi$ has a boundary label of the form $U_i\theta_{i+1}V_i^{-1}\theta_i^{-1}$ (see Relations (\ref{rel1})), where the word  $U_i$ (the word $V_i$)
has exactly one positive $q$-letter $q_i\in Q_i$ ($q'_i\in Q_i$). Hence the
boundary label of $\pi$ is  $q_iw_1(q')^{-1}_iw_2$ for some words $w_1, w_2.$

\begin{df}\label{activecell} The cell $\pi$ considered as a one-cell $q$-band with base
 $q_i$ is called \label{alrcell}{\it active from the right
(from the left) }  if the word $w_1$ (the word $w_2$) has at least one $a$-letter.
If $\pi$ with base $q_i$ is active from the right (from the left) then, by definition,
the same cell considered as a $q$-band with base $q_i^{-1}$ is active from the
left (resp., from the right).   A $(\theta,q)$-cell is called \label{passivec}{\it passive} if it is
not active either from the left or from the right.
\end{df}

The comparison with Definition \ref{active} shows that the cell $\pi$ with base $q_i^{\pm 1}$
is active from the left (resp., active from the right, passive) iff the base letter
$q_i^{\pm 1}$ is active on the left (resp., active on the right, passive)  for the rule
corresponding to the $\theta$-edges of $\pi.$

\begin{df}\label{activeband}  We say that a $q$-band with base $q_i$
  is \label{activelrb}{\it active from the left (from the right)}
  if every cell of it (with the same base) except for the first
  cell and the last one (if the first and/or the first cell corresponds to the rules $(12)^{\pm 1}$ or $(23)^{\pm 1}$), is active from the left (from the right). A $q$-band is called
  \label{passiveb}{\it passive} if every its cell is passive. Similarly one can speak on a $q$-band with
  base $q_i$ which is {\it passive from the left} or {\it passive from the right}.
\end{df}

\begin{rk} \label{s0} The letter $s_0$ in the standard base of $M_4$ corresponds to the left-most
$\alpha$-marker of the machines $M_1-M_4,$ and so every $s_0$-band is passive (from both sides).
\end{rk}

\label{sactivelrb} \begin{df}\label{stronglyactive} We say that a $q$-band $\cal C$ with base $q_i$ is
{\em strongly active from
the left (resp. right)} if every its cell $\pi$ is active from the left (from the right),
$\partial\pi$ has exactly one $a$-edge on the left side (right side) of $\cal C,$
and these $a$-letters are different for the cells corresponding to different
rules of the history of $\cal C.$
\end{df}
\begin{lemma} \label{*XII} Let $\Delta$ be any reduced diagram over 
$M.$ Let $\cal C$ be a
$q$-band, corresponding to the part $Q_i$ of $Q$. Suppose that $\cal C$ is strongly active from the left (resp. right).
Then $\Delta$ does not have an $a$-band starting and ending on the left side (resp., right side)  of $\cal C$.
\end{lemma}
\begin{proof} Suppose that an $a$-band $\aaa$ starts and ends on $\bott(\cal C)$ which is
the left side of $\cal C$. Let
$\Delta'$ be the subdiagram bounded by $\aaa$ and
$\cal C$.

Then $\Delta'$ has no other maximal $q$-bands except the part
${\cal C'= C}\cap \Delta'$ because $a$-bands and $q$-bands do not
intersect and $\Delta'$ has no $q$-annuli by Lemma \ref{NoAnnul}. Since maximal $a$-bands do not intersect, we can assume
without loss of generality that $\Delta'$ does not have any other
$a$-bands starting and ending on $\bott(\cal C')$. Since $\cal C'$ is strongly
active on the left, and the sides of $\aaa$ consist of
$\theta$-edges, we conclude that $\cal C'$ consists of two cells having
common $q$- and $a$-edges.
A $\theta$-cell in $\cal C$ is
completely determined by its $a$-letter on its bottom side (see Remark \ref{tb} for the argument).
Therefore those two $q$-cells cancel, a contradiction with the
assumption that $\Delta$ is reduced.
\end{proof}
The next Proposition summarizes previously proved properties of the various submachines of the $S$-machine $M$. We
formulate these properties in the language of \vk diagrams which makes it more convenient to apply these properties to the  group $G.$ 

Recall that the standard base of $M_3$ is denoted by $B$.
Note that the standard base of $M$ contains $L$ copies $B(i)$ of
$B$ ($i=1,3,\dots, 2L-1$) and $L$ copies  of $B\iv.$ 
We call the base of an admissible word of $M$\label{alignedb}{\it aligned}, if every maximal subword of this base without letters $t,
k, k', t'$ is a subword of a copy of $B^{\pm 1}$. 

\begin{rk}\label{align}
Since $||B||< N/2,$ Formula \ref{baza} and the definition of admissible words show that
every aligned base of length $\ge N/2$ must contain a $k^{\pm 1}$- or a $(k')^{\pm 1}$-letter or entirely consists of $t^{\pm 1}$- or $(t')^{\pm 1}$-letters. 
\end{rk}

A base of an admissible word of $M$ is called \label{normalb}{\em normal} if it is
a subword of a power of the base of the hub. (Recall that $t_1$ and $t_{2L+1}$
were identified in the definition of the machine $M.$)
A base is called \label{largeb}{\em large} if it contains a copy of $B^{\pm 1}$.

We shall say that a $(\theta,q)$-cell $\pi$ in a \vk diagram over $M$ is \label{oddc}{\em odd} if it contains exactly one $a$-edge
on its boundary, and its base is not $k$ or $k'$. A $\theta$-band with (1-letter) history of type (2) is called \label{oddb} {\em
odd} if it contains odd cells.

A trapezium over $M$ whose top label is one of $2L$  copies of the stop word of $M_4$ will be called \label{M4acct} $M_4$-{\em accepting
trapezium}. (The trapezium pictured in Subsection \ref{M4} is $M_4$-accepting and
in addition, its bottom label is an input word of $M_4$.)
A trapezium whose base is a copy of the standard base of $M_4$ is \label{standardM}{\it standard}
if
its bottom label is
in the domain of the rule $(12)$ 
and its top label is in the domain of $(23)^{-1}.$ 
By Lemmas
\ref{M40} and \ref{simul}, every standard trapezium has height $T_i$ for some $i$ and corresponds to a standard computation of $M_4.$  

\begin{prop}\label{summary} The following properties of the group $M$ hold. In all these properties we assume that we
are given a reduced \vk diagram $\Delta$ over $M$, all bands,  cells and edges are bands, cells and edges of that
diagram.
\begin{enumerate}[label=(\roman{*}), ref=(\roman{*})]

\item\label{order} A two letter base of a $\theta$-band is either a subword of the word (\ref{baza}) or of the inverse word, or it has form $q^{\pm 1}q^{\mp 1}$ for a base letter $q$.

\item\label{kk'} Every cell with base $k$ (every cell with base $k'$) corresponding to a rule of Step $1$ or Step $2$ except for the $(12)^{\pm 1}$-rule
(corresponding to a rule of Step $2$ or Step $3$ except for the $(23)^{\pm 1}$-rule),
is active from the left
(resp., from the right) and passive from the right (resp., from the left). Every non-$k$-cell (non-$k'$-cell) corresponding to a rule of Step 1 (resp., of Step 3)
is passive.

Every $t$-,$t'$-, and $s_0$-band is passive.

\item\label{cell}

(a) The boundary of every cell has at most two $a$-edges. It has either $0$ or $2$
$a$-edges if it is a $(\theta, q)$-cell corresponding to a control letter $p_i$,
otherwise is has at most one $a$-edge.

(b)If there are two $a$-letters $a$ and $a'$ in the boundary label of a cell $\pi$ then $a'$ is a copy
of $a^{-1}$ and $\pi$ is
either a $(\theta,a)$-cell or  a $(\theta, q)$-cell corresponding to a control letter $p_i$,
and the $a$-edges are separated by $q$-edges in $\partial\pi.$

(c) A $(\theta, a)$-cell has two mutually inverse $\theta$-letters in the boundary label.

(d) Two $(\theta,q)$-cells corresponding to control letters $p_i$ and $p_j$ with $i\ne j,$ have no common $a$-letter in the boundary labels.

\item\label{p} If a $\theta$-band $\cal T$ has $3$ consecutive cells $\pi_1, \pi_2, \pi_3,$ where   $\pi_2$
    corresponds to a control letter $p_i^{\pm 1}$ and $\pi_2$ is not active 
    from both sides
then one of the cells  $\pi_1,$ $\pi_3$  is a $(\theta,q)$-cell whose base is an $s$-letter.

\item\label{i} Let a $(\theta,q)$-cell $\pi_i$ of a $\theta$-band $\cal T$ have base $q$.
If $\ttt$ corresponds to a rule of Step 1 or to (12) (respectively,
Step 3 or to $(23)$), and the next cell, $\pi_{i+1}$ in $\ttt$ is a $(\theta,a)$-cell, then $q$ can be only one of
the following letters:  $t, p_1, k^{-1}, s_1^{-1}$ (resp.,
$k'$ or $(t')^{-1}$) (with some 
indices) . For other values of $q$ the next after $q$ letter in the base of $\ttt$ cannot be
$q^{-1}$.

\item\label{ii}  If in the base of a $\theta$-band, there is a subword $p_i^{\pm 1} p_i^{\mp 1}p_i^{\pm 1}$ for
some control letter $p_i$ and there are neither $k^{\pm 1}$- nor $(k')^{\pm 1}$-letters, then the active cells in this band are precisely the $p_i$-cells, and these cells
are active from both sides.

\item\label{iv} Suppose that $\cal C$ is a $k^{\pm 1}$-, $(k')^{\pm 1}$-,
or $p_i^{\pm 1}$-band with top path $\bf y$. Suppose that each cell of $\cal C$ has a common $a$-edge
with $\bf y$. Then no $a$-band of $\Delta$ can start and end on $\bf y$.

{\bf In the remaining parts of the Proposition, $\Delta$ is a trapezium. }

\item\label{kt} If $\Delta$ has base $k't'$ (base $tk$) then it cannot have Step history $(12)(2)(12)^{-1}$
(resp., $(23)^{-1}(2)(23)$).

\item\label{v} If $p_1p_1^{-1}s_0^{-1}$ is a subword of the base of $\Delta$, and the  history of $\Delta$ has the form
    $(12)H$, then $H$ is of type $(2),$ it has no rules $(12)^{\pm 1}$, $(23)^{\pm 1},$ and in the $H$-part of $\Delta,$ 
    all $p_1$-cells are active both from the left and from the right.

\item\label{vi} Suppose that $\Delta$ is $M_4$-accepting. Then the step history of $\Delta$ is a
subword of $ (2)(1)(2)(3).$
If $W$ is the label of the bottom path of $\Delta$, and
    $h$ is the height of $\Delta$, then $||W||_a\le 4h.$ 

\item\label{vii} If the history of $\Delta$ contains $(12)^{\pm 1}$ and $(23)^{\pm 1},$ then the base of $\Delta$
    is normal.

\item\label{viii} If (a) the length of the base of $\Delta$ is at least $N$ and its history contains both a rule $(12)^{\pm 1}$ and a rule $(23)^{\pm 1}$, or (b) the base of $\Delta$ is standard, then the  step history of $\Delta$ is a subword of
    $ (2)(1)(2)(3)(2)(1)(2).$

\item\label{ix} Suppose that the base of $\Delta$ is not aligned and the history is of type (2). Then the label of the top (and of the bottom) of every
maximal $s_j$-band of $\Delta$ admits a factorization of the form $u(b_1v_1b_1^{-1})\dots (b_m v_m b_m^{-1})w$
where $b_i^{\pm 1}$ is an $a$-letter or $1$ ($i=1,\dots,m$) , $v_i$ is a group word in $\theta$-letters,
 $b_i$ commutes with every letter of $v_i$ by virtue of  $(\theta, a)$-relations, and
each of $u,$ $w$ has at most one $a$-letter.

\item\label{x} If the base of $\Delta$ is large, and its history has the form  $H^3$, then $\Delta$ does not have
    odd cells $\pi$.

\item\label{xi} Suppose that the base of  $\Delta$ has the form $k^{-1}k$ or $k'(k')^{-1}$, and all $k$- (resp.
    $k'$-cells) of $\Delta$ are active. Let $W$ and $W'$ be the labels of the bottom and top of $\Delta$ respectively. Then
    the history of $\Delta$ has the form $H_1H_2^kH_3,$ where $k\ge 0$, $||H_1||\le ||W||_a/2$,
$||H_2||\le \min(|W|_a, |W'|_a)$, $||H_3||\le ||W'||_a/2$.

\item\label{xii} If the base of $\Delta$ is of length $\ge N$, and $\Delta$ has the step history
$$(12)(2)(23)(3)(23)^{-1}(2)(12)^{-1},$$ then the height of the
$(23)(3(23)^{-1})$-part of $\Delta$ is less than the sum of heights of the  $(12)(2)$- and $(2)(12)^{-1}$-parts
of it.

\item\label{xiii} Let $m>0$ be an integer such that for every standard trapezium with a
bottom label $W,$  inequality
$|W|_a\le m$ implies $||H||< \log m.$
Suppose that $\Delta$ is $M_4$-accepting, the history of $\Delta$ is $H$, and the bottom label $W'$ satisfies the
inequality $|W'|_a \le m$. Then
we have  $$||H'||\le 4|W'|_a + 3\log m.$$

\item\label{xiv} The set of numbers $m$ satisfying the assumption from \ref{xiii} is infinite.

\item\label{xv} Suppose $\Delta$ is $M_4$-accepting.  If the height $h$ of $\Delta$ exceeds $6|W|_a$, where $W$ is
    the
bottom label of $\Delta,$ then there exists a standard subtrapezium $\Delta'$ in $\Delta$
such that $h\in (h',9h'),$ where $h'$ is the height of $\Delta'.$
\end{enumerate}
\end{prop}

$\Box$

\ref{order} This follows from the definition of admissible word and from Lemma \ref{simul}.

\ref{kk'},\ref{cell} This follows from the definition of the rules of $M,$ the definitions  of  Relations (\ref{rel1}), and from Remark \ref{s0}.

\ref{p} Indeed, if a component $p_i\to ...$ of a rule from $M$ is not active from both sides, then it locks either $s_{i-1}p_i$-sector or the
$p_is_i$-sector,
and we can apply Property \ref{order} and Lemma  \ref{qqiv}.

\ref{i} Indeed, the rules from Step 1 and the rule (12) (resp. Step 3 and the rule (23)) lock all sectors except the
the $tk$-sectors, and $p_1s_1$-sectors (resp. the $k't'$-sectors) of the admissible words of $M$. It remains to use
Lemma \ref{qqiv}.

\ref{ii} This also follows from Lemma \ref{qqiv}: if a rule of $M$ does not lock the $p_is_i$-sectors or
$s_{i-1}p_i$-sectors, then its component involving $p_i$ has the form $p_i\to ap_i'b,$ where $a, b$ are tape letters, and all other components, exept for $k$- and $k'$-components,
do not involve tape letters.

\ref{iv} The condition means that the band $\cal C$ is active on the left and has no (passive) $(12)$-or $(23)$-cells. 
It follows from the definition of  $M$ that then
$k^{\pm 1}, (k')^{\pm 1}$ or $p_i$-band  is strongly active on the left. It remains to apply
Lemma \ref{*XII}.

\ref{kt} This follows from Lemmas \ref{M400} (a,b) and \ref {simul}.

\ref{v} Let us apply Lemma \ref{simul} and consider the reduced computation
corresponding to $\Delta.$ The rule $(12)$ switches on the copy of the machine
$\overrightarrow Z^{\theta,1}$ where $\theta=\theta_{start}$ is the start rule of $M_3.$
The (copies of the) rules of the form $\zeta_1(a)^{\pm 1}$ cannot follow by 
the (copy of the) rule $\zeta_2$ since $\zeta_2$ locks the $p_1s_1$-sector.
Also, by Lemma \ref{gen1}, it cannot follow by the rules $(12)^{\pm 1}$ or $(23)^{\pm 1}$  
locking the $s_0p_1$-sector. Therefore each of the rules of $H$ is of the form
$\zeta_1(a)^{\pm 1}$, and the statement follows.

\ref{vi} The first statement follows from Lemma \ref{M401} (b) because $\Delta$ is the trapezium corresponding to an accepting
computation of a copy of the machine $M_4$. The second property immediately follows from Lemma \ref{*IIIb}.

\ref{vii} Indeed, every sector of the standard base of $M$ is locked by either (12) or (23). It remains to use 
\ref{order} and Lemma \ref{qqiv}.

\ref{viii} Indeed, by Property \ref{vii}, the base of $\Delta$ is normal. Since its length is at least $N$, it must
contain a copy of the base of $M_4$, and it remains to use Lemma \ref{M401} (a).

\ref{ix} The base has non-aligned subword $B_0$ without $k$- and $t$-letters. Hence
the copy of $B_0^{\pm 1}$ is not a subword of the standard base of the machine $M_3.$
If $H$ is the history of the corresponding computation of $M_3,$ then by
Lemma \ref{ppm} (2), we have $\Pi_{32}(H)\equiv (\theta^{-1})
(\theta_1\theta_1^{-1})\dots(\theta_m\theta_m^{-1})(\theta')$ for some positive
rules $\theta,\theta_1,\dots,\theta'$ of $\tilde M_2$ ($\theta$ and/or $\theta'$ may be absent).

Recall that a $(\theta,s_j)$-cell has at most one $a$-edge, and
it has no $a$-edges, if it corresponds to a rule of one of the auxiliary machines $\overrightarrow Z$, $\overleftarrow Z.$ Hence the label of a side
of the $s_j$-band has form $u(b_1v_1b_1^{-1})\dots (b_m v_m b_m^{-1})w$
where $b_i^{\pm 1}$ is an $a$-letter or $1$ ($i=1,\dots,m$) , $v_i$ is a group word in $\theta$-letters, and each of $u,$ $w$ has at most one $a$-letter; and we should prove that $b_i$ commutes with $v_i$ if $b_i$
is involved in the rule $\bar\theta_i$. 

Let us consider the right side of the $s_j$-band. (The `left' case is similar.)
Then $\theta_i$ is a right rule, and
by Lemma \ref{ppm} (3), no rule of the subword
$\theta_i H'\theta_i^{-1}$ of $H$ locks the $s_jp^{j+1}$-sector, and so the letter $b_i$ commutes with every $\theta$-letter of $v_i$ by the definition of relations for the machine $M.$

\ref{x} follows from Lemma \ref{*V} since the cells corresponding to $M_3$-rules 
can have exactly one $a$-edge in the boundary only if they correspond to the rules of $\tilde M_2.$

\ref{xi} follows from Lemma \ref{gen2}.

\ref{xii} follows from Lemma \ref{*XV} because by Property \ref{vii} the base of the trapezium contains (as a subword)
a copy of the base of $M_4$.

\ref{xiii} follows from Lemma \ref{*IIIa} (b).

\ref{xiv} follows from Lemmas \ref{M32} and \ref{M40}.

\ref{xv} This is a reformulation of Lemma \ref{*IIIa} (a).
\endproof

\subsection{Diagrams with hubs} \label{dwh}

Given a reduced diagram $\Delta$ over the group $G,$ one can construct a planar graph whose vertices
are the hubs of this diagram plus one improper vertex outside $\Delta,$ and the
edges are maximal $t$-bands of $\Delta.$  

Let us consider two hubs $\Pi_1$ and $\Pi_2$ in a minimal diagram, 
connected by a
$t_i$-band
${\cal C}_i$ and a $t_{i+1}$-band ${\cal C}_{i+1}$, where
there are no other hubs between these $t$-bands. 
These bands, together with parts of
$\partial\Pi_1$ and $\partial\Pi_2,$ bound either a subdiagram 
having no cells,  
or  a trapezium $\Psi$ of height $\ge 1$. The former case is impossible since in this case
the hubs have a common $t$-edge and they are mirror copies of each other
contrary to the reducibility of the diagram. We want to show that
the latter case is not possible either. 

Indeed, in the latter case, both the top and the bottom of $\Psi$ are the subwords
of the hub $W_M^{\pm 1},$ i.e. , the history $H$ of $\Psi$ and $H^{-1}$ are the histories
of $M_4$-accepting subtrapezia of $\Psi$. Therefore, by Property \ref{viii} (b), the history $H$
is of type $(3).$ We may assume that the base of $\Psi$ has a subword $(k't')^{\epsilon}$
with $\epsilon = 1$ since
otherwise one can replace $\Psi$ by its mirror copy. Let $\Gamma$ be the maximal subtrapezium of $\Psi$ with base $k't'.$ Then every cell of the maximal $k'$-band $\cal C$ of $\Gamma$
is active from the right by Property \ref{kk'}. But the $a$-bands starting on $\cal C$ cannot
end on the passive (see Property \ref{kk'}) $t'$-band of $\Gamma.$ They also cannot end
on $\cal C$ by Property \ref{iv}. Hence $||H||=0,$ a contradiction.


Thus, any two hubs of a reduced diagram cannot be connected by  two $t$-bands,
such that the subdiagram bounded by them contains no other hubs.
This property makes the hub graph of a reduced diagram 
hyperbolic, in a sense, since the degree $L$ of every proper vertex (=hub) is high ($L\ge 40$).
Below we give a more precise formulation (proved for diagrams with such a hub graph, in particular, 
in \cite{SBR}, Lemma 11.4 and in  \cite{O1}, Lemma 3.2).

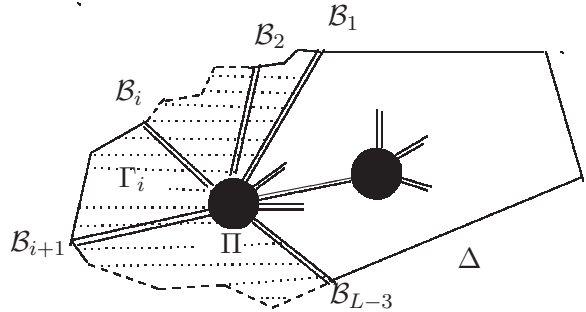
\begin{figure}[h!]

\unitlength 1mm 
\linethickness{0.4pt}
\ifx\plotpoint\undefined\newsavebox{\plotpoint}\fi 


\end{figure}

\begin{lemma} \label{extdisc} If a reduced diagram over the group $G$ contains a least one hub,
then there is a hub $\Pi$ in $\Delta$ such that $L-3$ consecutive maximal $t$-bands ${\cal B}_1,\dots
{\cal B}_{L-3} $ start on $\partial\Delta$ , end on the boundary $\partial\Pi$, and for any $i\in [1,L-4]$, 
there are no discs in the subdiagram $\Gamma_i$ bounded by ${\cal B}_i$, ${\cal B}_{i+1},$ $\partial\Pi,$ and $\partial\Delta.$ 
\end{lemma}

A maximal $q$-band starting on a hub of  a diagram is called a \label{spoke}
{\it spoke}. 

Lemma \ref{extdisc} implies

\begin{lemma} \label{mnogospits} If a reduced diagram $\Delta$ has $m\ge 1$ hubs then
the number of $q$-edges in the boundary path of $\Delta$ is greater than  $mLN/2$ .
\end{lemma}

\proof 
$\Delta$ has a hub $\Pi$ satisfying the assumption of Lemma \ref{extdisc}. Then we can separate
a subdiagram with only one hub $\Pi$ from $\Delta$ by making cuts along the $t$-bands ${\cal B}_1,$  ${\cal B}_{L-3} ,$ and along the part of $\partial\Pi$ having $3$ $t$-edges. Since by Lemma
\ref{NoAnnul}, every spoke of $\Gamma_i$ ($i\in [1,L-4]$) starting on $\Pi$ must end
on $\partial\Pi,$ the remaining diagram $\Delta'$ with $m-1$ hubs has
at most $|\partial\Delta|_q - (L-4)N + 4 N$ $q$-edges in the boundary. Since
$L-8 > L/2$ the statement follows by induction on $m$. \endproof

  \bigskip
  
 \label{papam} \subsection{Parameters} \label{param}

  The following constants will be used for the proofs in this paper.

\begin{equation}\label{const}
  L,N<<J<<\delta^{-1}<<(\delta')^{-1}<< c_0<<c_1<<\dots << c_7
  \end{equation}

  For each of the inequalities of this paper, one can find the highest
  constant (with respect to the order $<<$) involved in the inequality
  and see that for fixed lower constants, the inequality is correct as soon
  as the value of the highest one is sufficiently large. This principle
  makes the system of all inequalities used in this paper consistent.

  \subsection{Modified length of words and paths.} \label{newlength}
  
 Recall that the standard length $||w||$ of a word (a path) is called the {\it combinatorial length}. To introduce new length function on the group words in the generators of the groups $M$ and $G$ we first consider a word $w$ having no $q$-letters. We set the length $|a|$ of every
 $a$-letter $a$ equal to $\delta$. We set the length of any $\theta$-letter equal to $1,$
 but the the length $|v|$ of any $\theta a$-syllable, i.e., a $2$-letter word $v$ with one $\theta$-letter and one
 $a$-letter, will be equal to $1+\delta'.$ The {\it length of a decomposition} of $w$
 in a product of letters and $\theta a$-syllables is the sum of lengths of the factors of
 this decomposition. The {\it  length} $|w|$ of $w$ is the smallest length of such 
 decompositions. Finally, the \label{lengthw|} length \label{|.|} $|W|$ of arbitrary word $W\equiv w_0u_1\dots u_nw_n,$ where $u_i$-s are $q$-letters and the words $w_j$-s have no $q$-letters, is,
  by definition, $n+\sum_{i=0}^n |w_i|.$ The \label{lengthp|}{\it length of a path} in a diagram is the length
  of its label. The \label{perimd|}{\it perimeter} $|\partial\Delta|$ of a diagram is similarly defined by
  cyclic
  decompositions of its boundary $\partial\Delta.$
  
  Why do we need such a modification ? The assumption that $a$-edges are much shorter
  than other edges is used in Lemma \ref{main} (Step (2)) and in other lemmas. The assumption
  that $\delta'<<\delta,$ and so the length of a $\theta a$-syllable is less than
  the sum of lengths of its letters, is used in Lemma \ref{width1} and in many other lemmas.

\begin{lemma} \label{ochev}
Let $\bf s$ be a path in a diagram $\Delta$,
having $d$ $a$-edges and $e$ non-$a$-edges. Then

(a) $ e+d\delta\ge |{\bf s}|\ge e+d\delta'+\max(0,(d-e)(\delta-\delta'))\ge e+d\delta'$;

(b) if ${\bf s}={\bf s}_1{\bf s}_2,$ then $|{\bf s}_1|+|{\bf s}_2|\ge |{\bf s}|\ge |{\bf s}_1|+|{\bf s}_2|-(\delta-\delta')$ and
$|{\bf s}|=|{\bf s}_1|+|{\bf s}_2|$ if ${\bf s}_1$ ends or ${\bf s}_2$ starts with
a $q$-edge or if both these edges are not $a$-edges;

(c) if $\bf s$ is a top or a bottom of a $q$-band having $h$ cells, then
$h\le |{\bf s}|\le h(1+\delta')$; and $|{\bf s}|=h$ if $\bf s$ has no $a$-edges.

(d) $||{\bf s}||\ge|{\bf s}|\ge\delta ||{\bf s}||.$ 

\end{lemma}

\proof (a) Since every path is a product of $q$-, $\theta$-, and $a$-edges,
the first inequality follows. The second one is true because at most $e$ $a$-edges
can be joined with $\theta$-letters to form 2-edge subpaths of $\bf s$, and the remaining
$a$-edges has to be taken alone with coefficient $\delta$ when one calculate $|\bf s|.$  To make the reader more
familiar with the definition of $|\bf s|,$ we leave claims (b), (c), (d) for exercises.
\endproof

\section{Mixture on the boundary of a diagram}\label{mix}

Let $O$ be a circle with two-colored
finite set of points (or vertices) on it,
more precisely, let any vertex of this finite set be either black or white. We call $O$ a \label{neckl}{\it necklace} with black and white \label{bead}{\it beads} on it. We want to introduce the {\em mixture} of this finite set of beads.

Assume that there are $n$ white beads and $n'$ black ones on $O$. We define sets \label{Pk}
${\bf P}_j$ of
ordered pairs of distinct white beads as follows. A pair $(o_1,o_2)$
($o_1\ne o_2$) belongs to the set ${\bf P}_j$ if the simple arc of $O$
drawn from $o_1$ to $o_2$ in clockwise direction has at least $j$
black beads. We denote by \label{muKO}$\mu_J(O)$ the sum $\sum_{j=1}^J \card
 {\bf P}_j$ (the \label{Kmix}$J$-{\it mixture} on $O$). Below similar sets for
another necklace $O'$ are denoted by ${\bf P'}_J$. In this section, $J\ge
1$, but later on it will be a fixed large enough number $J$ from the list (\ref{const}).

\begin{lemma}\label{mixture} (a) $\mu_J(O)\le J(n^2-n)$.

 (b) Suppose a
necklace $O'$ is obtained from $O$ after removal of a   white bead
$v$. Then \\ $ \card {\bf P}_j-n < \card {\bf P'}_j \le \card {\bf
P}_j$ for every $j$, and $\mu_J(O)-Jn<\mu_J(O')\le \mu_J(O).$

(c) Suppose a necklace $O'$ is obtained from $O$ after removal of a
black bead $v$. Then  $\card {\bf P'}_j \le \card {\bf P}_j$ for
every $j,$ and $\mu_J(O')\le \mu_J(O).$

(d) Assume that there are three beads $v_1, v_2, v_3$ of a necklace
$O,$ such that the clockwise arc $v_1 - v_3$ contains $v_2$ and has
at most $J$ black beads (excluding $v_1$ and $v_3$), and the arcs
$v_1-v_2$ and $v_2-v_3$ have $m_1$ and $m_2$ white beads,
respectively. If $O'$ is obtained from $O$ by removal of $v_2$, then
$\mu_J(O')\le\mu_J(O)-m_1m_2.$
\end{lemma}

\proof (a) It is clear from the definition that $\card{\bf P}_j\le n^2-n$ 
, and the statement (a)
follows. The statements (b) and (c) are obvious.

(d) Let $o$ ($o'$) be a white bead on $v_1-v_2$ (on
$v_2-v_3$). Then for some $j\in \{1,\dots, J\},$ the pair $(o,o')$
belongs to ${\bf P}_j$ but does not belongs to ${\bf P}_{j+1}.$ Now,
on the one hand, the same pair $(o,o')$ considered on $O'$ does not
belong to ${\bf P'}_j$. On the other hand, we clearly have ${\bf
P'}_j\subseteq {\bf P}_j$. Therefore $\mu_J(O)-\mu_J(O')$ is at
least the number of such pairs $(o,o'),$ which is equal to
$m_1m_2.$ The lemma is proved.
\endproof

We will use also the mixture of beads on a closed interval ${\bf x}=[a,b]$ with real $a<b$. A \label{stringb}{\it string of beads} is a finite sets of
white and black beads on $\bf x$, but in the definition of mixture
$\mu^c(\bf x)$ we consider only pairs $(o,o')$ of white beads, where
$o<o'.$
 This gives us the mixture \label{muKc} $\mu^c_J(\bf x)$
as above.

\begin{lemma}\label{mixturec} Let $\bf x$ be a string of beads and $J\ge 1.$

(a) $\mu^c_J(x)\le J(n^2-n)/2$.

 (b) Suppose a
string $\bf x'$ is obtained from $\bf x$ after removal of a  white bead $v$.
Then \\ $\card {\bf P}_j-n < \card {\bf P'}_j \le \card {\bf P}_j$
for every $j$, and $\mu^c_J({\bf x})-Jn<\mu^c_J({\bf x'})\le \mu^c_J(\bf x).$

(c) Suppose a string $\bf x'$ is obtained from $\bf x$ after removal of a
black bead $v$. Then  $\card {\bf P'}_j \le \card {\bf P}_j$ for
every $j$, and $\mu^c_J({\bf x'})\le \mu^c_J(\bf x).$

(d) Assume that there are three  black beads $v_1< v_2<v_3$ on $\bf x$
such that the interval $(v_1, v_3)$ has at most $J$ black beads, and
the intervals $(v_1,v_2)$ and $(v_2,v_3)$ have $m_1$ and $m_2$ white
beads, respectively. If $\bf x'$ is obtained from $\bf x$ after removal of the
bead $v_2$, then $\mu^c_J({\bf x'})\le\mu^c_J({\bf x})-m_1m_2.$

(e) Assume that the set of black beads is non-empty. Then there is a
black bead $v,$ such that it divides $\bf x$ into two subsegments with
$m_1$ and $m_2$ white beads, respectively, $m_1\ge m_2$, and
$m_1m_2\le\mu^c_1({\bf x})\le (2m_1-1)m_2.$

\end{lemma}

\proof The proof of statements (a) - (d) is similar to the proof of
Lemma \ref{mixture}. To prove claim (e), we choose the black bead
$v$ so that the difference $|m_1-m_2|$ is minimal. We can assume
that $m_1\ge m_2.$ Since $m_1$ white beads are separated by $v$ from
$m_2$ white beads, we have $\mu^c_1(x)\ge m_1m_2.$ On the other
hand, there is a subsegment with $m_1-m_2$ pairwise non-separated
(by black beads) white beads. Therefore
$$\mu_1^c(x)\le \frac12(m_1+m_2)(m_1+m_2-1)-\frac12(m_1-m_2)(m_1-m_2-1)=(2m_1-1)m_2$$
\endproof

For any diagram $\Delta,$ we introduce the following invariant
$\kappa(\Delta)=\mu_{1} (\partial\Delta)$. To define them, we
consider the boundary $\partial(\Delta),$ as a \label{mu1neckl} $\kappa$-{\it necklace}, i.e.,
we consider a circle $O$ with $||\partial\Delta||$ edges labeled as the
boundary path of $\Delta.$ By definition, the
white beads are the mid-points  of the $\theta$-edges of
$O$ and black beads are the mid-points of the $q$-edges
$O$. 
Then, by definition, the \label{muK1mix} $\kappa$-{\it mixture}
on $\partial\Delta$ is $\kappa
(\Delta)=\mu_1(O).$

We will need an analogous parameter $\nu_J
(\Delta)$. The definition of
the \label{nuneckl} $\nu$-{\it necklace} on $\partial\Delta$ is similar,
but the black beads of it
 correspond to $t$- and $t'$-edges only while the set of white beads
 coincides with that for the $\kappa$-necklace. The $\nu$-necklace
 has $\nu_J$-mixture for every $J\ge 1,$ which is called the
 \label{nuKmix} $\nu_J$-{\it mixture} on $\partial\Delta$ and denoted by \label{nuK.} $\nu_J
(\Delta).$

Recall that a $\theta$-letter is said to be $(12)$-letter ($(23)$-letter) if it
corresponds to the rule $(12)$ (to $(23)$). Such a letter is {\it
special} if it is involved in a $((12),t)$-relation or
in a $((23),t')$-relation. An edge is
a $(12)$-edge (a $(23)$-edge, a \label{specialte} {\it special} edge) if it is labeled
by a $(12)$-letter (by a $(23)$-letter, by a special
$\theta$-letter, respectively). Note that if a $q$-band $\cal C$ has a special
$(12)$-edge (a special $(23)$-edge) on the left side, then the base of $\cal C$ is either 
a $t^{\pm 1}$ or a $k$ (resp., either a $(t')^{\pm 1}$ or a $(k')^{-1}$).

To define an auxiliary parameter $\lambda(\Delta)$ we consider, the \label{mu2neckl} $\lambda$-{\it necklace}, where
white beads are the middle points of all  $\theta$-edges of $O$
which are neither $(12)$-edges nor $(23)$-edges, and the black
beads are the middle points of all non-special  $(12)$- and
$(23)$-edges and all $q$-edges of $O$. The $\lambda$-necklace
defines the \label{mu2Kmix} $\lambda_J$-{\it mixture} on $\partial\Delta$ for every $J,$
and for $J=1$, we denote it
by $\lambda(\Delta).$

By definition,
\label{mu.}\label{mu11.}\label{mu21.}
$\mu(\Delta)= c_0\kappa(\Delta)+ \lambda(\Delta).$  The $\nu_J$-mixtures on the booundaries
will be later applied for a large enough $J.$

 Similarly we have \label{munustring} $\kappa(\bf x)$, $\lambda(\bf x),$ $\mu(\bf x)$, and $\nu_J(\bf x)$ for any path $\bf x$
 in a diagram. (Consider the strings of beads to define.) Clearly, each of this values remains unchanged if one replaces $\bf x$ by ${\bf x}^{-1}.$

\section{General properties of combs} \label{cmb}

By Lemma \ref{simul}, every property of a trapezium can be formulated as
a property of a computation of the S-machine $M$, and vice versa. Unfortunately
minimal diagrams can be much more complicated than trapezia. Now we define
diagrams which are the main subject of our research in this paper. 

As in \cite{OS}, we say that a reduced diagram $\Gamma$ over $M$ with reduced
boundary path (having no subpaths of the form $ee^{-1}$) is a \label{comb}{\it
comb} if it has a rim $q$-band $\cal C$ (the \label{handlec}{\it handle} of the
comb), 
and every maximal $\theta$-band of $\Gamma$ has
a cell in $\cal C.$ In particular, every trapezium is a comb.

Suppose that a maximal $q$-band $\cal C$ of a diagram $\Delta$ starts and
ends on $\partial\Delta$. Then it divides $\Delta$ into two
subdiagrams $\Gamma$ and $\Gamma',$ where $\Gamma'$ contains $\cal C$.  Suppose $\Gamma$ is a comb with handle $\cal C.$
Then we call $\Gamma$ a \label{subcombd}{\em subcomb} of
$\Delta.$

By Lemma \ref{NoAnnul}, any maximal $q$-band $\cal C'$ of a comb
$\Gamma$ is itself a handle of a subcomb $\Gamma'$ of $\Gamma$ which
does not contain (by definition of subcomb of a comb) cells from the handle
$\cal C$ of $\Gamma$ if $\cal C'\ne \cal C.$ In this case $\Gamma'$
is a \label{propersc}{\it proper} subcomb of the comb $\Gamma.$

  The \label{basewc}{\it base width} of a comb is by definition the maximal
number of letters in the bases of its $\theta$-bands. The \label{historyc}{\it history} $H$ and the Step history
of a comb are the history and Step history of its handle. If $H'$ is a subword of H
then \label{H'partc}$H'$-{\it part} of the comb is the union of all maximal $\theta$-bands corresponding
to $H'.$

\bigskip

  It will be convenient to view a comb $\Gamma$ with the handle on its right.
  Thus the bottom of the handle $\cal C$ is the right side of
  $\cal C,$ and it is the part of $\partial\Gamma.$ Respectively, every
  $q$-band of $\Gamma$ has the \label{rightlsqbc}right side and the left side.
  The words written on tops/bottoms of $\theta$-bands of $\Gamma$ and their bases will
  be  read from left to right, and so, for a base letter $q$, one can distinguish
  $q$- and $q^{-1}$-bands of $\Gamma.$ In particular, a $q$-band of $\Gamma$ can
  be active from the left, active from the right (or passive). 
  If a $q$-band
$\cal D$ is passive from the left (from the right), then $h=|{\bf y}'|\le |\bf y|,$ 
where $h$ is the number of cells in $\cal D,$ (resp., $|{\bf y}'|\ge |{\bf y}|=h$) by the definition of length
  and Lemma \ref{ochev}.

  \bigskip

  We introduce the following permanent notation for a comb $\Gamma$ with a handle $\cal C$.
 Denote by $H$ the history of $\cal C$ and set \label{hc} $h=||H||$, i.e., $h$ is the length of $\cal C$,
 the number of $q$-cells in $\cal C$. 
 The comb $\Gamma$ is a \label{onestepc}{\it one-Step}  comb if the history $H$ is one-Step,
 i.e., $H$ has one of the types $(1),$ $(2),$ or $(3).$

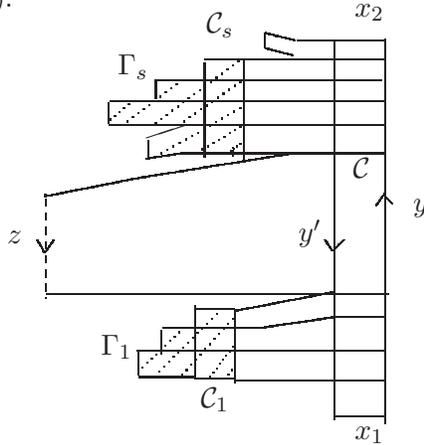
\unitlength 1mm 
\linethickness{0.4pt}
\ifx\plotpoint\undefined\newsavebox{\plotpoint}\fi 
\begin{picture}(112.5,65.25)(30,0)
\put(108.75,62.25){\line(0,-1){51.25}}
\put(102,62){\line(0,-1){50.5}}
\put(90,59.5){\line(0,-1){13.25}}
\put(84.75,59){\line(0,-1){12.5}}
\put(97,62){\line(1,0){11.5}}
\put(84.75,59.5){\line(1,0){24}}
\multiput(92.75,63)(.1333333,-.0333333){30}{\line(1,0){.1333333}}
\multiput(92.75,61)(.125,-.0333333){30}{\line(1,0){.125}}
\put(92.75,63){\line(0,-1){2}}
\put(78.25,56.75){\line(1,0){30}}
\put(73.75,54){\line(1,0){34.75}}
\put(72,50.75){\line(1,0){36.5}}
\put(81.75,47){\line(1,0){27}}
\multiput(81.75,47)(-.2065217,-.0326087){23}{\line(-1,0){.2065217}}
\put(77.25,48.75){\line(0,-1){2.5}}
\put(77,49){\line(3,1){5.25}}
\put(78.25,56.75){\line(0,-1){2.5}}
\put(72.25,54){\line(1,0){2.5}}
\put(72,53.75){\line(0,-1){2.75}}
\multiput(96.75,47)(-.211538462,-.033653846){104}{\line(-1,0){.211538462}}
\multiput(75,43.5)(-.17164179,-.03358209){67}{\line(-1,0){.17164179}}
\put(63.75,41.5){\line(-1,0){.25}}
\put(63.68,40.93){\line(0,-1){.9464}}
\put(63.68,39.037){\line(0,-1){.9464}}
\put(63.68,37.144){\line(0,-1){.9464}}
\put(63.68,35.251){\line(0,-1){.9464}}
\put(63.68,33.358){\line(0,-1){.9464}}
\put(63.68,31.465){\line(0,-1){.9464}}
\put(63.68,29.573){\line(0,-1){.9464}}
\put(63.68,27.68){\line(0,1){.125}}
\put(63.75,28.25){\line(1,0){38.25}}
\put(102,28.25){\line(1,0){7.25}}
\multiput(101.75,28.5)(-.17333333,-.03333333){75}{\line(-1,0){.17333333}}
\put(88.75,26.25){\line(0,-1){9.75}}
\put(88.75,16.75){\line(1,0){20}}
\put(83.5,26.25){\line(1,0){5.25}}
\put(83.5,17){\line(1,0){5.25}}
\put(83.5,26.75){\line(0,-1){9.5}}
\put(79.25,23.75){\line(1,0){12.75}}
\multiput(92,23.75)(.22777778,.03333333){45}{\line(1,0){.22777778}}
\put(102.25,25.25){\line(1,0){6.25}}
\put(75.75,20.75){\line(1,0){33}}
\put(76,20.75){\line(0,-1){3.25}}
\put(76,17.25){\line(1,0){7.25}}
\put(79,23.5){\line(0,-1){3.25}}
\put(102.25,12){\line(1,0){6.5}}
\multiput(72.18,51.43)(.6,.5){6}{{\rule{.4pt}{.4pt}}}
\multiput(74.43,51.18)(.725,.55){11}{{\rule{.4pt}{.4pt}}}
\multiput(78.18,50.93)(.73333,.55){16}{{\rule{.4pt}{.4pt}}}
\multiput(81.93,51.18)(.75,.525){11}{{\rule{.4pt}{.4pt}}}
\multiput(78.93,46.93)(.80769,.46154){14}{{\rule{.4pt}{.4pt}}}
\multiput(84.43,47.43)(.71429,.42857){8}{{\rule{.4pt}{.4pt}}}
\multiput(88.43,47.68)(0,-.5){3}{{\rule{.4pt}{.4pt}}}
\multiput(88.43,46.68)(-.125,.25){3}{{\rule{.4pt}{.4pt}}}
\multiput(88.18,47.18)(.4167,.3333){4}{{\rule{.4pt}{.4pt}}}
\put(73.25,57.5){$\Gamma_s$}
\multiput(76.43,18.43)(.61111,.58333){10}{{\rule{.4pt}{.4pt}}}
\multiput(78.68,17.93)(.63462,.63462){14}{{\rule{.4pt}{.4pt}}}
\multiput(81.68,17.43)(.61364,.68182){12}{{\rule{.4pt}{.4pt}}}
\multiput(85.43,17.18)(.42857,.57143){8}{{\rule{.4pt}{.4pt}}}
\multiput(87.43,16.93)(.4167,.5833){4}{{\rule{.4pt}{.4pt}}}
\put(71,20.25){$\Gamma_1$}
\put(84.25,13.25){${\cal C}_1$}
\put(85,63){${\cal C}_s$}
\put(104.5,43.75){${\cal C}$}
\multiput(108.75,41.75)(-.03289474,-.04605263){38}{\line(0,-1){.04605263}}
\multiput(108.5,41.5)(.03289474,-.03289474){38}{\line(0,-1){.03289474}}
\put(112.5,39.25){$y$}
\multiput(100.75,35.5)(.0333333,-.05){30}{\line(0,-1){.05}}
\multiput(103.25,35.5)(-.03289474,-.03289474){38}{\line(0,-1){.03289474}}
\put(97.25,35.25){$y'$}
\put(104.75,9){$x_1$}
\put(104.75,65.25){$x_2$}
\multiput(62.5,35.5)(.0333333,-.0583333){30}{\line(0,-1){.0583333}}
\multiput(63.5,33.75)(.03289474,.04605263){38}{\line(0,1){.04605263}}
\put(58.5,35){$z$}
\end{picture}

 The boundary of $\cal C$ is ${\bf x}_1{\bf y}{\bf x}_2{\bf y}'$, where ${\bf x}_1$ and ${\bf x}_2$ are the boundary
 $q$-edges of the band $\cal C$ and \label{z} $\bf yz$ is the boundary of $\Gamma$. (Thus, \label{y} $\bf y$ is the
 right side of $\cal C$, and \label{y'} $\bf (y')^{-1}$ is the left side.) Similarly
 we have the decomposition $({\bf y'})^{- 1}{\bf z'}$ for the boundary of $\Gamma\backslash\cal C$, where \label{z'} ${\bf z}={\bf x}_2{\bf z'}{\bf x}_1$.
By definition, $\area'(\Gamma)= \area(\Gamma\backslash\cal C).$ Since
$\bf z$ starts
 (ends) with the $q$-edge ${\bf x}_2$ (with ${\bf x}_1$), we have $|\partial\Gamma|=|\bf y|+|\bf z|$ by Lemma \ref{ochev} (b). 
 We also use ${\bf y}^{\Delta}, {\bf z}^{\Delta}$, ... instead of ${\bf y}, {\bf z},\dots$ if we want to stress
 that the notation relates to a particular comb $\Delta$.

 \begin{rk}\label{yzh} It follows from Lemma \ref{NoAnnul} that every maximal
 $\theta$-band crossing the handle of a comb $\Delta$ must ends on ${\bf z}^{\Delta}$.
 Therefore $|{\bf y}^{\Delta}|_{\theta} = |{\bf y'}^{\Delta}|_{\theta} = |{\bf z}^{\Delta}|_{\theta}=h.$
 \end{rk}
 
 \medskip

 For a comb $\Gamma$, we modify the notion of mixture. The \label{combmix}{\it comb mixtures} are \label{mucjK} $\kappa^c(\Gamma)=
 \kappa
 ({\bf z}) - \kappa({\bf y}),$ $\lambda^c(\Gamma)=
 \lambda
 ({\bf z}) - \lambda({\bf y})$, and
 similarly, \label{nucK} $\nu^c_J(\Gamma) = \nu_J({\bf z})-\nu_J({\bf y}) $
 ($\lambda(\Gamma)$ can be negative if
 $(12)$- or $(23)$-cells separate other $\theta$-cells of the handle !).
 By definition
 \label{muc.} $\mu^c(\Gamma)=c_0\kappa^c(\Gamma)+\lambda^c(\Gamma).$

 \begin{lemma}\label{positive} In the above notation, we have
 (a) $\kappa^c(\Gamma)\ge 0,$
 (b) $\nu^c_{J}(\Gamma)=\nu_J({\bf z})\ge 0,$\\
 (c) $\lambda^c(\Gamma)\ge 0$ if for every special edge $e$ of $\bf z,$ the edge $f$ of
 $\bf y$ connected with $e$ by a $\theta$-band, is also special.
 \end{lemma}

\proof (a), (b) Since the path $\bf y$ has no $q$-edges, we have $\kappa({\bf y})=0$
 ($\nu_{J}({\bf y})=0$, respectively), and so $\kappa^c(\Gamma)=\kappa({\bf z})\ge 0$ ( $\nu^c_{J}(\Gamma)
 = \nu_{J}({\bf z})\ge 0$, resp.).

 (c) Consider the strings of beads on $\bf z$ and on $\bf y$ used in the definitions of the comb
 mixture $\lambda^c(.)$. By Lemma \ref{NoAnnul}, the maximal $\theta$-bands of $\Gamma$
 establish a bijection between the white vertices of $\bf z$ and white vertices of $\bf y,$
 preserving the order of the beads on $\bf z$ and ${\bf y}^{-1}$, respectively. Every black bead on $\bf y$
 must belong to a non-special $\theta$-edge $f.$ By the condition of the lemma, we have
 a black bead on the corresponding edge $e$ of $\bf z.$ Hence one can apply Lemma \ref{mixture} (c)
 to the strings of beads on $\bf z$ and $\bf y$ several times to conclude that $\lambda({\bf z})\ge \lambda(\bf y),$
 and so $\lambda^c(\Gamma)\ge 0.$ \endproof

 \begin{lemma}\label{mu}
 Let $\Gamma$ be a proper subcomb of a diagram (of a comb) $\Delta.$
 Let $\Delta\backslash\Gamma$ be the compliment of $\Gamma$ in $\Delta,$ whose handle is the handle
 of $\Delta$ if $\Delta$ is a comb.
 Then

 (a) $\kappa(\Delta\backslash\Gamma)\le \kappa(\Delta)-\kappa^c(\Gamma)$ and
$\lambda(\Delta\backslash\Gamma)\le \lambda(\Delta)-\lambda^c(\Gamma),$

 (b) $\nu_J(\Delta\backslash\Gamma)\le \nu_J(\Delta)-\nu^c_J(\Gamma)$ for every $J\ge 1,$

 (c)
 $\kappa^c(\Delta\backslash\Gamma)\le \kappa^c(\Delta)-\kappa^c(\Gamma)$
and $\lambda^c(\Delta\backslash\Gamma)\le \lambda^c(\Delta)-\lambda^c(\Gamma)$
if $\Delta$ is a comb,

(d) $\nu^c_J(\Delta\backslash\Gamma)\le
\nu^c_J(\Delta)-\nu^c_J(\Gamma)$ for every $J\ge 1$ if $\Delta$ is a
comb,

(e) If $\bar\Delta$ is a subcomb of a diagram  $\Delta$ and $\Gamma$
is a subcomb of $\bar\Delta,$ then for every $J\ge 1$,
$0\le\nu^c_J(\bar\Delta) - \nu^c_J(\bar\Delta\backslash\Gamma)\le
\nu_J(\Delta)-\nu_J(\Delta\backslash\Gamma).$
(Also we have \\ $\nu^c_J(\bar\Delta) - \nu^c_J(\bar\Delta\backslash\Gamma)\le\nu^c_J(\Delta)-\nu^c_J(\Delta\backslash\Gamma)$
if $\Delta$ is a comb.)

 \end{lemma}

 \proof (a)  Let ${\bf y}={\bf y}^{\Gamma},$ and ${\bf z}={\bf z}^{\Gamma}$,
 $\bf xz$ the boundary path of $\Delta$. To obtain the necklace $O'$ corresponding to $\Delta\backslash\Gamma$, one replaces the subpath $\bf z$ of the boundary  by ${\bf y}^{-1}$. 
 Therefore the pairs of white beads counted to get $\lambda(\bf z)$
 are replaced by pairs counted to get $\lambda(\bf y).$ (Note that the white beads of $\bf z$
 are in bijective correspondence with white beads of $\bf y$ by the definition of comb and
 Lemma \ref{NoAnnul}.) Since every white bead of $\bf z$ is separated from any white bead of
 $\partial\Delta\backslash\partial\Gamma$
 by the black beads in the middle of the first and the last edges of $\bf z,$ Inequality (a) is proved for $\lambda$-mixtures. The case of $\kappa$-mixtures is similar.

 The proofs of claims (b), (c), and (d) are also similar.

 The path ${\bf y}^{\Gamma}$ has no $t$-edges, and the first inequality of (e) follows.
 Similarly, every pairs of white beads which makes a contribution to  $\nu^c_J(\bar\Delta)$
 but not to $\nu^c_J(\bar\Delta\backslash\Gamma)$ also contributes to
 $\nu_J(\Delta)$ but not to $\nu_J(\Delta\backslash\Gamma)$, and the second
 inequality of (e) follows. The proof of the version in  the parentheses is similar.
 \endproof

Let $\Gamma$ be a comb and ${\bf z}^1,\dots,{\bf z}^r$ the maximal subpaths of
${\bf z}={\bf z}^{\Gamma}$ containing
 no $q$-edges. We denote by $l^1,\dots,l^r$ their $\theta$-lengths,
 and define \label{lminus} $l_-=l_-^{\Gamma}$ to
 be  $h-\max_{i=1}^r l^i$. (Note that $h=h^{\Gamma}=\sum l^i$ by Lemma \ref{NoAnnul}.)

A $\theta$-band which starts on the handle $\cal C$ of a comb
$\Gamma$ will be called \label{simpletb} {\it simple} if it has no $(\theta, q)$-cells except
for the cell of $\cal C,$ and is maximal with respect to this
property.

We call a maximal $q$-band $\cal B$ a \label{derivqb} {\it derivative} band, if it
is not $\cal C$ but it can be connected with $\cal C$ by a simple
$\theta$-band. 
Throughout the paper, we  will use notation \label{Ci} ${\cal C}_1,\dots {\cal C}_s$
for derivative bands of a comb $\Gamma.$ It is possible that $s=0$, and
every maximal $\theta$-band is simple in this case.

Every derivative band
${\cal C}_i$ is a handle of a subcomb \label{Gi} $\Gamma_i$ (which does not
contain $\cal C$). We will use this notation and call $\Gamma_i$ a \label{dersc}
{\it derivative subcomb} of $\Gamma$. It follows from the definitions
that every cell of a comb belongs either to a derivative subcomb or
to a simple band of $\Gamma.$

Recall that every maximal $\theta$-band of a comb, in
particular, a maximal $\theta$-band crossing a derivative band
${\cal C}_i,$ must cross the handle $\cal C$. Therefore every cell
of ${\cal C}_i$ is connected with $\cal C$ by a $\theta$-band. Since
there is a simple $\theta$-band among these $\theta$-bands, no other
derivative ${\cal C}_j$ can intersect these connecting
$\theta$-bands by Lemma \ref{NoAnnul}, i.e., all of them are simple. 
It follows that different
derivative subcombs are disjoint, and if ${\cal C}_1,\dots,{\cal
C}_s$ is the system of all derivative bands in $\Gamma$ with
histories \label{derhist} $H_1,\dots,H_s$, then $H_1,\dots,H_s$ are pairwise
disjoint subwords in the history $H$ of $\Gamma$. Therefore
$\sum_{i=1}^s h_i\le h$, where \label{hi} $h_i=||H_i||$. We will also use \label{hminus} $h_-$
for $\sum_{i=1}^s h_i -\max_{i=1}^s h_i.$
\begin{lemma}\label{sravnim}
In the above notation, $l_-\ge \min(\sum_{i=1}^s h_i,\;
h-\max_{i=1}^s h_i)$. In particular,   \begin{equation}\label{hl}
  h_-\le l_-
\end{equation}
\end{lemma}

\proof Let $|{\bf z}^{i_0}|_\theta =l^{i_0}=\max_{i=1}^r l^i.$  Then,
either every maximal $\theta$-band ending on ${\bf z}^{i_0}$ crosses some
derivative band ${\cal C}_j$, where $j=j(i_0)$, or every maximal
$\theta$-band crossing ${\bf z}^{i_0}$ crosses no derivative bands because
otherwise a $q$-band would cross ${\bf z}^{i_0}.$ (This follows from the
definitions of comb, of ${\bf z}^i$-s and from Lemma \ref{NoAnnul}.) In the
former case, $l_-=h-l^{i_0}\ge h-\max_{i=1}^s h_i\ge h_-$, and in
the latter case, $l_-=h-l^{i_0}\ge \sum_{j=1}^s h_j\ge h_-.$
\endproof

 \begin{lemma}\label{lgamma} In the above notation, we have $hh_-\le hl_-\le 2\kappa^c(\Gamma).$
 \end{lemma}

\proof By (\ref{hl}), it suffices to prove the second inequality.
There are $h$ white beads on $\bf z$. Every such a bead $o$ belongs to
one of the paths ${\bf z}^i$ having $\theta$-length at most
$\max_{i=1}^rl^r$. Therefore for every such $o$, there are at least
$l_-$ white beads $o'$ on $\bf z$ such that $o$ and $o'$ are separated
on $\bf z$ by a black bead. Thus, we obtain at least $hl_-$ pairs
$(o,o')$ of white beads on $\bf z$ separated by black beads. Since one
of the pairs $(o,o')$ and $(o',o)$ contributes $1$ to
$\kappa^c(\Gamma)$, the lemma is proved.
\endproof

Let the handle $\cal C$ of a comb $\Gamma$ is a $t^{\pm 1}$- or $(t')^{\pm 1}$-band
with history having no $(23)$-rules or no $(12)$-rules, respectively;
and every derivative ${\cal C}_i$ is a $k^{\pm 1}$- or a $(k')^{\pm 1}$-band such that 
there are no special $\theta$-edges  (corresponding to the rules (12) and (23)) in the derivative subcomb $\Gamma_i.$
A subband $\cal B$ of some ${\cal C}_i$
which has neither $(12)$- nor $(23)$-edges
and is maximal with respect to this property, is called a \label{shortder}{\it short derivative} of $\cal C$.
By Property \ref{iv}, there are no maximal $a$-bands starting and ending on the same short derivative band.
Let \label{h'1...} $h'_1,\dots$ be the lengths of all short derivatives. Let \label{h'} $h'$ be the number of maximal $\theta$-bands in
$\Gamma$, which do not correspond to the rules $(12)$ and $(23)$.
Define \label{h'minus} $h'_- = h' - \max h'_j.$

\begin{lemma}\label{l'gamma} In the above notation,  we have $hh'_-\le 6\lambda(z^{\Gamma})=6\lambda^c(\Gamma)$.
 \end{lemma}

 \proof The sets of ends of the $\theta$-bands crossing  two short derivatives are separated
 in ${\bf z}^{\Gamma}$ either by a $q$-edge or by a non-special $\theta$-edge.
 Therefore arguing as in the proof of Lemma \ref{lgamma}, we come to inequality
 $h'h'_-\le 2\lambda({\bf z}^{\Gamma})= 2\lambda^c(\Gamma).$ (We note that under the assumption
 on the history,  $\lambda({\bf y}^{\Gamma})=0$
 since $\cal C$ is a $t^{\pm 1}$- or $(t')^{\pm 1}$-band.)
 This implies the statement of the lemma  if
 $h'_-\ne 0$ since in this case we have $h'\ge h/3$ because the handle
 $\cal C$, being a reduced diagram, cannot have two consecutive cells corresponding
 to the rules $(12), (23)$ (or inverse). If $h'_-=0$ the claim of the lemma is obvious.
 \endproof

\begin{lemma}\label{a-bands} Let $\Gamma$ be a comb with a handle
$\cal C$ of length $h$. Then the number $a_{ij}$ of all maximal $a$-bands of
$\Gamma$ starting on a derivative band ${\cal C}_i$ and ending on the bands ${\cal
C}_j$ with $j\ne i$  is at most $h_-.$ The total number of
cells in these $a$-bands over all $i<j$ does not exceed $hh_-\le
2\kappa^c(\Gamma).$
\end{lemma}

\proof Recall that derivative subcombs with different handles ${\cal
C}_i$ and ${\cal C}_j$ are disjoint and separated by these handles
(which are $q$-bands) from the remaining part of $\Gamma$. Therefore
every $a$-band $\cal A$ connecting some ${\cal C}_i$ and ${\cal
C}_j$ ($i\ne j$), connects $a$-edges of cells on the right sides of
these derivative bands. But every $q$-cell of ${\cal C}_i$ has at
most one $a$-edge on the right side of it by \ref{cell} (b). Besides, a connecting $a$-band $\cal A$ under consideration
either starts or ends on some ${\cal C}_j$, where $j\ne i.$ Thus
the total number of all connecting $a$-bands cannot exceed
$h-h_{r}$ for arbitrary $r\le s$. Now the first statement of the
lemma follows from the definition of $h_-.$ Since the number of
cells in $\cal A$ is at most $h$ by Lemma \ref{NoAnnul}, the second
statement is also proved by Lemma \ref{lgamma}.
\endproof

\begin{lemma}\label{a-bands'} Let $\Gamma$ be a comb and let  its handle $\cal C$ be a $t^{\pm 1}$- or $(t')^{\pm
1}$-band
with history having no $(23)$-rules or $(12)$-rules, respectively,
and every derivative ${\cal C}_i$ is a $k^{\pm 1}$- or a $(k')^{\pm
1}$-band such that there are
no special $\theta$-edges in the derivative subcomb $\Gamma_i.$
Let  ${\cal B}_1,\dots$ be the system of all short derivative bands. Then
the number of all the maximal $a$-bands of $\Gamma$ starting on a short derivative
${\cal B}_i$ and ending on some ${\cal B}_j$ ( $j\ne i$ ), where
${\cal B}_i$ and ${\cal B}_j$ are subbands of the same derivative
band, is at most $h'_-.$ The total number of cells in these
$a$-bands over all $i,j$ does not exceed $hh'_-\le 6\lambda({\bf z}^{\Gamma})\le
6\lambda^c(\Gamma).$
\end{lemma}

\proof The proof is similar to the proof of Lemma \ref{a-bands}, but
one should use Lemma \ref{l'gamma} instead of Lemma \ref{lgamma}.
\endproof

\begin{lemma}\label{deriv} (a) Let $\Gamma_1,\dots,\Gamma_s$ be
the derivative subcombs of a comb $\Gamma$. Then \\
$\sum_{i=1}^s\kappa^c(\Gamma_i)\le \kappa^c(\Gamma).$ 

(b) If  the history of a comb $\Gamma$ is $H\equiv H(1)\dots H(t),$
and $\Gamma(1),\dots,\Gamma(t)$ are $H(1)$- $,\dots, H(t)$-parts of $\Gamma$, resp.
($\Gamma(i)$ is absent if $H_i$ is empty), then 
$\sum_{i=1}^t\kappa^c(\Gamma(i))\le \kappa^c(\Gamma)$.
 \end{lemma}
 \proof (a),(b) Note that every white bead of $\Gamma_i$ (of $\Gamma(i)$)
 is placed on the boundary of $\Gamma$, and two white beads of
 $\partial\Gamma_i$ separated by a black bead are also separated by the same
 black bead on $\partial\Gamma.$ Since the sets of white
 beads of $\Gamma_i$ and $\Gamma_j$ (of $\Gamma(i)$ and $\Gamma(j)$) are disjoint for $i\ne j$,
 the statements (a) and (b) follow from the definition of $\kappa^c(.).$
\endproof

\begin{lemma} \label{simple} (a)Let $\Gamma$ be a comb.
Then the number $\alpha$ of $a$-edges in ${\bf z'}=({\bf z'})^{\Gamma}$ does not
exceed $(\delta')\iv(|{\bf z'}|-h)=(\delta')\iv(|{\bf z}|-h-2).$

(b) Assume in addition that the handle $\cal C$ is passive from the
left and there are no derivative bands ${\cal C}_i$ such that some
non-trivial $a$-band starts and ends on $\partial {\cal C}_i.$ Then the
total area of all simple $\theta$-bands ${\cal
S}_1,\dots,{\cal S}_h$ of $\Gamma$ is at most
$$h(h_-+\alpha+1)\le h( h_-+(\delta')\iv(|{\bf z}|-h-1)).$$
\end{lemma}

\proof (a) Notice that  $|{\bf z}'|\ge h+\delta'\alpha$ by Lemmas
\ref{NoAnnul} and \ref{ochev} (a), and so \\ $\alpha\le
(\delta')\iv(|{\bf z'}|-h)=(\delta')\iv(|{\bf z}|-h-2).$

(b) The total number of cells in all $a$-bands connecting the
derivative bands is at most $hh_-$ by Lemma \ref{a-bands}. If a
$(\theta,a)$-cell $\pi$ of a simple $\theta$-band ${\cal S}_j$ does not
belong to any such connecting $a$-bands, then one of the ends of the
maximal $a$-band containing $\pi$ must belong to $\partial\Gamma$
because $\cal C$ is passive from the left. The number of such
$a$-bands is at most $\alpha$,
 and the total number of their cells is at most
$\alpha h$, because their lengths  do not exceed $h$ by Lemma \ref{NoAnnul}
. Since a simple band has one $q$-cell, the number of cells in
all the simple $\theta$-bands is at most $h(h_- +\alpha+1)$.
\endproof

\begin{rk}\label{pustbudut} If ${\cal C}_i$ is a derivative band
of a comb $\Gamma$, then every $a$-band connecting two cells from
${\cal C}_i$ is of length at most $h_i$, and the total area of
of such bands crossing simple bands of $\Gamma$ at most $h_i^2/2.$
Hence if we omit the assumption that there are no derivative bands
${\cal C}_i$ such that some non-trivial $a$-band starts and ends on
${\cal C}_i,$ then we may add $\sum_{i=1}^s h_i^2/2\le \frac{h}{2} \sum_{i=1}^s h_i$
to the estimate of Lemma \ref{simple} (b), and in this case the
total number cells $n_s$ in all  simple bands satisfies
\begin{equation}
n_s\le h(h_-+\alpha+1+\sum_{i=1}^s h_i/2 )\le
h((\delta')\iv(|{\bf z}|-h-1)+\frac32\sum_{i=1}^s h_i) \label{ns}
\end{equation}

\end{rk}

If we use both Lemmas \ref{a-bands} and \ref{a-bands'} instead of
Lemma \ref{a-bands} in the proof of Lemma \ref{simple} we get

\begin{lemma} \label{comb'} Let $\Gamma$ be a comb and let its handle $\cal C$
be a $t^{\pm 1}$- or $(t')^{\pm 1}$-band
 with history having no $(23)$-rules or no $(12)$-rules, respectively,
and every derivative ${\cal C}_i$ is a $k^{\pm 1}$- or a $(k')^{\pm
1}$-band such that there are
 no special $\theta$-edges in the derivative  subcombs
of ${\Gamma }_i.$
Then the total area of all simple $\theta$-bands of $\Gamma$ is
at most $h(h_-+h'_-+\alpha+1)\le h(h_-+
h'_-+(\delta')\iv(|{\bf z}|-h-1)),$ where $\alpha$ is the number of
$a$-edges in ${\bf z'}=({\bf z'})^{\Gamma}$
\end{lemma} $\Box$

The proof of the following lemma can be obtained from the proof of
\cite[Lemma 4.10]{OS} by replacing  $|\Gamma|_a$ by $|{\bf z}|_a$ and replacing the constant $C$  by $2$ (since $C$ was the maximum  of the numbers of $a$-letters in
$(\theta,q)$-relations in \cite{OS}).

\begin{lemma} \label{comb} Let $h$ and $b$ be the height
and the base width of a comb $\Gamma$, respectively, and let
$\ttt_1,\dots \ttt_h$ be consecutive $\theta$-bands of $\Gamma$ . We
can assume that $\bott(\ttt_1)$ and $\topp(\ttt_h)$ are contained in
$\partial\Gamma$. Let $\alpha=|z^{\Gamma}|_a,$ and we denote by $\alpha_1$ the
number of $a$-edges on $\bott(\ttt_1)$. Then $\alpha + 8hb\ge
2\alpha_1$, and the area of $\Gamma$ does not exceed $4bh^2+2\alpha
h$.
\end{lemma}

\begin{rk}

For a comb $\Gamma$, we will use symbol \label{[Gamma]} $[\Gamma]$ to denote
the product $h^{\Gamma}(|{\bf z}^{\Gamma}|-|{\bf y}^{\Gamma}|).$ As we noted in Introduction,
an estimate of the form $\area(\Gamma) \le C[\Gamma]$ (where $C$ does not depend on $\Gamma$) would be perfect for the proof of the main theorem. It follows from the definition of comb
that every maximal $\theta$-band of $\Gamma$ starting on  ${\bf y}^{\Gamma}$ ends on 
${\bf z}^{\Gamma}$ and vice versa, that is $|{\bf z}^{\Gamma}|_{\theta}=|{\bf y}^{\Gamma}|_{\theta}=h^{\Gamma}.$ Clearly $|{\bf z}^{\Gamma}|_q-2\ge |{\bf y}^{\Gamma}|_{q}=0$
since the path ${\bf z}^{\Gamma}$ contains at least $2$ $q$-edges of the handle of $\Gamma,$
and therefore, when we estimate $|{\bf z}^{\Gamma}|-|{\bf y}^{\Gamma}|$ from below in the proofs
of several lemmas, our goal  is to obtain a lower bound for the difference $|{\bf z}^{\Gamma}|_{a}- |{\bf y}^{\Gamma}|_{a}$.

\end{rk}

We observed earlier that if the handle of a comb $\Gamma$ is passive from the right, then $ |{\bf y}^{\Gamma}|=h^{\Gamma},$ and so $[\Gamma]$ is equal to $h^{\Gamma}(|{\bf z}^{\Gamma}|-h^{\Gamma}),$ and therefore it is
positive. Moreover:

\begin{lemma}\label{nizko}
 If the height $h$ of a comb $\Gamma$ does not exceed $(\delta')^{-1}$  ,
then $|{\bf z}^{\Gamma}|-|{\bf y}^{\Gamma}|>0$ and the area of $\Gamma$ does not
exceed $4(\delta')^{-1}  [\Gamma]$.
\end{lemma}
\proof By Lemma \ref{ochev} (c), $|{\bf y}^{\Gamma}|\le h^{\Gamma}+1$. If the 
base width of $\Gamma$ is $b$, then, by Lemma \ref{NoAnnul}, ${\bf z}^{\Gamma}$ has at lest $2b$ $q$-edges and at least $h^{\Gamma}$ $\theta$-edges. Hence by Lemma \ref{ochev}(a),
$|{\bf z}^{\Gamma}|\ge 2b+h+\delta'|{\bf z}^{\Gamma}|_a.$ Therefore 
$|{\bf z}^{\Gamma}|-|{\bf y}^{\Gamma}|\ge (2b-1) + \delta'|{\bf z}^{\Gamma}|_a >0$ since $b\ge 1.$
Then, by Lemma \ref{comb}, $$Area(\Gamma) \le 4b(h^{\Gamma})^2+2|{\bf z}^{\Gamma}|_ah^{\Gamma} \le h^{\Gamma}(\delta')^{-1} (4(2b-1)+
2\delta'|{\bf z}^{\Gamma}|_a) \le 4(\delta')^{-1} h^{\Gamma}(|{\bf z}^{\Gamma}|-|{\bf y}^{\Gamma}|)$$
\endproof

\begin{rk}

Further we are finding appropriate estimates for the areas of  combs $\Gamma$-s
or for the areas of some proper subcombs of them provided the base width $b$ of $\Gamma$ is not
too small and not too large. It is not small in some lemmas because we need a choice to select a suitable subcomb of $\Gamma$,
and $b$ is not too large since the estimates of Lemma \ref{comb} and of other lemmas
depend on $b$. 
 The sufficiency
of the upper bound $b\le 15N$ will be seen later. 

\end{rk}

\begin{lemma} \label{h0} Let $\Gamma$ be a comb with base width $b\le 15N$ and
with passive handle $\cal C.$  Assume that $\Gamma$ has a derivative band ${\cal C}_{i_0}$
which contains an active from the right subband $\tilde{\cal C}$ of length   $h_0.$
Assume also that at most $(1-\delta) h_0$ maximal $a$-bands starting
on $\tilde{\cal C}$ and ending on one of the bands ${\cal C}$, ${\cal
C}_1,\dots, {\cal C}_s.$ Then (a) $\area(\Gamma)\le
(\delta')^{-2}[\Gamma]$ if $h_0\ge \delta h;$
(b) $\sum_{i=1}^s \area(\Gamma_i)\le
(\delta')^{-2}[\Gamma]$ for the set of derivative subcombs $\Gamma_1,\dots,\Gamma_s$ if
$h_0\ge \delta \sum h_i.$
\end{lemma}

 \proof To prove statement (a), we consider two cases.

 {\bf Case 1.} Assume that $\alpha = |{\bf z}|_a \ge \delta^2h/2.$ Then
 $|{\bf z}|-|{\bf y}|\ge\delta'\alpha\ge \delta'\delta^2h/2$ by Lemma  \ref{simple} (a) since $|{\bf y}|=h.$
 Hence by Lemma \ref{comb} with $b\le 15N$, we have
  $$\area(\Gamma)\le 60Nh^2 +2\alpha h\le  h(|{\bf z}|-|{\bf y}|)(60N\times 2(\delta')\iv\delta^{-2}+2(\delta')\iv)\le
  (\delta')^{-2}[\Gamma]$$
  since $(\delta')^{-1}> 120N \delta^{-2}+2.$

  {\bf Case 2.} Let $\alpha = |{\bf z}|_a < \delta^2h/2.$
  It follows from the condition of the lemma
  that at least $h_0 -2$ maximal $a$-bands start on $\tilde{\cal C}$  but at most $(1-\delta)h_0$
  end not on $\bf z.$ Therefore $\alpha \ge \delta h_0 -2\ge \delta^2h-2.$   The arising  in this
  case inequality $\delta^2h-2 < \delta^2h/2$ implies $h<4\delta^{-2}<(\delta')^{-1}$ by the choice of $\delta'.$
  Now Claim (a) follows from Lemma \ref{nizko}.

The proof of statement (b) is similar, but now two cases appear due to the comparison of $\alpha$ with
$\delta^2(\sum h_i)/2$, which leads to inequality $\sum h_i<4\delta^{-2}$ in the second case. Also one takes into account inequalities $\sum h_i\le h$ and $\sum (|{\bf z}_i|-|{\bf y}_i|)\le |{\bf z}|-|{\bf y}|$ in both cases.

\endproof

  \begin{lemma} \label{width1} Assume that a comb $\Gamma$ has no maximal $q$-bands except for
  its handle $\cal C$, and there are
  no non-trivial $a$-bands  both starting and terminating
  on ${\bf y'}=({\bf y'})^{\Gamma}$.

  (a) If $\cal C$ is active from the left or passive from the left, then $$\area'(\Gamma)\le(\delta')\iv h(|{\bf z'}|-|{\bf y'}|+1).$$

  (b) If $\cal C$ is active from the left or $\cal C$ is passive (from both sides), then \\
  $$\area(\Gamma)\le(\delta')\iv[\Gamma].$$
  \end{lemma}

\proof (a) Let $\cal T$ be a longest $\theta$-band in $\Gamma$, $d$ the number of $a$-cells in $\cal T$. Denote by
$T_1$ and $T_2$ the top and the bottom of $\cal T$.
Consider the families ${\bf S}_1$ and ${\bf S}_2$ of $a$-bands starting on $T_1$ and $T_2$,
respectively, which are maximal with respect to the requirement that these bands
do not contain cells from $\cal T.$ 
Observe that the cardinalities
$|{\bf S}_1|$ and $|{\bf S}_2|$ of these sets satisfy inequality
\begin{equation}\label{S1S2}
-1\le |{\bf S}_1|-|{\bf S}_2|\le 1
\end{equation}
since every maximal $a$-band of $\Gamma$ crossing $T_1$ has to cross $T_2$ (and vice versa),
with at most one exception for the $a$-band stating on the $a$-edge of the unique (see \ref{cell} (b))
$(\theta,q)$-cell of $\cal T.$

If there were  a non-trivial $a$-band from ${\bf S}_1$ and a non-trivial $a$-band from ${\bf S}_2$
both ending on the path $\bf y',$ then the maximal extension of one of them would connect
different edges of $\bf y'$ contrary to the assumption of the lemma.
Therefore either no non-trivial band from ${\bf S}_1$ ends on $\bf y'$
or no non-trivial band from ${\bf S}_2$ ends on $\bf y'.$

We may consider the former case only.

\bigskip

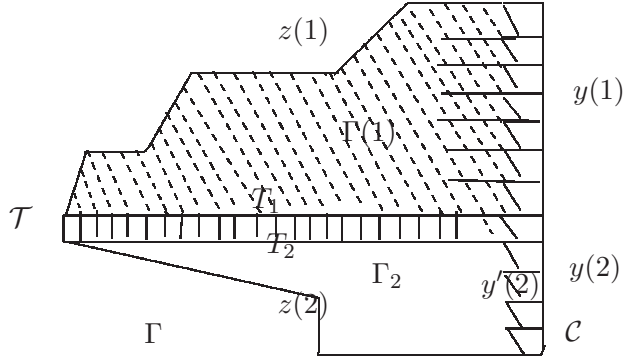
\unitlength 1mm 
\linethickness{0.4pt}
\ifx\plotpoint\undefined\newsavebox{\plotpoint}\fi 


\bigskip

The path $T_2$ cuts $\Gamma$ into two subdiagrams.
We denote by $\Gamma(1)$ the subdiagram of $\Gamma$ containing the bands from ${\bf S}_1$
and the $\theta$-band $\cal T$.
It has boundary ${\bf y}(1){\bf z}(1)T_2$, where ${\bf y}(1)$ and ${\bf z}(1)$ are subpaths of ${\bf y}={\bf y}(2){\bf y}(1)$ and ${\bf z}={\bf z}(1){\bf z}(2)$,
respectively. Similarly, ${\bf y'}={\bf y'}(1){\bf y'}(2)$, ${\bf z'}={\bf z'}(2){\bf z'}(1),$
and we define $\Gamma(2)$ as a sudiagram bounded by ${\bf z}(2){\bf y}(2)T_2^{-1}.$

Denote by $\bf S$ the family of maximal $a$-bands of $\Gamma(1)$ starting
on $\partial \cal C$.
   It follows from the choice of $\Gamma(1)$ that the families ${\bf S}_1$ and $\bf S$
  have at most one common band starting on the intersection of $\cal C$ and $\cal T$,
  and every $a$-band from these families must end on ${\bf z'}(1)$. Thus,
  \begin{equation}\label{S1}
  |{\bf z'}(1)|_a \ge |{\bf y'}(1)|_a+|{\bf S}_1|-1.
  \end{equation}

  {\bf Case 1.} If  $\cal C$ is active from the left, then $|{\bf y'}(1)|_a \ge |{\bf y'}(1)|_{\theta} -2$
 by the definition. Note that $|{\bf y'}(1)|_{\theta} = |{\bf z'}(1)|_{\theta}$ by Lemma \ref{NoAnnul},
 and so  $|{\bf y'}(1)|_a \ge |{\bf z'}(1)|_{\theta} -2$.
 This inequality together with (\ref{S1}) imply  $|{\bf z'}(1)|_a\ge |{\bf z}(1)|_{\theta}+|{\bf S}_1|-3$,
 and by Lemma \ref{ochev} (a),
$$|{\bf z'}(1)|\ge |{\bf z'}(1)|_{\theta}+\delta'(|{\bf z'}(1)|_{\theta}-3)+\delta(|{\bf S}_1|-3)=|{\bf y'}(1)|_{\theta}(1+\delta')+\delta(|{\bf S}_1|-3)-3\delta'\ge$$ 
\begin{equation} \label{0206}
|{\bf y'}(1)| +\delta(|{\bf S}_1|-3)-3\delta'=|{\bf y'}(1)| +\delta|{\bf S}_1|-3(\delta+\delta')
\end{equation}

Since at most $|{\bf S}_2|$  maximal $a$-bands of $\Gamma(2)$ starting on
$\cal C$ terminate on $\cal T$, Lemma \ref{ochev} and (\ref{S1S2})
give inequality $|{\bf y'}(2)|\le |{\bf z'}(2)|+\delta'|{\bf S}_2|\le
|{\bf z'}(2)|+\delta'(|{\bf S}_1|+1)$. Therefore by (\ref{0206}),
$$|{\bf y'}|\le |{\bf y'}(1)|+|{\bf y'}(2)|\le |{\bf z'}(1)|+|{\bf z'}(2)| - \delta
|{\bf S}_1|+3(\delta+\delta') +\delta'(|{\bf S}_1|+1).$$ Since by Lemma \ref{ochev} (b), $|{\bf z'}(1)|+|{\bf z'}(2)|
\le |{\bf z'}|+\delta-\delta'$ and also $|{\bf S}_1|\ge d-1,$
it follows
that $$|{\bf y'}|<|{\bf z'}|-(\delta-\delta')|{\bf S}_1|+4\delta\le |{\bf z'}|-(d-1)(\delta-\delta')+4\delta,$$
whence $d\le (\delta-\delta')^{-1} (|{\bf z'}|-|{\bf y'}|+5\delta).$

{\bf Case 2.} If $\cal C$ is passive from the left, we have $|{\bf y'}(1)|=|{\bf y'}(1)|_{\theta},$ and therefore
$$|{\bf z'}(1)|\ge |{\bf y'}(1)|+\delta'(|{\bf S}_1|-1)$$
by (\ref{S1}) and by Lemma \ref{ochev} (a).  Also we have $|{\bf z'}(2)|\ge |{\bf y'}(2)|$ by Lemmas
\ref{NoAnnul}
and \ref{ochev} (c), and so
$$|{\bf y'}|\le |{\bf y'}(1)|+|{\bf y'}(2)|< |{\bf z'}(1)|+|{\bf z'}(2)| - \delta'(|{\bf S}_1|-1)< |{\bf z'}| -\delta' (d-1)+\delta.$$
Hence $d \le (\delta')\iv(|{\bf z'}|-|{\bf y'}|)+1+\delta(\delta')\iv.$

Thus, in any case $d< (\delta')\iv(|{\bf z'}|-|{\bf y'}|+1)$ because $5\delta'<\delta<1/5.$

The number of maximal $\theta$-bands of $\Gamma$ is $h$, whence\\
$\area'(\Gamma)\le dh <(\delta')\iv h(|{\bf z'}|-|{\bf y'}|+1)$, as required.

(b)Now it follows from the assumption of the lemma and the definition of length,
that $|{\bf y}|\le |{\bf y'}|+2\delta'$ because if $\cal C$ is active from the left, then it
has at most two passive from the left cells. However, $|{\bf z}|=|{\bf z'}|+2$, and so $|{\bf z}|-|{\bf y}|\ge |{\bf z'}|-|{\bf y'}|+2-2\delta'.$ Now
by (a), and inequality $3\delta'<1,$ $$\area (\Gamma)=\area'(\Gamma)+h\le (\delta')\iv
h(|{\bf z'}|-|{\bf y'}|+1+\delta')\le (\delta')\iv h(|{\bf z}|-|{\bf y}|)  =(\delta')\iv[\Gamma]$$
\endproof

We say that a $q$-band  $\ccc'$ is \label{close} {\em close } to a $q$-band $\ccc$
in a diagram $\Delta$ without hubs (i.e. over the group $M$ ) if every maximal $\theta$-band
crossing $\ccc'$ also crosses $\ccc$.

Observe that a derivative band of a comb is close to the handle of
this comb.

\begin{df}
If $\cal C'$ is close to $\cal C$, then there is a unique minimal
subtrapezium in $\Delta$ containing both $\cal C'$ and a subband $\cal B$
of $\cal C,$ where the numbers of $(\theta,q)$-cells in $\cal C'$ and
in $\cal B$ are equal. We will denote this \label{fillingst} {\it
filling} subtrapezium by \label{Tp.} $Tp(\cal
   C',\cal C)$.
   \end{df}

\begin{lemma}\label{ppp} Assume that a comb $\Delta$ has no active $k$- or $k'$-cells
and contains a $\theta$-band $\cal T$ having a subword
$(pp^{-1}p)^{\pm 1}$ in the base, where $p$ is a control letter.
Then $\Delta$ has a one-Step
subcomb $\Gamma$ such that
$\area(\Gamma)\le(\delta')\iv[\Gamma].$
\end{lemma}
\proof Consider the maximal $p$-bands $\cal C'$ and $\cal C''$ of
$\Delta$ crossing $\cal T$ at the $(q,\theta)$-cells corresponding to
the first and the third letters of the subword $(pp^{-1}p)^{\pm 1}$,
respectively. Then $\cal C'$ is a handle of a subcomb $\Delta'$ of
$\Delta$, and the filling trapezium $Tp({\cal C',\cal C''})$ has
base $(pp^{-1}p)^{\pm 1}.$ By \ref{i}, \ref{ii}, all the $q$-cells
of $Tp({\cal C',\cal C''})$, in particular, of $\cal C',$ are active
from both sides and has one-step history.  It follows from Property
\ref{ii}
 applied to the comb $\Delta'\cup Tp({\cal C',\cal
C''}),$ that every $q$-band of $\Delta'$ is either active from both
sides or passive, and no non-trivial $a$-band of $\Delta'$ can start
and end on the same $q$-band by \ref{iv}. Hence there is a subcomb
$\Gamma$ of $\Delta'$ satisfying the condition of Lemma \ref{width1}
(b), and the statement is proved. \endproof

\section{Chains and quasicombs}

\subsection{Intersections of chains and $\theta$-bands}

Let the boundary $\partial\pi$ of a $(\theta,q)$-cell $\pi$
have a $p_i$-edge for a control state letter $p_i$. Then by Property \ref{cell} (a),
$\partial\pi$ either has no $a$-edges or contains two $a$-edges. Below we utilize this 
property in the definition of chain.

Assume that ${\cal A}_1,\dots, {\cal A}_m$ are maximal $a$-bands
such that ${\cal A}_i$ terminates on an $a$-edge of an active $p$-cell $\pi_i$ and ${\cal
A}_{i+1}$ starts with a different $a$-edge of $\pi_i$ ($i=1,\dots,m-1$). Then we say that
${\cal A}_1,\pi_1, {\cal A}_2,\pi_2,\dots,{\cal A}_m$ form a \label{chain} chain
with $m-1$ \label{link} links $\pi_1,\dots,\pi_{m-1}.$
By Property \ref{cell}, all links
are $p$-cells
for the same control base letter $p=p_j.$  A chain ${\bf
A}$ is called a \label{chain-ann}{\it chain-annulus} if the first $a$-edge of ${\cal A}_1$
coincides with the last one of ${\cal A}_m.$

It also follows that if ${\cal A}_1$ starts with
an edge $e_1$ and ${\cal
A}_m$ ends with $f_m,$ then the letters $\Lab(f_m)$ and $\Lab(e_1)$ are
either equal or they are copies of each other in different subalphabets 
$Y_i$ and $Y_{i+1}$ (corresponding, respectively, to some parts $Y'_j$ and $Y''_j$ of the
tape alphabet of the machine $M_3$). 

A chain ${\bf A}$ is {\it non-trivial} if it has at least one cell.

\begin{lemma}\label{nochain} Let $\bf x$ be a subpath of the boundary of a reduced
diagram $\Delta$ without hubs, and $\Lab(\bf x)$ a reduced word in a tape 
subalphabet $Y_i$ of the machine $M.$
Then no chain ${\bf A}$ can start and end on $\bf x.$
\end{lemma}
\proof Proving by contradiction, we assume that a non-trivial chain
${\bf A}$ starts and ends on $\bf x$. Notice that it cannot cross itself
since every cell has at most two $a$-edges by \ref{cell} (a). Thus it starts with an
edge $e$ of $\bf x$ and ends at an edge $f$ of $\bf x$, where $e\ne f,$ but
$\Lab(e)\equiv \Lab(f)^{-1}$ since these two letters belong to the
same subalphabet $Y_i.$ 
We may assume that ${\bf A}$ is chosen so that $e{\bf x'}f$ is a subpath
of $\bf x$, where the subpath $\bf x'$ is of minimal possible length.

Now the chain ${\bf A}$ and the path $\bf x'$ bound a subdiagram $\Delta'$ all of
whose boundary $a$-edges belong to $\bf x'.$ By our observation, every
$q$-edge on the boundary $\partial\Delta'$ is a $p^{\pm 1}$-edge for
the same  control base letter $p=p_j.$ Also every $p$-cell $\pi$ of
$\Delta'$ is active from both sides. Indeed  if a cell $\pi$ of $\Delta'$
is not active from both sides it does not belong to the chain ${\bf A}$ by definition.
Therefore $\pi$ must have a neighbor $(\theta, s_i)$-cell in $\Delta'$
by (\ref{p}) (consider the maximal $\theta$-band of $\Delta'$ containing $\pi$).
Then we may apply  Lemma \ref{NoAnnul} to the maximal $s_i$-band of $\Delta'$
containing the $(\theta,s_i)$-cell and obtain a letter of the form $s_i$ 
in the boundary label of
$\bf A$ or in $\Lab(\bf x),$ a contradiction. 

It follows that every $(\theta,q)$-cell in $\Delta'$ is a $p$-cell (corresponds
to a control letter) active from both sides. Since a maximal chain cannot end on a $p$-cell,  every
maximal chain of $\Delta'$ must start and end on $\bf x'$.

This property and the minimality of choice for $\bf x'$ imply that
$|{\bf x'}|=0,$ and so $\Lab(\bf x)$ has a non-empty freely trivial subword
$\Lab(ef);$ a contradiction. The lemma is proved.
\endproof

\begin{lemma}\label{2times} Let a non-trivial chain ${\bf A}$ crosses
a maximal $p$-band $\cal C$ of a reduced diagram $\Delta$ over $M$
twice at  cells $\pi'$ and $\pi''.$ Then there is a  $p$-cell $\pi_0$
in $\cal C$ between $\pi'$ and $\pi''$, which is not active from both sides.
\end{lemma}
\proof  As in the proof of Lemma \ref{nochain}, there is a
subdiagram $\Delta'$ bounded by the portion of ${\bf A}$ and a
segment $\bf x$ of the side of $\cal C$ between $\pi$ and $\pi'.$ As
there, arguing by contradiction, we may assume that $\bf x$ has no
$a$-edges. Hence, if there is no cell $\pi_0$ lying between $\pi'$
and $\pi''$ in $\cal C$ which is not active from both sides then the
cells $\pi$ and $\pi'$ must share a $p$-edge. Since the chain ${\bf
A}$ connects $\pi$ and $\pi',$ we obtain that these two cells must
have the same $a$-letter in the boundary labels (being letters from
the same subalphabet $Y_i$),
and they are mirror copies of each other
as it follows from the defining relations of the group $M,$ having
two $a$-letters.

We come to a contradiction because $\Delta$ is a reduced
diagram. The lemma is proved.
\endproof

\begin{lemma}\label{chain}  Assume that no $\theta$-band of a comb $\Delta$
has a base with a subword $(pp^{-1}p)^{\pm 1}$, where $p$ is a control letter.
Then every chain
${\bf A}$ of $\Delta$ has at most $9$ common $(\theta,a)$-cells  with any
$\theta$-band $\cal T$ of $\Delta.$
\end{lemma}

 \proof
 Let $\pi_1,\dots \pi_m$ be the common cells of ${\bf A}$ and ${\cal
 T}$
counted on $\cal T$
 from left to right. Denote by $\cal T'$ the subband of ${\cal T}$
 starting with $\pi_1$ and ending by $\pi_m.$ Every $(\theta,q)$-cell  of
 $\cal T'$ must be a $p$-cell since every maximal $q$-band
 of $\Delta$ crossing $\cal T',$ has to cross the chain ${\bf A}$
 too by Lemma \ref{NoAnnul}. Since by the assumption of the lemma and by Property \ref{order}, the
 base of $\cal T$ has no triples of consecutive $p^{\pm 1}$-letters, the subband $\cal T'$
 can have at most two $(\theta,q)$-cells.

 Assume that ${\bf A}$ crosses $\cal T$ consequently at $3$
 $(\theta,a)$-cells $\pi_{i_1}, \pi_{i_2}, \pi_{i_3}$, and
 $i_1<i_3<i_2$, i.e., ${\bf A}$ has a `convolution of a spiral' ${\bf A'}$,
starting at $\pi_{i_1}$ and ending at $\pi_{i_3}.$

\bigskip

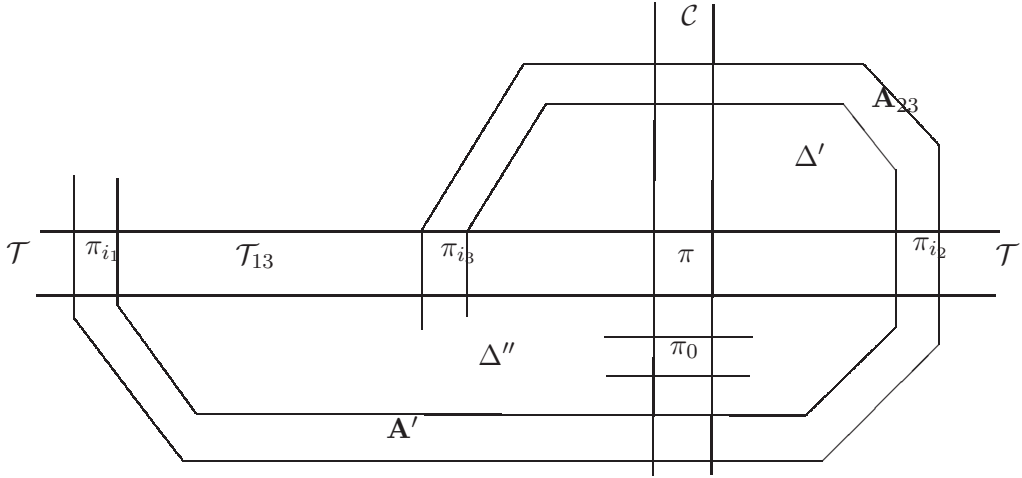
\unitlength 1mm 
\linethickness{0.4pt}
\ifx\plotpoint\undefined\newsavebox{\plotpoint}\fi 
\begin{picture}(146,71.25)(15,0)
\put(18.5,40.75){\line(1,0){127.5}}
\put(18,32.25){\line(1,0){127.5}}
\put(23,48.25){\line(0,-1){19}}
\multiput(23,29.25)(.0336391437,-.0443425076){327}{\line(0,-1){.0443425076}}
\multiput(34,14.75)(.033653846,-.043269231){104}{\line(0,-1){.043269231}}
\put(37.5,10.25){\line(1,0){85}}
\put(122.5,10.25){\line(1,1){15.5}}
\put(138,25.75){\line(0,1){26.5}}
\multiput(138,52.25)(-.0336700337,.0361952862){297}{\line(0,1){.0361952862}}
\put(128,63){\line(-1,0){45.25}}
\multiput(82.75,63)(-.0336658354,-.0548628429){401}{\line(0,-1){.0548628429}}
\put(69.25,41){\line(0,-1){13.25}}
\put(28.75,47.75){\line(0,-1){16.75}}
\multiput(28.75,31)(.0336538462,-.046474359){312}{\line(0,-1){.046474359}}
\multiput(39.25,16.5)(10.125,-.03125){8}{\line(1,0){10.125}}
\multiput(120.25,16.25)(.0343839542,.0336676218){349}{\line(1,0){.0343839542}}
\put(132.25,28){\line(0,1){21}}
\put(132.25,49){\line(-4,5){7}}
\put(125.25,57.75){\line(-1,0){39.5}}
\multiput(85.75,57.75)(-.0336538462,-.0544871795){312}{\line(0,-1){.0544871795}}
\put(75.25,40.75){\line(0,-1){11.25}}
\multiput(100.25,71.25)(-.03125,-7.875){8}{\line(0,-1){7.875}}
\multiput(108,71)(-.03125,-7.8125){8}{\line(0,-1){7.8125}}
\put(93.5,26.75){\line(1,0){19.75}}
\put(93.75,21.5){\line(1,0){19}}
\put(145.5,36.5){$\cal T$}
\put(64.5,13.25){$\bf A'$}
\put(24.5,37.75){$\pi_{i_1}$}
\put(71.75,37.5){$\pi_{i_3}$}
\put(134.5,38.25){$\pi_{i_2}$}
\put(103.25,36.75){$\pi$}
\put(102.25,24.5){$\pi_0$}
\put(103.75,68.25){$\cal C$}
\put(129,57.25){${\bf A}_{23}$}
\put(44.5,36.5){${\cal T}_{13}$}
\put(76.75,23){$\Delta''$}
\put(118.75,49.5){$\Delta'$}
\put(14,36.75){$\cal T$}
\end{picture}

By Lemma \ref{nochain} for the part $\Delta'$ of
$\Delta$ bounded by the subchain ${\bf A}_{23}$ of ${\bf A'}$ connecting $\pi_{i_2}$
and $\pi_{i_3}$ and the subband of $\cal T$ connecting the same
cells, $\cal T$ has a $p$-cell $\pi$ between $\pi_{i_2}$ and
$\pi_{i_3}.$ The maximal $p$-band $\cal C$ crossing $\cal T$ at $\pi$ must
cross ${\bf A'}$ at least twice (above and below $\cal T$).
Hence there is a part of ${\bf A'}$ satisfying, together with $\cal
C$, the condition of Lemma \ref{2times}, and so there is a cell
$\pi_0$ given by that lemma, inside the part $\Delta''$ of $\Delta$
bounded by ${\bf A'}$ and the part ${\cal T}_{13}$ of
$\cal T$ connected $\pi_{i_1}$ and $\pi_{i_3}.$  But then, according to
Property (\ref{p}),
  $\Delta''$
has to contain some $s_i$-cell neighboring  $\pi_0,$
contrary to Lemma \ref{NoAnnul} since
neither ${\bf A'}$ nor ${\cal T}_{13}$ have $s_i$-cells.
Thus
our assumption on the existence of ${\bf A'}$ leads to a
contradiction, and so $i_3>i_2$ if $i_2>i_1.$

Assume that ${\bf A}$ crosses $\cal T$ consequently at $4$
$(\theta,a)$-cells $\pi_{i_1}, \pi_{i_2}, \pi_{i_3}, \pi_{i_4}$. We
may assume that $i_1<i_2$, and so $i_1<i_2<i_3<i_4.$ Then again by
Lemma \ref{nochain}, we have at least 3 $p$-cells on $\cal T'$
(between $\pi_{i_u}$ and $\pi_{i_{u+1}}$ for $u=1,2,3$), a
contradiction. Hence such a series of common $(\theta,a)$-cells is
impossible, and since ${\bf A}$ has at most $2$ common
$(\theta,q)$-cells with $\cal T,$ we conclude that traveling along
${\bf A}$ one meets $m\le 3+1+3+1+3=11$ cells  of $\cal T$ (at most
$3$  $(\theta, a)$-cells, then a $(\theta,q)$-cell, and so on) , and
the number of $a$-cells among them does not exceed $9.$
\endproof

\begin{lemma}\label{chainan}
A comb $\Delta$ has no chain-annuli.
\end{lemma}
\proof Assume that ${\bf A}$ is a chain-annulus. By Lemma
\ref{NoAnnul}  it must have a link $\pi_i.$ Then there is a left-most
maximal $p_i^{\pm 1}$-band $\cal
C'$ of the comb $\Delta$ containing a link $\pi_i$ from $\bf A,$ where $p_i$ is a control letter, 
i.e., the subcomb $\Gamma$ of
$\Delta$ with handle $\cal C'$ has no links of ${\bf A}$  except for
those on $\cal C'$. But $\partial\pi_i$ has an $a$-edge from the
left by \ref{cell}, and so two different $a$-edges of $\cal C'$ are
connected by an $a$-band ${\cal A}_j$ from the left of $\cal C'.$ We will
assume that ${\cal A}_j$ is the shortest $a$-band with this property,
and so the $(\theta,q)$-cells $\pi^1,\dots, \pi^m$ situated on $\cal
C'$ between $\pi_{j-1}$ and $\pi_j$ (if any) are inside the
subdiagram $\Delta'$ bounded by the chain $\bf A$. By Lemma
\ref{NoAnnul} for $\Delta',$ these cells cannot have common edges
with $(\theta,q)$-cells which do not correspond to $p_i^{\pm 1}.$
Now it follows from \ref{p} that the cells $\pi^1,\dots,\pi^m$ are
active from both sides. Since the cells $\pi_{j-1}$ and $\pi_j$ of
the chain-annulus are also active from both sides, the existence of
the $a$-band ${\cal A}_j$ contradicts Property \ref{iv}.
 The lemma is proved. \endproof
 
 \subsection{An application of quasicombs}
 
 The surgery we use in Lemma \ref{qi} requires a slight modification of 
 the notion of comb. 
 
 Let $\Delta$ be a reduced diagram over $M$ with boundary $\bf yz$ such that every
 maximal $\theta$-band of $\Delta$  has exactly one $\theta$- edge from $\bf y.$ 
 Assume that one can construct a $q$-band $\cal Q$ with
 a top or bottom path ${\bf y}_0,$ and $\Lab(\bf y)$ can be obtained from $\Lab({\bf y}_0)$
 after deleting of some $a$-letters. Then we say that $\Delta$ is a \label{quasic}{\it quasicomb}
 with the {\it support} $\bf y.$
 As for combs, 
 we use the standard factorization \label{yq} \label{zq} ${\bf yz}={\bf y}^{\Delta}{\bf z}^{\Delta}$ of the
 boundary path of a quasicomb. The number \label{hq} $h=h^{\Delta}$ is the $\theta$-length
 $|{\bf y}|_{\theta};$ the notation \label{[q]} $[\Delta]$ is also extended to quasicombs as
well as \label{hminusq} $h_-,$ and the \label{muq} $\kappa$-, $\lambda$-, $\mu$-, and \label{nuq} $\nu$-mixtures.

 In particular, every comb is a quasicomb. (Take $\cal Q$ to be the handle of the comb.) 
 It follow from the definition that the history $H$ of the quasicomb $\Delta$ is the history of $\cal Q,$
 and by Lemma \ref{NoAnnul}, the boundary label of $\Delta$ uniquely
 determines the end edges of all maximal $\theta$-bands ${\cal T}_1,\dots, {\cal T}_h$
 starting on $\bf y.$ It is easy to see that the set of cells of ${\cal T}_1$ is also
 uniquely determined by the boundary of $\Delta,$ and by induction, the same is true
 for ${\cal T}_2,\dots, {\cal T}_h$. Hence every quasicomb (in particular, every comb)
 is a minimal diagram.

 \begin{rk} \label{quasi} The statements and the proofs of Lemmas \ref{chain} and
 \ref{chainan} remain valid if $\Delta$ is a quasicomb.
 \end{rk}
  
  We say that a diagram $\Delta$ \label{admit} {\it admits} a (proper) quasicomb $\Gamma$ if it has a 
  (proper) subcomb
  $\Gamma_0$ such that $\Lab({\bf z}^{\Gamma})\equiv \Lab({\bf z}^{\Gamma_0}),$ the handle of $\Gamma_0$
 serves for $\Gamma$ as $\cal Q$ in the definition of quasicomb, and the words 
 $\Lab({\bf y}^{\Gamma})$ and $\Lab({\bf y}^{\Gamma_0})$ are equal modulo $(\theta,a)$-relations.
 In particular, every subcomb of $\Delta$ is admitted.

\begin{lemma}\label {qi} Let $\Delta$ be a comb 
 with history of type $(2)$.
Assume that $\Delta$ has a  subcomb $\Gamma_0$ of base width $\le 3$ with a $s_j^{\pm 1}$-handle $\cal C'$ such that
the filling trapezium $Tp(\cal C',\cal C)$ is not aligned. 
Also assume that all maximal $q$-bands of $\Gamma_0$ except for $\cal C',$ are $p^{\pm 1}$-bands.
Then $\Delta$ admits a proper quasicomb $\Gamma$ such that $Area(\Gamma)\le 5(\delta')^{-1}[\Gamma].$ 
\end{lemma}
\proof By Property \ref{ix}, $\Lab({\bf y}^{\Gamma_0})$ has a factorization of the form\\ $u(b_1v_1b_1^{-1})\dots (b_m v_m b_m^{-1})w,$ 
where $b_i^{\pm 1}$ is an $a$-letter or $1$ ($i=1,\dots,m$) , $v_i$ is a group word in $\theta$-letters,
 $b_i$ commutes with every letter of $v_i$ by virtue of  $(\theta, a)$-relations, and
each of $u,$ $w$ has at most one $a$-letter. Similar property holds for $\Lab({\bf y'}^{\Gamma_0})$.
 Hence one can separate the band $\cal C'$ from the diagram $\Gamma_0$ and attach
 several $(\theta,a)$-cells to the left and to the right sides of it, and obtain an
 auxiliary subdiagram $E$ with the boundary $e{\bf x}f{\bf x'}$ whose label 'almost' equal to
 the boundary label of $\cal C',$ but $\Lab(\bf x)$ (resp., $\Lab(\bf x')$) are obtained
 from $\Lab(\bf y^{\Gamma_0})$ (from $\Lab(\bf y'^{\Gamma_0})$) by deleting of the $a$-letters $b_i^{\pm 1}$-s,
 and therefore each of $\bf x$ and $\bf x'$ has at most $2$ $a$-edges.
 
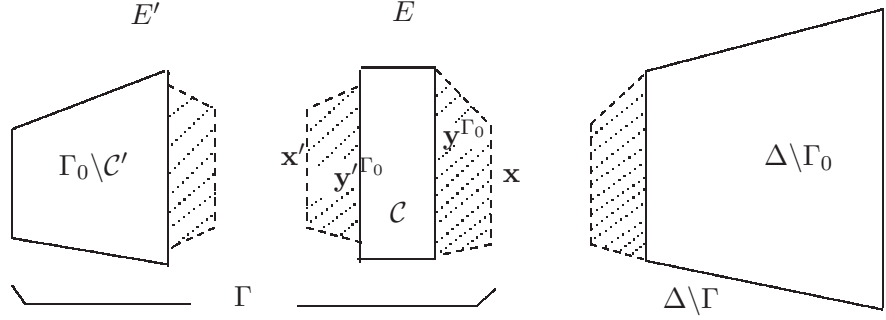
\unitlength 1mm 
\linethickness{0.4pt}
\ifx\plotpoint\undefined\newsavebox{\plotpoint}\fi 
\begin{picture}(125.5,54.5)(0,0)
\put(30.5,46.25){\line(0,-1){26}}
\multiput(30,46.25)(-.088043478,-.033695652){230}{\line(-1,0){.088043478}}
\put(9.75,38.5){\line(0,-1){14.5}}
\multiput(9.75,24)(.199519231,-.033653846){104}{\line(1,0){.199519231}}
\put(56,46.25){\line(0,-1){25.5}}
\put(56.25,46.75){\line(1,0){9.75}}
\put(66,46.75){\line(0,-1){25.5}}
\put(66,21.25){\line(-1,0){10.25}}
\multiput(94,46.25)(.1285714286,.0336734694){245}{\line(1,0){.1285714286}}
\put(125.5,54.5){\line(0,-1){32}}
\put(94,46.25){\line(0,-1){25.25}}
\multiput(94,21)(.163212435,-.033678756){193}{\line(1,0){.163212435}}
\put(125.5,22.75){\line(0,-1){8}}
\put(30.68,45.93){\line(-1,0){.125}}
\multiput(30.43,44.43)(.0600962,-.03125){13}{\line(1,0){.0600962}}
\multiput(31.992,43.617)(.0600962,-.03125){13}{\line(1,0){.0600962}}
\multiput(33.555,42.805)(.0600962,-.03125){13}{\line(1,0){.0600962}}
\multiput(35.117,41.992)(.0600962,-.03125){13}{\line(1,0){.0600962}}
\put(36.68,41.18){\line(0,-1){.9844}}
\put(36.68,39.211){\line(0,-1){.9844}}
\put(36.68,37.242){\line(0,-1){.9844}}
\put(36.68,35.273){\line(0,-1){.9844}}
\put(36.68,33.305){\line(0,-1){.9844}}
\put(36.68,31.336){\line(0,-1){.9844}}
\put(36.68,29.367){\line(0,-1){.9844}}
\put(36.68,27.398){\line(0,-1){.9844}}
\multiput(36.68,25.43)(-.068182,-.03125){11}{\line(-1,0){.068182}}
\multiput(35.18,24.742)(-.068182,-.03125){11}{\line(-1,0){.068182}}
\multiput(33.68,24.055)(-.068182,-.03125){11}{\line(-1,0){.068182}}
\multiput(32.18,23.367)(-.068182,-.03125){11}{\line(-1,0){.068182}}
\multiput(55.68,44.18)(-.070313,-.03125){12}{\line(-1,0){.070313}}
\multiput(53.992,43.43)(-.070313,-.03125){12}{\line(-1,0){.070313}}
\multiput(52.305,42.68)(-.070313,-.03125){12}{\line(-1,0){.070313}}
\multiput(50.617,41.93)(-.070313,-.03125){12}{\line(-1,0){.070313}}
\put(48.93,41.18){\line(0,-1){.9844}}
\put(48.93,39.211){\line(0,-1){.9844}}
\put(48.93,37.242){\line(0,-1){.9844}}
\put(48.93,35.273){\line(0,-1){.9844}}
\put(48.93,33.305){\line(0,-1){.9844}}
\put(48.93,31.336){\line(0,-1){.9844}}
\put(48.93,29.367){\line(0,-1){.9844}}
\put(48.93,27.398){\line(0,-1){.9844}}
\put(48.93,25.43){\line(0,1){0}}
\multiput(48.93,25.43)(.115079,-.031746){7}{\line(1,0){.115079}}
\multiput(50.541,24.985)(.115079,-.031746){7}{\line(1,0){.115079}}
\multiput(52.152,24.541)(.115079,-.031746){7}{\line(1,0){.115079}}
\multiput(53.763,24.096)(.115079,-.031746){7}{\line(1,0){.115079}}
\multiput(55.374,23.652)(.115079,-.031746){7}{\line(1,0){.115079}}
\multiput(66.18,46.43)(.0335648,-.0347222){18}{\line(0,-1){.0347222}}
\multiput(67.388,45.18)(.0335648,-.0347222){18}{\line(0,-1){.0347222}}
\multiput(68.596,43.93)(.0335648,-.0347222){18}{\line(0,-1){.0347222}}
\multiput(69.805,42.68)(.0335648,-.0347222){18}{\line(0,-1){.0347222}}
\multiput(71.013,41.43)(.0335648,-.0347222){18}{\line(0,-1){.0347222}}
\multiput(72.221,40.18)(.0335648,-.0347222){18}{\line(0,-1){.0347222}}
\put(73.43,38.93){\line(0,-1){.9844}}
\put(73.43,36.961){\line(0,-1){.9844}}
\put(73.43,34.992){\line(0,-1){.9844}}
\put(73.43,33.023){\line(0,-1){.9844}}
\put(73.43,31.055){\line(0,-1){.9844}}
\put(73.43,29.086){\line(0,-1){.9844}}
\put(73.43,27.117){\line(0,-1){.9844}}
\put(73.43,25.148){\line(0,-1){.9844}}
\multiput(73.43,23.18)(-.151042,-.03125){6}{\line(-1,0){.151042}}
\multiput(71.617,22.805)(-.151042,-.03125){6}{\line(-1,0){.151042}}
\multiput(69.805,22.43)(-.151042,-.03125){6}{\line(-1,0){.151042}}
\multiput(67.992,22.055)(-.151042,-.03125){6}{\line(-1,0){.151042}}
\multiput(93.68,45.68)(-.0353535,-.0328283){18}{\line(-1,0){.0353535}}
\multiput(92.407,44.498)(-.0353535,-.0328283){18}{\line(-1,0){.0353535}}
\multiput(91.134,43.316)(-.0353535,-.0328283){18}{\line(-1,0){.0353535}}
\multiput(89.862,42.134)(-.0353535,-.0328283){18}{\line(-1,0){.0353535}}
\multiput(88.589,40.952)(-.0353535,-.0328283){18}{\line(-1,0){.0353535}}
\multiput(87.316,39.771)(-.0353535,-.0328283){18}{\line(-1,0){.0353535}}
\put(86.68,39.18){\line(0,1){0}}
\put(86.68,39.18){\line(0,-1){.9412}}
\put(86.68,37.297){\line(0,-1){.9412}}
\put(86.68,35.415){\line(0,-1){.9412}}
\put(86.68,33.533){\line(0,-1){.9412}}
\put(86.68,31.65){\line(0,-1){.9412}}
\put(86.68,29.768){\line(0,-1){.9412}}
\put(86.68,27.886){\line(0,-1){.9412}}
\put(86.68,26.003){\line(0,-1){.9412}}
\put(86.68,24.121){\line(0,-1){.9412}}
\multiput(86.68,23.18)(.111111,-.031746){7}{\line(1,0){.111111}}
\multiput(88.235,22.735)(.111111,-.031746){7}{\line(1,0){.111111}}
\multiput(89.791,22.291)(.111111,-.031746){7}{\line(1,0){.111111}}
\multiput(91.346,21.846)(.111111,-.031746){7}{\line(1,0){.111111}}
\multiput(92.902,21.402)(.111111,-.031746){7}{\line(1,0){.111111}}
\put(67,36.5){${\bf y}^{\Gamma_0}$}
\put(75,31.5){$\bf x$}
\put(45.5,33.75){$\bf x'$}
\put(52.5,31){${\bf y'}^{\Gamma_0}$}
\put(60,26.25){$\cal C$}
\put(109.75,33.75){$\Delta\backslash\Gamma_0$}
\put(16,32.5){$\Gamma_0\backslash\cal C'$}
\put(25.5,52.5){$E'$}
\put(60.25,53){$E$}
\multiput(9.75,17.75)(.03365385,-.04807692){52}{\line(0,-1){.04807692}}
\put(11.5,15.25){\line(1,0){22}}
\put(47.75,15){\line(1,0){23.75}}
\multiput(71.5,15)(.03731343,.03358209){67}{\line(1,0){.03731343}}
\put(39.25,15){$\Gamma$}
\put(96.25,14.75){$\Delta\backslash\Gamma$}
\multiput(30.68,40.18)(.55,.45){6}{{\rule{.4pt}{.4pt}}}
\multiput(30.93,36.43)(.58333,.55556){10}{{\rule{.4pt}{.4pt}}}
\multiput(30.68,32.68)(.61111,.61111){10}{{\rule{.4pt}{.4pt}}}
\multiput(30.43,29.43)(.63889,.63889){10}{{\rule{.4pt}{.4pt}}}
\multiput(30.68,27.43)(.63889,.55556){10}{{\rule{.4pt}{.4pt}}}
\multiput(30.43,25.18)(.69444,.52778){10}{{\rule{.4pt}{.4pt}}}
\multiput(30.93,22.68)(.125,.125){3}{{\rule{.4pt}{.4pt}}}
\multiput(31.18,22.93)(.625,.5){9}{{\rule{.4pt}{.4pt}}}
\multiput(49.43,38.43)(.61111,.55556){10}{{\rule{.4pt}{.4pt}}}
\multiput(49.18,35.68)(.625,.55){11}{{\rule{.4pt}{.4pt}}}
\multiput(50.43,33.68)(.58333,.55556){10}{{\rule{.4pt}{.4pt}}}
\multiput(48.93,27.93)(.5833,.4167){4}{{\rule{.4pt}{.4pt}}}
\multiput(53.93,34.68)(.5,.5){4}{{\rule{.4pt}{.4pt}}}
\multiput(49.43,26.18)(.6875,.53125){9}{{\rule{.4pt}{.4pt}}}
\multiput(51.18,25.43)(.57143,.42857){8}{{\rule{.4pt}{.4pt}}}
\multiput(53.43,24.43)(.5,.4375){5}{{\rule{.4pt}{.4pt}}}
\multiput(66.18,42.93)(.5833,.5){4}{{\rule{.4pt}{.4pt}}}
\multiput(66.18,40.43)(.6,.5){6}{{\rule{.4pt}{.4pt}}}
\multiput(68.68,39.43)(.4,.4){6}{{\rule{.4pt}{.4pt}}}
\multiput(66.43,34.43)(.5,.4375){5}{{\rule{.4pt}{.4pt}}}
\multiput(66.18,31.93)(.64286,.53571){8}{{\rule{.4pt}{.4pt}}}
\multiput(66.18,29.68)(.675,.65){11}{{\rule{.4pt}{.4pt}}}
\multiput(66.18,27.43)(.675,.625){11}{{\rule{.4pt}{.4pt}}}
\multiput(65.93,25.68)(.7,.575){11}{{\rule{.4pt}{.4pt}}}
\multiput(66.18,23.43)(.675,.575){11}{{\rule{.4pt}{.4pt}}}
\multiput(66.68,22.18)(.65,.5){11}{{\rule{.4pt}{.4pt}}}
\multiput(69.68,22.68)(.6,.45){6}{{\rule{.4pt}{.4pt}}}
\multiput(86.68,37.43)(.675,.55){11}{{\rule{.4pt}{.4pt}}}
\multiput(87.18,35.68)(.69444,.52778){10}{{\rule{.4pt}{.4pt}}}
\multiput(86.68,32.93)(.7,.55){11}{{\rule{.4pt}{.4pt}}}
\multiput(86.68,31.18)(.725,.5){11}{{\rule{.4pt}{.4pt}}}
\multiput(86.93,29.18)(.69444,.5){10}{{\rule{.4pt}{.4pt}}}
\multiput(86.68,27.18)(.7,.475){11}{{\rule{.4pt}{.4pt}}}
\multiput(86.68,25.18)(.7,.525){11}{{\rule{.4pt}{.4pt}}}
\multiput(87.18,23.93)(.69444,.5){10}{{\rule{.4pt}{.4pt}}}
\multiput(88.68,22.93)(.75,.5){8}{{\rule{.4pt}{.4pt}}}
\multiput(90.68,22.43)(.55,.45){6}{{\rule{.4pt}{.4pt}}}
\multiput(92.18,21.93)(.5,.4167){4}{{\rule{.4pt}{.4pt}}}
\end{picture}

 Then we continue the surgery as follows. We can construct 
 the mirror copies of the $(\theta,a)$-cells attached to $\bf y'^{\Gamma_0}$ in $E,$ and attach
 these copies to the diagram $\Gamma_0\backslash \cal C'$ to obtain (after
 possible cancellations of some $(\theta, a)$-cells) a reduced quasicomb $E'$
 whose support can be denoted by $\bf x'$ since its label is $\Lab(\bf x'),$ and the boundary
 of $E'$ is ${\bf x'}({\bf z'})^{\Gamma_0}.$ Finally,
 we identify $E'$ and $E$ along $\bf x'$ and (after possible cancellations of $(\theta,a)$-cells) we have
 a desired quasicomb $\Gamma$ with support $\bf x.$ Now, to estimate the area of the
 minimal diagram $\Gamma$ from above it suffices to estimate $Area (E')$ (and use Lemma \ref{2diagr}) since
 obviously $Area(E)\le 3 (Area({\cal C'}))\le 3h.$ 
 
 Since the path $\bf x$ has at most $2$ $a$-edges
we obtain

\begin{equation} \label{igrek}
h^{\Gamma_0}\le |{\bf x}|\le h^{\Gamma_0}+2\delta'
\end{equation}
 by the definition of the length of a path.
 
 Note that at most two non-trivial chains of
$E'$ can start/end on $\bf x'$ since $|{\bf x'}|_a\le 2.$ Consequently, by Lemma
\ref{chain} (and Remark \ref{quasi}) , every $\theta$-band of $E'$ has at most $18$ cells
belonging to such chains.
 
 Let $d$ be the maximal number of $a$-cells in $\theta$-bands of
$E'$. 
Taking into account the argument of the previous paragraph and the lack of chain-annuli in $E',$ 
we conclude that there are at least $(d-18)/9$ maximal chains
having both ends on
$({\bf z'})^{\Gamma_0}.$  Now it follows from Lemma \ref{ochev} and Inequality (\ref{igrek}) that
$$|{\bf z}^{\Gamma}|\ge |({\bf z'})^{\Gamma_0}|+2\ge 2+h^{\Gamma_0}+2\delta'(d-18)/9 \ge
2+|{\bf x}|-2\delta'+2\delta'(d-18)/9>
|{\bf x}|+1+2\delta'(d+2)/9$$ because $\delta'\le 1/7.$ 
This implies $d+2\le 5(\delta')^{-1}(|{\bf z}^{\Gamma}|-|{\bf x}|-1),$  
and since every $\theta$-band of $E'$ has at most $2$ $p$-cells, we have, as required:
$$\area(\Gamma)\le Area(E')+Area(E) \le (2+d)h^{\Gamma}+3h^{\Gamma}< 5h^{\Gamma}(\delta')^{-1}(|{\bf z}^{\Gamma}|-|\bf x|)$$
\endproof

\begin{lemma} \label{nokpk} Let $\Delta$ be a comb
with base width $b\ge N.$
Assume that $\Delta$ has a one-Step history and has no active $k$- or $k'$-cells.
   Then  $\Delta$ admits a proper quasicomb $\Gamma$ such that
   $Area(\Gamma)\le 5(\delta')^{-1} [\Gamma].$

  \end{lemma}

 \proof
  We first assume that the history of $\Delta$ is of type $(2)$. It follows from the conditions of the lemma and \ref{kk'} that $\Delta$ cannot have $k$- or $k'$-cells at all, with the possible exception
  for the cells of the first and the last
  $(12)^{\pm 1}$ or $(23)^{\pm 1}$-bands. 
 By Lemma \ref{ppp}, we may also
assume that the base of any $\theta$-band
  of $\Delta$ does not have subword $(pp^{-1}p)^{\pm 1},$ where $p^{\pm 1}$ is any control letter.

Therefore a left-most maximal $s_i^{\pm 1}$-, $k^{\pm 1}$-, $(k')^{\pm 1}$-, $t^{\pm 1}$-, or
$(t')^{\pm 1}$-band $\cal C'$ of $\Delta$ is a handle of
 a subcomb $\Gamma$ of base width $b^{\Gamma}\le 3$
 and all other $q$-bands
of $\Gamma$ (if any exist) correspond to the same control letter
$p^{\pm 1}$ by \ref{order}. One may assume that $\cal C'$ has $h' > (\delta')^{-1}$
cells because otherwise Lemma \ref{nizko} is applicable to the subcomb with
handle $\cal C'.$

The filling trapezium $\Delta'=Tp(\cal
C',\cal C)$ has height $h'>(\delta')^{-1} >2$, and therefore it has a maximal $\theta$-band without $k$- or $k'$-cells.
Hence the base
of $\Delta'$ has no  $k$- or $k'$-letters.
Since $b\ge N$, $\Delta'$ is of base width $b'\ge
N-2,$ but having no $k$- or $k'$-letters, the
base of $\Delta'$ is not  aligned or has only $t^{\pm 1}$ or $(t')^{\pm 1}$-letters
by Remark \ref{align}. 
In the former case, we are done by Lemma \ref{qi}.
In the later case we may refer to Lemma \ref{width1}  since there are no derivative $p^{\pm 1}$-bands
for a $t$-band by Property \ref{order}.

If the history of $\Delta$ is of type $(1)$ or $(3),$ then it follows from
the absence of active $k$- and $k'$-cells in $\Delta$ that all $q$-cells of $\Delta$ are passive by Property \ref{kk'}, 
and the statement follows from Lemma \ref{width1}.

\endproof

\section{Combs with one-Step histories}

In this section we obtain the estimates of the areas of one-Step combs.
Lemmas \ref{t}, \ref{oneage}, \ref{13} will be used in the next sections.

\begin{lemma} \label{onek} Let $\Gamma$ be a comb
  whose proper subcombs have no active $k$- or $k'$-cells
  and the handle $\cal C$ of $\Gamma$ is passive from the left.
  Assume that $\Gamma$ has one-Step history and admits no proper quasicombs $\Delta$ such that
  $\area(\Delta)\le 5(\delta')^{-1}[\Delta].$
  Assume that $\Gamma$ has at most $L_0$ odd maximal
   $\theta$-bands for some $L_0$.
  Then  (a) $\area'(\Gamma)\le 5(\delta')^{-1}h(|{\bf z'}|-|{\bf y'}|+L_0);$
  (b) if $\cal C$ is also passive from the right, then
  $\area(\Gamma)< 5(\delta')^{-1}h(|{\bf z}|-|{\bf y}|+L_0);$
  \end{lemma}

  \proof (a) By Lemma \ref{nokpk}, we may assume that the base width $b$ of $\Gamma$ does not exceed $N.$
 By Lemma \ref{ppp}, we may assume that $\Delta$ has no bands with  $p^{\pm 1}p^{\mp
1}p^{\pm 1}$ in the base for arbitrary control letter $p.$ By Lemma
\ref{chain},  every chain of $\Gamma$ and every $\theta$-band of $\Gamma$
have at most $9$ $a$-cells in common. We also use below that $\Gamma$ has no chain-annuli by Lemma \ref{chainan}.

  At most $2(b-1)L_0$ non-trivial maximal chains can start/terminate on the odd
  $\theta$-bands of $\Gamma$ by \ref{cell} (a) - (c). 
  Let $\cal T$ be a $\theta$-band of $\Gamma$ having maximal number of $a$-cells $d$.
  Then among maximal chains crossing $\cal T$, we have at least
$(d-18(b-1)L_0)/9$ chains with both ends
  on $\bf z'$, and so $|{\bf z'}|_a\ge 2(d-18(b-1)L_0)/9.$ Therefore by Lemma \ref{ochev}, $$|{\bf z'}|-|{\bf y'}|\ge
  2(b-1)+2\delta'(d-18(b-1)L_0)/9>
  2\delta'(d+b-1)/9 - 4\delta'NL_0 $$ since $b\le N$. Hence $d+b-1\le (9/2)(\delta')^{-1}(|{\bf z'}|-|{\bf y'}|)+18NL_0 $,
  and therefore $$\area'(\Gamma)\le (b-1+d)h\le 5h(\delta')^{-1}(|{\bf z'}|-|{\bf y'}| +L_0)$$ since $18N\le (\delta')\iv/2$.

(b) Since $\cal C$ is passive from the right, we have $|{\bf y}|\le
|{\bf y'}|+2\delta'$ and $|{\bf z}|=|{\bf z'}|+2.$ Therefore $|{\bf z}|-|{\bf y}|>|{\bf z'}|-|{\bf y'}|+1,$ and so statement (a) implies (b)
because $\area({\cal C})\le h$.
  \endproof

\bigskip

\begin{lemma}\label{t}
Let $\Gamma$ be a one-Step comb of base width $b\le 15N$.
Assume that  its handle $\cal C $ is a $t^{\pm 1}$- or $(t')^{\pm 1}$-band.
 Then

 (1) $\Gamma$ has a subcomb $\Delta$ such
that its handle ${\cal C}^{\Delta}$ is passive from the right, and
\begin{equation} \label{subcomb} \area(\Delta)\le
 c_1([\Delta]+\frac12  h^{\Delta}h_-^{\Delta})
 \end{equation}

 (2) $\area(\Gamma)\le c_1([\Gamma]+\kappa^c(\Gamma)) $.
\end{lemma}

\proof (1) We may assume that $\Gamma$ has no $t^{\pm 1}$ or $(t')^{\pm 1}$-bands
except for $\cal C,$ since otherwise we obtain a smaller subcomb
$\Delta$, and $\Delta$ also satisfies the assumptions of the lemma.
In particular, the derivative bands ${\cal C}_1,\dots, {\cal C}_s$
are all $k^{-1}$- or $k'$-bands by Property \ref{order}.

Assume first that the history of $\Gamma$ is of Step $(1)$ or $(3)$,
and the derivative bands are not active from the right. Then either
$\Gamma$ has a maximal $k$- or $(k')\iv$-band $\cal C'$ active from the left
and passive from the right
(and there exists the left-most band with this property), or all its
$q$-bands are passive by Properties \ref{order} and \ref{kk'}. In any case Lemma \ref{width1} (b) is applicable
to a subcomb of $\Gamma$. Thus we may further assume that the history of
$\Gamma$ is of Step 2. Property \ref{i} implies that the derivative bands of
$\Gamma$ are active from the right.

The sum of areas of derivative subcombs $\Gamma_1,\dots,\Gamma_s$ is
at most $\sum (60 Nh_i^2+2\alpha h_i)$ by Lemma \ref{comb}, where
$\alpha=|{\bf z}|_a$. Hence
by Property \ref{iv} and Lemma \ref{simple},
\begin{equation}\label{otsenka}
\area(\Gamma)\le h(h_- +3\alpha+1)+60N \sum h_i^2
\end{equation}

We also recall that by Lemma \ref{simple},
$\alpha\le(\delta')\iv(|{\bf z}|-|{\bf y}|-1)$ since $|{\bf y}|=h.$

We assume first that $h_{i}\le 2\sum_{j\ne i}h_j$ for every $i$.
Then $h_- \ge \frac13\sum h_i$, and therefore $\sum h_i^{2}\le h\sum
h_i\le 3hh_-$. Now it follows from (\ref{otsenka}) and the
subsequent estimate for $\alpha$ that $$\area(\Gamma)\le (60N)3hh_-
+hh_- + 3(\delta')\iv[\Gamma] \le c_1[\Gamma]+c_1hh_-/2,$$ 
by the choice of $c_1,$ as required.

Now assume that there is $i_0$ such that $h_{i_0}>2\sum_{j\ne i_0}h_j$. Then
less than $h_{i_0}/2$ maximal $a$-bands starting on ${\cal C}_{i_0}$
end on other derivative bands, and we can refer to Lemma \ref{h0} (b). Therefore
$\sum\area(\Gamma_i)\le (\delta')^{-2}[\Gamma]$. Hence
$$\area(\Gamma)\le (\delta')^{-2}[\Gamma]+h(h_- +3\alpha+1)\le (\delta')^{-2}[\Gamma]+ hh_-+3(\delta')^{-1}[\Gamma]
\le c_1([\Gamma]+hh_-/2),$$ as required, since $c_1 > (\delta')^{-2}+3(\delta')\iv$
 Thus the desired
inequality is true in any case.

\medskip

(2) We will induct on the number of maximal $q$-bands in $\Gamma$.
By (1), we have a subcomb $\Delta$ of $\Gamma$ satisfying
(\ref{subcomb}), and therefore

\begin{equation} \label{areadelta}
\area(\Delta)\le
 c_1[\Delta]+c_1\kappa^c(\Delta)
 \end{equation}
by Lemma \ref{lgamma}.

One may assume that the subcomb $\Delta$ of $\Gamma$ is proper since
otherwise it is nothing to prove. Now $\Gamma$ is a union of
$\Delta$ and the remaining comb $\Delta'$ with the handle $\cal C.$
We observe that
\begin{equation} \label{obs}
{\bf y}^{\Delta'}={\bf y} \;\;and  \;\;|{\bf z}^{\Delta'}|= |{\bf z}| -|{\bf z}^{\Delta}|+|{\bf y}^{\Delta}|
\end{equation}
by Lemma \ref{ochev} (b).
By the inductive hypothesis,

\begin{equation} \label{areadelta'}
\area(\Delta')\le
 c_1[\Delta']+c_1\kappa^c(\Delta')
 \end{equation}

Now the sum of the first summands in the right-hand sides of (\ref{areadelta}) and
(\ref{areadelta'}) does not exceed $c_1h(|{\bf z}|-|{\bf y}|)$ because
$(|{\bf z}^{\Delta'}|-|{\bf y}^{\Delta'}|) +(|{\bf z}^{\Delta}|-|{\bf y}^{\Delta}|)=|{\bf z}|-|{\bf y}|$ by ({\ref{obs}}).   The sum of
second summands of (\ref{areadelta}) and (\ref{areadelta'}) does not
exceed $c_1\kappa^c(\Gamma)$ by Lemma \ref{mu}(c), and the lemma
is proved since $\area(\Gamma)=\area(\Delta)+\area(\Delta').$
\endproof

Let $\Gamma$ be a  comb with a handle $\cal C.$ Assume that $\Gamma$
is a  subcomb of (or can be embedded as a subcomb in) a larger comb
$\bar \Gamma$ with a handle $\bar{\cal C}$, and the filling
trapezium is $Tp(\cal C,\bar{\cal C})$. Then the comb $\bar\Gamma$ is called
an {\it extension} of $\Gamma$. The extension is called \label{rexec} {\it regular}
if the base width of $Tp(\cal C,\bar{\cal C})$ is at least $N.$
A comb $\Gamma$ is called \label{regcomb}{\it regular} if there exists a regular extension of $\Gamma.$
Recall that by definition of comb, every cell (of the handle) of a subcomb
of  $\Gamma$ is connected with $\cal C$ by a $\theta$-band. Therefore a subcomb of
a regular  comb  is  regular itself.

\begin{rk} A regular comb is organized better than a random one because its
history coincides with the history of the filling {\it trapezium} $Tp(\cal C,\bar{\cal C})$
having a sufficiently long  base. Therefore this history is a subject of some restrictions 
imposed on trapezia by Proposition \ref{summary}. Recall that by Lemma \ref{simul}, the properties
of trapezia reflect all the features of the computations executed by the machine $M.$
In particular, the next lemma uses Property \ref{x} based on the aperiodicity of the histories formulated in Lemma \ref{*V}, which, in turn goes back to Lemmas \ref{norep} and \ref{m0m1}(a). 
\end{rk}

\begin{lemma}\label{otrez}
Let $\Gamma$ be a one-Step regular 
comb, and the handle $\cal C$ of $\Gamma$ active from the left $k$- or
$(k')\iv$-band.
 Assume that $\Gamma$ has neither $t^{\pm 1}$- nor $(t')^{\pm
 1}$-bands and has no active from the left maximal $k$- or $(k')\iv$-bands
 except for $\cal C,$
 and $\Gamma$ admits no proper quasicomb $\Delta$ such that
$\area(\Delta)\le 5(\delta')^{-1}[\Delta].$
  Then
$\area(\Gamma)\le 16(\delta')^{-2}[\Gamma]$. 
\end{lemma}

\proof

Let ${\cal C}_1, {\cal C}_2,\dots {\cal C}_s$ be the set of
derivative subbands (connected with $\cal C$ by simple $\theta$-bands). Since none of them is a $t^{\pm 1}$- or
$(t')^{\pm 1}$-band, they must be $k^{-1}$- or $k'$-bands active from the
right by \ref{order}.  Moreover, since the derivative subcombs $\Gamma_1,\dots, \Gamma_s$
has no $k$- or $k'$-cells active from the left, it is easy to see from Property \ref{order} that $\Gamma$ has no
active $k^{\pm 1}$- or $(k')^{\pm 1}$- cells except for those belonging to
${\cal C}, {\cal C}_1,\dots,{\cal C}_s$. By \ref{iv}), every maximal $a$-band
starting on $\cal C$ ends either on some ${\cal C}_i$ or on ${\bf z}={\bf z}^{\Gamma}.$
Besides, if $\cal T$ is a simple $\theta$-band starting on $\cal C$,
then every maximal $a$-band starting on $\cal
T$ ends either on one of the bands ${\cal C}, {\cal C}_1,\dots {\cal
C}_s$, or on ${\bf z}$.

Assume that $\cal T$ is a simple $\theta$-band of maximal length $d$
among the simple $\theta$-bands starting on $\cal C$. Then the total number of
cells in all simple $\theta$-bands  is at most $dh$.

Let $T_1$ and $T_2$ be top and bottom of $\cal T.$ Since no
non-trivial $a$-band starts and ends on $\cal C$, either no $a$-band
$\cal A$ starting on $T_1$ (an having no cell from $\cal T$)
 ends on $\cal C$ or no $a$-band $\cal A$
starting on 
$T_2$ (and having no cells from $\cal T$)
ends on $\cal C$. We consider only the first variant.

If there are $m$ maximal $\theta$-bands above $\cal T$, then at least $d-1$
maximal $a$-bands starting on $\cal T$ and $m-1$ maximal $a$-bands
starting on $\cal C$ end on at most $m$ $(q,\theta)$-cells of the
derivative  bands and on $z$. Therefore at least $d-2$ of them end on
$\bf z$ , and so $|{\bf z}|-|{\bf y}|\ge 2+(d-2)\delta'$ by Lemma \ref{ochev} (a).
We set $\Delta_i=Tp({\cal C}_i,\cal C)$ ($i=1,\dots,s;$ and $d\ge 2$ if $s>0$ since otherwise
two neighbor $k$- or $k'$-cells of $\cal T$ form a non-reduced subdiagram).  By Property
(\ref{xi}),  $H_i\equiv u_iw_i^{k_i}v_i$,
 where $|u_i|, |v_i|\le (d-1)/2,$ $|w_i|\le d-1$.

Consider a regular extension $\bar\Gamma$ of the comb $\Gamma.$
Let $\bar{\cal C}$ be a handle of $\bar\Gamma$ and $\Pi_i=Tp({\cal C}_i,\bar {\cal C}).$ 

If the base of $\Pi_i$ is not aligned, and $\Gamma_i$ has an $s^{\pm 1}_j$-band for some $j$, then
it follows from Property \ref{order} and Lemma \ref{ppp} that $\Gamma_i$ satisfies
the assumptions of Lemma \ref{qi}, and we have a contradiction with the hypothesis
of Lemma \ref{otrez}. Therefore $\Gamma_i$ has no $s_j^{\pm 1}$-bands, and so, by Property
\ref{order}, it has no maximal $q$-bands at all except for ${\cal C}_i.$ Hence by
Lemma \ref{ppp}, $Area'(\Gamma_i)\le (\delta')^{-1}h_i(|{\bf z'}_i|-|{\bf y'}_i|+1)$ in standard
notation, if the base
of $\Pi_i$ is not aligned.

  Assume that the base of $\Pi_i$ is aligned. It starts with $k$ (or with $(k')^{-1}$)
and the second letter of base is not $k^{-1}$ (not $k'$) by Lemma \ref{qqiv}. Therefore,
by Property \ref{order}, the second letter is the copy of the first letter (or of the inverse of the last letter)
of the standard base $B$ of $M_3.$    Since the base of $\Pi$ is aligned and its base width $\ge N> 2||B||$, this base is large.
Therefore if  $k_i\ge 3$, then the $w^{k_i}$-part
of $\Pi$ has no odd $\theta$-bands
 by  (\ref{x}), and so the number of maximal odd $\theta$-bands is at most $|u_i|+|v_i|\le d-1$. If $k_i\le 2$,
 then $h_i\le |u_i|+|w_i|+|v_i|\le 3(d-1)$. Thus, in any case, the number $L_i$ of odd maximal $\theta$-bands in $\Pi_i$ (and in $\Gamma_i$) is at most $3(d-1).$

 Using ${\bf z}_i, {\bf z'}_i, {\bf y}_i, {\bf y'}_i$ in the
 standard way for the subcombs $\Gamma_i$-s, we have $\area'(\Gamma_i)\le 5(\delta')^{-1}h_i(|{\bf z'}_i|-|{\bf y'}_i|+L_i)$ 
 by Lemma \ref{onek}, if the base of $\Pi_i$ is aligned. Thus in any case
 $\area'(\Gamma_i)\le 5(\delta')^{-1}h_i(|{\bf z'}_i|-|{\bf y'}_i|+3(d-1)).$ 
 Since $|{\bf y'}_i|=h_i$ and $|{\bf y}|=h,$ the sum of these areas is at most $5(\delta')^{-1}h(|{\bf z}|-|{\bf y}|+3(d-1)).$
 Finally,
 \begin{equation}\label{arga}
 \area(\Gamma)\le
 dh+5(\delta')^{-1}h(|{\bf z}|-|{\bf y}|+3(d-1))
 \le h(d+5(\delta')^{-1}(|{\bf z}|-|{\bf y}|+3(d-1))
 \end{equation}
 Recall that $d-2\le (\delta')\iv(|{\bf z}|-|{\bf y}|)$ and $|{\bf z}|-|{\bf y}|\ge 2$ since $\bf y$ has no $a$-edges.
 Therefore
 $$d+5(\delta')^{-1}(|{\bf z}|-|{\bf y}|+3(d-1)) \le 16(\delta')^{-2}(|{\bf z}|-|{\bf y}|),$$
 and so by (\ref{arga}), $\area(\Gamma)\le 16(\delta')^{-2}[\Gamma],$ as required.
 \endproof

We call a (quasi)comb $\Gamma$ \label{longqc}{\it long} if
$|{\bf z}|=|{\bf z}^{\Gamma}|>|{\bf y}|=|{\bf y}^{\Gamma}|.$ 
 If  the handle of a comb $\Gamma$ is
passive from the right, then $\Gamma$ is long since $|{\bf z}|\ge 2+h$ and
$|{\bf y}|=h.$ Obviously a (quasi)comb $\Gamma$ is long if $\area(\Gamma)\le
c[\Gamma]$ for some $c>0.$ 

\begin{rk} \label{quasiotrez}
Observe that if $\Gamma$ is
a long subcomb of a diagram $\Delta$  then for the complimentary
subdiagram $\Delta'= \Delta\backslash\Gamma$ cut from $\Delta$ along the path $\bf y,$ we have
$|\partial\Delta'|<|\partial\Delta|$ by Lemma \ref{ochev}.

If $\Gamma$ is a long quasicomb
admitted by a minimal diagram $\Delta$ with boundary path $\bf zz'$ (where ${\bf z=z}^{\Gamma}$), then
there exists a minimal diagram $\Delta'$ with boundary label $\Lab({\bf y}^{-1})\Lab(\bf z').$ 
It follows from Lemma \ref{NoAnnul} that if $\Delta$ is a comb with handle $\cal C,$ and $\Delta$ admits a proper quasicomb $\Gamma$, then $\Delta'$ is also a comb with the same handle $\cal C$
but with fewer number of maximal $q$-bands than in $\Delta.$ 

It is clear that $Area(\Delta)\le Area(\Gamma)+Area(\Delta')$ (or  see Lemma \ref{2diagr}). We also use
notation $\Delta\backslash\Gamma$ for such a `compliment' $\Delta'$ of $\Gamma.$
Then all the statements of Lemma \ref{mu} hold if one replaces `subcomb $\Gamma$' by
`admitted quasicomb $\Gamma$' in their formulations because
the quasicomb $\Gamma$ and the subcomb $\Gamma_0$ from the definition of admitted quasicomb
have equal $\kappa$-, $\lambda$-, $\mu$-, and $\nu$-necklaces.
\end{rk}

\begin{lemma}\label{oneage} Let $\Gamma$ be a one-Step
regular comb of base width $b\in [2N,15N]$. Then $\Gamma$ admits a
long quasisubcomb $\Delta$ such that
  $$\area(\Delta)\le c_1([\Delta]+\frac{1}{2}h^{\Delta}h_-^{\Delta})
  \le c_1([\Delta]+\kappa^c(\Delta))$$
 \end{lemma}

 \proof For the beginning, we recall that every subcomb of a regular comb is regular.
 By Lemma \ref{t} (and Lemma \ref{lgamma}), one may assume that $\Gamma$ has neither $t$- nor $t'$-cells.
 If $\Gamma$ has a subcomb $\Delta$ of base width $\ge N$
 without active $k^{\pm 1}$- or $(k')^{\pm 1}$-cells, then one can apply
 Lemma \ref{nokpk} since $c_1\ge 5(\delta')\iv.$ Otherwise by \ref{kk'},  there
 is a maximal, active from the left or from the right $q$-band $\cal C'$ corresponding to a $k^{\pm 1}$- or $(k')^{\pm 1}$-letter,
 such that the subcomb $\Delta$ with handle $\cal C'$  has no other maximal active $k^{\pm 1}$- or $(k')^{\pm 1}$-bands and the base width
 of the filling trapezium $Tp(\cal C', \cal C),$ is at least $N$  since $b\ge 2N.$  If  $\cal C'$ is active from the left, then the statement follows from Lemma \ref{otrez}.
 Otherwise $\cal C'$ is active from the right, and from the right of $\cal C',$ there must
 be a maximal band $\cal C''\ne \cal C$ corresponding to $k$- or $(k')^{-1}$-letter  which therefore is
 active from the left. (Recall that $\Gamma$ has neither $t$- nor $t'$-cells, and so such ${\cal C}''$ exists by Property \ref{order}.) Lemma \ref{otrez} is now applicable
 to the subcomb with handle $\cal C'',$ and so Lemma \ref{oneage} is proved in any case.
\endproof

\begin{lemma}\label{13} Let $\Gamma$ be a  regular comb having history of type $(1)$ or $(3)$ 
and base width at most $15N.$ Assume that the handle $\cal C$ is 
a passive from the left 
$k^{\pm 1}$- or $(k')^{\pm 1}$-, or $s_0$-band. Then $\area'(\Gamma)\le c_1h(|{\bf z'}|-|{\bf y'}|+1)+c_1\kappa^c(\Gamma).$
\end{lemma}

\proof At first we will prove that either $\Gamma$ admits a
proper quasicomb $\Delta$ such that

\begin{equation}
\area(\Delta)\le c_1[\Delta]+c_1\kappa^c(\Delta)\label{areadel}
\end{equation}
or $\area'(\Gamma)\le c_1h(|{\bf z'}|-|{\bf y'}|+1)+c_1\kappa^c(\Gamma).$

By Lemma \ref{t}, we may assume that
$\Gamma$ has neither $t^{\pm 1}$-bands nor  $(t')^{\pm 1}$-bands,
and it also has neither $k$-bands nor  $(k')^{-1}$-bands
active from the left by Lemma \ref{otrez}. Therefore by Property \ref{order},
there are no active $k^{\pm 1}$- or $(k')^{\pm 1}$-cells except for
those in $\cal C.$ Since the step history of $\Gamma$ is $(1)$ or $(3),$ every other $q$-bands of $\Gamma$ is
passive by \ref{kk'}
, and if there exist
other $q$-bands, then we can find the desired subcomb $\Delta$ by
Lemma \ref{width1}(b). If $\Gamma$ has no maximal $q$-bands except
for the handle $\cal C,$ then statement follows from Lemma
\ref{width1}(a).

To complete the proof, we will induct on the number of maximal
$q$-bands in $\Gamma,$ as in Lemma \ref{t} (b). By the previous argument,
we may assume that $\Gamma$ admits a proper (quasi)subcomb $\Delta$ 
satisfying (\ref{areadel}), since otherwise there is
nothing to prove. Now $\Gamma$ is a union of $\Delta$ and the
`compliment' $\bar\Delta= \Gamma\backslash\Delta$ (see Remark \ref{quasiotrez})
with the handle $\cal C.$ By the inductive hypothesis,

\begin{equation} \label{areade'}
\area'(\bar\Delta)\le
 c_1 h^{\bar\Delta}(|({\bf z'})^{\bar\Delta}|-|({\bf y'})^{\bar\Delta}|+1)+c_1\kappa^c_{1,1}(\bar\Delta)
 \end{equation}

Note that $h^{\bar\Delta}\le h.$
 It follows that the sum of the first summands of (\ref{areadel})
and (\ref{areade'}) does not exceed $c_1h(|{\bf z'}|-|{\bf y'}|+1).$ The sum of
second summands of (\ref{areadel}) and (\ref{areade'}) does not
exceed $c_1\kappa_{1,1}^c(\Gamma)$ by Lemma \ref{mu}(c), and lemma is
proved since $\area'(\Gamma)=\area(\Delta)+\area'(\bar\Delta).$

\section{Combs with incomplete sets of Steps}

In this section, we analyze combs whose histories have no rules
of one of the Steps $(1)$ or $(3)$, and the main goal is Lemma \ref{bez12}. 
 Lemmas \ref{bez1} and \ref{bez12} utilize the $\lambda$-mixture, but
 unfortunately, this parameter can be negative for some other combs. Therefore
 first of all we have to 
 bound it from below in terms of other `quadratic' parameters of combs.

\begin{lemma}\label{mu21} Let $\Gamma$ be a comb whose handle $\cal C$ is either (a)
a $t^{\pm 1}$-band or (b) a $(t')^{\pm 1}$-band. Respectively, let the history $H$
of $\Gamma,$ either (a)
have rules $(23)^{\pm 1}$ but no rules $(12)^{\pm 1}$ or (b)
have
rules $(12)^{\pm 1}$ but no rules $(23)^{\pm 1}.$ Then \\
$\lambda^c(\Gamma)\ge - 8(\delta')^{-1}[\Gamma] -
36\kappa^c(\Gamma)$.
\end{lemma}
\proof We shall prove the variant (a) only. For this goal we
consider the set $\bf T$ of maximal $(23)$-bands of $\Gamma,$ 
which do not cross derivative bands
of $\Gamma$, and so 
both their edges on ${\bf y'}^{\Gamma}$ and on ${\bf z}^{\Gamma}$ are
labeled by the same $\theta$-letter, and therefore they are non-special edges by the definition of $\lambda^c.$
 (The set $\bf T$ may be empty.) Consider all (non-empty) maximal
combs $\Delta_j$  ($j=1,\dots, r$) in which $\Gamma$ is separated
by the bands of $\bf T.$ (The combs $\Delta_j$-s do not contain the
separating $\theta$-bands from $\bf T.$) Then
\begin{equation} \label{deltaj}
\lambda^c(\Gamma) = \sum_j\lambda^c(\Delta_j)
\end{equation}
because arbitrary two white beads which are not separated by a black
one in ${\bf z}={\bf z}^{\Gamma}$ or in ${\bf y}={\bf y}^{\Gamma}$  belong to the boundary
of some $\Delta_j$ since black beads are placed on $q$-edges and
also on non-special $\theta$-edges (and also, every bead of $\partial\Delta_j$ is
on ${\bf z}$). Below we call an edge {\it black} ({\it white}) if its middle point is a
black (white) bead.

Therefore we will estimate $\lambda^c(\Delta_j)$ for every
$j$ from below and then will use (\ref{deltaj}).
Clearly, this number is at least
$-\lambda^c({\bf y}(j))$ ($j=1,\dots, r$), where ${\bf y}(j)={\bf y}^{\Delta_j}.$ To give an upper bound for $\lambda^c({\bf y}(j)),$
we apply Lemma \ref{mixturec} (e) to ${\bf y}_j$ and select an
appropriate 'black' edge $e_j$ (if any exists) on ${\bf y}(j)$
such that ${\bf y}(j)= {\bf y}_2(j)e_j{\bf y}_1(j)$ with $m_1(j)$ and $m_2(j)$
white beads on ${\bf y}_1(j)$ and
${\bf y}_2(j)$, respectively, and
\begin{equation}\label{2mm}
\lambda^c({\bf y}(j)) \le 2m_1(j)m_2(j)
\end{equation}

\unitlength 1mm 
\linethickness{0.4pt}
\ifx\plotpoint\undefined\newsavebox{\plotpoint}\fi 
\begin{picture}(115.75,78)(0,0)
\put(96.25,76){\line(0,-1){71.25}}
\put(104,75.75){\line(0,-1){70.25}}
\put(30,30.25){\line(1,0){74}}
\put(30.25,30.25){\line(0,-1){4.5}}
\put(30.25,25.75){\line(0,1){0}}
\put(30.25,26.5){\line(1,0){73.75}}
\multiput(96.25,30.25)(-.033707865,.03511236){178}{\line(0,1){.03511236}}
\multiput(90.25,36.5)(-.033653846,.110576923){104}{\line(0,1){.110576923}}
\put(86.5,47.75){\line(0,1){7.5}}
\put(86.5,55.25){\line(-1,0){6}}
\put(80.5,55.25){\line(0,-1){7.5}}
\multiput(80.5,47.75)(.033613445,-.094537815){119}{\line(0,-1){.094537815}}
\multiput(84.5,36.5)(.033687943,-.040780142){141}{\line(0,-1){.040780142}}
\put(89,31.25){\line(0,-1){4}}
\multiput(89,27.25)(-.03125,.0625){8}{\line(0,1){.0625}}
\put(88.75,27.75){\line(-2,-5){5.5}}
\put(83.25,14){\line(1,0){6.5}}
\multiput(89.75,14)(.033707865,.074438202){178}{\line(0,1){.074438202}}
\put(104.25,74){\line(-1,0){8}}
\multiput(96.43,73.93)(-.0334448,-.0635452){13}{\line(0,-1){.0635452}}
\multiput(95.56,72.278)(-.0334448,-.0635452){13}{\line(0,-1){.0635452}}
\multiput(94.691,70.625)(-.0334448,-.0635452){13}{\line(0,-1){.0635452}}
\multiput(93.821,68.973)(-.0334448,-.0635452){13}{\line(0,-1){.0635452}}
\multiput(92.951,67.321)(-.0334448,-.0635452){13}{\line(0,-1){.0635452}}
\multiput(92.082,65.669)(-.0334448,-.0635452){13}{\line(0,-1){.0635452}}
\multiput(91.212,64.017)(-.0334448,-.0635452){13}{\line(0,-1){.0635452}}
\multiput(90.343,62.364)(-.0334448,-.0635452){13}{\line(0,-1){.0635452}}
\multiput(89.473,60.712)(-.0334448,-.0635452){13}{\line(0,-1){.0635452}}
\multiput(88.604,59.06)(-.0334448,-.0635452){13}{\line(0,-1){.0635452}}
\multiput(87.734,57.408)(-.0334448,-.0635452){13}{\line(0,-1){.0635452}}
\multiput(86.864,55.756)(-.0334448,-.0635452){13}{\line(0,-1){.0635452}}
\multiput(80.43,55.18)(-.0681511,-.0334008){13}{\line(-1,0){.0681511}}
\multiput(78.658,54.311)(-.0681511,-.0334008){13}{\line(-1,0){.0681511}}
\multiput(76.886,53.443)(-.0681511,-.0334008){13}{\line(-1,0){.0681511}}
\multiput(75.114,52.574)(-.0681511,-.0334008){13}{\line(-1,0){.0681511}}
\multiput(73.342,51.706)(-.0681511,-.0334008){13}{\line(-1,0){.0681511}}
\multiput(71.57,50.838)(-.0681511,-.0334008){13}{\line(-1,0){.0681511}}
\multiput(69.798,49.969)(-.0681511,-.0334008){13}{\line(-1,0){.0681511}}
\multiput(68.026,49.101)(-.0681511,-.0334008){13}{\line(-1,0){.0681511}}
\multiput(66.254,48.232)(-.0681511,-.0334008){13}{\line(-1,0){.0681511}}
\multiput(64.482,47.364)(-.0681511,-.0334008){13}{\line(-1,0){.0681511}}
\multiput(62.71,46.495)(-.0681511,-.0334008){13}{\line(-1,0){.0681511}}
\multiput(60.938,45.627)(-.0681511,-.0334008){13}{\line(-1,0){.0681511}}
\multiput(59.167,44.759)(-.0681511,-.0334008){13}{\line(-1,0){.0681511}}
\multiput(57.395,43.89)(-.0681511,-.0334008){13}{\line(-1,0){.0681511}}
\multiput(55.623,43.022)(-.0681511,-.0334008){13}{\line(-1,0){.0681511}}
\multiput(53.851,42.153)(-.0681511,-.0334008){13}{\line(-1,0){.0681511}}
\multiput(52.079,41.285)(-.0681511,-.0334008){13}{\line(-1,0){.0681511}}
\multiput(50.307,40.417)(-.0681511,-.0334008){13}{\line(-1,0){.0681511}}
\multiput(48.535,39.548)(-.0681511,-.0334008){13}{\line(-1,0){.0681511}}
\multiput(46.763,38.68)(-.0681511,-.0334008){13}{\line(-1,0){.0681511}}
\multiput(44.991,37.811)(-.0681511,-.0334008){13}{\line(-1,0){.0681511}}
\multiput(43.219,36.943)(-.0681511,-.0334008){13}{\line(-1,0){.0681511}}
\multiput(41.447,36.074)(-.0681511,-.0334008){13}{\line(-1,0){.0681511}}
\multiput(39.675,35.206)(-.0681511,-.0334008){13}{\line(-1,0){.0681511}}
\multiput(37.903,34.338)(-.0681511,-.0334008){13}{\line(-1,0){.0681511}}
\multiput(36.131,33.469)(-.0681511,-.0334008){13}{\line(-1,0){.0681511}}
\multiput(34.36,32.601)(-.0681511,-.0334008){13}{\line(-1,0){.0681511}}
\multiput(32.588,31.732)(-.0681511,-.0334008){13}{\line(-1,0){.0681511}}
\multiput(30.816,30.864)(-.0681511,-.0334008){13}{\line(-1,0){.0681511}}
\multiput(30.18,26.18)(.137013,-.031169){7}{\line(1,0){.137013}}
\multiput(32.098,25.743)(.137013,-.031169){7}{\line(1,0){.137013}}
\multiput(34.016,25.307)(.137013,-.031169){7}{\line(1,0){.137013}}
\multiput(35.934,24.871)(.137013,-.031169){7}{\line(1,0){.137013}}
\multiput(37.852,24.434)(.137013,-.031169){7}{\line(1,0){.137013}}
\multiput(39.771,23.998)(.137013,-.031169){7}{\line(1,0){.137013}}
\multiput(41.689,23.562)(.137013,-.031169){7}{\line(1,0){.137013}}
\multiput(43.607,23.125)(.137013,-.031169){7}{\line(1,0){.137013}}
\multiput(45.525,22.689)(.137013,-.031169){7}{\line(1,0){.137013}}
\multiput(47.443,22.252)(.137013,-.031169){7}{\line(1,0){.137013}}
\multiput(49.362,21.816)(.137013,-.031169){7}{\line(1,0){.137013}}
\multiput(51.28,21.38)(.137013,-.031169){7}{\line(1,0){.137013}}
\multiput(53.198,20.943)(.137013,-.031169){7}{\line(1,0){.137013}}
\multiput(55.116,20.507)(.137013,-.031169){7}{\line(1,0){.137013}}
\multiput(57.034,20.071)(.137013,-.031169){7}{\line(1,0){.137013}}
\multiput(58.952,19.634)(.137013,-.031169){7}{\line(1,0){.137013}}
\multiput(60.871,19.198)(.137013,-.031169){7}{\line(1,0){.137013}}
\multiput(62.789,18.762)(.137013,-.031169){7}{\line(1,0){.137013}}
\multiput(64.707,18.325)(.137013,-.031169){7}{\line(1,0){.137013}}
\multiput(66.625,17.889)(.137013,-.031169){7}{\line(1,0){.137013}}
\multiput(68.543,17.452)(.137013,-.031169){7}{\line(1,0){.137013}}
\multiput(70.462,17.016)(.137013,-.031169){7}{\line(1,0){.137013}}
\multiput(72.38,16.58)(.137013,-.031169){7}{\line(1,0){.137013}}
\multiput(74.298,16.143)(.137013,-.031169){7}{\line(1,0){.137013}}
\multiput(76.216,15.707)(.137013,-.031169){7}{\line(1,0){.137013}}
\multiput(78.134,15.271)(.137013,-.031169){7}{\line(1,0){.137013}}
\multiput(80.052,14.834)(.137013,-.031169){7}{\line(1,0){.137013}}
\multiput(81.971,14.398)(.137013,-.031169){7}{\line(1,0){.137013}}
\multiput(89.43,14.18)(.0328947,-.0342105){19}{\line(0,-1){.0342105}}
\multiput(90.68,12.88)(.0328947,-.0342105){19}{\line(0,-1){.0342105}}
\multiput(91.93,11.58)(.0328947,-.0342105){19}{\line(0,-1){.0342105}}
\multiput(93.18,10.28)(.0328947,-.0342105){19}{\line(0,-1){.0342105}}
\multiput(94.43,8.98)(.0328947,-.0342105){19}{\line(0,-1){.0342105}}
\put(95.93,7.68){\line(1,0){.125}}
\put(96.25,7.75){\line(1,0){8}}
\put(79,41.75){$D_j$}
\put(84.5,10.5){${\cal C}_{i(j)}$}
\put(99.25,69.75){$\cal C$}
\put(82.5,58){$f$}
\put(106.5,28.25){$e_j$}
\put(101.25,51.5){$y_1(j)$}
\put(100.25,16.75){$y_2(j)$}
\put(112,29.75){\vector(0,-1){.07}}\put(112,73.75){\vector(0,1){.07}}\put(112,73.75){\line(0,-1){44}}
\put(10.75,77.75){\vector(0,-1){.07}}\put(10.75,78){\vector(0,1){.07}}\put(10.75,78){\line(0,-1){.25}}
\put(115.75,51.5){$m_1(j)$}
\put(112,6.25){\vector(0,-1){.07}}\put(112,26.25){\vector(0,1){.07}}\put(112,26.25){\line(0,-1){20}}
\put(114.5,16.75){$m_2(j)$}
\put(91.75,57.5){$E_j$}
\put(64,8.5){$\Delta_j$}
\put(8,28.25){(23)-band  ${\cal T}_j$}
\multiput(35.43,29.93)(0,-.75){5}{{\rule{.4pt}{.4pt}}}
\multiput(40.93,30.18)(0,-.8125){5}{{\rule{.4pt}{.4pt}}}
\multiput(45.93,29.93)(0,-.75){5}{{\rule{.4pt}{.4pt}}}
\multiput(50.93,30.18)(0,-.8125){5}{{\rule{.4pt}{.4pt}}}
\multiput(56.43,29.93)(-.0625,-.9375){5}{{\rule{.4pt}{.4pt}}}
\multiput(61.18,29.93)(0,-.75){5}{{\rule{.4pt}{.4pt}}}
\multiput(65.68,29.93)(0,-.875){5}{{\rule{.4pt}{.4pt}}}
\multiput(71.18,30.18)(0,-.8){6}{{\rule{.4pt}{.4pt}}}
\multiput(76.43,29.93)(0,-.8125){5}{{\rule{.4pt}{.4pt}}}
\multiput(81.68,30.43)(0,-.9375){5}{{\rule{.4pt}{.4pt}}}
\end{picture}

Since the path $y$ has no $q$-edges, $e_j$ is a $(23)$-edge.
Let ${\cal T}_j$ be the maximal $(23)$-band of $\Delta_j$ containing
the edge $e_j$. By the choice of $\Delta_j,$ the $\theta$-band
${\cal T}_j$ crosses a derivative band ${\cal C}_{i(j)}$ of
$\Gamma,$ and this derivative band is a $k^{-1}$-band by Properties \ref{order} and \ref{i}. 
The Step history of ${\cal C}_{i(j)}$ has no
subwords $(23)^{-1}(2)(23)$ by \ref{kt} applied to the filling trapezium $Tp({\cal C}_{i(j)}, \cal C)$, and so the top or the bottom
path of ${\cal T}_j$ cuts the derivative band in two parts such that
one of them is a $k^{-1}$-subband ${\cal D}_j$ with Step history
$(2)$ (without the cell from ${\cal T}_j$). We denote by $d(j)$ the length of this subband.
Every maximal  $\theta$-band crossing a derivative band,
 also crosses the handle of $\Delta_j$, so we may assume  that
${\cal D}_j$ belongs to the union $E_j$ of the maximal $\theta$-bands of $\Gamma$
ending on the $m_1(j)$ white edges of ${\bf z}^{\Delta_j},$ because
one can interchange $m_1(j)$ and $m_2(j)$) in (\ref{2mm}), in particular, $d(j)\le m_1(j)$. We
consider $3$ cases.

(1) $d(j)\le m_1(j)/2.$  Then we say that $j\in J_1\subset
\{1,\dots,r\}$.

(2) $d(j)>m_1(j)/2$ and at least $d(j)/2$ maximal $a$-bands starting
on ${\cal D}_j$ end on derivative bands of $\Gamma$ non-equal to ${\cal C}_{i(j)}$. Then $j\in J_2.$

(3) $d(j)>m_1(j)/2$ and less than $d(j)/2$ maximal $a$-bands
starting on ${\cal D}_j$ end on derivative bands non-equal to ${\cal C}_{i(j)}$ . Then $j\in
J_3.$

{\bf Case (1).} In this case, the end $f$ of ${\cal D}_j$ is a $q$-edge factorizing
the path ${\bf z}={\bf z}^{\Delta_j}$ in a product ${\bf z'}f{\bf z''},$ where ${\bf z'}$ has at
least $m_1(j)/2$ white $\theta$-edges (the ends of  maximal
$\theta$-bands from $E_j$ non-crossing ${\cal D}_j$) and ${\bf z''}$ has at least
$m_2(j)$ white edges. Therefore $m_2(j)
m_1(j)/2 \le \kappa({\bf z}^{\Delta_j})=\kappa^c(\Delta_j).$  Hence by Inequality (\ref{2mm}) and by Lemma \ref{deriv} (a),
\begin{equation}\label{c1}
\sum_{j\in J_1} \lambda^c({\bf y}(j)) \le 2\sum_{j\in J_1} m_1(j)m_2(j)
\le 4\sum_{j\in J_1} \kappa^c(\Delta_j)\le 4 \kappa^c(\Gamma)
\end{equation}

{\bf Case (2).} In this case, at least $d(j)/2$ $a$-bands connect ${\cal C}_{i(j)}$ with
a different derivative band. Therefore the number of $a$-bands connecting pairwise different
derivative bands of $\Gamma$ is at least $\frac14 \sum_{j\in J_2}d(j)$. On the other hand,
the same number does not exceed $h_-=h_-^{\Gamma}$ by Lemma \ref{a-bands}, whence
$\sum_{j\in J_2}d(j)\le 4h_-.$ Since $m_1(j)\le 2d(j)$, this inequality 
implies
$\sum_{j\in J_2} m_1(j) 
\le 8h_-.$ Therefore by Lemma
\ref{a-bands} and (\ref{2mm}), we have
\begin{equation}\label{c2}
\sum_{j\in J_2}\lambda^c({\bf y}(j))\le 2(8h_-)\sum_{j\in J_2} m_2(j)
\le 16 h_-h\le 32\kappa^c(\Gamma)
\end{equation}

{\bf Case (3).} Since the Step history of the $k^{-1}$-band ${\cal D}_j$ is $(2)$, every cell of this band
is active from the right by \ref{kk'}, and so there are no $a$-bands starting and ending on
the same ${\cal D}_j$ by \ref{iv}. No $a$-band starting on ${\cal D}_j$ can cross ${\cal T}_j$  since there are no
$(\theta,a)$-cells between the intersection cells of ${\cal T}_j$ with ${\cal C}_{j(i)}$ and with $\cal C$ by \ref{i}.  Thus
in Case (3),  more than $d(j)/2$ $a$-bands starting
on ${\cal D}_j$ end on $\partial\Gamma.$ Therefore 
$|{\bf z}|_a \ge \sum_{j\in J_3}
d(j)/2\ge\sum_{j\in J_3} m_1(j)/4.$ Hence, by Lemma \ref{simple}
(a),
\begin{equation}\label{c3}
\sum_{j\in J_3}\lambda^c({\bf y}(j))\le 2\sum_{j\in J_3}m_1(j)m_2(j)\le
8|{\bf z}|_ah \le 8(\delta')^{-1}[\Gamma]
\end{equation}

Altogether, Inequalities (\ref{deltaj}) and (\ref{c1})--(\ref{c3}) imply the inequality
$$\lambda^c(\Gamma)\ge - \sum_{j\in  J_1\cup J_2\cup J_3}\lambda^c({\bf y}(j)) \ge -36
\kappa^c(\Gamma)-8(\delta')^{-1}[\Gamma].$$

\endproof

\begin{lemma} \label{styk}  Let $\Delta$ be a regular comb with step history  $(2)(1)(2)$ or $(2)(3)(2)$, where the
first or the last
$(2)$ can be absent. Let $H\equiv H(1)H(2)H(3)$ be the corresponding Step
decomposition of the history $H.$ Assume that the handle $\cal C$ is
$k^{-1}$- or $k'$-, or $s_0$-band and 
 the base width of
$\Delta$ is at most $15N.$
Let $l= ||H(1)||+||H(3)||,$ 
then $\area(\Delta)\le c_2( h(|{\bf z'}|-h
+l+1)+\kappa^c(\Delta))$, where $h=||H||$ and $z'=z'^{\Delta}.$
\end{lemma}
\proof 
Let $\Delta(i)$ be the $H(i)$-part of $\Delta,$ $i=1,2,3.$
We will abbreviate   $h(1)=h^{\Delta(1)},$ and so on. By \ref{kk'}, the handle $\cal C$
is passive from the left, and so by Lemma \ref{13},
\begin{equation}\label{d(2)}
\area'(\Delta(2))\le c_1h(2)(|{\bf z'}(2)|-h(2)+1)+c_1\kappa^c(\Delta(2))
\end{equation}
and therefore the statement is true if $l=0$ since $c_2>c_1+1$.
Further we assume that $l\ge 1.$

Since the base widths of $\Delta$ does not exceed $15N,$ we have by Lemma \ref{comb},
\begin{equation} \label{d1d3}
\area(\Delta(1))+\area(\Delta(3))\le
60N(h(1)^2+h(3)^2)+2(h(1)\alpha(1)+h(3)\alpha(3))
\end{equation}
where $\alpha(1)=|{\bf z}(1)|_a, \alpha(3) =|{\bf z}(3)|_a.$

Let ${\bf x}_{12}$ be the common part of the boundaries of $\Delta(1)$ and
$\Delta(2)$. By Lemma \ref{comb},
\begin{equation} \label{xle}
|{\bf x}_{12}|_a\le
(\alpha(1)+120Nh(1))/2\le(\alpha+|{\bf x}_{12}|_a+120Nh(1))/2,
\end{equation}
where
$\alpha=|{\bf z}|_a=|{\bf z}^{\Delta}|_a$ 
From (\ref{xle}) and similar inequality with ${\bf x}_{23},$ we have

\begin{equation}\label{x12}
|{\bf x}_{12}|_a\le \alpha+120Nh(1) \;\; and \;\; |{\bf x}_{23}|_a\le\alpha+120Nh(3)
\end{equation}
 Therefore
 \begin{equation}\label{alpha}
 \alpha(1)\le
|{\bf x}_{12}|_a+\alpha\le 2\alpha+120Nh(1)\;\; and \;\;\alpha(3)\le 2\alpha+120Nh(3)
\end{equation}
  Now we use
Inequalities (\ref{d1d3}), (\ref{alpha}), equality $h(1)+h(3)=l$, and Lemma \ref{simple} (a)to obtain inequality
\begin{equation}\label{1+3}
\begin{array}{l} \area(\Delta(1))+\area(\Delta(3))\le 60Nl^2+2l(2\alpha+120Nl)\\ \le
60Nl^2+2l(2(\delta')\iv(|z'|-h)+120Nl)\end{array}
\end{equation}
 We also have $|{\bf z'}(2)|\le
|{\bf z'}|+|{\bf x}_{12}|+|{\bf x}_{23}|\le |{\bf z'}|+2\alpha+120Nl$ by (\ref{x12}). Therefore
$$|{\bf z'}(2)|-h(2)\le |{\bf z'}|+ 2\alpha+120Nl - (h-l)\le
|{\bf z'}|-h+2\alpha+121Nl.$$ Hence $|{\bf z'}(2)|-h(2)\le
(|{\bf z'}|-h)(2(\delta')\iv+1)+121Nl$ by Lemma \ref{simple} (a). Using this
estimate we deduce from (\ref{d(2)}) and Lemma \ref{deriv}(b) that
$$\area'(\Delta(2))\le c_1h(2)(|{\bf z'}(2)|-h(2)+1)+c_1\kappa^c(\Delta(2))$$ $$\le
c_1h((|{\bf z'}|-h)(2(\delta')\iv+1)+122Nl)+c_1\kappa^c(\Delta)$$ In
turn, this inequality and (\ref{1+3}) show that
$$\begin{array}{l} \area (\Delta)\le \area(\Delta(1))+\area(\Delta(3)) +\area'(\Delta(2))+h\\
\le 60Nhl+2h(2(\delta')\iv(|{\bf z'}|-h)+120Nl)\\ +
c_1h((|{\bf z'}|-h)(2(\delta')\iv+1)+122Nl)+c_1\kappa^c(\Delta)+h\\ \le
c_2h(|{\bf z'}|-h+ l)+c_2\kappa^c(\Delta)\end{array}$$ because $l\ge
1$, $|{\bf z'}|-h\ge 0$ and $c_2\ge c_1\max(3(\delta')\iv, 123N).$ The
lemma is proved.
\endproof

\begin{lemma}\label{2etapa} Let $\Gamma$ be a regular comb of base width
$b\le 15N.$ Assume that it has no one-Step long subcombs $\Delta$ with
$\area(\Delta)\le c_1[\Delta]+c_1\kappa^c(\Delta).$
Let the history $H$ of $\Gamma$ either
(a) have rules $(23)^{\pm 1}$ but no rules $(12)^{\pm 1}$ or (b) have rules
$(12)^{\pm 1}$ but no rules $(23)^{\pm 1}.$ Also assume that in case (a), $\cal C$
is a $t^{\pm 1}$-band, and in case (b), $\cal C$ is a $(t')^{\pm
1}$-band. Then $\area (\Gamma)\le c_2 (2(\delta')^{-1}[\Gamma]+
6\kappa^c(\Gamma))\le c_3 [\Gamma]+ c_2\mu^c(\Gamma)$.
\end{lemma}

\proof  {\bf Case (b)}: We consider the system of derivative bands
${\cal C}_1,\dots,{\cal C}_s.$ 
As usual, $\Gamma_i$ is the derivative subcomb with handle ${\cal
C}_i$.
If a derivative band is a $(t')^{\mp 1}$-band, then it must have a one-step
history since the base of a $(12)$-band cannot contain
subwords $(t')^{\mp 1}(t')^{\pm 1}$ by \ref{i}. By Lemma \ref{t} (2)
, this would contradict our
assumption on long subcombs. Therefore all the derivative bands are
$k'$-bands by \ref{order}.

Let they have histories $H_1,\dots,
H_s$, respectively. Then $H_i$ ($i=1,\dots,s$) has no subwords of type
$(12)(2)(12)^{-1}$ by \ref{kt} applied to the filling
trapezium $Tp({\cal C}_i,\cal C)$ 
(a similar argument
was used in Lemma \ref{mu21}). Therefore
 $H_i\equiv H_i(1)H_i(2)H_i(3)$ ($i\le s$), where
$H_i(1)$ and $H_i(3)$ are of Step $(2)$ and $H_i(2)$ is of Step $(1),$
where some of the factors can be empty.
The lengths of the histories are denoted by $h_i$ and $h_i(j)$,
$j=1,2,3.$

 If $H_i(2)$ is non-empty, then by Lemma \ref{styk} for $\Gamma_i$ and
 equality $|{\bf z}|_i=|{\bf z'}_i|+2$,
$$\area(\Gamma_i)\le \area'(\Gamma)+h_i\le
c_2(h_i(|{\bf z}_i|-h_i+h_i(1)+h_i(3))+\kappa^c(\Gamma_i))$$
 If $H_i=H_i(1)$, then
  $\area(\Gamma_i)\le 60Nh_i(1)^2+2h_i(1)|{\bf z}_i|_a$ by Lemma \ref{comb}.

Since $c_2\ge 2(\delta')\iv$ and $|{\bf z}_i|_a\le (\delta')\iv(|{\bf z}|-h)$ by Lemma
\ref{ochev} (a), we derive from the inequalities of the previous paragraph and Lemma \ref{deriv}(b) that

\begin{equation}
\sum_{i=1}^s \area(\Gamma_i)\le c_2h(|{\bf z}|-h+
\sum(h_i(1)+h_i(3)))+c_2\kappa^c(\Gamma)+60N\sum
h_i(1)^2\label{predvar}
\end{equation}

Let $h_{i_0}=\max h_i.$ Then $\sum_{i\ne i_0}h_i = h_-$ by the
definition of $h_-$. By \ref{iv}, every maximal $a$-band starting on the
$H_{i_0}(1)$- or $H_{i_0}(3)$- part of ${\cal C}_{i_0}$ must end
either on the $H_{i}(1)$- or $H_{i}(3)$-part of some ${\cal C}_i$,
$i\ne i_0$, or on the path $\bf z$. (Indeed, the $H_i(1)$-part of a band ${\cal C}_i$ cannot be connected with
the $H_i(3)$-part by an $a$-band since the $H_i(2)$-part of ${\cal
C}_i$ has common $\theta$-edges  with $\cal C$ by \ref{i}.) Therefore $h_{i_0}(1)+h_{i_0}(3)\le
h_- +|{\bf z}|_a\le h_- +(\delta')^{-1}(|{\bf z}|-h)$ by Lemma \ref{ochev} (a), and so,
$\sum_{i=1}^s(h_i(1)+h_i(3))\le 2h_-+ (\delta')^{-1}(|{\bf z}|-h).$ It
follows from this estimate, Inequalities (\ref{predvar}) and
$hh_-\le 2\kappa^c(\Gamma)$ (see Lemma \ref{lgamma}) that

$$\sum_{i=1}^s \area(\Gamma_i)\le c_2h((|{\bf z}|-h)(1+(\delta')^{-1}) +2h_-) + c_2\kappa^c(\Gamma)
+60Nh(2h_-+(\delta')^{-1}(|{\bf z}|-h))$$
\begin{equation}\label{summaGi}
\le (c_2+c_2(\delta')^{-1}+60N(\delta')^{-1})[\Gamma] +
(4c_2+c_2+240N)\kappa^c(\Gamma)
\end{equation}

By Lemmas \ref{simple} and \ref{lgamma}, the total
area of all simple bands in $\Gamma$ does not exceed
$$h(h_-+|{\bf z}|_a+1)\le h(h_-+(\delta')\iv(|{\bf z}|-h-1))< (\delta')^{-1}[\Gamma] +2\kappa^c(\Gamma)$$
This inequality and (\ref{summaGi}) imply

\begin{equation}\label{areaG}
\area(\Gamma)\le (\frac32 (\delta')^{-1}c_2+60N(\delta')^{-1})[\Gamma] +
(5c_2+240N+2)\kappa^c(\Gamma)
\end{equation}
and to obtain inequality $\area (\Gamma)\le c_2 (2(\delta')^{-1}[\Gamma]+
6\kappa^c(\Gamma))$, it remains to use that $c_2>
 240N+2.$

To complete the proof, we observe that by Lemma \ref{mu21},
 $$\begin{array}{l} c_3 [\Gamma]+ c_2\mu^c(\Gamma)=c_3 [\Gamma]+ c_2(c_0\kappa^c(\Gamma)+\lambda^c(\Gamma))\\
 \ge c_3 [\Gamma]+ c_2c_0\kappa^c(\Gamma)-c_2(36\kappa^c(\Gamma)+8(\delta')^{-1} [\Gamma])\\
 =(c_3-8(\delta')^{-1}c_2)[\Gamma]) +c_2(c_0-36) \kappa^c(\Gamma)
 \ge  c_2 (2[\Gamma]+ 6\kappa^c(\Gamma))\end{array}$$
 since $c_3 \ge 9(\delta')^{-1}c_2,$ and $c_0 \ge 42.$

 {\bf Case (a)} of the lemma is completely analogous.
 \endproof

  \begin{lemma}\label{p1} Let $\Delta$ be a regular comb of base widths at most $15N,$
where the history $H^{\Delta}$ of $\Delta$ has no rules $(23)^{\pm 1}.$
 Assume that $\Delta$ has neither $t$- nor $t'$-cells and has no one-Step long subcombs $\Gamma'$ such that
$\area(\Gamma')\le c_1[\Gamma']+c_1\kappa^c(\Gamma').$
If $\Delta$ has a $(12)$-band
$\cal T$ with base of the form $\dots p_1p_1^{-1}s_0^{-1}\dots,$ then $\Delta$ has
a long subcomb $\Gamma$ with $\area (\Gamma)\le
c_3 [\Gamma]+ c_2\mu^c(\Gamma).$
\end{lemma}

\proof Let us consider the maximal $q$-bands $\cal C$ and ${\cal C}_0$ of $\Delta$ corresponding to the
distinguished 
letters $p_1$ and $s_0^{-1},$ resp., in the base of $\cal T$. By Property \ref{v} for the trapezium
$Tp({\cal C}, {\cal C}_0),$  the history $H$ of $\cal C$ has no subhistories of the form $(12)(2)(12)^{-1},$
It follows that $H$ has type $(2)(1)(2)$ (or a subword of $(2)(1)(2)$), in particular, $H$ has at most
two rules $(12)^{\pm 1}$ ),
and by \ref{v}, all $p_1$-cells corresponding to the rules of Step 2 are active from both sides 
in the subcomb $\Gamma$ with handle $\cal C$ as well. Below the usual notation $h, {\bf y,z,y', z'}$
will be used for $\Gamma.$

We consider the system of derivative bands ${\cal C}_1,\dots,{\cal
C}_s$ of $\Gamma.$ Let they have histories $H_1,\dots, H_s$, respectively. Then
$H_i\equiv H_i(1)H_i(2)H_i(3)$ ($i\le s$), where $H_i(1)$ and $H_i(3)$ are
of type $(2)$ and $H_i(2)$ is of type $(1),$ and some of the factors
can be empty. The lengths of the histories are denoted by $h_i$ and
$h_i(j)$, $j=1,2,3.$ As usual, $\Gamma_i$ is the derivative subcomb
with handle ${\cal C}_i$.

Assume that ${\cal C}_i$ is a derivative $p_1^{-1}$-band. Then $H_i$
has neither rules $(12)^{\pm 1}$ nor rules of Step $1$ 
by \ref{i}.
Hence the history $H_i$ is one-Step (of Step
$2$) and moreover all $p_1^{\pm 1}$-bands of $\Gamma_i$ are active from both sides
and all other $q$-bands of $\Gamma_i$ are passive by \ref{ii}. 
Therefore one can apply Lemma \ref{width1} (b) to the subcomb whose
handle is a left-most maximal $q$-band of $\Gamma_i$ which
contradicts the assumption of the lemma about long subcombs. Hence, by \ref{order},
all the derivative bands are $s_0$-bands. By the same argument and
by Lemma \ref{13}, derivative subcombs  cannot have one-Step
histories. (Indeed, since the $s_0$-band is passive by \ref{kk'},
 we have $|{\bf z}^{\Gamma_i}|-|{\bf y}^{\Gamma_i}|=|({\bf z'})^{\Gamma_i}|-|({\bf y'})^{\Gamma_i}|+2,$ and Lemma \ref{13}  would imply $\area(\Gamma_i)=\area'(\Gamma_i)+h^{\Gamma_i}\le
c_1[\Gamma_i]+c_1\kappa^c(\Gamma_i)$
contrary to the condition of the lemma.) Therefore each of
the derivative bands crosses a $(12)$-band, and so $s\le 2.$

We factorize the history $H$ as $H'H(3)H'',$ where $H(3)$ is of Step $1$
including the $(12)^{\pm 1}$-rules, and $H', H''$ are of Step 2. If the $H'$-part ($H''$-part)
of $\Gamma$ is non-empty and it is crossed by a derivative band of $\Gamma$, then this unique band has to cross
the $(12)^{\pm 1}$-band separating the $H(3)$-part from the $H'$-part (from $H''$-part). Using
this observation,
we further factorize $H'\equiv H(1)H(2)$ so that  every maximal $\theta$-band of
the $H(2)$-part $\Gamma(2)$ of
$\Gamma$ crosses a derivative
$s_0$-band ($H(2)$ can be empty) and the $H(1)$-part $\Gamma(1)$ has no derivative bands.
Similarly, we have $H''\equiv H(4)H(5).$

Thus $H\equiv H(1)\dots H(5).$
Let $h(1),\dots,h(5)$ be corresponding
lengths of the histories (some of them may be zero), $\Gamma(i)$ ($i=1,\dots,5$) be the corresponding parts of $\Gamma$ (some of them may be empty) with handles ${\cal C}(i)$, and
${\bf z}={\bf z}(5)\dots {\bf z}(1),$ where ${\bf z}(i)$ is the common part of $\bf z$ and $\partial\Gamma(i)$.
Note that if the maximal $(12)$-band $\cal T$ separating $\Gamma(3)$ and $\Gamma(2)$
is not crossed by a derivative band,
then by \ref{i}, $\cal T$ has no cells except for the intersection cell with the handle, and so
the boundaries of $\Gamma(3)$ and $\Gamma(2)$ have no common edges except for that in $\cal C.$
Similar note is true for the combs $\Gamma(3)$ and $\Gamma(4)$.

Let $\bf x$ be the extension of the path ${\bf z}(1)^{-1}$ along $z^{-1}$, such that $|{\bf x}|_{\theta}= |{\bf z}(1)|_{\theta}$ and
the last edge of $\bf x$ is the $s_0^{-1}$-edge of ${\bf z}(2)$. Then every maximal $a$-band starting
on ${\bf y'}(2){\bf y'}(1)$ must end on $\bf x$ by Lemma \ref{NoAnnul} because (a) the derivative $s_0$-band
is passive, (b) it has a common $(12)$-edges with $\cal C$
 by \ref{i}, and (c) every maximal $a$-band of $\Gamma(2)$ starting on the side ${\bf y'}(2){\bf y'}(1)$ 
 of ${\cal C}(1){\cal C}(2)$ cannot end on ${\bf y'}(2){\bf y'}(1)$ by \ref{iv}. Since all cells of ${\cal C}(1){\cal C}(2)$ are active, the
 path $\bf x$ has at least $h(1)+h(2)$ $a$-edges and only $h(1)$ $\theta$-edges. 
It follows therefore from Lemma \ref{ochev} (a)  that
 \begin{equation}\label{p1(4)}
 |{\bf x}|- h(1) \ge 1+\delta h(2)+ \delta' h(1) 
 \end{equation}
 
\unitlength 1mm 
\linethickness{0.4pt}
\ifx\plotpoint\undefined\newsavebox{\plotpoint}\fi 
\begin{picture}(253.5,102.75)(0,0)
\put(103.25,102.75){\line(0,-1){94}}
\put(114.25,102.25){\line(0,-1){92.75}}
\put(114,64){\line(-1,0){50.75}}
\put(114.25,45.25){\line(-1,0){57.25}}
\multiput(86.25,83.5)(.0337022133,-.0382293763){497}{\line(0,-1){.0382293763}}
\put(79.5,83.25){\line(1,0){7.25}}
\multiput(79.5,83.25)(.0337186898,-.0370905588){519}{\line(0,-1){.0370905588}}
\put(96.75,64.5){\line(0,-1){4}}
\multiput(103,45)(-.0428790199,-.0336906585){653}{\line(-1,0){.0428790199}}
\put(74.75,23.25){\line(-1,0){5.75}}
\multiput(69,23.25)(.0417304747,.0336906585){653}{\line(1,0){.0417304747}}
\put(96.25,45.25){\line(0,1){3.75}}
\put(114,48){\line(-1,0){57}}
\put(57.25,48){\line(0,-1){2.5}}
\put(75,23.25){\line(1,0){39.25}}
\multiput(85.5,23.25)(-.066105769,-.033653846){208}{\line(-1,0){.066105769}}
\put(71.75,16.25){\line(0,-1){3.25}}
\put(71.75,13){\line(1,0){19.25}}
\multiput(91,13)(.104910714,-.033482143){112}{\line(1,0){.104910714}}
\put(102.75,9.25){\line(1,0){11.5}}
\multiput(69,23.5)(-.113402062,.033505155){97}{\line(-1,0){.113402062}}
\multiput(58,26.75)(-.0337301587,.0426587302){252}{\line(0,1){.0426587302}}
\multiput(49.75,37.5)(.033632287,.035874439){223}{\line(0,1){.035874439}}
\multiput(64,48.25)(-.091463415,.033536585){164}{\line(-1,0){.091463415}}
\put(49,53.75){\line(0,1){10}}
\put(49,63.75){\line(1,0){14.75}}
\multiput(63.75,64.25)(-.0614478114,.0336700337){297}{\line(-1,0){.0614478114}}
\put(45.5,74.25){\line(0,1){5.5}}
\put(45.5,79.75){\line(1,0){37.5}}
\put(85.75,83.25){\line(1,0){28.75}}
\put(88.5,83.5){\line(-2,1){18}}
\put(70.5,92.5){\line(0,1){9.5}}
\put(70.5,102){\line(1,0){43.75}}
\put(96,14.5){$\Gamma(1)$}
\put(62.75,34.25){$\Gamma(2)$}
\put(66.5,56){$\Gamma(3)$}
\put(63.25,73.75){$\Gamma(4)$}
\put(78.75,97){$\Gamma(5)$}
\put(253.25,85.5){\rule{.25\unitlength}{.25\unitlength}}
\put(78.75,11){${\bf z}(1)$}
\put(50.25,30.75){${\bf z}(2)$}
\put(45.5,61.75){${\bf z}(3)$}
\put(47.5,82.75){${\bf z}(4)$}
\put(67.75,93.75){${\bf z}(5)$}
\put(104.25,19.75){${\bf y'}(1)$}
\put(116,16.5){${\bf y}(1)$}
\put(107.75,12.25){${\cal C}(1)$}
\put(101.5,32.5){${\bf y'}(2)$}
\put(106.75,40.25){${\cal C}(2)$}
\put(115.75,34.5){${\bf y}(2)$}
\put(101.75,55){${\bf y'}(3)$}
\put(116.25,55.5){${\bf y}(3)$}
\put(107,51){${\cal C}(3)$}
\put(101,77.75){${\bf y'}(4)$}
\put(107.75,71.5){${\cal C}(4)$}
\put(116,75){${\bf y}(4)$}
\put(100.75,97.5){${\bf y'}(5)$}
\put(116.25,90.5){${\bf y}(5)$}
\put(105.5,91){${\cal C}(5)$}
\put(35,24){\vector(0,-1){.07}}\put(35,83){\vector(0,1){.07}}\put(35,83){\line(0,-1){59}}
\put(36.75,52){$E$}
\put(20.25,8.5){\vector(0,-1){.07}}\put(20.5,102.5){\vector(0,1){.07}}\multiput(20.5,102.5)(-.03125,-11.75){8}{\line(0,-1){11.75}}
\put(22.75,55){$\Gamma$}
\put(76.5,21.25){${\bf x}$}
\put(92.25,24.75){${\bf p}_{12}$}
\put(93,85.5){${\bf p}_{45}$}
\put(63.5,47.75){\line(0,-1){2.25}}
\put(69.25,47.75){\line(0,-1){2.5}}
\put(74.75,47.75){\line(0,-1){2.25}}
\put(79.75,47.75){\line(0,-1){2.5}}
\put(85,47.75){\line(0,-1){2.25}}
\put(90.5,47.75){\line(0,-1){2.25}}
\put(52.25,46.5){$\cal T$}
\end{picture}

 (Here $1$ is added for the $q$-edge of the handle $\cal C.$ ) Since
 the path ${\bf z}(2)$ has $h(2)$ $\theta$-edges which do not belong to $\bf x$, we get from \ref{p1(4)}:
 \begin{equation}\label{p1(3)}
 |{\bf z}(2){\bf z}(1)| -h(1)-h(2) \ge 1+\delta h(2)+ \delta' h(1)\; and\;  
  |{\bf z}(5){\bf z}(4)| -h(4)-h(5) \ge 1+\delta h(4)+ \delta' h(5)
 \end{equation}

 The band ${\cal C}(3)$ is passive by \ref{kk'}; therefore ${\bf y}(3)=h(3)$ and so $|{\bf y}|=|{\bf y'}|=h+\delta' (h(1)+h(2)+h(4)+h(5)).$
 This together with equality $|{\bf z}(i)|\ge |{\bf z}(i)|_{\theta} = h(i),$
 Lemma \ref{ochev} (b), and (\ref{p1(3)}) give rise to inequality

 $$|{\bf z}|-|{\bf y}|\ge |{\bf z}(1)|+\dots+|{\bf z}(5)|- |{\bf y}(1)|-\dots-|{\bf y}(5)| - 4(\delta-\delta')\ge$$
 \begin{equation}\label{p1(5)}
 2+(\delta - \delta')(h(2)+h(4)-4)> \max(1, (\delta-\delta')(h(2)+h(4)))
 \end{equation}
In particular, $\Gamma$ is a long subcomb. 

Let ${\bf p}_{1,2}$ be the common segment of $\partial\Gamma(2)$ and $\partial\Gamma(1);$ it is a
top/bottom path of a $\theta$-band, and it consists of
a $p_1$-edge and of $a$-edges. The path ${\bf p}_{4,5}$ is defined similarly. Since every
maximal $a$-band of $\Gamma(2)$ starting on ${\bf p}_{1,2}$ must end on ${\bf y}'(2)=({\bf y'})^{\Gamma(2)},$
we have

\begin{equation}\label{p1(1)} |{\bf p}_{1,2}|_a\le|{\bf y'}(2)|_a=h(2) \;\;\; and\;\; also \;\;\;|{\bf p}_{4,5}|_a\le|{\bf y'}(4)|_a=h(4)
\end{equation}

Now we set $E=\Gamma(2)\cup\Gamma(3)\cup\Gamma(4)$. Then ${\bf z}^{E}={\bf p}_{1,2}{\bf z}(2){\bf z}(3){\bf z}(4){\bf p}_{45}$ and
 \begin{equation}\label{p1(6)}
 |{\bf z}^{E}|-|{\bf y}^{E}|\le
 |{\bf z}|-|{\bf y}|+ \delta(|{\bf p}_{1,2}|_a +|{\bf p}_{4,5}|_a) = |{\bf z}|-|{\bf y}|+ \delta(|h(2)+h(4)|)
 \end{equation}
  by (\ref{p1(1)}), since the maximal $\theta$- and $a$-bands
 starting on ${\bf y'}(1)$ and ${\bf y'}(5)$  end on ${\bf z}(1)$ and ${\bf z}(5)$, respectively.

 The comb $E$ has $s\le 2$ derivative subcombs $E_j=\Gamma_j$ ($j\le s\le 2$) whose handles are  ${\cal C}_j$-s. By Lemma \ref{styk},
 \begin{equation}\label{13fev}
 \sum_j \area(E_j)\le c_2 ((\sum_j |({\bf z'})^{E_j}|-h^{E_j})+h(2)+h(4)+1)(h(2)+h(3)+h(4))
 +\sum_j\kappa^c(E_j))
 \end{equation}
 Since $\sum_j (|({\bf z'})^{E_j}|-|h^{E_j}|) \le
  |{\bf z'}| - |{\bf z}(1)|-|{\bf z}(5)| - h(2)-h(3)-h(4)$ and $h(2)+h(3)+h(4)\le h,$ the  Inequalities (\ref{p1(5)}) and (\ref{13fev}) imply

  $$\sum_j \area(E_j)\le $$ \begin{equation}\label{15234}
  c_2h(|{\bf z'}| - |{\bf z}(1)|-|{\bf z}(5)| - h(2)-h(3)-h(4)+
  (\delta-\delta')^{-1}(|{\bf z}|-|{\bf y}|)+1)
  +c_2\sum\kappa^c(E_j)
  \end{equation}

 Observe that any simple band of $\Gamma(3)$ has no cells except for one  cell $\pi$ of the handle ${\cal C}(3)$
 because there are no $(\theta,a)$-cells from the left of $\pi$ by \ref{i}.
 Hence the comb $E$ consists of the cells of $E_j$-s, the cells of the handle  of $E$ and the cells of
 maximal $a$-bands connecting this handle with ${\bf p}_{1,2}$ and ${\bf p}_{4,5}$ and intersecting at most $h(2)$ and $h(4)$
 $\theta$-bands, respectively, by Lemma \ref{NoAnnul}.  Therefore we obtain  from (\ref{15234}), (\ref{p1(1)}),  and  Lemma \ref{deriv} (a) that
 $$\area(E) \le c_2 h(|{\bf z'}|-|{\bf z}(1)|-|{\bf z}(5)|-h(2)-h(3)-h(4)+$$
 $$(\delta-\delta')^{-1}(|{\bf z}|-|{\bf y}|)+1)+c_2\kappa^c(E)+h+
 h(2)^2+h(4)^2 \le $$
 \begin{equation}\label{234}
  c_2h(|{\bf z}|
 -|{\bf z}(1)|-|{\bf z}(5)|-h(2)-h(3)-h(4)+(\delta-\delta')^{-1}(|{\bf z}|-|{\bf y}|)) +c_2\kappa^c(E)+h(h(2)+h(4))
 \end{equation}
because $|{\bf z}|=|{\bf z'}|+2.$ Since
$|{\bf y}(2)|+|{\bf y}(3)|+|{\bf y}(4)|=h(2)+h(3)+h(4)+\delta'(h(2)+h(4))\le h(2)+h(3)+h(4)+\delta'(\delta-\delta')^{-1}(|{\bf z}|-|{\bf y}|)$ by (\ref{p1(5)}), Inequality (\ref{234})
yields
   $$\area(E)\le c_2h(|{\bf z}| -|{\bf z}(1)|-|{\bf z}(5)|-{\bf y}(2)-{\bf y}(3)-{\bf y}(4)+$$
  \begin{equation} \label{areaD}
  + (\delta'+1)(\delta-\delta')^{-1}(|{\bf z}|-|{\bf y}|)))+c_2\kappa^c(E)+h(h(2)+h(4))
  \end{equation}

Since the combs $\Gamma(1)$ and $\Gamma(5)$ have no derivative
bands, applying Lemma \ref{width1} (b), we obtain the
upper estimates
\begin{equation}\label{ar1}
\area(\Gamma(1))\le
(\delta')^{-1}h(|{\bf z}(1)|+|{\bf p}_{1,2}|- |{\bf y}(1)|)\le
c_2h(|{\bf z}(1)|+1+\delta h(2)- |{\bf y}(1)|)
\end{equation}
as $|{\bf p}_{1,2}|\le 1+\delta h(2)$ by (\ref{p1(1)}) and Lemma \ref{ochev}; and similarly,
\begin{equation}\label{ar5}
\area(\Gamma(5))\le c_2h(|{\bf z}(5)|+1+\delta h(4)-|{\bf y}(5)|)
\end{equation}
Summing Inequalities  (\ref{areaD})-(\ref{ar5}) and using
(\ref{p1(5)}), we get $$\area(\Gamma)\le c_2h(|{\bf z}|-|{\bf y}|)(2+(\delta'+1+\delta)(\delta-\delta')^{-1})
+c_2\kappa^c(E)\le$$
\begin{equation}\label{p1(8)}
2\delta^{-1}c_2[\Gamma]
+c_2\kappa^c(E)
\end{equation}

The handle $\cal C$ has (at most) two $(12)$-cells  (see the first paragraph in
the proof of the lemma), and so it has at most $3$ maximal subbands without
$(12)$-cells. It follows from the definition of the $\lambda$-mixture that
 \begin{equation}\label{vot}
 \lambda({\bf y})\le
 (h(1)+h(2))(h(3)+h(4)+h(5))+(h(4)+h(5))h(3)
 \end{equation}
If $h(1)+ h(5) \ge h(2)+h(4)$, then $2h(1)+2h(5) \ge (h(1)+h(2))+(h(4)+h(5))$, and so
the right-hand side of (\ref{vot}) does not
exceed $2h(1)(h-h(1))+2h(5)(h-h(5))$ which, in turn, does not exceed
$4\kappa^c(\Gamma)$ since the ends of $s_0$-bands separate the
$\theta$-edges of $z$ in parts with $h(1), h(5),$ and $h-h(1)-h(5)$
$\theta$-edges. If $h(1)+h(5)<h(2)+h(4)$, then the right-hand side
of (\ref{vot}) does not exceed $2(h(2)+h(4))h \le 2(\delta-\delta')^{-1}h(|{\bf z}|-|{\bf y}|)$
by (\ref{p1(5)}). Thus, in any case $ \lambda^c(\Gamma)\ge-\lambda({\bf y})\ge
-4\kappa^c(\Gamma)-2(\delta-\delta')^{-1}[\Gamma].$  Therefore
by this inequality and (\ref{p1(8)}),

$$ \area(\Gamma)\le 2\delta^{-1}c_2[\Gamma] +c_2\kappa^c(\Gamma)\le
(c_3-2(\delta-\delta')^{-1}c_2)[\Gamma]
+c_2(c_0-4)\kappa^c(\Gamma)=$$  $$c_3[\Gamma]
+c_2(c_0\kappa^c(\Gamma) -
4\kappa^c(\Gamma)-2(\delta-\delta')^{-1}[\Gamma]) \le c_3[\Gamma]
+c_2(c_0\kappa^c(\Gamma)+ \lambda^c(\Gamma))= c_3[\Gamma]
+c_2\mu^c(\Gamma)$$ since $c_0\ge 5$ and $c_3>5\delta^{-1}c_2$. The
lemma is proved. \endproof

 \begin{lemma} \label{bez1} Let $\Gamma$ be a regular comb of base width
$b\le 15N$. Assume that it has no maximal $t^{\pm 1}$-, $(t')^{\pm
1}$-bands except for the handle $\cal C,$ and $\cal C$  is either (a) a
$(t')^{\pm 1}$-band or (b) a $t^{\pm 1}$-band.
Also assume that there are no special $\theta$-edges in any derivative subcomb.
 Let $\Gamma$ in case (a),
have  $(23)^{\pm 1}$-bands but have no $(12)^{\pm 1}$-bands, and in case (b), it has
$(12)^{\pm 1}$-bands but has no $(23)^{\pm 1}$-bands.  Then $\area (\Gamma)\le
(\delta')^{-2} [\Gamma]+ c_2\mu^c(\Gamma).$
 \end{lemma}

 \proof We will consider the case (b) only. Observe that the $(12)$-cells of the
 $t^{\pm 1}$-band $\cal C$
 are special, and so $\lambda({\bf y})=0.$ Let ${\cal C}_1,\dots,{\cal
C}_s$ be the system of derivative bands of $\Gamma$. It follows from \ref{order} and
the assumptions of the lemma that all of them must be $k^{-
1}$-bands.

First assume that a derivative band ${\cal C}_i$ is of length $h_i=1.$ 
Then 
$\area(\Gamma_i)
\le 4(\delta')^{-}[\Gamma_i]$ by Lemma \ref{nizko}. Since the
statement of the lemma can by induction be assumed true for the comb
$\Gamma\backslash\Gamma_i$, this implies that the statement is true
for $\Gamma$ since $\mu^c(\Gamma)\ge \mu^c
(\Gamma\backslash\Gamma_i)\ge 0$ and $[\Gamma]=
[\Gamma_i]+[\Gamma\backslash\Gamma_i].$ Thus we may assume further
that $h_i>1$ for every $i$.

{\bf Case 1.} There is no derivative band whose length $h_{i_0}$ is at
least $0.8\sum_{i=1}^s h_i$. Then $\sum h_i^2\le 4hh_-\le
8\kappa^c(\Gamma_i)$ by Lemma \ref{lgamma}. By Lemmas \ref{comb}
and \ref{simple} (a), $\area(\Gamma_i)\le
60Nh_i^2+2(\delta')^{-1}h_i(|{\bf z}_i|-|h_i|)$, and therefore by Lemma
\ref{deriv}, $\sum \area(\Gamma_i)\le
480N\kappa^c(\Gamma)+2(\delta')^{-1}[\Gamma].$

{\bf Case 2.} There is a derivative band ${\cal C}_{i_0}$ with
$h_{i_0}\ge 0.8\sum_{i=1}^s h_i$, and there is a short derivative
${\cal C}'_j$ in it of length $h'_j>0.6h' \ge 0.2\sum h_i.$  (Here
we assume that the total some of lengths of short derivatives $h'$
is at least $\sum h_i/3$ because  $h_i>1$ for every $i$.) Then, by \ref{kk'} and \ref{iv}, at
least $h'_j/3$ $a$-bands starting on  ${\cal C'}_j$ must end on $z$. Therefore
by Lemmas \ref{comb} and \ref{simple} (a),, $$\sum \area(\Gamma_i)\le
\sum 60Nh_i^2+ 2(\delta')^{-1}[\Gamma] $$ $$ \le
60Nh|{\bf z}|_a(0.2)^{-1}(1/3)^{-1}+ 2(\delta')^{-1}[\Gamma] \le
900N(\delta')^{-1}[\Gamma]+ 2(\delta')^{-1}[\Gamma] \le  0.5
(\delta')^{-2}[\Gamma]$$ since $(\delta')^{-1}\ge 2000N.$

{\bf Case 3.} There is a derivative ${\cal C}_{i_0}$ with
$h_{i_0}\ge 0.8\sum_{i=1}^s h_i$, and there are no short derivatives
${\cal C'}_j$ of length $h'_j>0.6h'$. Then $h'_-\ge h'-0.6h'\ge
0.4\sum h_i/3 = \frac{2}{15}\sum h_i$. Hence by Lemmas \ref{comb} and \ref{l'gamma} , $$\sum
\area(\Gamma_i)\le 60Nh_i^2 +2(\delta')^{-1}[\Gamma]\le \frac{15}{2}60 N
hh'_-+2(\delta')^{-1}[\Gamma] \le
2700N\lambda^c(\Gamma)+2(\delta')^{-1}[\Gamma]$$

Thus, in any case 
\begin{equation} \label{apr1}
\sum \area(\Gamma_i)\le
480N\kappa^c(\Gamma)+2700N\lambda^c(\Gamma)
+0.5(\delta')^{-2}[\Gamma]
\end{equation}
By Lemmas \ref{comb'}, \ref{lgamma},
and \ref{l'gamma},
 the number of cells in all simple bands of $\Gamma$ does not
exceed 
\begin{equation}\label{apr2}
h(h_-+h'_-+(\delta')^{-1}(|{\bf z}|-h-1)\le
2\kappa^c(\Gamma)+6\lambda^c(\Gamma)+ (\delta')^{-1}[\Gamma]
\end{equation}
The  two upper bounds (\ref{apr1}) and (\ref{apr2}) together prove the lemma since $c_0\ge 1$
and $c_2>2800N.$ \endproof

\begin{lemma} \label{bez12} Let $\Gamma$ be a regular comb of base width $b$, where $3N\le b\le15N$.
Assume that its history either (a) contains $(12)$-rules but does
not contain $(23)$-rules or (b) vice versa. Then it admits a long
quasicomb $\Delta$ such that $\area(\Delta)\le c_3[\Delta]
+c_2\mu^c(\Delta).$
\end{lemma}
 \proof Here we consider case (a) only. If $\Gamma$ has a $(t')^{\pm 1}$-band, the statement follows from
 Lemma \ref{2etapa} provided this band crosses a maximal $(12)$-band, and it follows from 
 Lemma \ref{t} (2) otherwise since $\lambda$ takes non-negative values on one-Step combs
 and every subcomb with (passive) $(t')^{\pm 1}$-handle is long. We may therefore assume that $\Gamma$ has no such bands. 
 
 Assume $\Gamma$ has a left-most
 maximal $t^{\pm 1}$-band $\cal B$. If $\cal B$ does not have $(12)$-cells, then again the statement follows from Lemma
 \ref{t} (2). So we assume further that $\cal B$ has $(12)$-cells.

 Let $\Delta$ be the subcomb with the handle $\cal B$. If there are no special $\theta$-edges  from the left of each derivative band of $\Delta,$ then the statement follows from Lemma \ref{bez1} since $c_3>(\delta')^{-2}.$
 Therefore we may assume that $\Delta$ has
 a maximal $(12)$-band $\cal T$ crossing a derivative band ${\cal B}_i$ of $\cal B$ and having a special $\theta$-edge
 from the left of ${\cal B}_i$. 
 
 Note that ${\cal B}_i$ is a $k^{-1}$-band by the choice of $\cal B$ and by \ref{order}. Since only $t$-bands and $k^{\pm 1}$-bands can have special $(12)$-edges (and $k^{-1}$-band can have them from the right only),
  it follows that   $\cal T$ has a $k$-cell from the left of ${\cal B}_i$. 
 Since $\Gamma$ has no $(t')$-cells,
 the base of $\cal T$ is not aligned between the letter $k$ and the next letter $k^{-1}$
 by \ref{order}. Moreover  it
has  a subword $p_1p_1^{-1}s_0^{-1}$ between the $k$- and $k^{-1}$-letters by \ref{i}.
  So the statement of the lemma follows from Lemma \ref{p1}. (Again, we take into account that $\lambda$ is non-negative on one-Step combs.) Thus we may further
 assume that $\Gamma$ has no $t$-cells.

 By Lemma \ref{oneage}, one may assume that $\Gamma$ has
 no one-Step subcombs of base width $>2N.$ Since $b\ge 3N,$ this implies the existence of a
 $(12)$-band $\cal T$ with base $B_0$ of length $\ge N.$ Since $N>2||B||$ and $B_0$ has neither $t$- nor $t'$-letters, it must have at least two subwords of the form $q^{\pm 1}q^{\mp 1}$ for some  base letter $q$ (see \ref{order}). But the existence of $s_1^{-1}s_1$ excludes the possibility of all other
  subwords $q^{\pm 1}q^{\mp 1}$ by \ref{i}, and also by \ref{i}, the existence of $k^{-1}k$ implies
  the existence of at least one subword $s_0p_1p_1^{-1}s_0^{-1}.$  Thus  in any case $B$
 must have a subword $p_1p_1^{-1}s_0^{-1},$ which finishes the proof as in the previous paragraph.
 \endproof

 \section{Combs with multi-Step histories}
 
 In this section, we allow  all three Steps in comb histories.
 Although Lemma \ref{itog} gives no estimate of the area if the size of a comb
 is close, in a sense, to one of the numbers $T_i$-s, this lemma (together with
 the lemmas of the next section) will imply that the Dehn functions of the groups $M$ and $G$
  are almost quadratic  because  the set of $T_i$-s has infinitely many
 very long gaps. Again, to obtain upper estimates of areas for various combs one should
 apply a skillful combination of a number of quadratic parameter. For example,
 Lemma \ref{nuK} (and also Lemma \ref{areaD1} in the next Section) shows the use of the $\nu$-mixture.

 Let $H$ be the history of a comb $\Gamma.$ Consider a factorization
 $H\equiv H(1)\dots H(m),$ where no two non-empty factors are separated by empty ones. We say that this factorization
 is \label{firm} {\it firm} if for every $i=1,\dots, m-1$,

 (a) for non-empty $H_i$ and $H_{i+1}$, the last letter of $H_{i}$ and the first letter of $H_{i+1}$ must belong to different Steps;
so one of these two letters  is $(12)^{\pm 1}$- or $(23)^{\pm 1}$-letter calling   $(i)-(i+1)$ {\it transition letters}; the maximal $\theta$-band of $\Gamma$ corresponding to the $(i)-(i+1)$ transition letter
is an \label{(i)-b} $(i)-(i+1)$-{\it transition} band;

 (b) the transition $\theta$-bands of $\Gamma$ 
 are not simple.

There might be many firm factorizations of $\Gamma.$ Observe that if a factorization $H\equiv H(1)\dots H(m)$ is firm,
then $H^{-1}\equiv H(m)^{-1}\dots H(1)^{-1}$ is a firm factorization for the history of the mirror copy $\Gamma^{-1}$ of the comb $\Gamma$.

\begin{lemma} \label{two01} Let $\Gamma$ be a comb of base width $b\le 15N$
with a firm factorization of the history $H\equiv H(1)H(2)H(3)$, where
$H(2),$ $H(3)$ are one-Step histories, and $h(2)\ge 3h(3)$ (o\emph{}r
$h(3)\ge 3h(2)$). Assume that the handle $\cal C$ of $\Gamma$ is a
$t^{\pm 1}$- or $(t')^{\pm 1}$-band, and the $(1)-(2)$ transition band  has no $(\theta,a)$-cells
between $\cal C$ and the derivative band crossing this transition band, and the  $H(2)$-part (resp., the
$H(3)$-part) of $\Gamma$ has passive $k$- or $k'$-cells only in the $(12)$- or in the $(23)$-
bands. Then provided
$h(2)\ge 0.01 h$ (respectively, $h(3)\ge 0.01h$), there is a long
subcomb $\Delta$ in $\Gamma$ with $\area(\Delta)\le c_3 [\Delta]+
c_2\mu^c(\Delta).$
\end{lemma}

\proof   We will prove the lemma assuming that $h(2)\ge 3h(3)$ and
$h(2)\ge 0.01h \;(>0)$ since the proof of the second version of the lemma
is similar.

Consider the system of derivative bands ${\cal C}_1,\dots,{\cal
C}_s$ of $\Gamma$. Let ${\cal C}_1,\dots,{\cal C}_u$ have histories
$H_1,\dots,H_u$ which are subwords of $H(1)$, ${\cal C}_{u+1}$ have
history $H_{u+1}\equiv H_{u+1}(1)H_{u+1}(2)$, where $H_{u+1}(1)$ and
$H_{u+1}(2)$ are a suffix and a prefix of $H(1)$ and $H(2)$,
respectively: ${\cal C}_{u+1}$ is a union of bands ${\cal
C}_{u+1}(1)$ and ${\cal C}_{u+1}(2)$ having these two histories, and
$h_{u+1}=h_{u+1}(1)+h_{u+1}(2)> 0$. Similarly we define subwords $H_{u+2},\dots, H_{v}$
of $H(2)$ and
$H_{v+1}\equiv H_{v+1}(2)H_{v+1}(3)$, while  $H_{v+2},\dots, H_{s}$ are subword of $H(3)$.
(It is also possible that ${\cal C}_{u+1}$ has history
$H_{u+1}(1)H_{u+1}(2) H_{u+1}(3),$ where $H_{u+1}(2)\equiv H(2),$ and we
will come back to this case later on.)

Proving by contradiction, we assume that $\Gamma$ has no subcombs
$\Delta$ with area satisfying  the statement of the lemma.

The band ${\cal C}_{u+1}(2),$ if it has non-zero length, is not a $t^{\pm 1}$- or
$(t')^{\pm 1}$-band by the condition on the $(1)-(2)$-transition band and by \ref{i}.
If some ${\cal C}_i$ is a $t^{\pm 1}$- or
$(t')^{\pm 1}$-band, for $i\in [u+2, v+1]$, then the derivative subcomb
$\Gamma_i$ satisfies the condition of Lemma \ref{t}, 
a contradiction since $c_3>c_1$. Therefore all the derivative bands
of the system ${\cal C}_{u+1}(2), {\cal C}_{u+2},\dots,{\cal
C}_{v+1}(2)$ are $k^{-1}$- or $k'$-bands by \ref{i}.

Assume that one of the numbers
$h_{u+1}(2),h_{u+2},\dots,h_{v+1}(2)$, redenote it by $g$, is at least
$0.9h(2)\ge 0.009h>\delta h$

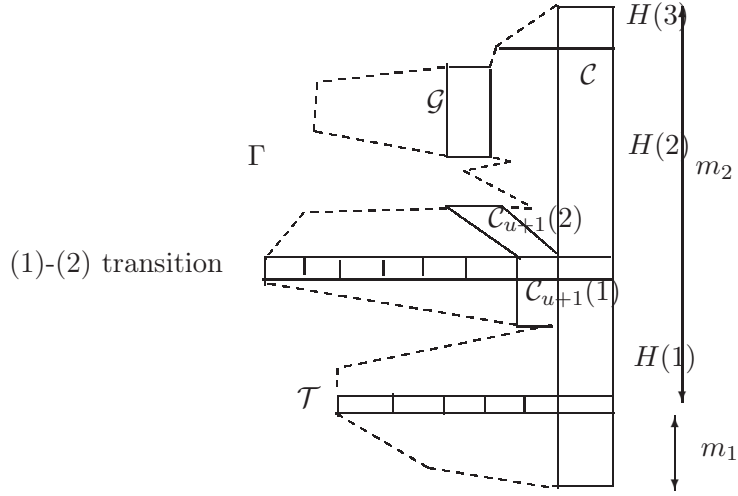
\unitlength 1mm 
\linethickness{0.4pt}
\ifx\plotpoint\undefined\newsavebox{\plotpoint}\fi 
\begin{picture}(100.25,71.75)(0,0)
\put(76.5,71.75){\line(0,-1){64.25}}
\put(83.75,71.25){\line(0,-1){63.5}}
\put(84,17.5){\line(-1,0){37}}
\put(83.5,19.75){\line(-1,0){36.25}}
\put(47.25,19.75){\line(0,-1){2.75}}
\put(83.75,35.25){\line(-1,0){46.5}}
\put(83.5,38.25){\line(-1,0){45.75}}
\put(37.5,38.25){\line(0,-1){3.75}}
\multiput(61.5,45)(.048076923,-.033653846){208}{\line(1,0){.048076923}}
\put(61.75,45){\line(1,0){7.25}}
\multiput(69,45)(.038978495,-.033602151){186}{\line(1,0){.038978495}}
\put(71,38.25){\line(0,-1){9.25}}
\put(71,29){\line(1,0){4.75}}
\put(61.75,64){\line(0,-1){12.5}}
\put(61.75,51.5){\line(1,0){5.75}}
\put(67.5,51.5){\line(0,1){11.75}}
\put(62,63.5){\line(1,0){5.5}}
\put(76.5,71.5){\line(1,0){7.5}}
\put(84,71.5){\line(0,1){0}}
\put(68.75,66){\line(1,0){15.25}}
\multiput(76.68,71.43)(-.03125,-.03125){4}{\line(0,-1){.03125}}
\multiput(76.18,71.68)(-.0484848,-.0333333){15}{\line(-1,0){.0484848}}
\multiput(74.725,70.68)(-.0484848,-.0333333){15}{\line(-1,0){.0484848}}
\multiput(73.271,69.68)(-.0484848,-.0333333){15}{\line(-1,0){.0484848}}
\multiput(71.816,68.68)(-.0484848,-.0333333){15}{\line(-1,0){.0484848}}
\multiput(70.362,67.68)(-.0484848,-.0333333){15}{\line(-1,0){.0484848}}
\multiput(68.907,66.68)(-.0484848,-.0333333){15}{\line(-1,0){.0484848}}
\multiput(68.18,65.93)(-.03125,-.09375){8}{\line(0,-1){.09375}}
\multiput(67.68,64.43)(-.03125,-.09375){8}{\line(0,-1){.09375}}
\put(61.43,63.43){\line(-1,0){.9444}}
\put(59.541,63.207){\line(-1,0){.9444}}
\put(57.652,62.985){\line(-1,0){.9444}}
\put(55.763,62.763){\line(-1,0){.9444}}
\put(53.874,62.541){\line(-1,0){.9444}}
\put(51.985,62.319){\line(-1,0){.9444}}
\put(50.096,62.096){\line(-1,0){.9444}}
\put(48.207,61.874){\line(-1,0){.9444}}
\put(46.319,61.652){\line(-1,0){.9444}}
\put(44.43,61.43){\line(0,-1){.9286}}
\put(44.287,59.573){\line(0,-1){.9286}}
\put(44.144,57.715){\line(0,-1){.9286}}
\put(44.001,55.858){\line(0,-1){.9286}}
\multiput(44.18,54.93)(.153509,-.028509){6}{\line(1,0){.153509}}
\multiput(46.022,54.588)(.153509,-.028509){6}{\line(1,0){.153509}}
\multiput(47.864,54.245)(.153509,-.028509){6}{\line(1,0){.153509}}
\multiput(49.706,53.903)(.153509,-.028509){6}{\line(1,0){.153509}}
\multiput(51.548,53.561)(.153509,-.028509){6}{\line(1,0){.153509}}
\multiput(53.39,53.219)(.153509,-.028509){6}{\line(1,0){.153509}}
\multiput(55.232,52.877)(.153509,-.028509){6}{\line(1,0){.153509}}
\multiput(57.074,52.535)(.153509,-.028509){6}{\line(1,0){.153509}}
\multiput(58.917,52.193)(.153509,-.028509){6}{\line(1,0){.153509}}
\multiput(60.759,51.851)(.153509,-.028509){6}{\line(1,0){.153509}}
\multiput(67.93,51.43)(.15,-.033333){5}{\line(1,0){.15}}
\multiput(69.43,51.096)(.15,-.033333){5}{\line(1,0){.15}}
\multiput(70.18,50.93)(-.15625,-.03125){5}{\line(-1,0){.15625}}
\multiput(68.617,50.617)(-.15625,-.03125){5}{\line(-1,0){.15625}}
\multiput(67.055,50.305)(-.15625,-.03125){5}{\line(-1,0){.15625}}
\multiput(65.492,49.992)(-.15625,-.03125){5}{\line(-1,0){.15625}}
\multiput(63.93,49.68)(.0594406,-.0314685){13}{\line(1,0){.0594406}}
\multiput(65.475,48.862)(.0594406,-.0314685){13}{\line(1,0){.0594406}}
\multiput(67.021,48.043)(.0594406,-.0314685){13}{\line(1,0){.0594406}}
\multiput(68.566,47.225)(.0594406,-.0314685){13}{\line(1,0){.0594406}}
\multiput(70.112,46.407)(.0594406,-.0314685){13}{\line(1,0){.0594406}}
\multiput(71.657,45.589)(.0594406,-.0314685){13}{\line(1,0){.0594406}}
\put(72.93,44.68){\line(-1,0){.875}}
\put(71.18,44.805){\line(-1,0){.875}}
\put(61.68,44.68){\line(-1,0){.95}}
\put(59.78,44.63){\line(-1,0){.95}}
\put(57.88,44.58){\line(-1,0){.95}}
\put(55.98,44.53){\line(-1,0){.95}}
\put(54.08,44.48){\line(-1,0){.95}}
\put(52.18,44.43){\line(-1,0){.95}}
\put(50.28,44.38){\line(-1,0){.95}}
\put(48.38,44.33){\line(-1,0){.95}}
\put(46.48,44.28){\line(-1,0){.95}}
\put(44.58,44.23){\line(-1,0){.95}}
\put(42.68,44.18){\line(0,1){0}}
\multiput(42.68,44.18)(-.0328125,-.0390625){16}{\line(0,-1){.0390625}}
\multiput(41.63,42.93)(-.0328125,-.0390625){16}{\line(0,-1){.0390625}}
\multiput(40.58,41.68)(-.0328125,-.0390625){16}{\line(0,-1){.0390625}}
\multiput(39.53,40.43)(-.0328125,-.0390625){16}{\line(0,-1){.0390625}}
\multiput(38.48,39.18)(-.0328125,-.0390625){16}{\line(0,-1){.0390625}}
\put(37.43,37.93){\line(1,0){.125}}
\put(37.68,37.93){\line(0,1){.125}}
\multiput(37.93,34.68)(.188571,-.031429){5}{\line(1,0){.188571}}
\multiput(39.815,34.365)(.188571,-.031429){5}{\line(1,0){.188571}}
\multiput(41.701,34.051)(.188571,-.031429){5}{\line(1,0){.188571}}
\multiput(43.587,33.737)(.188571,-.031429){5}{\line(1,0){.188571}}
\multiput(45.473,33.423)(.188571,-.031429){5}{\line(1,0){.188571}}
\multiput(47.358,33.108)(.188571,-.031429){5}{\line(1,0){.188571}}
\multiput(49.244,32.794)(.188571,-.031429){5}{\line(1,0){.188571}}
\multiput(51.13,32.48)(.188571,-.031429){5}{\line(1,0){.188571}}
\multiput(53.015,32.165)(.188571,-.031429){5}{\line(1,0){.188571}}
\multiput(54.901,31.851)(.188571,-.031429){5}{\line(1,0){.188571}}
\multiput(56.787,31.537)(.188571,-.031429){5}{\line(1,0){.188571}}
\multiput(58.673,31.223)(.188571,-.031429){5}{\line(1,0){.188571}}
\multiput(60.558,30.908)(.188571,-.031429){5}{\line(1,0){.188571}}
\multiput(62.444,30.594)(.188571,-.031429){5}{\line(1,0){.188571}}
\multiput(64.33,30.28)(.188571,-.031429){5}{\line(1,0){.188571}}
\multiput(66.215,29.965)(.188571,-.031429){5}{\line(1,0){.188571}}
\multiput(68.101,29.651)(.188571,-.031429){5}{\line(1,0){.188571}}
\multiput(69.987,29.337)(.188571,-.031429){5}{\line(1,0){.188571}}
\multiput(75.68,29.18)(-.158333,-.031944){6}{\line(-1,0){.158333}}
\multiput(73.78,28.796)(-.158333,-.031944){6}{\line(-1,0){.158333}}
\multiput(71.88,28.413)(-.158333,-.031944){6}{\line(-1,0){.158333}}
\multiput(69.98,28.03)(-.158333,-.031944){6}{\line(-1,0){.158333}}
\multiput(68.08,27.646)(-.158333,-.031944){6}{\line(-1,0){.158333}}
\multiput(66.18,27.263)(-.158333,-.031944){6}{\line(-1,0){.158333}}
\multiput(64.28,26.88)(-.158333,-.031944){6}{\line(-1,0){.158333}}
\multiput(62.38,26.496)(-.158333,-.031944){6}{\line(-1,0){.158333}}
\multiput(60.48,26.113)(-.158333,-.031944){6}{\line(-1,0){.158333}}
\multiput(58.58,25.73)(-.158333,-.031944){6}{\line(-1,0){.158333}}
\multiput(56.68,25.346)(-.158333,-.031944){6}{\line(-1,0){.158333}}
\multiput(54.78,24.963)(-.158333,-.031944){6}{\line(-1,0){.158333}}
\multiput(52.88,24.58)(-.158333,-.031944){6}{\line(-1,0){.158333}}
\multiput(50.98,24.196)(-.158333,-.031944){6}{\line(-1,0){.158333}}
\multiput(49.08,23.813)(-.158333,-.031944){6}{\line(-1,0){.158333}}
\put(47.18,23.43){\line(0,1){0}}
\put(47.18,23.43){\line(0,-1){.8125}}
\put(47.055,21.805){\line(0,-1){.8125}}
\multiput(46.93,17.43)(.0544444,-.0322222){15}{\line(1,0){.0544444}}
\multiput(48.563,16.463)(.0544444,-.0322222){15}{\line(1,0){.0544444}}
\multiput(50.196,15.496)(.0544444,-.0322222){15}{\line(1,0){.0544444}}
\multiput(51.83,14.53)(.0544444,-.0322222){15}{\line(1,0){.0544444}}
\multiput(53.463,13.563)(.0544444,-.0322222){15}{\line(1,0){.0544444}}
\multiput(55.096,12.596)(.0544444,-.0322222){15}{\line(1,0){.0544444}}
\multiput(56.73,11.63)(.0544444,-.0322222){15}{\line(1,0){.0544444}}
\multiput(58.363,10.663)(.0544444,-.0322222){15}{\line(1,0){.0544444}}
\put(59.18,10.18){\line(0,1){0}}
\multiput(59.18,10.18)(.188889,-.030556){5}{\line(1,0){.188889}}
\multiput(61.069,9.874)(.188889,-.030556){5}{\line(1,0){.188889}}
\multiput(62.957,9.569)(.188889,-.030556){5}{\line(1,0){.188889}}
\multiput(64.846,9.263)(.188889,-.030556){5}{\line(1,0){.188889}}
\multiput(66.735,8.957)(.188889,-.030556){5}{\line(1,0){.188889}}
\multiput(68.624,8.652)(.188889,-.030556){5}{\line(1,0){.188889}}
\multiput(70.513,8.346)(.188889,-.030556){5}{\line(1,0){.188889}}
\multiput(72.402,8.041)(.188889,-.030556){5}{\line(1,0){.188889}}
\multiput(74.291,7.735)(.188889,-.030556){5}{\line(1,0){.188889}}
\put(75.75,7.75){\line(1,0){7.75}}
\put(86.5,23.75){$H(1)$}
\put(85.75,52){$H(2)$}
\put(85.5,69.25){$H(3)$}
\put(59,57.75){$\cal G$}
\put(72.25,32.5){${\cal C}_{u+1}(1)$}
\put(67.25,42.25){${\cal C}_{u+1}(2)$}
\put(79.5,61.25){$\cal C$}
\put(41.75,18.25){$\cal T$}
\put(92,7.25){\vector(0,-1){.07}}\put(92,17){\vector(0,1){.07}}\put(92,17){\line(0,-1){9.75}}
\put(95.5,11.75){$m_1$}
\put(42.75,38){\line(0,-1){2}}
\put(47.5,37.75){\line(0,-1){2.5}}
\put(53.25,37.75){\line(0,-1){2}}
\put(58.5,38){\line(0,-1){2.25}}
\put(64.25,38){\line(0,-1){2.5}}
\put(54.5,20){\line(0,-1){2.25}}
\put(61.25,19.5){\line(0,-1){2.25}}
\put(66.75,19.5){\line(0,-1){1.75}}
\put(72,19.5){\line(0,-1){2}}
\put(35.25,50.5){$\Gamma$}
\put(93,19){\vector(0,-1){.07}}\put(93,71.5){\vector(0,1){.07}}\put(93,71.5){\line(0,-1){52.5}}
\put(94.75,49.25){$m_2$}
\put(3.5,36.25){(1)-(2) transition}
\end{picture}

Denote by ${\cal G}$ the corresponding derivative (sub)band of length $g$. Since $h(2)\ge 3h(3)$, we see that
at most $h(2)/3+0.1h(2)\le  h(2)/2< 2g/3$ maximal $a$-bands starting on $\cal G$ end on
the other derivative bands of $\Gamma.$
So we
may apply Lemma \ref{h0} (a) to $\cal G$ and get inequality

\begin{equation}\label{19000}
\area(\Gamma)\le (\delta')^{-2}[\Gamma]
\end{equation}

Now assume that each of $h_{u+1}(2),h_{u+2}(2),\dots,h_{v+1}(2)$ is less
than $0.9h(2)$. It follows from this assumption that $\max_{i=1}^s
h_i < h-h(2)+0.9h(2)\le (1-0.001)h$ since $h(2)\ge 0.01h.$
Therefore,  by  Lemma \ref{sravnim}, $l_-\ge \min(\sum_{i=1}^s
h_i,\; h-\max_{i=1}^s h_i) \ge 0.001\sum_{i=1}^{s} h_i$. Hence by
Lemmas \ref{comb}, \ref{simple} (1), and by inequality $\delta'^{-1}> \max(40N, 4000),$ we have

$$\sum_{i=1}^{s} \area(\Gamma_i)\le 60N
\sum_{i=1}^{s}h_i^{2}+\sum_{i=1}^{s}2\alpha_ih_i  \le$$
\begin{equation} 
60Nh\times 1000 l_-+2h((\delta')\iv(|{\bf z}|-h)
 \le 2(\delta')\iv[\Gamma]+(\delta')^{-2}hl_-,\label{sumi}
\end{equation}
where $\alpha_i=|{\bf z}^{\Gamma_i}|_a.$

By the inequalities (\ref{ns}) and $l_-\ge 0.001\sum_{i=1}^s h_i$, the
number $n_s$ of cells in all the simple $\theta$-bands of $\Gamma$
satisfies inequality

$$n_s \le h(\frac32\sum_{i=1}^sh_i +(\delta')\iv(|{\bf z}|-|{\bf y}|))\le
h(1500l_- +(\delta')\iv(|{\bf z}|-|{\bf y}|)) $$

 From this inequality,
(\ref{sumi}), and by Lemma \ref{lgamma}, we have
\begin{equation}\label{not56}
\area(\Gamma)\le c_1 [\Gamma]+c_1 h l_-/2 \le
c_1([\Gamma]+\kappa^c(\Gamma))
\end{equation}
since $c_1\ge 3(\delta')^{-2}$.

If $\lambda^c(\Gamma)\ge 0,$ the statement of the lemma follows from Inequalities (\ref{19000})
and (\ref{not56}). Then we will assume that $-\lambda({\bf y})\le\lambda^c(\Gamma)<0.$

To estimate $\lambda^c(\Gamma)$ from below, we again start with the assumption that
$g\ge 0.9h(2),$ and so at least $g/2$ maximal $a$-bands end on
$\bf z$. We have $|{\bf z}|_a\ge g/2\ge 0.4h(2) \ge 0.004h.$ Hence
$|{\bf z}|-h\ge\delta'h/250$ by Lemma \ref{simple} (a). Then by Lemma
\ref{mixturec} (a),
$$\lambda^c(\Gamma) \ge -\lambda({\bf y})\ge -h^2 \ge -250(\delta')^{-1}[\Gamma]$$
Since $c_1\ge 2(\delta')^{-2},$ this estimate together with
(\ref{19000}) yields $\area(\Gamma)\le
c_1[\Gamma]+\lambda^c(\Gamma)\le c_1[\Gamma]+\mu^c(\Gamma).$
Here, the right-hand side does not exceed
$c_3[\Gamma]+c_2\mu^c(\Gamma)$ because $c_3 > c_1c_2,$ and so
the lemma proved in case $g\ge 0.9h(2).$

Now let $g< 0.9h(2).$ Since $\lambda({\bf y})>0,$  by Lemma
\ref{mixturec}(e), there is a maximal $(12)$- or $(23)$-band $\cal T$ 
such that there
are $m_1$ $\theta$-bands crossing $\cal C$ below $\cal T$, $m_2$
$\theta$-bands crossing $\cal C$ above $\cal T,$ and
$\lambda({\bf y})\le 2m_1m_2.$

Assume first that $\cal T$ belongs to the $H(1)$-part of $\Gamma.$
Then one of the two ends of the derivative band ${\cal C}_{u+1}$ lies
above $\cal T$ but there are at least $0.1h(2)$ $\theta$-bands above
this end since $g\le 0.9h(2).$ Therefore by Lemma \ref{mixturec}
(d), $$\kappa^c(\Gamma)\ge 0.1m_1h(2)\ge 0.001m_1h \ge
0.001m_1m_2\ge \lambda({\bf y})/2000.$$ Then assume that $\cal T$
belongs to the $H(2)H(3)$-part of $\Gamma.$ Then $m_2\le h(3)\le
h(2)/3,$ since the one step history $H(2)$ has no $(12)^{\pm 1}$- or
$(23)^{\pm 1}$-rules. On the other hand, there is an (lower) end
of the derivative ${\cal C}_{v+1}$ such that there are at least
$h(3)$ maximal $\theta$-bands above it and at least $0.1h(2)$ below
it since $g\le 0.9h(2).$ Hence
$$\kappa^c(\Gamma)\ge 0.1h(3)h(2)\ge 0.1m_2(0.01h) \ge 0.001m_1m_2\ge \lambda({\bf y})/2000.$$
Thus $\lambda^c(\Gamma)\le 2000\kappa(\Gamma)$ if $g\le 0.9h(2)$. This inequality
and (\ref{not56}) imply $\area(\Gamma) \le
c_1([\Gamma]+\mu^c(\Gamma))$ because $c_0\ge 2001.$
This leads to a contradiction since $c_3>c_2 >c_1.$

The case where ${\cal C}_{u+1}$ had history
$H_{u+1}(1)H_{u+1}(2) H_{u+1}(3)$ with $H_{u+1}(2)\equiv H(2),$
can be treated as the above subcase with $g\ge 0.9h(2),$ since now $g=h(2).$
\endproof

We omit the proof of the following lemma since the argument would be just a
simplified version of the proof given above for Lemma \ref{two01}:
instead of the inequalities $h(3)\ge 0.01h$ and $h(3)\ge 3h(2),$ below we 
have that $h''\ge 0.7 h$ (and so $h''\ge \frac{7}{3}h'$) and $H(1)$ is empty.

\begin{lemma}\label{podslovo1}.
Let $\Gamma$ be a comb of basic width $b\le 15N$ with a $t^{\pm 1}$- or
$(t')^{\pm 1}$-handle $\cal C,$ and let the history of $\Gamma$ have a
firm factorization $H\equiv H'H'',$ where $h''\ge 0.7h$ and $H''$ is of Step
$(2)$. Let the derivative bands ${\cal C}_i$ be all a $k^{\pm 1}$-
or $(k')^{\pm 1}$-bands. Then $\Gamma$ has a long subcomb $\Delta$
with $\area(\Delta)\le c_3 [\Delta]+ c_2\mu^c (\Delta)$.
\end{lemma}
 $\Box$

\begin{lemma}\label{nuK} 
Let $\Gamma'$ be a subcomb of a comb
$\Gamma$ of base width $b\le 15N,$ $\cal C'$ and $\cal C$ their
handles with histories $H'$ and $H$, respectively, and each of these
handles a $t^{\pm 1}$- or a $(t')^{\pm 1}$-band. Assume that $h'<
h/2,$ and $H$ has at most $6$ letters $(12)^{\pm 1}$ and
$(23)^{\pm 1}.$ Then either $\Gamma$ has a long subcomb $\Delta$
with $\area(\Delta)\le c_3[\Delta] +c_2\mu^c(\Delta)$ or \\
$\area(\Gamma')\le
c_1([\Gamma']+\lambda^c(\Gamma')+\nu_J^c(\Gamma)-
\nu_J^c(\Delta')) ,$ where $\Delta'= \Gamma\backslash\Gamma'.$
\end{lemma}

\proof If there is a maximal $t$- or $t'$-band in $\Gamma,$ having no
$(12)$- or $(23)$-cells,
then by
Lemma \ref{t} (b), it is a handle of a long subcomb $\Delta$ with
$$\area(\Delta)\le c_1[\Delta] +c_1\kappa^c(\Delta)\le c_3[\Delta] +c_2\mu^c(\Delta)$$
because $\lambda^c(\Delta)\ge -\lambda({\bf y}^{\Delta})=0$ for a one-Step $\Delta,$
and any subcomb with passive from the right handle is long.

Therefore we  may assume that every $t$- or $t'$-band
of $\Gamma$ intersects a $\theta$-band corresponding to one of the
$\theta$-letters $(12)^{\pm 1}$, $(23)^{\pm 1}$. Since their base widths are at most
$15N$, the number of maximal $t$- and $t'$-bands in $\Delta$ does
not exceed $90N<J/2.$

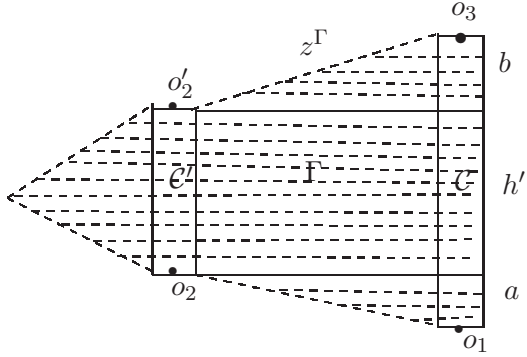
\unitlength 1mm 
\linethickness{0.4pt}
\ifx\plotpoint\undefined\newsavebox{\plotpoint}\fi 
\begin{picture}(270.25,62.5)(-20,0)
\put(33,34.5){\line(0,-1){22.75}}
\put(33,11.75){\line(1,0){5.75}}
\put(38.75,11.75){\line(0,1){22}}
\put(38.75,33.75){\line(-1,0){5.75}}
\put(71,43.75){\line(0,-1){39}}
\put(71,4.75){\line(1,0){6}}
\put(77,4.75){\line(0,1){38.75}}
\put(77,43.5){\line(-1,0){6}}
\multiput(70.68,43.68)(-.103175,-.031746){9}{\line(-1,0){.103175}}
\multiput(68.823,43.108)(-.103175,-.031746){9}{\line(-1,0){.103175}}
\multiput(66.965,42.537)(-.103175,-.031746){9}{\line(-1,0){.103175}}
\multiput(65.108,41.965)(-.103175,-.031746){9}{\line(-1,0){.103175}}
\multiput(63.251,41.394)(-.103175,-.031746){9}{\line(-1,0){.103175}}
\multiput(61.394,40.823)(-.103175,-.031746){9}{\line(-1,0){.103175}}
\multiput(59.537,40.251)(-.103175,-.031746){9}{\line(-1,0){.103175}}
\multiput(57.68,39.68)(-.103175,-.031746){9}{\line(-1,0){.103175}}
\multiput(55.823,39.108)(-.103175,-.031746){9}{\line(-1,0){.103175}}
\multiput(53.965,38.537)(-.103175,-.031746){9}{\line(-1,0){.103175}}
\multiput(52.108,37.965)(-.103175,-.031746){9}{\line(-1,0){.103175}}
\multiput(50.251,37.394)(-.103175,-.031746){9}{\line(-1,0){.103175}}
\multiput(48.394,36.823)(-.103175,-.031746){9}{\line(-1,0){.103175}}
\multiput(46.537,36.251)(-.103175,-.031746){9}{\line(-1,0){.103175}}
\multiput(44.68,35.68)(-.103175,-.031746){9}{\line(-1,0){.103175}}
\multiput(42.823,35.108)(-.103175,-.031746){9}{\line(-1,0){.103175}}
\multiput(40.965,34.537)(-.103175,-.031746){9}{\line(-1,0){.103175}}
\multiput(39.108,33.965)(-.103175,-.031746){9}{\line(-1,0){.103175}}
\multiput(32.43,33.93)(-.0509511,-.0326087){16}{\line(-1,0){.0509511}}
\multiput(30.799,32.886)(-.0509511,-.0326087){16}{\line(-1,0){.0509511}}
\multiput(29.169,31.843)(-.0509511,-.0326087){16}{\line(-1,0){.0509511}}
\multiput(27.538,30.799)(-.0509511,-.0326087){16}{\line(-1,0){.0509511}}
\multiput(25.908,29.756)(-.0509511,-.0326087){16}{\line(-1,0){.0509511}}
\multiput(24.278,28.712)(-.0509511,-.0326087){16}{\line(-1,0){.0509511}}
\multiput(22.647,27.669)(-.0509511,-.0326087){16}{\line(-1,0){.0509511}}
\multiput(21.017,26.625)(-.0509511,-.0326087){16}{\line(-1,0){.0509511}}
\multiput(19.386,25.582)(-.0509511,-.0326087){16}{\line(-1,0){.0509511}}
\multiput(17.756,24.538)(-.0509511,-.0326087){16}{\line(-1,0){.0509511}}
\multiput(16.125,23.495)(-.0509511,-.0326087){16}{\line(-1,0){.0509511}}
\multiput(14.495,22.451)(-.0509511,-.0326087){16}{\line(-1,0){.0509511}}
\multiput(13.68,21.93)(.0643813,-.0326087){13}{\line(1,0){.0643813}}
\multiput(15.354,21.082)(.0643813,-.0326087){13}{\line(1,0){.0643813}}
\multiput(17.028,20.234)(.0643813,-.0326087){13}{\line(1,0){.0643813}}
\multiput(18.701,19.386)(.0643813,-.0326087){13}{\line(1,0){.0643813}}
\multiput(20.375,18.538)(.0643813,-.0326087){13}{\line(1,0){.0643813}}
\multiput(22.049,17.691)(.0643813,-.0326087){13}{\line(1,0){.0643813}}
\multiput(23.723,16.843)(.0643813,-.0326087){13}{\line(1,0){.0643813}}
\multiput(25.397,15.995)(.0643813,-.0326087){13}{\line(1,0){.0643813}}
\multiput(27.071,15.147)(.0643813,-.0326087){13}{\line(1,0){.0643813}}
\multiput(28.745,14.299)(.0643813,-.0326087){13}{\line(1,0){.0643813}}
\multiput(30.419,13.451)(.0643813,-.0326087){13}{\line(1,0){.0643813}}
\multiput(32.093,12.604)(.0643813,-.0326087){13}{\line(1,0){.0643813}}
\multiput(38.68,11.68)(.159314,-.033088){6}{\line(1,0){.159314}}
\multiput(40.591,11.283)(.159314,-.033088){6}{\line(1,0){.159314}}
\multiput(42.503,10.886)(.159314,-.033088){6}{\line(1,0){.159314}}
\multiput(44.415,10.489)(.159314,-.033088){6}{\line(1,0){.159314}}
\multiput(46.327,10.091)(.159314,-.033088){6}{\line(1,0){.159314}}
\multiput(48.239,9.694)(.159314,-.033088){6}{\line(1,0){.159314}}
\multiput(50.15,9.297)(.159314,-.033088){6}{\line(1,0){.159314}}
\multiput(52.062,8.9)(.159314,-.033088){6}{\line(1,0){.159314}}
\multiput(53.974,8.503)(.159314,-.033088){6}{\line(1,0){.159314}}
\multiput(55.886,8.106)(.159314,-.033088){6}{\line(1,0){.159314}}
\multiput(57.797,7.709)(.159314,-.033088){6}{\line(1,0){.159314}}
\multiput(59.709,7.312)(.159314,-.033088){6}{\line(1,0){.159314}}
\multiput(61.621,6.915)(.159314,-.033088){6}{\line(1,0){.159314}}
\multiput(63.533,6.518)(.159314,-.033088){6}{\line(1,0){.159314}}
\multiput(65.444,6.121)(.159314,-.033088){6}{\line(1,0){.159314}}
\multiput(67.356,5.724)(.159314,-.033088){6}{\line(1,0){.159314}}
\multiput(69.268,5.327)(.159314,-.033088){6}{\line(1,0){.159314}}
\put(62.68,40.68){\line(1,0){.95}}
\put(64.58,40.68){\line(1,0){.95}}
\put(66.48,40.68){\line(1,0){.95}}
\put(68.38,40.68){\line(1,0){.95}}
\put(70.28,40.68){\line(1,0){.95}}
\put(72.18,40.68){\line(1,0){.95}}
\put(74.08,40.68){\line(1,0){.95}}
\put(75.98,40.68){\line(1,0){.95}}
\put(55.68,38.68){\line(1,0){.9545}}
\put(57.589,38.68){\line(1,0){.9545}}
\put(59.498,38.68){\line(1,0){.9545}}
\put(61.407,38.68){\line(1,0){.9545}}
\put(63.316,38.68){\line(1,0){.9545}}
\put(65.225,38.68){\line(1,0){.9545}}
\put(67.134,38.68){\line(1,0){.9545}}
\put(69.043,38.68){\line(1,0){.9545}}
\put(70.952,38.68){\line(1,0){.9545}}
\put(72.862,38.68){\line(1,0){.9545}}
\put(74.771,38.68){\line(1,0){.9545}}
\put(50.68,37.43){\line(1,0){.963}}
\put(52.606,37.393){\line(1,0){.963}}
\put(54.532,37.356){\line(1,0){.963}}
\put(56.457,37.319){\line(1,0){.963}}
\put(58.383,37.282){\line(1,0){.963}}
\put(60.309,37.245){\line(1,0){.963}}
\put(62.235,37.207){\line(1,0){.963}}
\put(64.161,37.17){\line(1,0){.963}}
\put(66.087,37.133){\line(1,0){.963}}
\put(68.013,37.096){\line(1,0){.963}}
\put(69.939,37.059){\line(1,0){.963}}
\put(71.865,37.022){\line(1,0){.963}}
\put(73.791,36.985){\line(1,0){.963}}
\put(75.717,36.948){\line(1,0){.963}}
\put(44.93,35.93){\line(1,0){.9697}}
\put(46.869,35.899){\line(1,0){.9697}}
\put(48.808,35.869){\line(1,0){.9697}}
\put(50.748,35.839){\line(1,0){.9697}}
\put(52.687,35.808){\line(1,0){.9697}}
\put(54.627,35.778){\line(1,0){.9697}}
\put(56.566,35.748){\line(1,0){.9697}}
\put(58.505,35.718){\line(1,0){.9697}}
\put(60.445,35.687){\line(1,0){.9697}}
\put(62.384,35.657){\line(1,0){.9697}}
\put(64.324,35.627){\line(1,0){.9697}}
\put(66.263,35.596){\line(1,0){.9697}}
\put(68.202,35.566){\line(1,0){.9697}}
\put(70.142,35.536){\line(1,0){.9697}}
\put(72.081,35.505){\line(1,0){.9697}}
\put(74.021,35.475){\line(1,0){.9697}}
\put(75.96,35.445){\line(1,0){.9697}}
\put(29.93,31.93){\line(1,0){.9947}}
\put(31.919,31.898){\line(1,0){.9947}}
\put(33.908,31.866){\line(1,0){.9947}}
\put(35.898,31.834){\line(1,0){.9947}}
\put(37.887,31.802){\line(1,0){.9947}}
\put(39.877,31.77){\line(1,0){.9947}}
\put(41.866,31.738){\line(1,0){.9947}}
\put(43.855,31.706){\line(1,0){.9947}}
\put(45.845,31.674){\line(1,0){.9947}}
\put(47.834,31.642){\line(1,0){.9947}}
\put(49.823,31.611){\line(1,0){.9947}}
\put(51.813,31.579){\line(1,0){.9947}}
\put(53.802,31.547){\line(1,0){.9947}}
\put(55.791,31.515){\line(1,0){.9947}}
\put(57.781,31.483){\line(1,0){.9947}}
\put(59.77,31.451){\line(1,0){.9947}}
\put(61.759,31.419){\line(1,0){.9947}}
\put(63.749,31.387){\line(1,0){.9947}}
\put(65.738,31.355){\line(1,0){.9947}}
\put(67.728,31.323){\line(1,0){.9947}}
\put(69.717,31.291){\line(1,0){.9947}}
\put(71.706,31.259){\line(1,0){.9947}}
\put(73.696,31.228){\line(1,0){.9947}}
\put(75.685,31.196){\line(1,0){.9947}}
\put(26.18,29.93){\line(1,0){.9853}}
\put(28.15,29.89){\line(1,0){.9853}}
\put(30.121,29.851){\line(1,0){.9853}}
\put(32.091,29.812){\line(1,0){.9853}}
\put(34.062,29.773){\line(1,0){.9853}}
\put(36.033,29.734){\line(1,0){.9853}}
\put(38.003,29.694){\line(1,0){.9853}}
\put(39.974,29.655){\line(1,0){.9853}}
\put(41.944,29.616){\line(1,0){.9853}}
\put(43.915,29.577){\line(1,0){.9853}}
\put(45.886,29.538){\line(1,0){.9853}}
\put(47.856,29.498){\line(1,0){.9853}}
\put(49.827,29.459){\line(1,0){.9853}}
\put(51.797,29.42){\line(1,0){.9853}}
\put(53.768,29.381){\line(1,0){.9853}}
\put(55.739,29.341){\line(1,0){.9853}}
\put(57.709,29.302){\line(1,0){.9853}}
\put(59.68,29.263){\line(1,0){.9853}}
\put(61.65,29.224){\line(1,0){.9853}}
\put(63.621,29.185){\line(1,0){.9853}}
\put(65.591,29.145){\line(1,0){.9853}}
\put(67.562,29.106){\line(1,0){.9853}}
\put(69.533,29.067){\line(1,0){.9853}}
\put(71.503,29.028){\line(1,0){.9853}}
\put(73.474,28.989){\line(1,0){.9853}}
\put(75.444,28.949){\line(1,0){.9853}}
\put(23.68,28.18){\line(1,0){.9861}}
\put(25.652,28.143){\line(1,0){.9861}}
\put(27.624,28.106){\line(1,0){.9861}}
\put(29.596,28.069){\line(1,0){.9861}}
\put(31.569,28.032){\line(1,0){.9861}}
\put(33.541,27.995){\line(1,0){.9861}}
\put(35.513,27.957){\line(1,0){.9861}}
\put(37.485,27.92){\line(1,0){.9861}}
\put(39.457,27.883){\line(1,0){.9861}}
\put(41.43,27.846){\line(1,0){.9861}}
\put(43.402,27.809){\line(1,0){.9861}}
\put(45.374,27.772){\line(1,0){.9861}}
\put(47.346,27.735){\line(1,0){.9861}}
\put(49.319,27.698){\line(1,0){.9861}}
\put(51.291,27.661){\line(1,0){.9861}}
\put(53.263,27.624){\line(1,0){.9861}}
\put(55.235,27.587){\line(1,0){.9861}}
\put(57.207,27.55){\line(1,0){.9861}}
\put(59.18,27.513){\line(1,0){.9861}}
\put(61.152,27.476){\line(1,0){.9861}}
\put(63.124,27.439){\line(1,0){.9861}}
\put(65.096,27.402){\line(1,0){.9861}}
\put(67.069,27.365){\line(1,0){.9861}}
\put(69.041,27.328){\line(1,0){.9861}}
\put(71.013,27.291){\line(1,0){.9861}}
\put(72.985,27.254){\line(1,0){.9861}}
\put(74.957,27.217){\line(1,0){.9861}}
\put(21.93,26.68){\line(1,0){.9864}}
\put(23.902,26.634){\line(1,0){.9864}}
\put(25.875,26.589){\line(1,0){.9864}}
\put(27.848,26.543){\line(1,0){.9864}}
\put(29.821,26.498){\line(1,0){.9864}}
\put(31.793,26.452){\line(1,0){.9864}}
\put(33.766,26.407){\line(1,0){.9864}}
\put(35.739,26.362){\line(1,0){.9864}}
\put(37.712,26.316){\line(1,0){.9864}}
\put(39.684,26.271){\line(1,0){.9864}}
\put(41.657,26.225){\line(1,0){.9864}}
\put(43.63,26.18){\line(1,0){.9864}}
\put(45.602,26.134){\line(1,0){.9864}}
\put(47.575,26.089){\line(1,0){.9864}}
\put(49.548,26.043){\line(1,0){.9864}}
\put(51.521,25.998){\line(1,0){.9864}}
\put(53.493,25.952){\line(1,0){.9864}}
\put(55.466,25.907){\line(1,0){.9864}}
\put(57.439,25.862){\line(1,0){.9864}}
\put(59.412,25.816){\line(1,0){.9864}}
\put(61.384,25.771){\line(1,0){.9864}}
\put(63.357,25.725){\line(1,0){.9864}}
\put(65.33,25.68){\line(1,0){.9864}}
\put(67.302,25.634){\line(1,0){.9864}}
\put(69.275,25.589){\line(1,0){.9864}}
\put(71.248,25.543){\line(1,0){.9864}}
\put(73.221,25.498){\line(1,0){.9864}}
\put(75.193,25.452){\line(1,0){.9864}}
\put(17.68,24.43){\line(1,0){.9833}}
\put(19.646,24.413){\line(1,0){.9833}}
\put(21.613,24.396){\line(1,0){.9833}}
\put(23.58,24.38){\line(1,0){.9833}}
\put(25.546,24.363){\line(1,0){.9833}}
\put(27.513,24.346){\line(1,0){.9833}}
\put(29.48,24.33){\line(1,0){.9833}}
\put(31.446,24.313){\line(1,0){.9833}}
\put(33.413,24.296){\line(1,0){.9833}}
\put(35.38,24.28){\line(1,0){.9833}}
\put(37.346,24.263){\line(1,0){.9833}}
\put(39.313,24.246){\line(1,0){.9833}}
\put(41.28,24.23){\line(1,0){.9833}}
\put(43.246,24.213){\line(1,0){.9833}}
\put(45.213,24.196){\line(1,0){.9833}}
\put(47.18,24.18){\line(1,0){.9833}}
\put(49.146,24.163){\line(1,0){.9833}}
\put(51.113,24.146){\line(1,0){.9833}}
\put(53.08,24.13){\line(1,0){.9833}}
\put(55.046,24.113){\line(1,0){.9833}}
\put(57.013,24.096){\line(1,0){.9833}}
\put(58.98,24.08){\line(1,0){.9833}}
\put(60.946,24.063){\line(1,0){.9833}}
\put(62.913,24.046){\line(1,0){.9833}}
\put(64.88,24.03){\line(1,0){.9833}}
\put(66.846,24.013){\line(1,0){.9833}}
\put(68.813,23.996){\line(1,0){.9833}}
\put(70.78,23.98){\line(1,0){.9833}}
\put(72.746,23.963){\line(1,0){.9833}}
\put(74.713,23.946){\line(1,0){.9833}}
\put(14.68,22.18){\line(1,0){.9919}}
\put(16.664,22.188){\line(1,0){.9919}}
\put(18.647,22.196){\line(1,0){.9919}}
\put(20.631,22.204){\line(1,0){.9919}}
\put(22.615,22.212){\line(1,0){.9919}}
\put(24.599,22.22){\line(1,0){.9919}}
\put(26.583,22.228){\line(1,0){.9919}}
\put(28.567,22.236){\line(1,0){.9919}}
\put(30.551,22.244){\line(1,0){.9919}}
\put(32.535,22.252){\line(1,0){.9919}}
\put(34.518,22.26){\line(1,0){.9919}}
\put(36.502,22.268){\line(1,0){.9919}}
\put(38.486,22.276){\line(1,0){.9919}}
\put(40.47,22.285){\line(1,0){.9919}}
\put(42.454,22.293){\line(1,0){.9919}}
\put(44.438,22.301){\line(1,0){.9919}}
\put(46.422,22.309){\line(1,0){.9919}}
\put(48.406,22.317){\line(1,0){.9919}}
\put(50.389,22.325){\line(1,0){.9919}}
\put(52.373,22.333){\line(1,0){.9919}}
\put(54.357,22.341){\line(1,0){.9919}}
\put(56.341,22.349){\line(1,0){.9919}}
\put(58.325,22.357){\line(1,0){.9919}}
\put(60.309,22.365){\line(1,0){.9919}}
\put(62.293,22.373){\line(1,0){.9919}}
\put(64.276,22.381){\line(1,0){.9919}}
\put(66.26,22.389){\line(1,0){.9919}}
\put(68.244,22.397){\line(1,0){.9919}}
\put(70.228,22.406){\line(1,0){.9919}}
\put(72.212,22.414){\line(1,0){.9919}}
\put(74.196,22.422){\line(1,0){.9919}}
\put(16.93,20.18){\line(1,0){.9917}}
\put(18.913,20.171){\line(1,0){.9917}}
\put(20.896,20.163){\line(1,0){.9917}}
\put(22.88,20.155){\line(1,0){.9917}}
\put(24.863,20.146){\line(1,0){.9917}}
\put(26.846,20.138){\line(1,0){.9917}}
\put(28.83,20.13){\line(1,0){.9917}}
\put(30.813,20.121){\line(1,0){.9917}}
\put(32.796,20.113){\line(1,0){.9917}}
\put(34.78,20.105){\line(1,0){.9917}}
\put(36.763,20.096){\line(1,0){.9917}}
\put(38.746,20.088){\line(1,0){.9917}}
\put(40.73,20.08){\line(1,0){.9917}}
\put(42.713,20.071){\line(1,0){.9917}}
\put(44.696,20.063){\line(1,0){.9917}}
\put(46.68,20.055){\line(1,0){.9917}}
\put(48.663,20.046){\line(1,0){.9917}}
\put(50.646,20.038){\line(1,0){.9917}}
\put(52.63,20.03){\line(1,0){.9917}}
\put(54.613,20.021){\line(1,0){.9917}}
\put(56.596,20.013){\line(1,0){.9917}}
\put(58.58,20.005){\line(1,0){.9917}}
\put(60.563,19.996){\line(1,0){.9917}}
\put(62.546,19.988){\line(1,0){.9917}}
\put(64.53,19.98){\line(1,0){.9917}}
\put(66.513,19.971){\line(1,0){.9917}}
\put(68.496,19.963){\line(1,0){.9917}}
\put(70.48,19.955){\line(1,0){.9917}}
\put(72.463,19.946){\line(1,0){.9917}}
\put(74.446,19.938){\line(1,0){.9917}}
\put(20.93,18.43){\line(1,0){.9911}}
\put(22.912,18.421){\line(1,0){.9911}}
\put(24.894,18.412){\line(1,0){.9911}}
\put(26.876,18.403){\line(1,0){.9911}}
\put(28.858,18.394){\line(1,0){.9911}}
\put(30.84,18.385){\line(1,0){.9911}}
\put(32.823,18.376){\line(1,0){.9911}}
\put(34.805,18.367){\line(1,0){.9911}}
\put(36.787,18.358){\line(1,0){.9911}}
\put(38.769,18.349){\line(1,0){.9911}}
\put(40.751,18.34){\line(1,0){.9911}}
\put(42.733,18.331){\line(1,0){.9911}}
\put(44.715,18.323){\line(1,0){.9911}}
\put(46.698,18.314){\line(1,0){.9911}}
\put(48.68,18.305){\line(1,0){.9911}}
\put(50.662,18.296){\line(1,0){.9911}}
\put(52.644,18.287){\line(1,0){.9911}}
\put(54.626,18.278){\line(1,0){.9911}}
\put(56.608,18.269){\line(1,0){.9911}}
\put(58.59,18.26){\line(1,0){.9911}}
\put(60.573,18.251){\line(1,0){.9911}}
\put(62.555,18.242){\line(1,0){.9911}}
\put(64.537,18.233){\line(1,0){.9911}}
\put(66.519,18.224){\line(1,0){.9911}}
\put(68.501,18.215){\line(1,0){.9911}}
\put(70.483,18.206){\line(1,0){.9911}}
\put(72.465,18.198){\line(1,0){.9911}}
\put(74.448,18.189){\line(1,0){.9911}}
\put(25.18,16.18){\line(1,0){.9856}}
\put(27.151,16.189){\line(1,0){.9856}}
\put(29.122,16.199){\line(1,0){.9856}}
\put(31.093,16.209){\line(1,0){.9856}}
\put(33.064,16.218){\line(1,0){.9856}}
\put(35.035,16.228){\line(1,0){.9856}}
\put(37.007,16.237){\line(1,0){.9856}}
\put(38.978,16.247){\line(1,0){.9856}}
\put(40.949,16.257){\line(1,0){.9856}}
\put(42.92,16.266){\line(1,0){.9856}}
\put(44.891,16.276){\line(1,0){.9856}}
\put(46.862,16.285){\line(1,0){.9856}}
\put(48.834,16.295){\line(1,0){.9856}}
\put(50.805,16.305){\line(1,0){.9856}}
\put(52.776,16.314){\line(1,0){.9856}}
\put(54.747,16.324){\line(1,0){.9856}}
\put(56.718,16.334){\line(1,0){.9856}}
\put(58.689,16.343){\line(1,0){.9856}}
\put(60.66,16.353){\line(1,0){.9856}}
\put(62.632,16.362){\line(1,0){.9856}}
\put(64.603,16.372){\line(1,0){.9856}}
\put(66.574,16.382){\line(1,0){.9856}}
\put(68.545,16.391){\line(1,0){.9856}}
\put(70.516,16.401){\line(1,0){.9856}}
\put(72.487,16.41){\line(1,0){.9856}}
\put(74.459,16.42){\line(1,0){.9856}}
\put(28.43,14.18){\line(1,0){.9948}}
\put(30.419,14.169){\line(1,0){.9948}}
\put(32.409,14.159){\line(1,0){.9948}}
\put(34.398,14.148){\line(1,0){.9948}}
\put(36.388,14.138){\line(1,0){.9948}}
\put(38.378,14.128){\line(1,0){.9948}}
\put(40.367,14.117){\line(1,0){.9948}}
\put(42.357,14.107){\line(1,0){.9948}}
\put(44.346,14.096){\line(1,0){.9948}}
\put(46.336,14.086){\line(1,0){.9948}}
\put(48.326,14.076){\line(1,0){.9948}}
\put(50.315,14.065){\line(1,0){.9948}}
\put(52.305,14.055){\line(1,0){.9948}}
\put(54.294,14.044){\line(1,0){.9948}}
\put(56.284,14.034){\line(1,0){.9948}}
\put(58.273,14.023){\line(1,0){.9948}}
\put(60.263,14.013){\line(1,0){.9948}}
\put(62.253,14.003){\line(1,0){.9948}}
\put(64.242,13.992){\line(1,0){.9948}}
\put(66.232,13.982){\line(1,0){.9948}}
\put(68.221,13.971){\line(1,0){.9948}}
\put(70.211,13.961){\line(1,0){.9948}}
\put(72.201,13.951){\line(1,0){.9948}}
\put(74.19,13.94){\line(1,0){.9948}}
\put(49.93,9.68){\line(1,0){.9907}}
\put(51.911,9.643){\line(1,0){.9907}}
\put(53.893,9.606){\line(1,0){.9907}}
\put(55.874,9.569){\line(1,0){.9907}}
\put(57.856,9.532){\line(1,0){.9907}}
\put(59.837,9.495){\line(1,0){.9907}}
\put(61.819,9.457){\line(1,0){.9907}}
\put(63.8,9.42){\line(1,0){.9907}}
\put(65.782,9.383){\line(1,0){.9907}}
\put(67.763,9.346){\line(1,0){.9907}}
\put(69.745,9.309){\line(1,0){.9907}}
\put(71.726,9.272){\line(1,0){.9907}}
\put(73.707,9.235){\line(1,0){.9907}}
\put(75.689,9.198){\line(1,0){.9907}}
\put(57.68,7.68){\line(1,0){.9868}}
\put(59.653,7.653){\line(1,0){.9868}}
\put(61.627,7.627){\line(1,0){.9868}}
\put(63.601,7.601){\line(1,0){.9868}}
\put(65.574,7.574){\line(1,0){.9868}}
\put(67.548,7.548){\line(1,0){.9868}}
\put(69.522,7.522){\line(1,0){.9868}}
\put(71.495,7.495){\line(1,0){.9868}}
\put(73.469,7.469){\line(1,0){.9868}}
\put(75.443,7.443){\line(1,0){.9868}}
\put(67.43,5.68){\line(1,0){.9}}
\put(69.23,5.78){\line(1,0){.9}}
\put(71.03,5.88){\line(1,0){.9}}
\put(72.83,5.98){\line(1,0){.9}}
\put(74.63,6.08){\line(1,0){.9}}
\put(73.25,23){$\cal C$}
\put(35.25,23.25){$\cal C'$}
\put(39,33.5){\line(1,0){37.75}}
\put(39.25,11.75){\line(1,0){37.75}}
\put(79.75,8.5){$a$}
\put(79.5,22.5){$h'$}
\put(79,38.75){$b$}
\put(74.25,2){$o_1$}
\put(35.5,9){$o_2$}
\put(35,36.25){$o'_2$}
\put(72.75,46.25){$o_3$}
\put(53.25,24.25){$\Gamma$}
\put(52.5,41){$z^{\Gamma}$}
\put(35.75,34.25){\circle*{1.118}}
\put(35.75,12.25){\circle*{1}}
\put(74,43.25){\circle*{1.5}}
\put(73.75,4.5){\circle*{1.118}}
\put(294,62.25){\circle*{.5}}
\end{picture}

Now we will prove that
\begin{equation}\label{numinus}
\nu^c_J(\Delta')<\nu^c_J(\Gamma)-(h')^2
\end{equation}
Recall that $\nu^c_J(\Gamma)=\nu_J({\bf z}^{\Gamma})$ by Lemma \ref{positive} (b). So we consider the two-colored string of beads
responsible for the $\nu_J$-mixture of ${\bf z}^{\Gamma}.$
 Denote by $o_1$ and $o_3$ the black beads on the
two ends of $\cal C$ and by $o_2, o_2'$ the black beads on the two
ends of $\cal C'.$ We have $h'$ white beads between $o_2$ and
$o'_2$, $a$ white beads between $o_1$ and $o_2$ and $b$ white beads
between $o'_2$ and $o_3$ for some $a,b\ge 0.$ Thus, $a+h'+b=h.$ When
we pass from ${\bf z}^{\Gamma}$ to to ${\bf z}^{\Delta'}$, we delete at least two
black beads $o_2,o'_2.$ But the number of black beads between the
vertices $o_1, o_3$ is less than $J$. Hence we may apply Lemma
\ref{mixturec}, parts (d,c),  and obtain that $\nu^c_J(\Delta')\le
\nu^c_J(\Gamma)-a(h'+b)-b(h'+a).$ But here $a(h'+b)+b(h'+a)>(h')^2$
since $a+b>h'.$ So the inequality \ref{numinus} is obtained.

Now by Lemmas \ref{comb} and \ref{simple},
\begin{equation}\label{arg'}
\area(\Gamma')\le
60N(h')^2+ 2(\delta')\iv[\Gamma']
\end{equation}
 Since  by Lemma \ref{mixturec}(a)
and (\ref{numinus}),
\begin{equation}\label{muc21}
\lambda^c(\Gamma')\ge
-\lambda^c({\bf y}^{\Gamma'})> - (h')^2/2 \ge (-\nu_J^c(\Gamma)+ \nu_J^c(\Delta'))/2
\end{equation}
we deduce  from (\ref{arg'}) and (\ref{muc21}) that
$$\area(\Gamma')\le (60N +c_1/2- c_1/2)(\nu_J^c(\Gamma)- \nu_J^c(\Delta'))+
2(\delta')\iv[\Gamma']$$ $$= (60N+\frac{c_1}{2}) (\nu_J^c(\Gamma)- \nu_J^c(\Delta')) -\frac{c_1}{2}(\nu_J^c(\Gamma)- \nu_J^c(\Delta'))+
2(\delta')\iv[\Gamma']\le $$ $$(60N+\frac{c_1}{2})(\nu_J^c(\Gamma)- \nu_J^c(\Delta'))  +c_1\lambda^c(\Gamma')
+2(\delta')\iv[\Gamma']
 \le c_1(\nu_J^c(\Gamma)- \nu_J^c(\Delta') +\lambda^c(\Gamma')+[\Gamma']),$$
because $c_1\ge
61N+c_1/2$  and  $c_1\ge 2(\delta')^{-1},$ and the lemma is proved.

\endproof

 \begin{lemma} \label{123} Let $\Delta$ be a
 comb with
history $H^{\Delta}$  of type $(1)(12)(2)(23)(3)$, where the (1)-part and the (3)-part of $\Delta$ can be empty. Assume that the base
width $b$ of $\Delta$ satisfies inequalities $4N<b\le 15N.$ Then either
(a)   $\Gamma$ admits a long quasicomb with
\begin{equation}\label{b}
\area(\Gamma)\le
c_3[\Gamma]+c_2\mu^c(\Gamma)+c_3(\nu_J^c(\Delta)-
\nu_J^c(\Delta\backslash \Gamma)) \;\;\;\; or,
\end{equation}

(b) $\Delta$ has a maximal $t^{\pm 1}$ or $(t')^{\pm 1}$-band of
length $l$, where $T_i\le l<10 T_i$ for some $i$.
\end{lemma}

\proof $\Delta$ has a regular subcomb $\Delta_1$
of base widths $b_1>3N$
such that the base widths of the trapezium ${\bf T}=Tp({\cal D}_1,
\cal D)$
is at least $N+1$, where ${\cal D}_1$
and ${\cal D}$ are the handles of $\Delta_1$
and $\Delta$, respectively.  If the history of $\Delta_1$ has one
Step, then the Property (a) of the lemma is a consequence of Lemma
\ref{oneage} since in this case $\lambda^c(\Delta)\ge -\lambda({\bf y}^{\Delta})=0,$
and $\nu_J^c(\Delta)-
\nu_J^c(\Delta\backslash \Gamma)\ge 0$ by Lemmas \ref{positive} and \ref{mu} (d). It is a consequence of Lemma \ref{bez12} if the
history of $\Delta_1$ has no $(12)^{\pm 1}$- or $(23)^{\pm 1}$-rules.  Hence we may assume that the
history of $\Delta_1$ is of type $(1)(12)(2)(23)(3)$ as well.

 Since the history of ${\bf T}$ has
both  rules $(12)^{\pm 1}$ and $(23)^{\pm 1}$, the base of ${\bf T}$
is normal by \ref{vii}, and since the base of ${\bf T}$ has at least $N+1$ letters,
$\bf T$ contains a standard subtrapezium, and so the subtrapezium of ${\bf T}$ bounded
by the $(12)$- and $(23)$-bands has height $T_i$ for some $i$.

Since $\bf T$ has a normal base of length $\ge N+1,$ its base must contain a letter $t^{\pm 1}$ and a
letter $(t')^{\pm 1}$.  Denote by $\cal C$ ($\cal C'$) a maximal
$t^{\pm 1}$-band ($(t')^{\pm 1}$-band) of $\Delta$ crossing ${\bf T}.$
We may assume that neither of them corresponds to
the first letter of the base of $\bf T$
since otherwise this  normal base of length $N+1$ has one more $t^{\pm 1}$ or $(t')^{\pm 1}$-letter, respectively, and one can
select one of the bands $\cal C, C'$ closer to $\cal D.$

By $\Gamma$ and $\Gamma'$, we denote the subcombs with handles $\cal
C$ and $\cal C'$, respectively. Let the histories of these handles
be $H$ and $H'$. Without loss of the generality of our further
proof, we assume that $\Gamma'$ is contained in $\Gamma$.
Since $0\le \nu^c_J(\Gamma)-\nu^c_J(\Gamma\backslash\Gamma')\le
\nu^c_J(\Delta)-\nu^c_J(\Delta\backslash\Gamma')$ by Lemmas \ref{positive} (b) and \ref{mu}
(d,e), we may assume by  Lemma \ref{nuK}, that $h'=h^{\Gamma'}\ge
h/2$.

Let $H\equiv H(1)H(2)H(3)$ and $H'\equiv H'(1)H'(2)H'(3)$ be the Step
factorizations.
Since the left-most $q$-band of ${\bf T}$ is not a
subband of $\cal C'$ or $\cal C,$
the maximal
$(12)$- and $(23)$-bands of $\Gamma'$ and $\Gamma$ are not simple, and so the factorizations
$H\equiv H(1)H(2)H(3)$ and $H'\equiv H'(1)H'(2)H'(3)$ are firm.
Recall  that by \ref{kk'}, every $k$-cell of the $H(1)$- and $H(2)$-parts of $\Gamma$ (every $k'$-cell of the $H(2)$- and $H(3)$-parts of $\Gamma$) is not passive unless it belongs to a $(12)$- or $(23)$-band.

One may assume
that $10h'(2)\le h'$ because $h'(2)=T_i$ and if $10h'(2) >  h'$  then the length of $\cal C'$ belongs
to the segment $[T_i, 10T_i)$,
and we obtain Property (b).
Similarly we may assume that $10h(2)\le h.$

If $h(1)\ge 0.3 h$, then $h(1)\ge 3h(2)$, and one can apply Lemma
\ref{two01} to $\Gamma^{-1}$ and obtain the desired estimate (\ref{b})
for $\area(\Gamma)$. If $h(1)<0.3 h$, then $h'(1)\le h(1)<0.6 h'$
since $h'\ge h/2$. It follows that $h'(3)=h'-h'(1)-h'(2) \ge
(1-0.6-0.1)h'\ge 0.3h'\ge 3h'(2).$ Now one can apply Lemma \ref{two01}
to $\Gamma'$ and obtain the required estimate  (\ref{b}) for
$\area(\Gamma').$
\endproof

\begin{lemma}\label{podslovo}
Let $\Gamma$ be a regular comb of width $b\le 15N$ with a handle
$\cal C$ containing both $(12)$- and $(23)$-cells, and the history
$H$ of $\Gamma$ contains, in its Step factorization,  a product
$H(1)H(2)H(3)$, where $h(2)\ge h/30$ and $h(1)+h(3)< h(2)/2$. Let
one of the derivative bands ${\cal C}_i$ be a $k^{\pm 1}$- or
$(k')^{\pm 1}$-band which crosses all the maximal $\theta$-bands of
the $H(2)$-part of $\Gamma.$
Assume also that either

(a) $\cal C$ is a $t^{\pm 1}$-band, and $H(1)H(2)H(3)$ is of the form
$(2)(1)(2)$ or

(b) $\cal C$ is a $(t')^{\pm 1}$-band, and $H(1)H(2)H(3)$ is of the form
$(2)(3)(2).$

Then $\area(\Gamma)\le c_3[\Gamma]+c_2\mu^c(\Gamma).$
\end{lemma}

\proof It follows from \ref{viii} that $H$ has no subwords of type
$(1)(2)(1)$ or $(3)(2)(3)$ because $\Gamma$ is regular and has both
rules $(12)^{\pm 1}$ and $(23)^{\pm 1}$ in its history. An $a$-band starting on the
$H(2)$-part ${\cal C}_i(2)$ of ${\cal C}_i$ cannot cross a $(23)$-band in case (a) or
$(12)$-band in case (b) by \ref{i}.
Also it cannot end on ${\cal C}_i(2)$ by \ref{kk'} and \ref{iv}.
Hence every maximal $a$-band
starting on ${\cal C}_i(2)$ must end
either on the parts ${\cal C}_i(1),$ ${\cal C}_i(3)$ or on the path
${\bf z}={\bf z}^{\Gamma}$.
Now inequalities $h(1)+h(3) < h(2)/2$ and $h(2)\ge h/30\ge \delta h$
make possible to apply Lemma \ref{h0} to $\Gamma.$ Hence
\begin{equation}\label{chast1}
\area(\Gamma)\le (\delta')^{-2}[\Gamma]
\end{equation}

Since more than $\frac{1}{2} h(2)-2\ge \frac{1}{60}h-2$ maximal $a$-bands end on $\bf z$, we
have   $|{\bf z}|-|{\bf y}|>\delta'(\frac{1}{60}h-2)+1> \frac{\delta'}{60}h$ by Lemma \ref{simple} (a), i.e., $h<60(\delta')^{-1}(|{\bf z}|-|{\bf y}|).$ Hence by Lemma \ref{mixturec} (a),
$$\lambda^c(\Gamma) \ge -\lambda^c({\bf y})\ge -h^2/2 \ge-30(\delta')^{-1}h(|{\bf z}|-|{\bf y}|)
=-30(\delta')^{-1}[\Gamma],$$
This inequality and (\ref{chast1}) complete the proof of the lemma
since $$(\delta')^{-2}[\Gamma]=((\delta')^{-2}+30(\delta')\iv c_2)[\Gamma]-30(\delta')^{-1}c_2[\Gamma]
\le c_3[\Gamma]+c_2\lambda^c(\Gamma)\le c_3[\Gamma]+c_2\mu^c(\Gamma)$$
by the choice of $c_3$, the definition of $\mu^c(\Gamma)$, and by Lemma \ref{positive} (a).
\endproof

\begin{lemma} \label{5parts}
Let $\Gamma$ be a regular comb of width $b\le 15N$ whose
handle $\cal$ is a $t^{\pm 1}$-band (or $(t')^{\pm 1}$-band) with
firm factorization of the history $H\equiv H(1)\dots H(5),$ where
$H(2)$ and $H(4)$ are both of type $(1)$ (or both of type $(3)$,
respectively), and $h(2)+h(4)\ge 0.9h$. Assume that $H(3)$
contains  both   $(12)^{\pm 1}$ and $(23)^{\pm 1},$
and one of the derivative
bands ${\cal C}_i$ crosses all the maximal $(12)$- and $(23)$-bands of
$\Gamma.$
 Then $\Gamma$ has a long sumcomb $\Delta$ with $\area(\Delta)\le c_3 [\Delta]+c_2\mu^c(\Delta).$
\end{lemma}

\proof We will prove only the first version of the lemma. The
history $H_i$ of ${\cal C}_i$ contains $H(3)$, and so it
contains the rule $(23)^{\pm 1}$, and therefore the band ${\cal C}_i$ cannot
be a $t^{\pm 1}$-band; it is $k^{-1}$-band by \ref{order}. It follows from the assumption of
the lemma that the derivative ${\cal C}_i$ must cross $H(2)$-,
$H(3)$- and $H(4)$-parts of the comb $\Gamma$, and therefore
$h_i>0.9h.$ Moreover the sum of lengths of $H(2)$- and
$H(4)$-parts of ${\cal C}_i$ is at least $0.9h,$ and so $\max (h(2), h(4))\ge 0.4h.$

There are at most $0.1h$ maximal $a$-bands
starting on ${\cal C}_i$ and ending on other derivative bands.
There
are no $a$-bands starting on the $H(2)$-part and ending on the
$H(4)$-part of ${\cal C}_i$ by the condition on $H(3)$, because an
$a$-band cannot cross both $(12)$- and $(23)$-bands by \ref{i}. 
Besides both the $H(2)$- and the $H(4)$-part of ${\cal C}_i$ are active from the right
by \ref{kk'}.
Therefore we can
apply Lemma \ref{h0} to $\Gamma$:
\begin{equation}\label{onlyarea}
\area(\Gamma)\le (\delta')^{-2}[\Gamma]
\end{equation}
Note that at least
$0.9h-4-0.1h =0.8h-4$ maximal $a$-bands end on $\bf z,$ and so $|{\bf z}|-|{\bf y}|-2=|{\bf z'}|-|{\bf y'}|\ge
(0.8h-4)\delta'$ by Lemma \ref{simple} (1), whence 
\begin{equation}\label{07}
h\le \frac{5}{4}(\delta')\iv (|{\bf z}|-|{\bf y}|)
\end{equation}
 Hence by Lemma \ref{mixturec} (a) and Inequality (\ref{07}), we get
  $$\lambda^c(\Gamma)\ge
-\lambda({\bf y})> -h^2/2 > -h(\delta')^{-1}(|{\bf z}|-|{\bf y}|) =
-(\delta')\iv[\Gamma]$$ This inequality and
(\ref{onlyarea}) complete the proof as in Lemma \ref{podslovo} because
$$(\delta')^{-2}[\Gamma]=((\delta')^{-2}+(\delta')\iv c_2)[\Gamma]-(\delta')^{-1}c_2[\Gamma]
\le c_3[\Gamma]+c_2\lambda^c(\Gamma)\le c_3[\Gamma]+c_2\mu^c(\Gamma)$$

 \endproof

\begin{lemma} \label{grebenki} Let $\Delta$ be a regular comb. Assume that
history $H^{\Delta}$ has $m\le 6$ letters $(12)^{\pm 1}$ and $(23)^{\pm 1},$
and  $\max(4, 2m)N<b\le
15N$ for the base width $b$ of $\Delta.$ Then either $\Delta$ admits a
long quasicomb $\Delta'$ with
$$\area(\Delta')\le c_3[\Delta']+c_2\mu^c(\Delta')+
c_3(\nu^c_J(\Delta)-\nu^c_J(\Delta\backslash\Delta'))$$ or $\Delta$
has a maximal $t^{\pm 1}$ or $(t')^{\pm 1}$-band of length $l$, where
$T_i\le l< 200 T_i$ for some $T_i$.
\end{lemma}

\proof If $m\le 2$ or the handle of $\Delta$ does not contain either
$(12)$-cells or $(23)$-cells, then the statement follows from Lemmas
\ref{oneage}, \ref{bez12}, and \ref{123}  because $\lambda^c$ is
non-negative for one Step (quasi)combs, the third summand in the
above inequality is positive by Lemmas \ref{positive} (a) and
\ref{mu} (d), and $c_3>c_2>c_1.$ Then we will induct 
on $m$ assuming
that $3\le m\le 6$ and that the history $H^{\Delta}$ of $\Delta$
contains both  rules $(12)^{\pm 1}$ and $(23)^{\pm 1}.$

The comb $\Delta$ has a regular subcomb $\Delta_1$ of base width $b_1>
N(2m-1)$ such that the base width of the filling trapezium ${\bf
T}=Tp({\cal D}_1, \cal D)$ is $N+1$, where ${\cal D}_1$ and ${\cal
D}$ are the handles of $\Delta_1$ and $\Delta$, respectively. If the
history of $\Delta_1$ has $m_1$ letters $(12)^{\pm 1}$ and $(23)^{\pm 1},$ and $m_1<m,$  then the statement
of the lemma is a consequence of the inductive hypothesis since
$2m_1\le 2m-1$. Hence we may assume that $m_1=m$. Similarly,
$\Delta_1$ has a regular subcomb $\Delta_2$ of width $b_2>(2m-2)N$
with handle ${\cal D}_2$ and the filling trapezium ${\bf
T'}=Tp({\cal D}_2, {\cal D}_1)$ of width $N+1$, and we may assume
that $m_2=m_1=m$ since otherwise $2m_2\le 2m-2$ and one may apply
the inductive conjecture to $\Delta_2.$

\bigskip

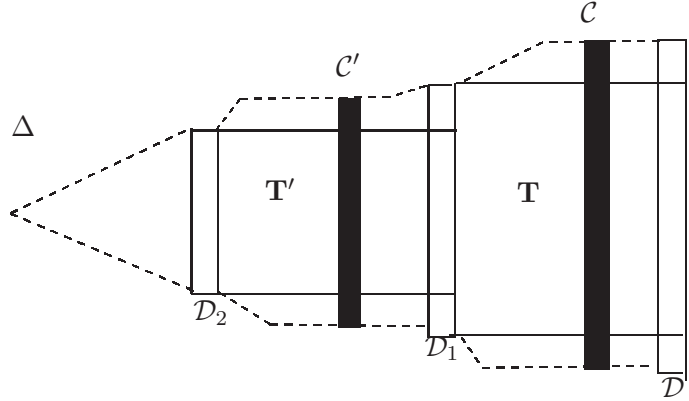
\unitlength 1mm 
\linethickness{0.4pt}
\ifx\plotpoint\undefined\newsavebox{\plotpoint}\fi 
\begin{picture}(20.25,55.75)(15,0)
\put(125.25,53.25){\line(0,-1){44}}
\put(125.25,9.25){\line(0,1){0}}
\put(125.25,9.25){\line(0,-1){1.25}}
\put(125.25,53.25){\line(-1,0){3.75}}
\put(121.5,53.25){\line(0,-1){44.25}}
\put(121.5,9){\line(1,0){3.25}}
\put(94.75,47.5){\line(1,0){30.5}}
\put(94.5,47.5){\line(0,-1){33.75}}
\put(94.5,13.75){\line(-1,0){3.5}}
\put(91,13.75){\line(0,1){33.5}}
\put(91,47.25){\line(1,0){3}}
\put(93.75,14){\line(1,0){31}}
\put(59.5,41.25){\line(1,0){35.25}}
\put(59.5,41.5){\line(0,-1){22}}
\put(59.5,19.5){\line(1,0){34.75}}
\put(63,41.25){\line(0,-1){21.5}}
\put(121.68,53.18){\line(-1,0){.9688}}
\put(119.742,53.148){\line(-1,0){.9688}}
\put(117.805,53.117){\line(-1,0){.9688}}
\put(115.867,53.086){\line(-1,0){.9688}}
\put(113.93,53.055){\line(-1,0){.9688}}
\put(111.992,53.023){\line(-1,0){.9688}}
\put(110.055,52.992){\line(-1,0){.9688}}
\put(108.117,52.961){\line(-1,0){.9688}}
\multiput(106.18,52.93)(-.066964,-.032738){12}{\line(-1,0){.066964}}
\multiput(104.573,52.144)(-.066964,-.032738){12}{\line(-1,0){.066964}}
\multiput(102.965,51.358)(-.066964,-.032738){12}{\line(-1,0){.066964}}
\multiput(101.358,50.573)(-.066964,-.032738){12}{\line(-1,0){.066964}}
\multiput(99.751,49.787)(-.066964,-.032738){12}{\line(-1,0){.066964}}
\multiput(98.144,49.001)(-.066964,-.032738){12}{\line(-1,0){.066964}}
\multiput(96.537,48.215)(-.066964,-.032738){12}{\line(-1,0){.066964}}
\multiput(90.68,47.18)(-.098958,-.03125){8}{\line(-1,0){.098958}}
\multiput(89.096,46.68)(-.098958,-.03125){8}{\line(-1,0){.098958}}
\multiput(87.513,46.18)(-.098958,-.03125){8}{\line(-1,0){.098958}}
\put(85.93,45.68){\line(-1,0){.9524}}
\put(84.025,45.656){\line(-1,0){.9524}}
\put(82.12,45.632){\line(-1,0){.9524}}
\put(80.215,45.608){\line(-1,0){.9524}}
\put(78.311,45.584){\line(-1,0){.9524}}
\put(76.406,45.561){\line(-1,0){.9524}}
\put(74.501,45.537){\line(-1,0){.9524}}
\put(72.596,45.513){\line(-1,0){.9524}}
\put(70.692,45.489){\line(-1,0){.9524}}
\put(68.787,45.465){\line(-1,0){.9524}}
\put(66.882,45.442){\line(-1,0){.9524}}
\multiput(65.93,45.43)(-.0331633,-.0433673){14}{\line(0,-1){.0433673}}
\multiput(65.001,44.215)(-.0331633,-.0433673){14}{\line(0,-1){.0433673}}
\multiput(64.073,43.001)(-.0331633,-.0433673){14}{\line(0,-1){.0433673}}
\multiput(63.144,41.787)(-.0331633,-.0433673){14}{\line(0,-1){.0433673}}
\multiput(58.68,41.18)(-.0662393,-.031339){13}{\line(-1,0){.0662393}}
\multiput(56.957,40.365)(-.0662393,-.031339){13}{\line(-1,0){.0662393}}
\multiput(55.235,39.55)(-.0662393,-.031339){13}{\line(-1,0){.0662393}}
\multiput(53.513,38.735)(-.0662393,-.031339){13}{\line(-1,0){.0662393}}
\multiput(51.791,37.92)(-.0662393,-.031339){13}{\line(-1,0){.0662393}}
\multiput(50.069,37.106)(-.0662393,-.031339){13}{\line(-1,0){.0662393}}
\multiput(48.346,36.291)(-.0662393,-.031339){13}{\line(-1,0){.0662393}}
\multiput(46.624,35.476)(-.0662393,-.031339){13}{\line(-1,0){.0662393}}
\multiput(44.902,34.661)(-.0662393,-.031339){13}{\line(-1,0){.0662393}}
\multiput(43.18,33.846)(-.0662393,-.031339){13}{\line(-1,0){.0662393}}
\multiput(41.457,33.032)(-.0662393,-.031339){13}{\line(-1,0){.0662393}}
\multiput(39.735,32.217)(-.0662393,-.031339){13}{\line(-1,0){.0662393}}
\multiput(38.013,31.402)(-.0662393,-.031339){13}{\line(-1,0){.0662393}}
\multiput(36.291,30.587)(-.0662393,-.031339){13}{\line(-1,0){.0662393}}
\multiput(35.43,30.18)(.074846,-.031636){12}{\line(1,0){.074846}}
\multiput(37.226,29.42)(.074846,-.031636){12}{\line(1,0){.074846}}
\multiput(39.022,28.661)(.074846,-.031636){12}{\line(1,0){.074846}}
\multiput(40.819,27.902)(.074846,-.031636){12}{\line(1,0){.074846}}
\multiput(42.615,27.143)(.074846,-.031636){12}{\line(1,0){.074846}}
\multiput(44.411,26.383)(.074846,-.031636){12}{\line(1,0){.074846}}
\multiput(46.207,25.624)(.074846,-.031636){12}{\line(1,0){.074846}}
\multiput(48.004,24.865)(.074846,-.031636){12}{\line(1,0){.074846}}
\multiput(49.8,24.106)(.074846,-.031636){12}{\line(1,0){.074846}}
\multiput(51.596,23.346)(.074846,-.031636){12}{\line(1,0){.074846}}
\multiput(53.393,22.587)(.074846,-.031636){12}{\line(1,0){.074846}}
\multiput(55.189,21.828)(.074846,-.031636){12}{\line(1,0){.074846}}
\multiput(56.985,21.069)(.074846,-.031636){12}{\line(1,0){.074846}}
\multiput(58.782,20.309)(.074846,-.031636){12}{\line(1,0){.074846}}
\multiput(59.68,19.93)(-.03125,-.03125){4}{\line(0,-1){.03125}}
\put(59.43,19.68){\line(0,1){0}}
\multiput(63.18,19.68)(.0535714,-.0337302){14}{\line(1,0){.0535714}}
\multiput(64.68,18.735)(.0535714,-.0337302){14}{\line(1,0){.0535714}}
\multiput(66.18,17.791)(.0535714,-.0337302){14}{\line(1,0){.0535714}}
\multiput(67.68,16.846)(.0535714,-.0337302){14}{\line(1,0){.0535714}}
\multiput(69.18,15.902)(.0535714,-.0337302){14}{\line(1,0){.0535714}}
\put(69.93,15.43){\line(1,0){.9762}}
\put(71.882,15.406){\line(1,0){.9762}}
\put(73.834,15.382){\line(1,0){.9762}}
\put(75.787,15.358){\line(1,0){.9762}}
\put(77.739,15.334){\line(1,0){.9762}}
\put(79.692,15.311){\line(1,0){.9762}}
\put(81.644,15.287){\line(1,0){.9762}}
\put(83.596,15.263){\line(1,0){.9762}}
\put(85.549,15.239){\line(1,0){.9762}}
\put(87.501,15.215){\line(1,0){.9762}}
\put(89.454,15.192){\line(1,0){.9762}}
\multiput(95.43,13.68)(.0327381,-.047619){14}{\line(0,-1){.047619}}
\multiput(96.346,12.346)(.0327381,-.047619){14}{\line(0,-1){.047619}}
\multiput(97.263,11.013)(.0327381,-.047619){14}{\line(0,-1){.047619}}
\put(98.18,9.68){\line(1,0){.9583}}
\put(100.096,9.701){\line(1,0){.9583}}
\put(102.013,9.721){\line(1,0){.9583}}
\put(103.93,9.742){\line(1,0){.9583}}
\put(105.846,9.763){\line(1,0){.9583}}
\put(107.763,9.784){\line(1,0){.9583}}
\put(109.68,9.805){\line(1,0){.9583}}
\put(111.596,9.826){\line(1,0){.9583}}
\put(113.513,9.846){\line(1,0){.9583}}
\put(115.43,9.867){\line(1,0){.9583}}
\put(117.346,9.888){\line(1,0){.9583}}
\put(119.263,9.909){\line(1,0){.9583}}
\put(111.75,9.5){\rule{3.25\unitlength}{43.75\unitlength}}
\put(79,15){\rule{3\unitlength}{30.5\unitlength}}
\put(102.75,31.75){$\bf T$}
\put(69.25,32){$\bf T'$}
\put(59.75,16){${\cal D}_2$}
\put(90.5,11.25){${\cal D}_1$}
\put(122,5.75){$\cal D$}
\put(78.75,48.5){$\cal C'$}
\put(111.25,55.75){$\cal C$}
\put(35.5,40.5){$\Delta$}
\end{picture}

Thus both ${\cal D}_1$ and ${\cal D}_2$ have $(12)$- and
$(23)$-cells. Hence the base of ${\bf T}$ is normal by \ref{vii},
and so it contains a letter $t^{\pm 1}$ and a letter $(t')^{\pm 1}$
which are not the first letter in this base . The same is true for
${\bf T'}.$ Denote by $\cal C$ ($\cal C'$) a maximal $t^{\pm
1}$-band or $(t')^{\pm 1}$-band of $\Delta$ crossing $\bf T$
(crossing $\bf T'$) and corresponding to this letter of the base. By
$\Gamma$ and $\Gamma'$, we denote the subcombs with handles $\cal C$
and $\cal C'$, respectively. The histories of these handles are $H$
and $H'$. We will assume that $h'\ge h/2$ for their length, because
otherwise one can apply Lemma \ref{nuK} to $\Gamma$ since
$c_1<c_2<c_3.$

Observe that there are derivative bands in both $\Gamma$ and
$\Gamma'$ crossing all the $\theta$-bands of $\Gamma$ corresponding
to the rules $(12)^{\pm 1}$ and $(23)^{\pm 1}.$ This follows from
the equality of the numbers of $(12)$- and $(23)$-cells in ${\cal D}$ and ${\cal D}_2.$
Hence such a derivative band is a $k^{- 1}$ or $k'$-band by \ref{order} and \ref{i},
and there exist firm  factorizations $H=H^{(1)}\dots H^{(m+1)}$ and
$H'\equiv (H')^{(1)}\dots (H')^{(m+1)}$ for $\Gamma$ and $\Gamma'$, where
$(H')^{i}\equiv H^i$ for $i=2,\dots, m-1.$

Besides one may assume that all other
derivative bands of $\Gamma$ and $\Gamma'$ (if any) are also either $k^{-1}$- or
$k'$-bands. Indeed, they do not cross $(12)$- and $(23)$-bands and so
if a derivative band is $t^{\pm}$- or $(t')\iv$-band, one can apply Lemma \ref{t} to a
derivative diagram $\Delta'$, and the statement of our lemma follows. (Similarly, the comb $\Lambda$ from Case 4
below, also enjoys this property of $\Gamma$ and $\Gamma'$ by the same reason.)

By Property \ref{viii} the step history of $\Delta$ is a subword of
$(2)(1)(2)(3)(2)(1)(2)$. Since $m\ge 3$ and one always can replace
$\Delta$ by $\Delta^{-1}$ (and $H$ by
$H^{-1}$),
we have to consider the
following $6$ step histories: $(1)(2)(3)(2)$, $(3)(2)(1)(2)$,
$(1)(2)(3)(2)(1)$, $(2)(1)(2)(3)(2),$ $(1)(2)(1)(2)(3)(2),$ and
$(2)(1)(2)(1)(2)(3)(2).$

{\bf Case 1}.  The history $H$ is of type $(1)(2)(3)(2)$, and
$H^{(1)}H^{(2)}H^{(3)}H^{(4)}$ is the corresponding firm
factorization. In this case,
 we select $\cal C$ to be a $t^{\pm
1}$-band and $\cal C'$ a $(t')^{\pm 1}$-band.

If $h^{(4)}\ge 0.01h$, then one can apply Lemma \ref {two01} to
$\Gamma$ with $H(1)\equiv H^{(1)}H^{(2)}H^{(3)}, H(2)\equiv H^{(4)}$ and
$H(3)\equiv \emptyset,$ since the condition on the $(1)-(2)$-transition
holds by \ref{i}, and the passive cells of $H(2)$-part of $\cal C$ 
must be $(12)$- or $(23)$-cells
by \ref{kk'}. Hence we obtain a required subcomb.  Therefore we may
further assume that $h^{(4)}<0.01h$.

Since the base of ${\bf T}$ is normal and has $\ge N$ letters, this trapezium contains
a standard subtrapezium with history $H^{(2)}$, and so $h^{(2)}= T_i$
for some $i$. 
We may assume that $T_i\le  h/200$
because otherwise $h<200T_i$, and the  lemma is true. Thus
$h^{(2)}+h^{(4)}<h/60.$

Assume that $h^{(3)}\ge h/30$. Then $(h')^{(3)}=h^{(3)}\ge h'/30$ and
$(h')^{(2)}+(h')^{(4)}\le h^{(2)}+h^{(4)}<h/60\le h^{(3)}/2=(h')^{(3)}/2$.
Hence one can apply Lemma \ref{podslovo}(b) to $\Gamma'$. (Here
$H(1)\equiv (H')^{(2)}$, $H(2)\equiv (H')^{(3)}$ and $H(3)\equiv (H')^{(4)}$. ) Therefore we
can further assume that $h^{(3)}<h/30.$

Now, $h^{(1)}>h-h^{(2)}-h^{(3)}-h^{(4)} > h(1-1/60-1/30)= 0.95h$.
Therefore Lemma \ref{5parts} is applicable to $\Gamma$ with
$H(1)\equiv \emptyset$, $H(2)\equiv H^{(1)}$, $H(3)\equiv  H^{(2)}H^{(3)}H^{(4)}$,
$H(4)\equiv H(5)\equiv \emptyset.$ This completes Case 1.

{\bf Case 2}. The history  $H$ is of type $(3)(2)(1)(2)$, and
$H^{(1)}H^{(2)}H^{(3)}H^{(4)}$ is the corresponding firm
factorization. In this case we will assume that $\cal C$ is a
$(t')^{\pm 1}$-band and $\cal C'$ a $t^{\pm 1}$-band. Then the proof
coincides with that in Case 1.

{\bf Case 3}. The history $H$ is of type $(1)(2)(3)(2)(1)$, and
$H^{(1)}H^{(2)}H^{(3)}H^{(4)}H^{(5)}$ is the corresponding firm
factorization. In this case we will assume that $\cal C$ is a
$t^{\pm 1}$-band.

As in Case 1, one may assume that $\max(h^{(2)}, h^{(4)})\le h/200$.
Then $h^{(3)}<h/100$ by \ref{xii}. Therefore $h^{(1)}+h^{(5)}> h -
2h/100= 0.98h$. Therefore one can apply Lemma \ref{5parts} to
$\Gamma$ with $H(1)=\emptyset$, $H(2)\equiv H^{(1)}$, $H(3)\equiv H^{(2)}H^{(3)}H^{(4)},$
$H(4)\equiv H^{(5)},$
and $H(5)\equiv \emptyset$.

{\bf Case 4}. The history $H$ is of type $(2)(1)(2)(3)(2)$, and
$H^{(1)}H^{(2)}H^{(3)}H^{(4)}H^{(5)}$ is the corresponding firm
factorization. In this case we will assume that both $\cal C$ and
$\cal C'$  are $t^{\pm 1}$-bands and consider an auxiliary maximal
$(t')^{\pm 1}$-band $\cal B$ between them. It exists since the base of
$Tp(\cal C', \cal C)$ is normal, and determines a subcomb
$\Lambda$ of $\Delta$ whose history $G$ has a firm factorization
$G^{(1)}\dots G^{(5)}$.

If $h^{(5)}\ge 0.01h$, then one can apply Lemma \ref {two01} to
$\Gamma$ with $H(1)\equiv H^{(1)}H^{(2)}H^{(3)}H^{(4)},$ $ H(2)=H^{(5)}$ and
$H(3)=\emptyset.$ Hence we may assume that $h^{(5)}<0.01h$.

Since $\Delta$ is regular, $h^{(3)}= T_i$ for some $i$. We
may assume that $T_i\le  h/200$ because otherwise $h<200T_i$, as
desired. Thus $h^{(3)}+h^{(5)}<h/60.$

Assume that $h^{(4)}\ge h/30$. Then $g^{(4)}=h^{(4)}\ge g/30$ and
$g^{(3)}+g^{(5)}\le h^{(3)}+h^{(5)}<h/60\le h^{(4)}/2=g^{(4)}/2$.
Hence one can apply Lemma \ref{podslovo} to $\Lambda$. (Here
$H(1)\equiv G^{(3)}$, $H(2)\equiv G^{(4)}$ and $H(3)\equiv G^{(5)}$.) Therefore we can
further assume that $h^{(4)}<h/30.$

Suppose $h^{(1)}\ge 0.7h$. Then Lemma \ref{podslovo1}  can be
applied to $\Gamma^{-1}$ with  $H''\equiv (H^{(1)})^{-1}$. Hence we may assume
that $h^{(1)}<0.7h$, and therefore $h^{(2)}>h(1-0.7 -
1/30-1/60)=h/4$.

If $g^{(1)}\ge 0.01g$, then Lemma \ref{two01} is applicable to
$\Lambda^{-1}$ with $H(2)\equiv (G^{(1)})^{-1}$ and $H(3)\equiv \emptyset$. Therefore
we may assume that $g^{(1)}<0.01g.$

Now $(h')^{(2)}=h^{(2)}>h/4\ge h'/4$ and $(h')^{(1)}+(h')^{(3)}\le
g^{(1)}+h^{(3)}<0.01 g+ h/200\le 0.015 h <(h')^{(2)}/2$. Hence Lemma
\ref{podslovo} is applicable to $\Gamma'$, and the lemma is proved
in Case 4.

{\bf Case 5}. The history $H$ is of type $(1)(2)(3)(2)(1)(2)$, and
$H^{(1)}H^{(2)}H^{(3)}H^{(4)}H^{(5)}H^{(6)}$ is the corresponding
firm factorization. In this case we will assume that $\cal C$ is a
$(t')^{\pm 1}$-band and $\cal C'$ is a $t^{\pm 1}$-band.

If $h^{(6)}\ge 0.01h$, then one can apply Lemma \ref {two01} to
$\Gamma$ with $H(1)\equiv H^{(1)}H^{(2)}H^{(3)}H^{(4)}H^{(5)},$
$H(2)\equiv H^{(6)}$ and $H(3)\equiv \emptyset.$ Hence we may assume that
$h^{(6)}<0.01h$. Then, as in Case 3, we may assume that $\max
(h^{(2)},h^{(4)})\le h/200$ and $h^{(3)}<h/100$. Since $h'\ge h/2$, it
follows that
$$(h')^{(1)}+(h')^{(5)}=h'-(h')^{(2)}-(h')^{(3)}-(h')^{(4)}-(h')^{(6)}\ge $$ $$h'-h^{(2)}-h^{(3)}-h^{(4)}-h^{(6)}>
h'-0.03h\ge h'-0.06h'=0.94h'.$$
 Therefore one
can apply Lemma \ref{5parts} to $\Gamma'$ with $H(1)\equiv\emptyset$,
$H(2)\equiv (H')^{(1)}$, $H(3)\equiv (H')^{(2)}(H')^{(3)}(H')^{(4)}$, $H(4)\equiv (H')^{(5)}$
and $H(5)\equiv (H')^{(6)}$.

{\bf Case 6}. The history $H$ is of type $(2)(1)(2)(3)(2)(1)(2)$, and
$H^{(1)}H^{(2)}H^{(3)}H^{(4)}H^{(5)}H^{(6)}H^{(7)}$ is the
corresponding firm factorization. In this case we will assume that
$\cal C$ is a $(t')^{\pm 1}$-band and $\cal C'$ is a $t^{\pm 1}$-band.

If $h^{(7)}\ge 0.01h$, then one can apply Lemma \ref {two01} to
$\Gamma$ with $H(1)\equiv H^{(1)}H^{(2)}H^{(3)}H^{(4)}H^{(5)}H^{(6)},$ $
H(2)\equiv H^{(7)}$ and $H(3)\equiv \emptyset.$ Hence we may assume that
$h^{(7)}<0.01h$. Similarly, $h^{(1)}<0.01h$
 Then, as
in Cases 3 and 5 we may assume that $\max (h^{(3)},h^{(5)})<h/200$
and $h^{(4)}<h/100$. Therefore
$$(h')^{(2)}+(h')^{(6)}= h'-(h')^{(1)}-(h')^{(3)}-(h')^{(4)}-(h')^{(5)}-(h')^{(7)}$$ $$\ge
h'-h^{(1)}-h^{(3)}-h^{(4)}-h^{(5)}-h^{(7)}
> h' - 0.04h \ge h'-0.08h'=0.92 h'.$$ Therefore one can apply
Lemma \ref{5parts} to $\Gamma'$ with $H(1)\equiv (H')^{(1)}$,
$H(2)\equiv (H')^{(2)}$, $H(3)\equiv (H')^{(3)}(H')^{(4)}(H')^{(5)}$, $H(4)\equiv (H')^{(6)}$
and $H(5)\equiv (H')^{(7)}$.

The lemma is proved in any case. \endproof

\begin{lemma}\label{itog} Let $\Delta$ be a comb of base width $b>13N.$ Then either $\Delta$ admits a
long quasicomb $\Gamma'$ with
\begin{equation}\label{smotri}
\area(\Gamma')\le c_3[\Gamma']+c_2\mu^c(\Gamma')+
c_3(\nu^c_J(\Delta)-\nu^c_J(\Delta\backslash\Gamma'))
\end{equation} or $\Delta$
has a maximal $t^{\pm 1}$ or $(t')^{\pm 1}$-band of length $l$, where
$T_i\le l< 200 T_i$ for some $T_i$,
and this band is a handle of a subcomb of base width $\le 14N.$
\end{lemma}

\proof Recall that the third term in the right-hand side of (\ref{smotri}) is positive
for every sub(quasi)comb $\Gamma'$ by Lemma \ref{mu} (e) and Remark \ref{quasiotrez}.

Then we observe that $\Delta$ has a subcomb $\Delta_0$ of base withs $b^{\Delta_0}\in (13N,15N],$ and in turn, $\Delta_0$
has a regular subcomb $\Gamma$ with  $12N<b^{\Gamma}\le 14N.$
If  $\Gamma$ is a one Step comb, then by Lemma
\ref{oneage}, it admits a  long quasicomb $\Gamma'$ with
$\area(\Gamma')\le c_1([\Gamma']+\kappa^c(\Gamma'))$.
Here the right-hand side does not exceed $c_3[\Gamma']+c_2\mu^c(\Gamma')$
since $c_1<c_2<c_3$ and $\lambda^c(\Gamma')\ge -\lambda(y^{\Gamma'})=0$
for one Step comb $\Gamma'.$ Inequality (\ref{smotri}) follows in this case.

Then we may assume that the history $H$ of $\Gamma$ has one of the rules $(12)^{\pm 1},$ $(23)^{\pm 1}.$
If $H$ has no $(12)^{\pm 1}$ or no $(23)^{\pm 1},$ then the statement of the lemma
follows from Lemma \ref{bez12}. Otherwise $H$ has at most $6$ letters $(12)^{\pm 1}$
and $(23)^{\pm 1}$ by Properties \ref{viii} and \ref{kt},
since $\Gamma$ is a regular comb (and so there exists a trapezium of width $\ge N$ with
history $H$). Now the application of Lemmas \ref{grebenki} and \ref{mu} (e) completes the proof.
\endproof

\bigskip

 \section{Separation of a hub}\label{separ}

In this section we consider  
minimal diagrams over the group $G$ with cyclically reduced boundary
paths. Thus in contrast to previous sections, we study diagrams with hubs.

\subsection{Solid diagrams}

Let $\pi$ be a hub in a diagram $\Delta,$ connected with the
boundary $\partial\Delta$ by  $t$-spokes 
$\cal B$ and $\cal B'.$ We denote by
\label{cl.} $cl(\pi,{\cal B},{\cal B'})$ the subdiagram bounded by these spokes
(and including them) and by subpaths of the boundaries of $\Delta$
and $\pi,$ and call this subdiagram a \label{clove} {\it clove} if it has no hubs.

  \begin{lemma}\label{rimexists} (a) Let $\Psi=cl(\pi,{\cal B},{\cal B'})$ be a clove in a reduced
  diagram $\Delta$. Assume that $\Psi$ contains a rim
  $\theta$-band $\cal T$,
  which crosses neither $\cal B$ nor $\cal B'$, and every rim $\theta$-band of $\Psi$ with this
  property has at least $2LN$  $q$-cells. Then there is a maximal $q$-band $\cal C$ in
  $\Psi$ and a subcomb $\Gamma$ with handle $\cal C$ such that the base width of $\Gamma$
  is $15N$ and no $q$-band of $\Gamma$ is a subband of a spoke of $\Delta.$

  (b) Assume that a reduced diagram $\Delta$
  contains cells
  but has no hubs. Then either it has a rim band of base width
  $<2LN$ or it has a subcomb of base width $15N.$ 
  \end{lemma}
  \proof (a) Since (1) a hub has $LN$ spokes, (2) no $q$-band of $\Psi$ intersects $\cal T$ twice by Lemma
  \ref{NoAnnul},
  (3) $\cal T$ has at least $2LN$  $q$-cells, and (4) $L>30$, there exists a maximal $q$-band $\cal C'$
 such that a subdiagram $\Gamma'$ separated from $\Psi$ by $\cal C'$ contains no edges of the spokes
  of $\pi$ and the part of $\cal T$ belonging to $\Gamma'$ has at least $15N$  $q$-cells.
  
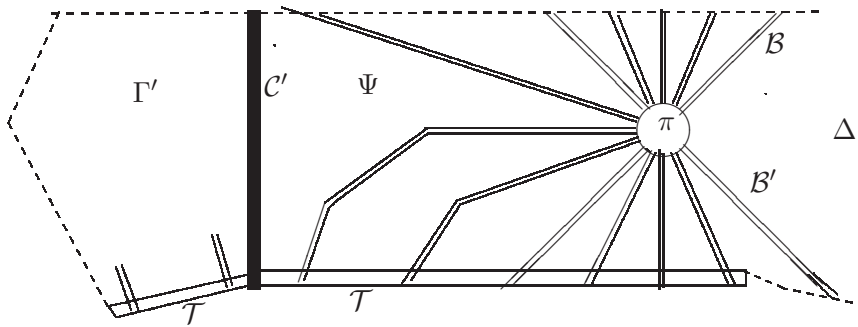
\unitlength 1mm 
\linethickness{0.4pt}
\ifx\plotpoint\undefined\newsavebox{\plotpoint}\fi 
\begin{picture}(131.5,60.75)(0,0)
\multiput(117.43,40.18)(-.03125,.03125){4}{\line(0,1){.03125}}
\multiput(52.18,46.68)(-.0625,.03125){4}{\line(-1,0){.0625}}
\put(22.43,50.93){\line(1,0){.9903}}
\put(24.41,50.935){\line(1,0){.9903}}
\put(26.391,50.939){\line(1,0){.9903}}
\put(28.371,50.944){\line(1,0){.9903}}
\put(30.352,50.949){\line(1,0){.9903}}
\put(32.333,50.954){\line(1,0){.9903}}
\put(34.313,50.959){\line(1,0){.9903}}
\put(36.294,50.964){\line(1,0){.9903}}
\put(38.274,50.969){\line(1,0){.9903}}
\put(40.255,50.973){\line(1,0){.9903}}
\put(42.236,50.978){\line(1,0){.9903}}
\put(44.216,50.983){\line(1,0){.9903}}
\put(46.197,50.988){\line(1,0){.9903}}
\put(48.177,50.993){\line(1,0){.9903}}
\put(50.158,50.998){\line(1,0){.9903}}
\put(52.138,51.003){\line(1,0){.9903}}
\put(54.119,51.007){\line(1,0){.9903}}
\put(56.1,51.012){\line(1,0){.9903}}
\put(58.08,51.017){\line(1,0){.9903}}
\put(60.061,51.022){\line(1,0){.9903}}
\put(62.041,51.027){\line(1,0){.9903}}
\put(64.022,51.032){\line(1,0){.9903}}
\put(66.003,51.037){\line(1,0){.9903}}
\put(67.983,51.041){\line(1,0){.9903}}
\put(69.964,51.046){\line(1,0){.9903}}
\put(71.944,51.051){\line(1,0){.9903}}
\put(73.925,51.056){\line(1,0){.9903}}
\put(75.905,51.061){\line(1,0){.9903}}
\put(77.886,51.066){\line(1,0){.9903}}
\put(79.867,51.07){\line(1,0){.9903}}
\put(81.847,51.075){\line(1,0){.9903}}
\put(83.828,51.08){\line(1,0){.9903}}
\put(85.808,51.085){\line(1,0){.9903}}
\put(87.789,51.09){\line(1,0){.9903}}
\put(89.77,51.095){\line(1,0){.9903}}
\put(91.75,51.1){\line(1,0){.9903}}
\put(93.731,51.104){\line(1,0){.9903}}
\put(95.711,51.109){\line(1,0){.9903}}
\put(97.692,51.114){\line(1,0){.9903}}
\put(99.672,51.119){\line(1,0){.9903}}
\put(101.653,51.124){\line(1,0){.9903}}
\put(103.634,51.129){\line(1,0){.9903}}
\put(105.614,51.134){\line(1,0){.9903}}
\put(107.595,51.138){\line(1,0){.9903}}
\put(109.575,51.143){\line(1,0){.9903}}
\put(111.556,51.148){\line(1,0){.9903}}
\put(113.537,51.153){\line(1,0){.9903}}
\put(115.517,51.158){\line(1,0){.9903}}
\put(117.498,51.163){\line(1,0){.9903}}
\put(119.478,51.168){\line(1,0){.9903}}
\put(121.459,51.172){\line(1,0){.9903}}
\put(123.439,51.177){\line(1,0){.9903}}
\multiput(23.43,51.18)(-.033088,-.072304){12}{\line(0,-1){.072304}}
\multiput(22.636,49.444)(-.033088,-.072304){12}{\line(0,-1){.072304}}
\multiput(21.841,47.709)(-.033088,-.072304){12}{\line(0,-1){.072304}}
\multiput(21.047,45.974)(-.033088,-.072304){12}{\line(0,-1){.072304}}
\multiput(20.253,44.239)(-.033088,-.072304){12}{\line(0,-1){.072304}}
\multiput(19.459,42.503)(-.033088,-.072304){12}{\line(0,-1){.072304}}
\multiput(18.665,40.768)(-.033088,-.072304){12}{\line(0,-1){.072304}}
\multiput(17.871,39.033)(-.033088,-.072304){12}{\line(0,-1){.072304}}
\multiput(17.077,37.297)(-.033088,-.072304){12}{\line(0,-1){.072304}}
\multiput(16.68,36.43)(.0326355,-.0603448){14}{\line(0,-1){.0603448}}
\multiput(17.594,34.74)(.0326355,-.0603448){14}{\line(0,-1){.0603448}}
\multiput(18.507,33.05)(.0326355,-.0603448){14}{\line(0,-1){.0603448}}
\multiput(19.421,31.361)(.0326355,-.0603448){14}{\line(0,-1){.0603448}}
\multiput(20.335,29.671)(.0326355,-.0603448){14}{\line(0,-1){.0603448}}
\multiput(21.249,27.981)(.0326355,-.0603448){14}{\line(0,-1){.0603448}}
\multiput(22.162,26.292)(.0326355,-.0603448){14}{\line(0,-1){.0603448}}
\multiput(23.076,24.602)(.0326355,-.0603448){14}{\line(0,-1){.0603448}}
\multiput(23.99,22.912)(.0326355,-.0603448){14}{\line(0,-1){.0603448}}
\multiput(24.904,21.223)(.0326355,-.0603448){14}{\line(0,-1){.0603448}}
\multiput(25.818,19.533)(.0326355,-.0603448){14}{\line(0,-1){.0603448}}
\multiput(26.731,17.844)(.0326355,-.0603448){14}{\line(0,-1){.0603448}}
\multiput(27.645,16.154)(.0326355,-.0603448){14}{\line(0,-1){.0603448}}
\multiput(28.559,14.464)(.0326355,-.0603448){14}{\line(0,-1){.0603448}}
\multiput(29.473,12.775)(.0326355,-.0603448){14}{\line(0,-1){.0603448}}
\put(42.75,51){\rule{.5\unitlength}{.25\unitlength}}
\put(48.5,14.25){\rule{1.75\unitlength}{37\unitlength}}
\put(103.75,35.5){\circle{6.519}}
\put(106,38.25){\line(1,1){12.75}}
\put(106,37.5){\line(1,1){13.75}}
\put(104,39){\line(0,1){12}}
\put(103.5,39){\line(0,1){12.5}}
\multiput(110,51)(-.033557047,-.080536913){149}{\line(0,-1){.080536913}}
\multiput(110.75,51)(-.033653846,-.078525641){156}{\line(0,-1){.078525641}}
\multiput(22.75,185.75)(-.03125,-.03125){8}{\line(0,-1){.03125}}
\multiput(53,51.75)(.1045353982,-.0337389381){452}{\line(1,0){.1045353982}}
\put(100,35.75){\line(-1,0){28}}
\multiput(58,50.75)(.1041666667,-.0337009804){408}{\line(1,0){.1041666667}}
\multiput(72.5,35.75)(-.0462962963,-.0336700337){297}{\line(-1,0){.0462962963}}
\put(58.75,25.75){\line(-1,-3){3.5}}
\put(100.25,35){\line(-1,0){27.75}}
\multiput(72.5,35)(-.0454545455,-.0336363636){275}{\line(-1,0){.0454545455}}
\multiput(60,25.75)(-.0326087,-.0326087){23}{\line(0,-1){.0326087}}
\multiput(59.25,25)(-.033505155,-.097938144){97}{\line(0,-1){.097938144}}
\put(103.75,32.25){\line(0,-1){17.75}}
\put(103.5,32.25){\line(1,0){.25}}
\put(103.25,32.75){\line(0,-1){18.5}}
\put(106,33){\line(1,-1){18.25}}
\put(106.5,33.5){\line(1,-1){18.5}}
\multiput(104.75,32.5)(.033695652,-.077173913){230}{\line(0,-1){.077173913}}
\multiput(105.5,32.5)(.033695652,-.076086957){230}{\line(0,-1){.076086957}}
\put(101,38.25){\line(-1,1){12.75}}
\put(89.25,51){\line(1,-1){12.75}}
\multiput(96.5,51.25)(.033653846,-.078525641){156}{\line(0,-1){.078525641}}
\multiput(97.75,50.75)(.033687943,-.083333333){141}{\line(0,-1){.083333333}}
\put(101.25,33.25){\line(-1,-1){19}}
\multiput(101.5,33)(.0333333,-.0333333){15}{\line(0,-1){.0333333}}
\put(101.75,32.75){\line(-1,-1){18}}
\put(102.25,32.75){\line(-1,-2){9}}
\multiput(102.75,32.5)(-.0337301587,-.0694444444){252}{\line(0,-1){.0694444444}}
\put(94.25,15){\line(0,1){0}}
\multiput(100.5,34.5)(-.0962301587,-.0337301587){252}{\line(-1,0){.0962301587}}
\multiput(76.25,26)(-.03372093,-.051162791){215}{\line(0,-1){.051162791}}
\multiput(100.75,34)(-.0952380952,-.0337301587){252}{\line(-1,0){.0952380952}}
\multiput(76.75,25.5)(-.033678756,-.051813472){193}{\line(0,-1){.051813472}}
\put(117.25,45.75){$\cal B$}
\put(115.25,27.25){$\cal B'$}
\put(103,35.5){$\pi$}
\put(50.75,40){$\cal C'$}
\put(63,40.25){$\Psi$}
\put(33.25,39.25){$\Gamma'$}
\put(126.25,34.25){$\Delta$}
\multiput(119.43,14.43)(.142857,-.029762){6}{\line(1,0){.142857}}
\multiput(121.144,14.073)(.142857,-.029762){6}{\line(1,0){.142857}}
\multiput(122.858,13.715)(.142857,-.029762){6}{\line(1,0){.142857}}
\multiput(124.573,13.358)(.142857,-.029762){6}{\line(1,0){.142857}}
\multiput(126.287,13.001)(.142857,-.029762){6}{\line(1,0){.142857}}
\multiput(128.001,12.644)(.142857,-.029762){6}{\line(1,0){.142857}}
\multiput(129.715,12.287)(.142857,-.029762){6}{\line(1,0){.142857}}
\put(124,15){\line(1,0){.25}}
\multiput(123.75,15.25)(.03125,-.03125){8}{\line(0,-1){.03125}}
\multiput(123,16.25)(.033505155,-.033505155){97}{\line(0,-1){.033505155}}
\multiput(124.25,15.75)(-.03125,.03125){8}{\line(0,1){.03125}}
\put(123.75,16.25){\line(0,1){.25}}
\multiput(123.75,16.25)(.03963415,-.03353659){82}{\line(1,0){.03963415}}
\multiput(30,12)(.148809524,.033730159){126}{\line(1,0){.148809524}}
\put(50,16.75){\line(1,0){64.75}}
\put(114.75,16.75){\line(0,-1){2}}
\put(114.75,14.75){\line(-1,0){66}}
\multiput(48.75,14.75)(-.140873016,-.033730159){126}{\line(-1,0){.140873016}}
\multiput(31,10.5)(-.0333333,.05){30}{\line(0,1){.05}}
\multiput(114.43,16.68)(.075,-.032143){10}{\line(1,0){.075}}
\multiput(115.93,16.037)(.075,-.032143){10}{\line(1,0){.075}}
\multiput(117.43,15.394)(.075,-.032143){10}{\line(1,0){.075}}
\multiput(118.93,14.751)(.075,-.032143){10}{\line(1,0){.075}}
\put(62,11.5){$\cal T$}
\put(39.75,9.75){$\cal T$}
\multiput(43.75,21.5)(.03333333,-.11666667){60}{\line(0,-1){.11666667}}
\multiput(44.5,21.5)(.03333333,-.1125){60}{\line(0,-1){.1125}}
\multiput(31.75,17.5)(.03358209,-.08955224){67}{\line(0,-1){.08955224}}
\multiput(31,17.25)(.03333333,-.09583333){60}{\line(0,-1){.09583333}}
\end{picture}

  If $\Gamma'$ is not a comb, and so a maximal $\theta$-band of it does not cross $\cal C',$
  then $\Gamma'$ must contain another rim band $\cal T'$ having at least $2LN$
  $q$-cells by the assumption of the lemma. This makes possible to find a subdiagram $\Gamma''$ of
  $\Gamma'$ such that a part of $\cal T'$ is a rim band of $\Gamma''$ containing at least $LN>15N$
  $q$-cells, and $\Gamma''$ does not contain $\cal C'$.
  Since $\area(\Gamma')>\area(\Gamma'')>\dots$ , such a procedure must stop. Hence, for some $i$, we
  obtain a subcomb $\Gamma^{(i)}$ of width $b\ge 15N$ intersected by no spokes.
  If $b>15N,$ then a derived subcomb of it has width $b-1\ge 15N.$ Finally we obtain the desired $\Gamma.$

  (b) The proof is easier than that for (a): one should just
  ignore the hub.
\endproof

We call a minimal diagram \label{solidd}{\it solid} if it has no rim
$\theta$-bands of base width $\le 2LN,$ no subcombs of base width
$15N$ and no one-Step subcombs 
whose handles are $t^{\pm 1}$-
or $(t')^{\pm 1}$-bands. Here we focus on solid diagrams
since the proof of Theorem \ref{mainth} will be reduced to
them in the next section.

For a clove $\Psi=cl(\pi,{\cal B},{\cal B'})$ in a diagram $\Delta$,
we denote by \label{pPsi} ${\bf p}(\Psi)$ the common subpath of $\partial\Psi$ and
$\partial\Delta$ starting with the $q$-edge of ${\cal B}$ and ending
with the $q$-edge of $\cal B'.$

\begin{lemma}\label{psi1} Let $\Psi=cl(\pi,{\cal B},{\cal B'})$
be a clove in a solid diagram $\Delta$. Then every maximal
$\theta$-band of $\Psi$ crosses either ${\cal B}$ or ${\cal B'}$;
the base width of any $\theta$-band of $\Psi$  is less than $2LN,$
and $\area(\Psi)\le (2LN(h+h')+ \delta^{-1}|{\bf p}(\Psi)|)(h+h'),$ where $h$
and $h'$ are the lengths of the bands ${\cal B}$ and ${\cal B'}$,
respectively.
\end{lemma}
\proof If the first claim were wrong, then one could find a rim $\theta$-band which
crosses neither ${\cal B}$ nor ${\cal B'}$. Then by Lemma \ref{rimexists} (a), either
first  or the second condition in the definition of solid diagram would be violated,
a contradiction. Thus the first statement of the lemma is proved, and $\Psi$ has at most
$h+h'$ maximal $\theta$-bands.

     Now we consider a maximal $\theta$-band
$\cal T$ in $\Psi$ ($\cal T$ is not an annulus by Lemma \ref{NoAnnul}). If its base width is at least $2LN$, then there is a maximal $q$-band $\cal C$ intersecting $\cal T$ which does not
start/end on the hub $\pi$ because the number of spokes starting on the same hub cell is $LN$. Moreover, as in the proof of Lemma \ref{rimexists}, one
can select $\cal C$ so that $\cal C$ separates a comb of base width
$ > (2LN-LN)/2\ge 15N$ from $\Psi,$ contrary to the assumption that
$\Delta$ is solid. Thus, the base width of $\cal T$ is less than
$2LN$. Therefore the number of $(\theta, q)$-cells in $\Psi$ is at
most $2LN\times (h+h')$. Every $(\theta, q)$-cell has at most two $a$-edges
by \ref{cell}. Hence the number of maximal $a$-bands
starting and ending on the $(\theta, q)$-cells of $\Psi$ (but not on $\partial\Delta$) is at most
$2LN(h+h').$ Their lengths do not exceed $h+h'$ by Lemma \ref{NoAnnul} since the number
of maximal $\theta$-bands in $\Psi$ is at most $h+h'.$ Thus the total area of these $a$-bands
does not exceed $2LN(h+h')^2.$
Arbitrary other maximal $a$-band and a maximal
$q$-band  of $\Psi$ is also of length at most $h+h'$  by the same reason, but it
has at least one edge on ${\bf p}(\Psi)$. Therefore their total area
is at most $(h+h')(|{\bf p}(\Psi)|_a+|{\bf p}(\Psi)|_q)\le \delta^{-1}|{\bf p}(\Psi)|$ by Lemma \ref{ochev} (d). Since every cell of $\Psi$ belongs to a $\theta$-band, every $(\theta,a)$-cell belongs to $a$-band and every maximal
$a$-band starts and ends on a $(\theta,q)$-cell or on $\partial\Delta, $ the sum of these two inequalities gives
the inequality from the lemma.
\endproof

Let $\Psi$ be a clove at a hub $\pi$ in a solid diagram $\Delta$.
Assume that it has more than $L$ $t$- and $t'$-spokes. (Recall that
$\partial\pi$ has $2L$ $t$- and $t'$-edges.) Then we denote
by $\bar\Delta$ the subdiagram formed by $\pi$ and $\Psi$, and
denote by $\bf \bar p$ the path $\topp({\cal B}) {\bf u}^{-1}
\bott({\cal B'})^{-1},$ where $\bf u$ is a subpath on $\partial\pi,$ such that
$\bf \bar p$ separate $\bar\Delta$ from the remaining subdiagram $\Psi'$
of $\Delta.$
It follows that the total number of $t$- and
$t'$-edges in $u$ is less than $L,$ $|u|<LN,$ and the number of $t$- end $t'$-edges
in ${\bf p}(\Psi)$ is at least $L+1.$

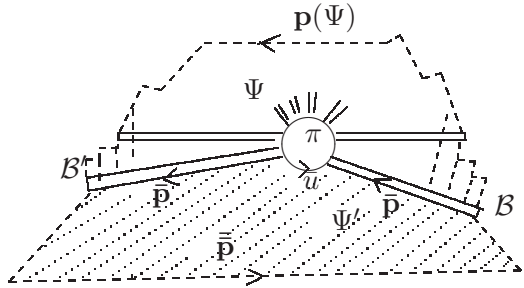
\unitlength 1mm 
\linethickness{0.4pt}
\ifx\plotpoint\undefined\newsavebox{\plotpoint}\fi 
\begin{picture}(100,48.5)(0,0)
\put(71.75,22.75){\circle{7.649}}
\multiput(74.75,21)(.097355769,-.033653846){208}{\line(1,0){.097355769}}
\multiput(74.25,20)(.094951923,-.033653846){208}{\line(1,0){.094951923}}
\multiput(94.5,14.25)(-.0333333,-.0833333){15}{\line(0,-1){.0833333}}
\multiput(68.25,21)(-.192164179,-.03358209){134}{\line(-1,0){.192164179}}
\multiput(42.5,16.5)(-.03125,.21875){8}{\line(0,1){.21875}}
\multiput(42.25,18.25)(.214285714,.033613445){119}{\line(1,0){.214285714}}
\put(75.5,23.25){\line(1,0){17.25}}
\put(75.5,24.25){\line(1,0){17}}
\put(92.5,24.25){\line(0,-1){1}}
\put(46.5,24.25){\line(1,0){21.75}}
\put(46.5,24.25){\line(0,-1){1}}
\put(46.5,23.25){\line(1,0){20.75}}
\multiput(42.68,16.18)(-.0332817,-.0363777){19}{\line(0,-1){.0363777}}
\multiput(41.415,14.797)(-.0332817,-.0363777){19}{\line(0,-1){.0363777}}
\multiput(40.15,13.415)(-.0332817,-.0363777){19}{\line(0,-1){.0363777}}
\multiput(38.886,12.033)(-.0332817,-.0363777){19}{\line(0,-1){.0363777}}
\multiput(37.621,10.65)(-.0332817,-.0363777){19}{\line(0,-1){.0363777}}
\multiput(36.356,9.268)(-.0332817,-.0363777){19}{\line(0,-1){.0363777}}
\multiput(35.091,7.886)(-.0332817,-.0363777){19}{\line(0,-1){.0363777}}
\multiput(33.827,6.503)(-.0332817,-.0363777){19}{\line(0,-1){.0363777}}
\multiput(32.562,5.121)(-.0332817,-.0363777){19}{\line(0,-1){.0363777}}
\put(31.93,4.43){\line(1,0){.9855}}
\put(33.901,4.473){\line(1,0){.9855}}
\put(35.872,4.517){\line(1,0){.9855}}
\put(37.843,4.56){\line(1,0){.9855}}
\put(39.814,4.604){\line(1,0){.9855}}
\put(41.785,4.647){\line(1,0){.9855}}
\put(43.756,4.691){\line(1,0){.9855}}
\put(45.727,4.734){\line(1,0){.9855}}
\put(47.698,4.778){\line(1,0){.9855}}
\put(49.669,4.821){\line(1,0){.9855}}
\put(51.64,4.864){\line(1,0){.9855}}
\put(53.611,4.908){\line(1,0){.9855}}
\put(55.582,4.951){\line(1,0){.9855}}
\put(57.553,4.995){\line(1,0){.9855}}
\put(59.524,5.038){\line(1,0){.9855}}
\put(61.495,5.082){\line(1,0){.9855}}
\put(63.466,5.125){\line(1,0){.9855}}
\put(65.437,5.169){\line(1,0){.9855}}
\put(67.408,5.212){\line(1,0){.9855}}
\put(69.379,5.256){\line(1,0){.9855}}
\put(71.35,5.299){\line(1,0){.9855}}
\put(73.321,5.343){\line(1,0){.9855}}
\put(75.292,5.386){\line(1,0){.9855}}
\put(77.263,5.43){\line(1,0){.9855}}
\put(79.234,5.473){\line(1,0){.9855}}
\put(81.205,5.517){\line(1,0){.9855}}
\put(83.176,5.56){\line(1,0){.9855}}
\put(85.147,5.604){\line(1,0){.9855}}
\put(87.118,5.647){\line(1,0){.9855}}
\put(89.089,5.691){\line(1,0){.9855}}
\put(91.06,5.734){\line(1,0){.9855}}
\put(93.031,5.778){\line(1,0){.9855}}
\put(95.002,5.821){\line(1,0){.9855}}
\put(96.973,5.864){\line(1,0){.9855}}
\put(98.944,5.908){\line(1,0){.9855}}
\multiput(99.93,5.93)(-.0334225,.0374332){17}{\line(0,1){.0374332}}
\multiput(98.793,7.202)(-.0334225,.0374332){17}{\line(0,1){.0374332}}
\multiput(97.657,8.475)(-.0334225,.0374332){17}{\line(0,1){.0374332}}
\multiput(96.521,9.748)(-.0334225,.0374332){17}{\line(0,1){.0374332}}
\multiput(95.384,11.021)(-.0334225,.0374332){17}{\line(0,1){.0374332}}
\multiput(94.248,12.293)(-.0334225,.0374332){17}{\line(0,1){.0374332}}
\multiput(46.68,24.43)(.03125,.0625){12}{\line(0,1){.0625}}
\multiput(47.43,25.93)(.03125,.0625){12}{\line(0,1){.0625}}
\multiput(48.18,27.43)(.03125,.0625){12}{\line(0,1){.0625}}
\multiput(48.93,28.93)(.03125,.0625){12}{\line(0,1){.0625}}
\put(49.68,30.43){\line(0,1){0}}
\multiput(49.68,30.43)(.166667,-.033333){5}{\line(1,0){.166667}}
\multiput(51.346,30.096)(.166667,-.033333){5}{\line(1,0){.166667}}
\multiput(52.18,29.93)(.0324074,.0354938){18}{\line(0,1){.0354938}}
\multiput(53.346,31.207)(.0324074,.0354938){18}{\line(0,1){.0354938}}
\multiput(54.513,32.485)(.0324074,.0354938){18}{\line(0,1){.0354938}}
\multiput(55.68,33.763)(.0324074,.0354938){18}{\line(0,1){.0354938}}
\multiput(56.846,35.041)(.0324074,.0354938){18}{\line(0,1){.0354938}}
\put(57.43,35.68){\line(0,1){0}}
\multiput(92.18,24.18)(-.030556,.086111){10}{\line(0,1){.086111}}
\multiput(91.569,25.902)(-.030556,.086111){10}{\line(0,1){.086111}}
\multiput(90.957,27.624)(-.030556,.086111){10}{\line(0,1){.086111}}
\multiput(90.346,29.346)(-.030556,.086111){10}{\line(0,1){.086111}}
\multiput(89.735,31.069)(-.030556,.086111){10}{\line(0,1){.086111}}
\multiput(89.43,32.18)(-.104167,-.03125){8}{\line(-1,0){.104167}}
\multiput(87.763,31.68)(-.104167,-.03125){8}{\line(-1,0){.104167}}
\multiput(86.93,31.43)(-.032967,.0631868){13}{\line(0,1){.0631868}}
\multiput(86.073,33.073)(-.032967,.0631868){13}{\line(0,1){.0631868}}
\multiput(85.215,34.715)(-.032967,.0631868){13}{\line(0,1){.0631868}}
\multiput(84.358,36.358)(-.032967,.0631868){13}{\line(0,1){.0631868}}
\put(83.93,37.18){\line(0,1){0}}
\multiput(83.93,37.18)(-.058333,-.033333){10}{\line(-1,0){.058333}}
\multiput(82.763,36.513)(-.058333,-.033333){10}{\line(-1,0){.058333}}
\put(82.18,36.18){\line(-1,0){.97}}
\put(80.24,36.18){\line(-1,0){.97}}
\put(78.3,36.18){\line(-1,0){.97}}
\put(76.36,36.18){\line(-1,0){.97}}
\put(74.42,36.18){\line(-1,0){.97}}
\put(72.48,36.18){\line(-1,0){.97}}
\put(70.54,36.18){\line(-1,0){.97}}
\put(68.6,36.18){\line(-1,0){.97}}
\put(66.66,36.18){\line(-1,0){.97}}
\put(64.72,36.18){\line(-1,0){.97}}
\put(62.78,36.18){\line(-1,0){.97}}
\put(60.84,36.18){\line(-1,0){.97}}
\put(58.9,36.18){\line(-1,0){.97}}
\put(42.18,18.43){\line(0,1){.75}}
\put(42.18,19.93){\line(0,1){.75}}
\put(42.18,20.68){\line(1,0){.875}}
\put(43.93,20.68){\line(0,1){.75}}
\put(43.93,22.18){\line(1,0){.8333}}
\put(45.596,22.18){\line(1,0){.8333}}
\put(46.43,22.18){\line(0,1){.5}}
\put(43.93,20.68){\line(0,-1){.9167}}
\put(43.93,18.846){\line(0,-1){.9167}}
\put(46.18,22.43){\line(0,-1){.9375}}
\put(46.305,20.555){\line(0,-1){.9375}}
\put(48.43,27.43){\line(0,-1){.9063}}
\put(48.43,25.617){\line(0,-1){.9063}}
\put(48.43,23.805){\line(0,-1){.9063}}
\put(48.43,21.992){\line(0,-1){.9063}}
\multiput(94.18,14.43)(.03125,.09375){8}{\line(0,1){.09375}}
\multiput(94.68,15.93)(.03125,.09375){8}{\line(0,1){.09375}}
\multiput(95.18,17.43)(.03125,.09375){8}{\line(0,1){.09375}}
\put(95.43,18.18){\line(-1,0){.6667}}
\put(94.096,18.346){\line(-1,0){.6667}}
\multiput(93.43,18.43)(.033333,.133333){5}{\line(0,1){.133333}}
\multiput(93.763,19.763)(.033333,.133333){5}{\line(0,1){.133333}}
\put(93.93,20.43){\line(-1,0){.75}}
\put(92.43,20.596){\line(-1,0){.75}}
\put(91.68,20.68){\line(0,1){.75}}
\put(91.846,22.18){\line(0,1){.75}}
\multiput(93.43,17.93)(-.033333,-.1){6}{\line(0,-1){.1}}
\multiput(93.03,16.73)(-.033333,-.1){6}{\line(0,-1){.1}}
\multiput(92.63,15.53)(-.033333,-.1){6}{\line(0,-1){.1}}
\multiput(91.43,20.43)(-.029762,-.107143){7}{\line(0,-1){.107143}}
\multiput(91.013,18.93)(-.029762,-.107143){7}{\line(0,-1){.107143}}
\multiput(90.596,17.43)(-.029762,-.107143){7}{\line(0,-1){.107143}}
\multiput(90.18,26.68)(-.029221,-.12013){7}{\line(0,-1){.12013}}
\multiput(89.771,24.998)(-.029221,-.12013){7}{\line(0,-1){.12013}}
\multiput(89.362,23.316)(-.029221,-.12013){7}{\line(0,-1){.12013}}
\multiput(88.952,21.634)(-.029221,-.12013){7}{\line(0,-1){.12013}}
\multiput(88.543,19.952)(-.029221,-.12013){7}{\line(0,-1){.12013}}
\multiput(88.134,18.271)(-.029221,-.12013){7}{\line(0,-1){.12013}}
\multiput(74.5,25.5)(.03731343,.03358209){67}{\line(1,0){.03731343}}
\multiput(73.75,25.75)(.03358209,.04104478){67}{\line(0,1){.04104478}}
\multiput(72.5,27)(.03125,.3125){8}{\line(0,1){.3125}}
\put(71.75,27){\line(0,1){2.5}}
\multiput(70.5,27)(-.0333333,.1333333){15}{\line(0,1){.1333333}}
\multiput(70,26.75)(-.0326087,.076087){23}{\line(0,1){.076087}}
\multiput(69,26)(-.03289474,.05263158){38}{\line(0,1){.05263158}}
\multiput(69,25.25)(-.03365385,.04326923){52}{\line(0,1){.04326923}}
\put(71.25,23){$\pi$}
\multiput(65,36)(.07894737,.03289474){38}{\line(1,0){.07894737}}
\multiput(65.25,36.25)(.05263158,-.03289474){38}{\line(1,0){.05263158}}
\multiput(63.75,6)(.076087,-.0326087){23}{\line(1,0){.076087}}
\multiput(63.5,4)(.0666667,.0333333){30}{\line(1,0){.0666667}}
\multiput(70.5,20.5)(.0416667,-.0333333){30}{\line(1,0){.0416667}}
\multiput(71.75,19.5)(-.03289474,-.03289474){38}{\line(0,-1){.03289474}}
\multiput(81,17.5)(-.0333333,.0333333){15}{\line(0,1){.0333333}}
\put(80.5,18){\line(1,0){2.25}}
\multiput(80.5,18)(.0326087,-.076087){23}{\line(0,-1){.076087}}
\multiput(54.25,19.5)(-.03888889,-.03333333){45}{\line(-1,0){.03888889}}
\multiput(52.5,18)(.0869565,-.0326087){23}{\line(1,0){.0869565}}
\put(59.5,8){$\bf\bar{\bar p}$}
\put(71,16.75){$\bar u$}
\put(81.5,13.75){$\bf\bar p$}
\put(51,15){$\bf\bar p$}
\put(96.5,13.25){$\cal B$}
\put(38.75,18.25){$\cal B'$}
\put(63,28.75){$\Psi$}
\put(69.75,38.5){${\bf p}(\Psi)$}
\put(74.75,11.25){$\Psi'$}
\multiput(35.93,8.68)(.8,.46667){16}{{\rule{.4pt}{.4pt}}}
\multiput(34.68,4.93)(.82692,.51923){27}{{\rule{.4pt}{.4pt}}}
\multiput(38.68,4.93)(.77778,.51852){28}{{\rule{.4pt}{.4pt}}}
\multiput(41.68,4.93)(.80357,.52679){29}{{\rule{.4pt}{.4pt}}}
\multiput(45.68,5.18)(.79808,.55769){27}{{\rule{.4pt}{.4pt}}}
\multiput(48.18,5.18)(0,-.25){3}{{\rule{.4pt}{.4pt}}}
\multiput(48.68,4.93)(.79,.54){26}{{\rule{.4pt}{.4pt}}}
\multiput(52.18,5.18)(.65,.45){6}{{\rule{.4pt}{.4pt}}}
\multiput(60.93,10.68)(.70455,.54545){12}{{\rule{.4pt}{.4pt}}}
\multiput(55.93,5.43)(.75,.54167){25}{{\rule{.4pt}{.4pt}}}
\multiput(59.18,5.43)(.78261,.55435){24}{{\rule{.4pt}{.4pt}}}
\multiput(62.43,5.43)(.75,.5625){13}{{\rule{.4pt}{.4pt}}}
\multiput(76.18,14.68)(.7,.5){6}{{\rule{.4pt}{.4pt}}}
\multiput(66.43,5.68)(.6875,.5625){9}{{\rule{.4pt}{.4pt}}}
\multiput(69.68,5.68)(.66667,.53333){16}{{\rule{.4pt}{.4pt}}}
\multiput(73.18,5.93)(.73438,.57813){17}{{\rule{.4pt}{.4pt}}}
\multiput(75.68,5.93)(.76667,.56667){16}{{\rule{.4pt}{.4pt}}}
\multiput(78.93,5.93)(.66667,.53333){16}{{\rule{.4pt}{.4pt}}}
\multiput(82.68,5.93)(.72917,.625){13}{{\rule{.4pt}{.4pt}}}
\multiput(85.93,6.18)(.65909,.54545){12}{{\rule{.4pt}{.4pt}}}
\multiput(88.93,5.93)(.6875,.59375){9}{{\rule{.4pt}{.4pt}}}
\multiput(91.93,6.18)(.60714,.46429){8}{{\rule{.4pt}{.4pt}}}
\multiput(94.68,6.18)(.5,.4){6}{{\rule{.4pt}{.4pt}}}
\multiput(97.68,6.43)(.25,.25){3}{{\rule{.4pt}{.4pt}}}
\multiput(40.93,13.93)(.60714,.42857){8}{{\rule{.4pt}{.4pt}}}
\end{picture}

\begin{lemma}\label{ppsi} If $|\partial\Delta|=n$ and $|{\bf p}(\Psi)|\ge 2LN\max(h,h')$, then, in the preceding notation,
$|{\bf p}(\Psi)|-|\bf \bar p|>0$ and
 $$\area(\bar\Delta) \le c_4(\mu(\Delta)- \mu(\Psi')) +
 c_5(\nu_J(\Delta)- \nu_J(\Psi')) +c_6n (n-|\partial\Psi'|)$$ 
 \end{lemma}

 \proof Let us present $\partial\Psi'$ in the form $\bf \bar p \bar{\bar p}. $ By Lemma \ref{ochev}(b),
 $|\partial\Psi'|\le |{\bf \bar p}|+|{\bf\bar{\bar p}}|$ and $n = |{\bf p}(\Psi)|+|\bf \bar{\bar p}|$ since the first and
 the last edges of ${\bf p}(\Psi)$ are $q$-edges.
 Hence
 \begin{equation}\label{pochemu}
 n-|\partial\Psi'|\ge |{\bf p}(\Psi)|-|\bf \bar p|
 \end{equation}

Note that by the definition of $\Psi,$ we have $|{\bf \bar p}|\le h+h'+|u|\le h+h'+LN-1,$ and $|{\bf p}(\Psi)|\ge LN+1$. Therefore
in case $h=h'=0,$ we have $|{\bf p}(\Psi)|-|{\bf \bar p}|\ge\max(2, \delta |{\bf p}(\Psi)|)$ by Lemma \ref{ochev}(d), since $\delta^{-1}\ge LN.$ If $\max(h,h')\ge 1,$ then by the second assumption of the lemma
$\frac{3}{4}|{\bf p}(\Psi)|-|{\bf \bar p}|> 1.5LN\max(h,h')-2\max(h,h')-LN\ge 2$, and
 since $\delta<1/4,$ in any case we obtain
 \begin{equation} \label{razn}
 |{\bf p}(\Psi)|-|{\bf \bar p}|\ge \max (\delta |{\bf p}(\Psi)|, 2)
 \end{equation}

Inequalities (\ref{pochemu}), (\ref{razn}), and the assumption of the lemma on $|{\bf p}(\Psi)|$ imply

\begin{equation}\label{tri}
|{\bf p}(\Psi)|\le \delta\iv (n-|\partial\Psi'|), \;\;\; n-|\partial\Psi'|\ge 2,\;\; and\;\; h+h' \le n/LN
 \end{equation}

Now by  Lemma \ref{psi1} and Inequalities (\ref{tri}), we have

 $$\area(\bar\Delta)=\area(\Psi)+1\le (2LN(h+h')+ \delta^{-1}|{\bf p}(\Psi)|)(h+h')+1
\le $$
$$
(2|{\bf p}(\Psi)|+\delta\iv|{\bf p}(\Psi)|)(h+h')+1\le(2+\delta\iv)\delta\iv (n-|\partial\Psi'|)(h+h')+1$$
\begin{equation}\label{areabar}
\le 3(\delta^2LN)\iv (n-|\partial\Psi'|)n\le c_6n (n-|\partial\Psi'|)/2
\end{equation}

Recall now that in the definition
of $\kappa$- and $\lambda$-mixtures, the middle point of every boundary $q$-edge is a black bead of the
necklace on the boundary of diagram, and every white bead is
a middle point of a boundary $\theta$-edge (see Section \ref{mix} for details).
It follows
that $\kappa(\Delta)- \kappa(\Psi') \ge - |{\bf \bar p}| n$ and
$\lambda(\Delta)- \lambda(\Psi') \ge - |{\bf \bar p}| n$ because new pairs of white beads $(o,o')$ separated by black beads
can appear in the necklace on $\partial\Psi'$ (in comparison with the necklace on $\partial\Delta$)
only if one of the beads $o, o'$ belongs to $\bf \bar p.$  Hence by (\ref{razn}),
$$\min(\kappa(\Delta)- \kappa(\Psi'), \lambda(\Delta)- \lambda(\Psi'))\ge
-\delta\iv(|{\bf p}(\Psi)|-|{\bf \bar p}|)n,$$  
and therefore by (\ref{pochemu}),
$$ c_4(\mu(\Delta)- \mu(\Psi'))=c_4((c_0\kappa(\Delta)+ \lambda(\Delta))-(c_0\kappa(\Psi')+ \lambda(\Psi'))=
c_4(c_0(\kappa(\Delta)- $$ \begin{equation} \label{c4} \kappa(\Psi')))+
 (\lambda(\Delta)- \lambda(\Psi'))\ge -c_4(c_0+1)\delta^{-1} (|{\bf p}(\Psi)|-|{\bf \bar p}|)n
\ge -c_6n(n-|\partial\Psi'|)/4
\end{equation}

Recall also that the number of $t$- and $t'$-edges in the path $\bf\bar p$ (or in the path $\bf u$)
does not exceed the similar number for ${\bf p}(\Psi).$ Therefore any two white beads $o, o'$ of the $\nu$-necklace
on $\partial\Delta$, provided they both  belong to ${\bf p}(\Psi),$ are separated by at least the same
number of black beads in the $\nu$-necklace for $\Delta$ as in the $\nu$-necklace for $\Psi'$  (either the clockwise arc $o-o'$ includes ${\bf p}(\Psi)$ or not). So such a pair  contributes
to $\nu_J(\Delta)$ at least the amount it contributes to $\nu_J(\Psi')$.
Thus, to estimate $\nu_J(\Delta)- \nu_J(\Psi')$ from below, it suffices to consider
the contribution to $\nu_J(\Psi')$ for the pairs $o, o', $ where one of the two beads
lies on $\bf\bar p.$ Then the argument we used above for $\kappa$- and $\lambda$-mixtures, yields
$\nu_J(\Delta)- \nu_J(\Psi') \ge -J\delta\iv(|{\bf p}(\Psi)|-|{\bf \bar p}|)n.$ Hence
$c_5(\nu_J(\Delta)-\nu_J(\Psi'))\ge -c_6n(n-|\partial\Psi'|)/4$ by (\ref{pochemu}) since $c_6>4J\delta^{-1}c_5.$
This inequality together with (\ref{areabar}) and (\ref{c4}) prove the lemma.
\endproof

A clove $cl(\pi,{\cal B},{\cal B'})$ will be called a \label{crescent}{\it crescent}
if

(1) it contains $l\ge L-20>L/2$ consecutive  $t^{\pm 1}$-spokes ${\cal
C}_1=\cal B,$ ${\cal C}_2,\dots$ ${\cal C}_l=\cal B'$ connecting $\partial\Delta$ and
$\partial\pi$;

(2) every maximal $\theta$-band of this clove crosses either ${\cal
C}_1$ or ${\cal C}_l$; 
moreover, either all maximal $\theta$-bands
of $\Psi$ cross ${\cal C}_1$, or all of them cross ${\cal C}_l$, or
there exists $i$, $2\le i \le l-2$   such that the $\theta$-bands
crossing ${\cal C}_l$ but not ${\cal C}_1$, do not cross ${\cal
C}_i$, and the $\theta$-bands crossing ${\cal C}_1$ but not ${\cal
C}_l$, do not cross ${\cal C}_{i+1}$;

(3) every maximal  $(12)$- or $(23)$-band of $\Psi$ crossing ${\cal C}_{1}$ ( crossing
${\cal C}_{l}$) also crosses ${\cal C}_2$ (crosses ${\cal
C}_{l-1}),$ and every spoke of the clove is crossed by at most $3$ $(12)$- or $(23)$-bands.

\begin{lemma}\label{narrow} Assume a solid diagram $\Delta$ has
a hub. Then it contains a crescent $\Psi=cl(\pi,{\cal C}_1,{\cal
C}_l)$ such that the cloves $cl(\pi,{\cal C}_2,{\cal C}_l)$ and
$cl(\pi,{\cal C}_1,{\cal C}_{l-1})$ are also crescents.
\end{lemma}

\proof We  consider a hub $\pi$ provided by Lemma  \ref{extdisc}. 
There are consecutive maximal $t^{\pm 1}$-bands ${\cal
B}_1,\dots,{\cal B}_{L-3}$ connecting  (counter-clockwise) $\partial\Delta$ and $\partial\pi$,
such that the subdiagram $\bar\Psi$ bounded by ${\cal B}_1$,
${\cal B}_{L-3}$, $\partial\Delta,$ and $\partial\pi$
contains all these $t^{\pm 1}$-bands but does not contain hubs.  Observe
that by Lemma \ref{rimexists},  every maximal $\theta$-band of
$\bar\Psi=cl(\pi,{\cal B}_1,{\cal B}_{L-3})$ crosses either ${\cal
B}_1$ or ${\cal B}_{L-3}$ because $\Delta$ is solid.

Consider a subdiagram $\Psi(0)=cl(\pi,{\cal B}_u,{\cal B}_v)$ of
$\bar\Psi$  with $u-v = L-k$ for some $k$, $6\le k<L,$ $u>1, v<L-3.$
Every maximal $\theta$-band of $\Psi(0)$ crosses either ${\cal B}_u$
or ${\cal B}_v$.

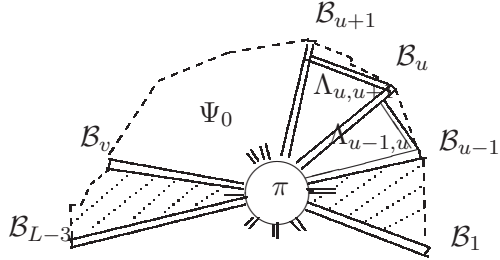
\unitlength 1mm 
\linethickness{0.4pt}
\ifx\plotpoint\undefined\newsavebox{\plotpoint}\fi 
\begin{picture}(58.75, 45.25)(-30, 0)
\put(35.75,13.5){\circle{8.544}}
\multiput(39.5,11.25)(.087365591,-.033602151){186}{\line(1,0){.087365591}}
\multiput(39.75,12.25)(.091292135,-.033707865){178}{\line(1,0){.091292135}}
\multiput(56,6.25)(-.0326087,-.0434783){23}{\line(0,-1){.0434783}}
\multiput(39.75,14.75)(.151442308,.033653846){104}{\line(1,0){.151442308}}
\multiput(54.75,18.25)(-.03125,.125){8}{\line(0,1){.125}}
\put(54.5,19.25){\line(-4,-1){15}}
\put(37,27.75){\line(0,1){0}}
\multiput(8.25,6.25)(.135964912,.033625731){171}{\line(1,0){.135964912}}
\multiput(8,7.25)(.138888889,.033625731){171}{\line(1,0){.138888889}}
\multiput(8.25,7.5)(.03125,-.125){8}{\line(0,-1){.125}}
\multiput(13,16.75)(.22256098,-.03353659){82}{\line(1,0){.22256098}}
\multiput(13.25,18)(.18556701,-.033505155){97}{\line(1,0){.18556701}}
\multiput(13.5,17.75)(-.03125,-.15625){8}{\line(0,-1){.15625}}
\multiput(39,16.75)(.0400641026,.0336538462){312}{\line(1,0){.0400641026}}
\multiput(50.75,28)(.0326087,-.0326087){23}{\line(0,-1){.0326087}}
\multiput(50.75,27.75)(-.0411184211,-.0337171053){304}{\line(-1,0){.0411184211}}
\multiput(40.5,33.25)(-.033505155,-.154639175){97}{\line(0,-1){.154639175}}
\multiput(39.75,33.25)(-.033653846,-.141826923){104}{\line(0,-1){.141826923}}
\multiput(50,25.75)(.033687943,-.04787234){141}{\line(0,-1){.04787234}}
\multiput(50.5,28)(-.091517857,.033482143){112}{\line(-1,0){.091517857}}
\put(40.5,26.75){$\Lambda_{u,u+1}$}
\put(39.75,33.75){\line(1,0){1.25}}
\multiput(40.68,33.43)(.079365,-.031746){9}{\line(1,0){.079365}}
\multiput(42.108,32.858)(.079365,-.031746){9}{\line(1,0){.079365}}
\multiput(43.537,32.287)(.079365,-.031746){9}{\line(1,0){.079365}}
\multiput(44.965,31.715)(.079365,-.031746){9}{\line(1,0){.079365}}
\multiput(45.68,31.43)(-.03125,-.09375){4}{\line(0,-1){.09375}}
\multiput(45.43,30.68)(.0625,-.03125){12}{\line(1,0){.0625}}
\multiput(46.93,29.93)(.0625,-.03125){12}{\line(1,0){.0625}}
\multiput(48.43,29.18)(.0625,-.03125){12}{\line(1,0){.0625}}
\multiput(51.43,27.43)(.030556,-.083333){10}{\line(0,-1){.083333}}
\multiput(52.041,25.763)(.030556,-.083333){10}{\line(0,-1){.083333}}
\multiput(52.652,24.096)(.030556,-.083333){10}{\line(0,-1){.083333}}
\multiput(53.263,22.43)(.030556,-.083333){10}{\line(0,-1){.083333}}
\multiput(53.874,20.763)(.030556,-.083333){10}{\line(0,-1){.083333}}
\put(55.18,17.93){\line(0,-1){.9167}}
\put(55.221,16.096){\line(0,-1){.9167}}
\put(55.263,14.263){\line(0,-1){.9167}}
\put(55.305,12.43){\line(0,-1){.9167}}
\put(55.346,10.596){\line(0,-1){.9167}}
\put(55.388,8.763){\line(0,-1){.9167}}
\put(8.43,7.43){\line(0,1){.9}}
\put(8.33,9.23){\line(0,1){.9}}
\put(8.23,11.03){\line(0,1){.9}}
\put(8.18,11.93){\line(1,0){.75}}
\multiput(9.68,11.93)(.03125,.06875){10}{\line(0,1){.06875}}
\multiput(10.305,13.305)(.03125,.06875){10}{\line(0,1){.06875}}
\multiput(10.93,14.68)(.041667,.033333){12}{\line(1,0){.041667}}
\multiput(11.93,15.48)(.041667,.033333){12}{\line(1,0){.041667}}
\multiput(12.93,16.28)(.041667,.033333){12}{\line(1,0){.041667}}
\multiput(13.43,18.68)(.0318627,.0477941){17}{\line(0,1){.0477941}}
\multiput(14.513,20.305)(.0318627,.0477941){17}{\line(0,1){.0477941}}
\multiput(15.596,21.93)(.0318627,.0477941){17}{\line(0,1){.0477941}}
\multiput(16.68,23.555)(.0318627,.0477941){17}{\line(0,1){.0477941}}
\multiput(17.763,25.18)(.0318627,.0477941){17}{\line(0,1){.0477941}}
\multiput(18.846,26.805)(.0318627,.0477941){17}{\line(0,1){.0477941}}
\multiput(19.93,28.43)(.076446,.030992){11}{\line(1,0){.076446}}
\multiput(21.612,29.112)(.076446,.030992){11}{\line(1,0){.076446}}
\multiput(23.293,29.793)(.076446,.030992){11}{\line(1,0){.076446}}
\multiput(24.975,30.475)(.076446,.030992){11}{\line(1,0){.076446}}
\multiput(26.657,31.157)(.076446,.030992){11}{\line(1,0){.076446}}
\multiput(28.339,31.839)(.076446,.030992){11}{\line(1,0){.076446}}
\multiput(29.18,32.18)(.195455,.031818){5}{\line(1,0){.195455}}
\multiput(31.134,32.498)(.195455,.031818){5}{\line(1,0){.195455}}
\multiput(33.089,32.816)(.195455,.031818){5}{\line(1,0){.195455}}
\multiput(35.043,33.134)(.195455,.031818){5}{\line(1,0){.195455}}
\multiput(36.998,33.452)(.195455,.031818){5}{\line(1,0){.195455}}
\multiput(38.952,33.771)(.195455,.031818){5}{\line(1,0){.195455}}
\multiput(39.93,13.93)(-.0625,.03125){4}{\line(-1,0){.0625}}
\put(40,14){\line(1,0){3.5}}
\put(40,13.75){\line(1,0){.5}}
\put(40,13.5){\line(1,0){3.5}}
\multiput(34.25,20)(.0333333,-.1333333){15}{\line(0,-1){.1333333}}
\multiput(33.5,19.75)(.0333333,-.1166667){15}{\line(0,-1){.1166667}}
\multiput(32.25,19.25)(.0333333,-.05){30}{\line(0,-1){.05}}
\multiput(31.75,18.75)(.0333333,-.0416667){30}{\line(0,-1){.0416667}}
\put(28.75,13.25){\line(1,0){2.5}}
\put(35,13.25){$\pi$}
\put(58.75,6.5){${\cal B}_1$}
\put(40.75,36.25){${\cal B}_{u+1}$}
\put(9.5,19.75){${\cal B}_v$}
\put(51.5,30.5){${\cal B}_u$}
\put(42.5,20){$\Lambda_{u-1,u}$}
\multiput(38.25,10.25)(.03289474,-.04605263){38}{\line(0,-1){.04605263}}
\multiput(37.5,9.75)(.03289474,-.03289474){38}{\line(0,-1){.03289474}}
\put(35.75,9.5){\line(0,-1){2}}
\put(35.25,9.25){\line(0,-1){1.5}}
\multiput(33.25,10.25)(-.03289474,-.03947368){38}{\line(0,-1){.03947368}}
\multiput(32.5,10.25)(-.0333333,-.0333333){30}{\line(0,-1){.0333333}}
\put(49.5,25.5){\line(2,-3){4}}
\multiput(39.5,31.25)(.089285714,-.033730159){126}{\line(1,0){.089285714}}
\put(0,7.75){${\cal B}_{L-3}$}
\put(25.25,23){$\Psi_0$}
\multiput(15.68,15.68)(-.65909,-.56818){12}{{\rule{.4pt}{.4pt}}}
\multiput(19.68,15.68)(-.70833,-.5625){13}{{\rule{.4pt}{.4pt}}}
\multiput(21.93,14.93)(-.65625,-.59375){9}{{\rule{.4pt}{.4pt}}}
\multiput(24.93,14.43)(-.60714,-.5){8}{{\rule{.4pt}{.4pt}}}
\multiput(28.43,14.18)(-.7,-.45){6}{{\rule{.4pt}{.4pt}}}
\multiput(45.68,15.93)(-.71429,-.5){8}{{\rule{.4pt}{.4pt}}}
\multiput(49.43,16.18)(-.75,-.52778){10}{{\rule{.4pt}{.4pt}}}
\multiput(54.43,16.93)(-.75,-.52083){13}{{\rule{.4pt}{.4pt}}}
\multiput(54.68,14.18)(-.725,-.45){11}{{\rule{.4pt}{.4pt}}}
\multiput(55.18,11.43)(-.75,-.39286){8}{{\rule{.4pt}{.4pt}}}
\put(57.25,19.25){${\cal B}_{u-1}$}
\end{picture}

 Let $\Lambda_{u-1,u}$ (let $\Lambda_{u+1,u}$ ) be the trapezium formed by all $\theta$-bands
 starting on ${\cal B}_{u-1}$ (starting on ${\cal B}_{u+1}$, resp.) and ending on ${\cal B}_u.$ It contains $M_4$-accepting subtrapezia with the same histories, and so  there are at
most 3 $(12)$- and $(23)$-bands among them by \ref{vi} and \ref{kt}; similarly for $\Lambda_{v,v+1}$ and $\Lambda_{v,v-1}.$
Since ${\cal B}_u$
(resp. ${\cal B}_v$) must belong to one of trapezia $\Lambda_{u-1,u}$ and $\Lambda_{u+1,u}$
(to one of $\Lambda_{v,v+1}$ and $\Lambda_{v,v-1},$ resp.)
the number of maximal $( (12),t)$- and $((23),t)$-cells in ${\cal B}_u$ (in ${\cal B}_v$), and therefore in any spoke of $\Psi(0),$
is at most $3$, and the number of maximal $(12)$- and $(23)$-bands in $\Psi(0)$
is at most $6$.
 We want to obtain a clove $\Psi(1)=cl(\pi,{\cal B}_{u'},{\cal B}_{v'})$
 ($u'\ge u, v'\le v$) applying one of the following transitions changing the pair $(u,v).$

 (a) If a $(12)$- or $(23)$-band crosses ${\cal B}_u$ but not ${\cal B}_{u+2}$
 (${\cal B}_v$ but not ${\cal B}_{v-2}$), then we set $u'=u+2, v'=v$ ($u'=u, v'=v-2$).

 (b) Notice that either all maximal $\theta$-bands of $\Psi(0)$ cross
 ${\cal B}_u$ or all of them cross  ${\cal B}_v$, or there exists $i$ ($u\le i<v$)
such that the $\theta$-bands crossing ${\cal B}_v$ but  not ${\cal
B}_u$, do not cross ${\cal B}_i$, and the $\theta$-bands crossing
${\cal B}_u$ but not ${\cal B}_v$, do not cross ${\cal B}_{i+1}$. If
in the latter case $i\ge v-2$, then we set $u'=u, v'=v-2$.
Similarly we set $u'=u+2, v'=v$ if $i\le u+1.$

 After a transition (a) or (b), we obtain a
 clove $\Psi(1)=cl(\pi,{\cal B}_{u'},{\cal B}_{v'})$
 with $v'-u'\ge L-k'$, where $k'=k+2.$
 Let us start
 with $u=2$ and $v=L-4$ (i.e., $k=6$) and apply a maximal series of transitions of type (a).
 Since  every transition of type (a) removes a maximal $(12)$- or $(23)$-band and the number of such bands
 in $\Psi(0)$ is at most $6,$ the length
 of the series is also at most $6.$ Then, if possible, we apply a transition of type (b). Note
 that no transition of  types (a) and (b) is applicable to a clove $\Psi(m)$ with $m\le 7$.
  We have
  $k^{(m)}\le 6+2\times 7 =20.$

  It remains to set ${\cal C}_1={\cal B}_{u^{(m)}},$
 and ${\cal C}_l={\cal B}_{v^{(m)}}.$ Then $\Psi=\Psi(m)=cl(\pi,{\cal C}_1,{\cal
C}_l)$ satisfies the conditions (1)-(3) from the definition of crescent. Indeed condition (1)
holds since $k^{(m)}-1\le 19 $ and $L>40,$ condition (2) (condition (3)) holds since  no  transition of type
(b) (of type (a), resp.)  is applicable to $\Psi$.
 The cloves $cl(\pi,{\cal C}_2,{\cal C}_l)$ and
$cl(\pi,{\cal C}_1,{\cal C}_{l-1})$ are also crescents since (1) $20 \le 41 /2$ and (2),(3) no transitions
of type (b) and (a)) are applicable to $\Psi.$
\endproof

\begin{lemma} \label{malot} The number of  maximal $t$- and $t'$-bands
in a crescent $\Psi$ is less then $13LN<J/2.$
\end{lemma}

\proof The number of maximal $t$- and $t'$-bands ending on $\pi$
is less than $2L$. Any other maximal $t$- or $t'$-band $\cal C$ starts and
ends on the subpath ${\bf p}(\Psi)$ of the boundary of $\Delta$ and separates
a subdiagram $\Phi$ from the crescent $\Psi=cl(\pi,\cal B, B')$ such that
$\Phi$ contains $\cal C$ but has no cells from $\cal B$ or $\cal B'$.
However by the definition of crescent, every cell from $\Phi$ is connected
with either $\cal B$ or $\cal B'$ by a $\theta$-band. It follows that
every maximal $\theta$-band of $\Phi$ has to cross $\cal C,$ i.e., $\Phi$
is a subcomb of $\Delta$ with handle $\cal C.$
This subcomb is not
one-Step since $\Delta$ is a solid diagram, and therefore $\cal C$ has
ether $(12)$- or $(23)$-cell. In other words, it crosses
one of the maximal  $\theta$-bands crossing ${\cal C}_1$ or ${\cal
C}_l$ and corresponding to one of the rules $(12),$ $(13).$ The
number of such $\theta$-bands is at most $6$ by Properties (3) and
(2) from the definition of crescent.
Their base widths $<2LN$ by Lemma \ref{psi1}, and so the  number of
maximal $t$- and $t'$-bands $\cal C$ which have no ends on $\pi$, is less than $12LN$. The lemma is
proved because $12LN+2L<13LN.$
\endproof

\subsection{Surgery removing a hub}\label{srh}

When we induct on the number of hubs, we want to cut up a subdiagram
$\Delta_1$ with one hub so that $\area(\Delta_1)$ to be
bounded in `quadratic terms' (as we did earlier for subcombs). The
estimates of Lemmas \ref{case2} and \ref{areaD1} will be applied in the
final Section \ref{aqub}.

In this subsection 
we consider a solid diagram $\Delta$ with the hub
$\pi$ and the crescent $\Psi=cl(\pi,{\cal C}_1,{\cal C}_l)$ provided by
Lemma \ref{narrow}. We will assume that the $t$-spokes ${\cal C}_1,{\cal C}_l$ are
enumerated counter-clockwise with respect to the hub $\pi,$ and the histories $H_1,\dots, H_l$ 
of ${\cal C}_1,{\cal C}_l)$ are read towards $\pi.$ 
We have $\partial\Psi=\topp({\cal C}_1)^{-1}{\bf p}
\bott({\cal C}_l){\bf s}^{-1},$ where ${\bf p}={\bf p}(\Psi)$ and $\bf s$ is a subpath in
$\partial\pi.$ Then $\partial\Delta={\bf pp}_2$. The diagram $\Delta$ is the union of
$\Psi, \pi,$ and the remaining subdiagram $\Psi'$. Let now $h_i=||H_i||,$ $i=1,\dots,l.$

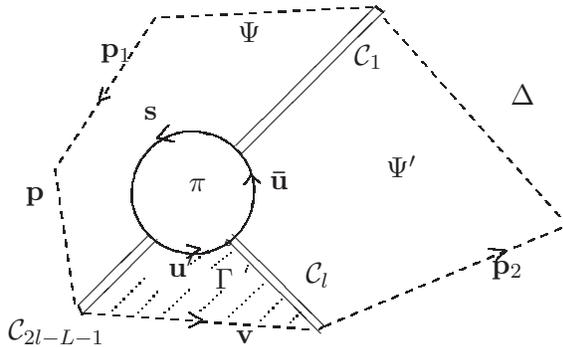
\unitlength 1mm 
\linethickness{0.4pt}
\ifx\plotpoint\undefined\newsavebox{\plotpoint}\fi 
\begin{picture}(77,52.25)(-30,0)
\put(34.393,21.5){\line(0,1){.4554}}
\put(34.381,21.955){\line(0,1){.454}}
\put(34.342,22.409){\line(0,1){.4512}}
\put(34.279,22.861){\line(0,1){.4469}}
\put(34.19,23.307){\line(0,1){.4412}}
\multiput(34.077,23.749)(-.027603,.086841){5}{\line(0,1){.086841}}
\multiput(33.939,24.183)(-.032417,.085161){5}{\line(0,1){.085161}}
\multiput(33.776,24.609)(-.030941,.069345){6}{\line(0,1){.069345}}
\multiput(33.591,25.025)(-.029803,.057863){7}{\line(0,1){.057863}}
\multiput(33.382,25.43)(-.032993,.056105){7}{\line(0,1){.056105}}
\multiput(33.151,25.823)(-.031569,.047401){8}{\line(0,1){.047401}}
\multiput(32.899,26.202)(-.030374,.040499){9}{\line(0,1){.040499}}
\multiput(32.625,26.566)(-.032592,.038737){9}{\line(0,1){.038737}}
\multiput(32.332,26.915)(-.031236,.033168){10}{\line(0,1){.033168}}
\multiput(32.02,27.247)(-.033042,.031369){10}{\line(-1,0){.033042}}
\multiput(31.689,27.56)(-.038606,.032746){9}{\line(-1,0){.038606}}
\multiput(31.342,27.855)(-.040377,.030536){9}{\line(-1,0){.040377}}
\multiput(30.978,28.13)(-.047274,.031759){8}{\line(-1,0){.047274}}
\multiput(30.6,28.384)(-.055973,.033218){7}{\line(-1,0){.055973}}
\multiput(30.208,28.616)(-.057743,.030035){7}{\line(-1,0){.057743}}
\multiput(29.804,28.827)(-.069221,.031218){6}{\line(-1,0){.069221}}
\multiput(29.389,29.014)(-.08503,.032758){5}{\line(-1,0){.08503}}
\multiput(28.964,29.178)(-.086729,.027951){5}{\line(-1,0){.086729}}
\put(28.53,29.318){\line(-1,0){.4408}}
\put(28.089,29.433){\line(-1,0){.4465}}
\put(27.643,29.523){\line(-1,0){.4509}}
\put(27.192,29.589){\line(-1,0){.4538}}
\put(26.738,29.629){\line(-1,0){.9109}}
\put(25.827,29.632){\line(-1,0){.4542}}
\put(25.373,29.596){\line(-1,0){.4514}}
\put(24.922,29.534){\line(-1,0){.4473}}
\put(24.474,29.447){\line(-1,0){.4417}}
\multiput(24.033,29.336)(-.08695,-.027255){5}{\line(-1,0){.08695}}
\multiput(23.598,29.199)(-.08529,-.032075){5}{\line(-1,0){.08529}}
\multiput(23.171,29.039)(-.069469,-.030663){6}{\line(-1,0){.069469}}
\multiput(22.755,28.855)(-.057982,-.029571){7}{\line(-1,0){.057982}}
\multiput(22.349,28.648)(-.056237,-.032768){7}{\line(-1,0){.056237}}
\multiput(21.955,28.419)(-.047527,-.031379){8}{\line(-1,0){.047527}}
\multiput(21.575,28.168)(-.04062,-.030212){9}{\line(-1,0){.04062}}
\multiput(21.209,27.896)(-.038867,-.032436){9}{\line(-1,0){.038867}}
\multiput(20.86,27.604)(-.033293,-.031103){10}{\line(-1,0){.033293}}
\multiput(20.527,27.293)(-.031501,-.032916){10}{\line(0,-1){.032916}}
\multiput(20.212,26.963)(-.032901,-.038474){9}{\line(0,-1){.038474}}
\multiput(19.915,26.617)(-.030698,-.040254){9}{\line(0,-1){.040254}}
\multiput(19.639,26.255)(-.031948,-.047146){8}{\line(0,-1){.047146}}
\multiput(19.384,25.878)(-.033442,-.055839){7}{\line(0,-1){.055839}}
\multiput(19.149,25.487)(-.030266,-.057622){7}{\line(0,-1){.057622}}
\multiput(18.938,25.084)(-.031496,-.069095){6}{\line(0,-1){.069095}}
\multiput(18.749,24.669)(-.033098,-.084898){5}{\line(0,-1){.084898}}
\multiput(18.583,24.244)(-.028298,-.086617){5}{\line(0,-1){.086617}}
\put(18.442,23.811){\line(0,-1){.4403}}
\put(18.325,23.371){\line(0,-1){.4462}}
\put(18.232,22.925){\line(0,-1){.4506}}
\put(18.165,22.474){\line(0,-1){.4537}}
\put(18.123,22.021){\line(0,-1){1.3651}}
\put(18.151,20.655){\line(0,-1){.4517}}
\put(18.211,20.204){\line(0,-1){.4476}}
\put(18.296,19.756){\line(0,-1){.4421}}
\multiput(18.406,19.314)(.03363,-.10882){4}{\line(0,-1){.10882}}
\multiput(18.54,18.879)(.031733,-.085418){5}{\line(0,-1){.085418}}
\multiput(18.699,18.452)(.030384,-.069591){6}{\line(0,-1){.069591}}
\multiput(18.881,18.034)(.029339,-.0581){7}{\line(0,-1){.0581}}
\multiput(19.086,17.627)(.032542,-.056368){7}{\line(0,-1){.056368}}
\multiput(19.314,17.233)(.031188,-.047652){8}{\line(0,-1){.047652}}
\multiput(19.564,16.852)(.030048,-.040741){9}{\line(0,-1){.040741}}
\multiput(19.834,16.485)(.03228,-.038997){9}{\line(0,-1){.038997}}
\multiput(20.125,16.134)(.030969,-.033417){10}{\line(0,-1){.033417}}
\multiput(20.434,15.8)(.03279,-.031633){10}{\line(1,0){.03279}}
\multiput(20.762,15.484)(.038342,-.033055){9}{\line(1,0){.038342}}
\multiput(21.107,15.186)(.040131,-.030859){9}{\line(1,0){.040131}}
\multiput(21.469,14.908)(.047018,-.032137){8}{\line(1,0){.047018}}
\multiput(21.845,14.651)(.055705,-.033665){7}{\line(1,0){.055705}}
\multiput(22.235,14.416)(.0575,-.030497){7}{\line(1,0){.0575}}
\multiput(22.637,14.202)(.068968,-.031772){6}{\line(1,0){.068968}}
\multiput(23.051,14.011)(.084765,-.033438){5}{\line(1,0){.084765}}
\multiput(23.475,13.844)(.086502,-.028645){5}{\line(1,0){.086502}}
\put(23.907,13.701){\line(1,0){.4398}}
\put(24.347,13.582){\line(1,0){.4458}}
\put(24.793,13.488){\line(1,0){.4504}}
\put(25.243,13.419){\line(1,0){.4535}}
\put(25.697,13.376){\line(1,0){.4552}}
\put(26.152,13.357){\line(1,0){.4556}}
\put(26.608,13.365){\line(1,0){.4544}}
\put(27.062,13.397){\line(1,0){.4519}}
\put(27.514,13.455){\line(1,0){.4479}}
\put(27.962,13.539){\line(1,0){.4426}}
\multiput(28.404,13.647)(.10896,.0332){4}{\line(1,0){.10896}}
\multiput(28.84,13.78)(.085544,.031391){5}{\line(1,0){.085544}}
\multiput(29.268,13.937)(.069712,.030105){6}{\line(1,0){.069712}}
\multiput(29.686,14.117)(.058217,.029106){7}{\line(1,0){.058217}}
\multiput(30.094,14.321)(.056498,.032316){7}{\line(1,0){.056498}}
\multiput(30.489,14.547)(.047777,.030997){8}{\line(1,0){.047777}}
\multiput(30.872,14.795)(.045969,.033621){8}{\line(1,0){.045969}}
\multiput(31.239,15.064)(.039126,.032123){9}{\line(1,0){.039126}}
\multiput(31.591,15.353)(.033541,.030835){10}{\line(1,0){.033541}}
\multiput(31.927,15.662)(.031764,.032663){10}{\line(0,1){.032663}}
\multiput(32.244,15.988)(.033208,.038209){9}{\line(0,1){.038209}}
\multiput(32.543,16.332)(.031019,.040007){9}{\line(0,1){.040007}}
\multiput(32.822,16.692)(.032325,.046889){8}{\line(0,1){.046889}}
\multiput(33.081,17.067)(.029652,.048623){8}{\line(0,1){.048623}}
\multiput(33.318,17.456)(.030727,.057378){7}{\line(0,1){.057378}}
\multiput(33.533,17.858)(.032048,.068841){6}{\line(0,1){.068841}}
\multiput(33.726,18.271)(.028148,.070525){6}{\line(0,1){.070525}}
\multiput(33.895,18.694)(.028992,.086387){5}{\line(0,1){.086387}}
\multiput(34.04,19.126)(.03014,.10984){4}{\line(0,1){.10984}}
\put(34.16,19.565){\line(0,1){.4454}}
\put(34.256,20.011){\line(0,1){.4501}}
\put(34.327,20.461){\line(0,1){.4533}}
\put(34.372,20.914){\line(0,1){.5857}}
\put(31,14.75){\circle{.707}}
\put(30.75,15.25){\line(1,-1){12.5}}
\put(31.5,16){\line(1,-1){12.25}}
\multiput(43.75,3.75)(-.05,-.0333333){15}{\line(-1,0){.05}}
\put(32.5,26.5){\line(1,1){19.25}}
\put(32,27.75){\line(1,1){18.5}}
\multiput(50.5,46.25)(.0434783,-.0326087){23}{\line(1,0){.0434783}}
\put(21.5,15.25){\line(-1,-1){10}}
\multiput(11.5,5.25)(-.0333333,.0833333){15}{\line(0,1){.0833333}}
\put(11,6.5){\line(1,1){9.25}}
\put(11.68,5.43){\line(1,0){.9688}}
\put(13.617,5.289){\line(1,0){.9688}}
\put(15.555,5.148){\line(1,0){.9688}}
\put(17.492,5.008){\line(1,0){.9688}}
\put(19.43,4.867){\line(1,0){.9688}}
\put(21.367,4.727){\line(1,0){.9688}}
\put(23.305,4.586){\line(1,0){.9688}}
\put(25.242,4.445){\line(1,0){.9688}}
\put(27.18,4.305){\line(1,0){.9688}}
\put(29.117,4.164){\line(1,0){.9688}}
\put(31.055,4.023){\line(1,0){.9688}}
\put(32.992,3.883){\line(1,0){.9688}}
\put(34.93,3.742){\line(1,0){.9688}}
\put(36.867,3.602){\line(1,0){.9688}}
\put(38.805,3.461){\line(1,0){.9688}}
\put(40.742,3.32){\line(1,0){.9688}}
\multiput(25.75,5.5)(.0583333,-.0333333){30}{\line(1,0){.0583333}}
\multiput(27.5,4.5)(-.0583333,-.0333333){30}{\line(-1,0){.0583333}}
\put(50.18,46.18){\line(-1,0){.9667}}
\put(48.246,46.08){\line(-1,0){.9667}}
\put(46.313,45.98){\line(-1,0){.9667}}
\put(44.38,45.88){\line(-1,0){.9667}}
\put(42.446,45.78){\line(-1,0){.9667}}
\put(40.513,45.68){\line(-1,0){.9667}}
\put(38.58,45.58){\line(-1,0){.9667}}
\put(36.646,45.48){\line(-1,0){.9667}}
\put(34.713,45.38){\line(-1,0){.9667}}
\put(32.78,45.28){\line(-1,0){.9667}}
\put(30.846,45.18){\line(-1,0){.9667}}
\put(28.913,45.08){\line(-1,0){.9667}}
\put(26.98,44.98){\line(-1,0){.9667}}
\put(25.046,44.88){\line(-1,0){.9667}}
\put(23.113,44.78){\line(-1,0){.9667}}
\multiput(21.18,44.68)(-.031851,-.0480769){16}{\line(0,-1){.0480769}}
\multiput(20.16,43.141)(-.031851,-.0480769){16}{\line(0,-1){.0480769}}
\multiput(19.141,41.603)(-.031851,-.0480769){16}{\line(0,-1){.0480769}}
\multiput(18.122,40.064)(-.031851,-.0480769){16}{\line(0,-1){.0480769}}
\multiput(17.103,38.526)(-.031851,-.0480769){16}{\line(0,-1){.0480769}}
\multiput(16.084,36.987)(-.031851,-.0480769){16}{\line(0,-1){.0480769}}
\multiput(15.064,35.449)(-.031851,-.0480769){16}{\line(0,-1){.0480769}}
\multiput(14.045,33.91)(-.031851,-.0480769){16}{\line(0,-1){.0480769}}
\multiput(13.026,32.372)(-.031851,-.0480769){16}{\line(0,-1){.0480769}}
\multiput(12.007,30.834)(-.031851,-.0480769){16}{\line(0,-1){.0480769}}
\multiput(10.987,29.295)(-.031851,-.0480769){16}{\line(0,-1){.0480769}}
\multiput(9.968,27.757)(-.031851,-.0480769){16}{\line(0,-1){.0480769}}
\multiput(8.949,26.218)(-.031851,-.0480769){16}{\line(0,-1){.0480769}}
\multiput(7.93,24.68)(.0275,-.19){5}{\line(0,-1){.19}}
\multiput(8.205,22.78)(.0275,-.19){5}{\line(0,-1){.19}}
\multiput(8.48,20.88)(.0275,-.19){5}{\line(0,-1){.19}}
\multiput(8.755,18.98)(.0275,-.19){5}{\line(0,-1){.19}}
\multiput(9.03,17.08)(.0275,-.19){5}{\line(0,-1){.19}}
\multiput(9.305,15.18)(.0275,-.19){5}{\line(0,-1){.19}}
\multiput(9.58,13.28)(.0275,-.19){5}{\line(0,-1){.19}}
\multiput(9.855,11.38)(.0275,-.19){5}{\line(0,-1){.19}}
\multiput(10.13,9.48)(.0275,-.19){5}{\line(0,-1){.19}}
\multiput(10.405,7.58)(.0275,-.19){5}{\line(0,-1){.19}}
\multiput(23,30.5)(-.03333333,-.03888889){45}{\line(0,-1){.03888889}}
\multiput(21.5,28.75)(.15,-.0333333){15}{\line(1,0){.15}}
\multiput(25.5,14.5)(.1166667,-.0333333){15}{\line(1,0){.1166667}}
\multiput(27.25,14)(-.04605263,-.03289474){38}{\line(-1,0){.04605263}}
\multiput(33.75,23.5)(-.0326087,-.0434783){23}{\line(0,-1){.0434783}}
\put(33,22.5){\line(0,1){0}}
\multiput(33.75,23.5)(.03125,.0625){8}{\line(0,1){.0625}}
\multiput(34,23.75)(.0416667,-.0333333){30}{\line(1,0){.0416667}}
\put(25.75,21.75){$\pi$}
\put(36.5,22){$\bf{\bar u}$}
\put(4,20.5){$\bf p$}
\put(14,39.5){${\bf p}_1$}
\multiput(14.25,35.5)(-.0333333,-.1166667){15}{\line(0,-1){.1166667}}
\multiput(13.75,33.75)(.1,.0333333){15}{\line(1,0){.1}}
\multiput(51.68,44.68)(.0332237,-.0375){19}{\line(0,-1){.0375}}
\multiput(52.942,43.255)(.0332237,-.0375){19}{\line(0,-1){.0375}}
\multiput(54.205,41.83)(.0332237,-.0375){19}{\line(0,-1){.0375}}
\multiput(55.467,40.405)(.0332237,-.0375){19}{\line(0,-1){.0375}}
\multiput(56.73,38.98)(.0332237,-.0375){19}{\line(0,-1){.0375}}
\multiput(57.992,37.555)(.0332237,-.0375){19}{\line(0,-1){.0375}}
\multiput(59.255,36.13)(.0332237,-.0375){19}{\line(0,-1){.0375}}
\multiput(60.517,34.705)(.0332237,-.0375){19}{\line(0,-1){.0375}}
\multiput(61.78,33.28)(.0332237,-.0375){19}{\line(0,-1){.0375}}
\multiput(63.042,31.855)(.0332237,-.0375){19}{\line(0,-1){.0375}}
\multiput(64.305,30.43)(.0332237,-.0375){19}{\line(0,-1){.0375}}
\multiput(65.567,29.005)(.0332237,-.0375){19}{\line(0,-1){.0375}}
\multiput(66.83,27.58)(.0332237,-.0375){19}{\line(0,-1){.0375}}
\multiput(68.092,26.155)(.0332237,-.0375){19}{\line(0,-1){.0375}}
\multiput(69.355,24.73)(.0332237,-.0375){19}{\line(0,-1){.0375}}
\multiput(70.617,23.305)(.0332237,-.0375){19}{\line(0,-1){.0375}}
\multiput(71.88,21.88)(.0332237,-.0375){19}{\line(0,-1){.0375}}
\multiput(73.142,20.455)(.0332237,-.0375){19}{\line(0,-1){.0375}}
\multiput(74.405,19.03)(.0332237,-.0375){19}{\line(0,-1){.0375}}
\multiput(75.667,17.605)(.0332237,-.0375){19}{\line(0,-1){.0375}}
\multiput(43.68,3.68)(.075231,.03125){12}{\line(1,0){.075231}}
\multiput(45.485,4.43)(.075231,.03125){12}{\line(1,0){.075231}}
\multiput(47.291,5.18)(.075231,.03125){12}{\line(1,0){.075231}}
\multiput(49.096,5.93)(.075231,.03125){12}{\line(1,0){.075231}}
\multiput(50.902,6.68)(.075231,.03125){12}{\line(1,0){.075231}}
\multiput(52.707,7.43)(.075231,.03125){12}{\line(1,0){.075231}}
\multiput(54.513,8.18)(.075231,.03125){12}{\line(1,0){.075231}}
\multiput(56.319,8.93)(.075231,.03125){12}{\line(1,0){.075231}}
\multiput(58.124,9.68)(.075231,.03125){12}{\line(1,0){.075231}}
\multiput(59.93,10.43)(.075231,.03125){12}{\line(1,0){.075231}}
\multiput(61.735,11.18)(.075231,.03125){12}{\line(1,0){.075231}}
\multiput(63.541,11.93)(.075231,.03125){12}{\line(1,0){.075231}}
\multiput(65.346,12.68)(.075231,.03125){12}{\line(1,0){.075231}}
\multiput(67.152,13.43)(.075231,.03125){12}{\line(1,0){.075231}}
\multiput(68.957,14.18)(.075231,.03125){12}{\line(1,0){.075231}}
\multiput(70.763,14.93)(.075231,.03125){12}{\line(1,0){.075231}}
\multiput(72.569,15.68)(.075231,.03125){12}{\line(1,0){.075231}}
\multiput(74.374,16.43)(.075231,.03125){12}{\line(1,0){.075231}}
\multiput(65.5,14)(.3125,-.03125){8}{\line(1,0){.3125}}
\multiput(68,13.75)(-.03947368,-.03289474){38}{\line(-1,0){.03947368}}
\put(66,10.75){${\bf p}_2$}
\put(23,10.25){$\bf u'$}
\put(19.75,31.25){$\bf s$}
\put(32.25,41.5){$\Psi$}
\put(47.75,38.5){${\cal C}_1$}
\put(41.25,9.5){${\cal C}_l$}
\put(29.25,8.75){$\Gamma$}
\multiput(15.43,5.68)(.53125,.53125){9}{{\rule{.4pt}{.4pt}}}
\multiput(18.93,5.43)(.5,.55){6}{{\rule{.4pt}{.4pt}}}
\multiput(25.93,11.93)(.5625,.4375){5}{{\rule{.4pt}{.4pt}}}
\multiput(22.43,5.43)(.54167,.5){7}{{\rule{.4pt}{.4pt}}}
\multiput(28.68,11.43)(.5,.4){6}{{\rule{.4pt}{.4pt}}}
\multiput(27.18,5.68)(.55,.5){6}{{\rule{.4pt}{.4pt}}}
\multiput(33.18,10.93)(.125,.375){3}{{\rule{.4pt}{.4pt}}}
\multiput(30.93,4.93)(.5625,.53125){9}{{\rule{.4pt}{.4pt}}}
\multiput(35.18,4.43)(.55,.55){6}{{\rule{.4pt}{.4pt}}}
\multiput(39.18,4.18)(.4167,.4167){4}{{\rule{.4pt}{.4pt}}}
\put(52,23.5){$\Psi '$}
\put(68.5,33.25){$\Delta$}
\put(1.75,2){${\cal C}_{2l-L-1}$}
\put(32,1.5){$\bf v$}
\end{picture}

\medskip

{\bf Without loss of generality, we will assume further that $h_1\ge
h_l$ for $\Psi.$} Under this assumption, we will use the following
special surgery for $\Delta$. Denote by $e_j$ the common $t$-edge of $\partial {\cal C}_j$  and $\partial\pi$. Consider the reduced subpath $e_{2l-L-1}{\bf u'}e_{l}$ of ${\bf s}.$ Denote by $\Gamma$ the
subdiagram without hubs bounded by  $\bott({\cal C}_{2l-L-1})^{-1}({\bf v})
\topp({\cal C}_{l})({\bf u'})^{-1}$, where $\bf v$ is a subpath of
$\partial\Delta$. We have ${\bf p}(\Psi)={\bf p}={\bf p}_1 {\bf v}f$ for some ${\bf p}_1,$
where $f$ is the common edge of $\partial{\cal C}_l$ end $\partial\Delta.$

There is a reduced path $e_1^{-1}{\bf\bar u} e_{l}^{-1}$
, where $({\bf\bar u})^{-1}$ is a subpath of $\partial\pi$.
Then  the paths ${\bf w}_1=\topp({\cal C}_{l})({\bf u'})^{-1}$ is obtained from   ${\bf w}=\bott({\cal
C}_{l})\bar {\bf u}^{-1}$ by a $t_{i}$-reflection since ${\cal C}_l$ is a $t_{i}$-band for some $i$ (see definitions in Remark \ref{diasym}). 
Therefore the following surgery is possible.

(1) Cut $\Delta$ along $\bf w$.

(2) Construct a diagram $\Gamma_1$ obtained from $\Gamma$ by the $t_i$-reflection
(see Remark \ref{diasym}) and take a standard mirror copy
$\Gamma_2$ of $\Gamma_1$ (where the mirror edges have equal labels).  Glue $\Gamma_1$ and $\Gamma_2$ together along the path $\bf r$
obtained by the $t_{i}$-reflection from $\bott({\cal C}_{2l-L-1})^{-1}\bf v$, and
obtain a diagram $\Pi$ with boundary $\bf w'w''$, where $\Lab ({\bf w'})\equiv
\Lab (({\bf w''})^{-1})\equiv \Lab ({\bf w}^{-1})$.

\unitlength 1mm 
\linethickness{0.4pt}
\ifx\plotpoint\undefined\newsavebox{\plotpoint}\fi 
\begin{picture}(105,51)(-10,0)
\put(32,23.5){\circle{11.597}}
\put(35.25,18.75){\circle{.5}}
\multiput(35.25,19)(.033536585,-.057926829){164}{\line(0,-1){.057926829}}
\multiput(40.75,9.5)(-.076087,-.0326087){23}{\line(-1,0){.076087}}
\multiput(39,8.75)(-.033557047,.062080537){149}{\line(0,1){.062080537}}
\multiput(36,28)(.0336879433,.0425531915){282}{\line(0,1){.0425531915}}
\multiput(45.5,40)(-.076087,.0326087){23}{\line(-1,0){.076087}}
\multiput(43.75,40.75)(-.0336538462,-.0451923077){260}{\line(0,-1){.0451923077}}
\multiput(27,20.25)(-.0460199,-.03358209){201}{\line(-1,0){.0460199}}
\multiput(17.75,13.5)(-.0333333,.0583333){30}{\line(0,1){.0583333}}
\multiput(16.75,15.25)(.049731183,.033602151){186}{\line(1,0){.049731183}}
\multiput(44.25,28.25)(.0326087,-.1521739){23}{\line(0,-1){.1521739}}
\multiput(45,24.75)(-.03358209,-.0858209){67}{\line(0,-1){.0858209}}
\multiput(42.75,19)(.033536585,-.056402439){164}{\line(0,-1){.056402439}}
\put(48.25,9.75){\line(0,1){0}}
\multiput(44.25,28.75)(.033653846,-.052884615){156}{\line(0,-1){.052884615}}
\put(48.75,21.25){\line(0,-1){11.75}}
\multiput(60.75,29)(-.033653846,-.049679487){156}{\line(0,-1){.049679487}}
\multiput(61,29)(.03333333,-.06666667){45}{\line(0,-1){.06666667}}
\put(62.5,26.25){\line(0,-1){6.25}}
\multiput(62.25,21)(.03333333,-.13666667){75}{\line(0,-1){.13666667}}
\multiput(55.25,21.25)(.0337078652,-.0393258427){267}{\line(0,-1){.0393258427}}
\multiput(70,28.75)(.03333333,-.06666667){45}{\line(0,-1){.06666667}}
\put(71.5,25.75){\line(0,-1){4.5}}
\multiput(71.5,21.25)(.033625731,-.052631579){171}{\line(0,-1){.052631579}}
\multiput(77.25,12.25)(.0333333,-.05){15}{\line(0,-1){.05}}
\multiput(70,29)(.0344827586,.0336990596){319}{\line(1,0){.0344827586}}
\multiput(80.93,39.18)(.0435924,-.0336134){17}{\line(1,0){.0435924}}
\multiput(82.412,38.037)(.0435924,-.0336134){17}{\line(1,0){.0435924}}
\multiput(83.894,36.894)(.0435924,-.0336134){17}{\line(1,0){.0435924}}
\multiput(85.376,35.751)(.0435924,-.0336134){17}{\line(1,0){.0435924}}
\multiput(86.858,34.608)(.0435924,-.0336134){17}{\line(1,0){.0435924}}
\multiput(88.34,33.465)(.0435924,-.0336134){17}{\line(1,0){.0435924}}
\multiput(89.823,32.323)(.0435924,-.0336134){17}{\line(1,0){.0435924}}
\multiput(91.305,31.18)(.0435924,-.0336134){17}{\line(1,0){.0435924}}
\multiput(92.787,30.037)(.0435924,-.0336134){17}{\line(1,0){.0435924}}
\multiput(94.269,28.894)(.0435924,-.0336134){17}{\line(1,0){.0435924}}
\multiput(95.751,27.751)(.0435924,-.0336134){17}{\line(1,0){.0435924}}
\multiput(97.233,26.608)(.0435924,-.0336134){17}{\line(1,0){.0435924}}
\multiput(98.715,25.465)(.0435924,-.0336134){17}{\line(1,0){.0435924}}
\multiput(100.198,24.323)(.0435924,-.0336134){17}{\line(1,0){.0435924}}
\multiput(101.68,23.18)(-.0676638,-.0327635){13}{\line(-1,0){.0676638}}
\multiput(99.92,22.328)(-.0676638,-.0327635){13}{\line(-1,0){.0676638}}
\multiput(98.161,21.476)(-.0676638,-.0327635){13}{\line(-1,0){.0676638}}
\multiput(96.402,20.624)(-.0676638,-.0327635){13}{\line(-1,0){.0676638}}
\multiput(94.643,19.772)(-.0676638,-.0327635){13}{\line(-1,0){.0676638}}
\multiput(92.883,18.92)(-.0676638,-.0327635){13}{\line(-1,0){.0676638}}
\multiput(91.124,18.069)(-.0676638,-.0327635){13}{\line(-1,0){.0676638}}
\multiput(89.365,17.217)(-.0676638,-.0327635){13}{\line(-1,0){.0676638}}
\multiput(87.606,16.365)(-.0676638,-.0327635){13}{\line(-1,0){.0676638}}
\multiput(85.846,15.513)(-.0676638,-.0327635){13}{\line(-1,0){.0676638}}
\multiput(84.087,14.661)(-.0676638,-.0327635){13}{\line(-1,0){.0676638}}
\multiput(82.328,13.809)(-.0676638,-.0327635){13}{\line(-1,0){.0676638}}
\multiput(80.569,12.957)(-.0676638,-.0327635){13}{\line(-1,0){.0676638}}
\multiput(78.809,12.106)(-.0676638,-.0327635){13}{\line(-1,0){.0676638}}
\put(12.18,50.68){\line(0,1){.125}}
\multiput(44.18,40.68)(-.192105,-.028947){5}{\line(-1,0){.192105}}
\multiput(42.259,40.39)(-.192105,-.028947){5}{\line(-1,0){.192105}}
\multiput(40.338,40.101)(-.192105,-.028947){5}{\line(-1,0){.192105}}
\multiput(38.417,39.811)(-.192105,-.028947){5}{\line(-1,0){.192105}}
\multiput(36.495,39.522)(-.192105,-.028947){5}{\line(-1,0){.192105}}
\multiput(34.574,39.232)(-.192105,-.028947){5}{\line(-1,0){.192105}}
\multiput(32.653,38.943)(-.192105,-.028947){5}{\line(-1,0){.192105}}
\multiput(30.732,38.653)(-.192105,-.028947){5}{\line(-1,0){.192105}}
\multiput(28.811,38.364)(-.192105,-.028947){5}{\line(-1,0){.192105}}
\multiput(26.89,38.074)(-.192105,-.028947){5}{\line(-1,0){.192105}}
\multiput(25.93,37.93)(-.0333333,-.0441176){17}{\line(0,-1){.0441176}}
\multiput(24.796,36.43)(-.0333333,-.0441176){17}{\line(0,-1){.0441176}}
\multiput(23.663,34.93)(-.0333333,-.0441176){17}{\line(0,-1){.0441176}}
\multiput(22.53,33.43)(-.0333333,-.0441176){17}{\line(0,-1){.0441176}}
\multiput(21.396,31.93)(-.0333333,-.0441176){17}{\line(0,-1){.0441176}}
\multiput(20.263,30.43)(-.0333333,-.0441176){17}{\line(0,-1){.0441176}}
\multiput(19.13,28.93)(-.0333333,-.0441176){17}{\line(0,-1){.0441176}}
\multiput(17.996,27.43)(-.0333333,-.0441176){17}{\line(0,-1){.0441176}}
\put(17.43,26.68){\line(0,-1){.9792}}
\put(17.263,24.721){\line(0,-1){.9792}}
\put(17.096,22.763){\line(0,-1){.9792}}
\put(16.93,20.805){\line(0,-1){.9792}}
\put(16.763,18.846){\line(0,-1){.9792}}
\put(16.596,16.888){\line(0,-1){.9792}}
\multiput(18.18,13.18)(.155303,-.032197){6}{\line(1,0){.155303}}
\multiput(20.043,12.793)(.155303,-.032197){6}{\line(1,0){.155303}}
\multiput(21.907,12.407)(.155303,-.032197){6}{\line(1,0){.155303}}
\multiput(23.771,12.021)(.155303,-.032197){6}{\line(1,0){.155303}}
\multiput(25.634,11.634)(.155303,-.032197){6}{\line(1,0){.155303}}
\multiput(27.498,11.248)(.155303,-.032197){6}{\line(1,0){.155303}}
\multiput(29.362,10.862)(.155303,-.032197){6}{\line(1,0){.155303}}
\multiput(31.225,10.475)(.155303,-.032197){6}{\line(1,0){.155303}}
\multiput(33.089,10.089)(.155303,-.032197){6}{\line(1,0){.155303}}
\multiput(34.952,9.702)(.155303,-.032197){6}{\line(1,0){.155303}}
\multiput(36.816,9.316)(.155303,-.032197){6}{\line(1,0){.155303}}
\put(45.25,18){$\Gamma_1$}
\put(59,21){$\Gamma_2$}
\put(83,25.5){$\Psi'$}
\put(29.75,34){$\Psi$}
\put(15.25, 30){${\bf p}_1$}
\multiput(90.25,18.75)(.1333333,-.0333333){15}{\line(1,0){.1333333}}
\multiput(92.25,18.25)(-.03289474,-.03289474){38}{\line(0,-1){.03289474}}
\put(95.5,17.5){${\bf p}_2$}
\put(31,23.5){$\pi$}
\put(10.5,11.75){${\cal C}_{2l-L-1}$}
\put(30.5,07.25){$\bf v$}
\put(47.25,42){${\cal C}_1$}
\put(41.5,7.75){${\cal C}_l$}
\put(38,18){$\bf w$}
\put(43.25,14.5){$\bf w'$}
\put(64,20.5){$\bf w''$}
\put(28.5,14.75){$\Gamma$}
\put(52.75,20.5){$r$}
\put(58.25,14.25){${\bf v}_1$}
\put(65,33){\vector(1,0){.07}}\put(44.75,33.25){\vector(-1,0){.07}}\multiput(44.68,33.18)(.96429,-.0119){22}{{\rule{.4pt}{.4pt}}}
\put(54.5,34.5){$\Pi$}
\put(51,3.75){\vector(1,0){.07}}\put(16.5,3.75){\vector(-1,0){.07}}\multiput(16.43,3.68)(.985714,0){36}{{\rule{.4pt}{.4pt}}}
\put(105,4){\vector(1,0){.07}}\put(54.75,3.75){\vector(-1,0){.07}}\multiput(54.68,3.68)(.985294,.004902){52}{{\rule{.4pt}{.4pt}}}
\put(33,1.25){$\Delta_1$}
\put(78,1.25){$\Delta_2$}
\put(71.5,16.75){$\bf w$}
\multiput(17.93,26.68)(.69737,.60526){20}{{\rule{.4pt}{.4pt}}}
\multiput(17.18,21.68)(.70536,.64286){29}{{\rule{.4pt}{.4pt}}}
\multiput(16.93,17.18)(.65385,.61538){14}{{\rule{.4pt}{.4pt}}}
\multiput(25.43,25.18)(.68478,.6413){24}{{\rule{.4pt}{.4pt}}}
\multiput(19.18,13.68)(.66667,.65){16}{{\rule{.4pt}{.4pt}}}
\multiput(32.43,26.43)(.64583,.60417){13}{{\rule{.4pt}{.4pt}}}
\multiput(22.93,12.68)(.7,.6){21}{{\rule{.4pt}{.4pt}}}
\multiput(28.68,13.18)(.6875,.53125){9}{{\rule{.4pt}{.4pt}}}
\multiput(31.18,10.68)(.6875,.5625){9}{{\rule{.4pt}{.4pt}}}
\multiput(34.93,10.18)(.7,.55){6}{{\rule{.4pt}{.4pt}}}
\multiput(44.18,22.43)(.45,.55){6}{{\rule{.4pt}{.4pt}}}
\multiput(47.68,22.43)(-.4,-.4){6}{{\rule{.4pt}{.4pt}}}
\multiput(48.18,17.43)(-.45,-.55){6}{{\rule{.4pt}{.4pt}}}
\multiput(47.93,13.93)(-.5,-.4167){4}{{\rule{.4pt}{.4pt}}}
\multiput(46.43,12.68)(0,0){3}{{\rule{.4pt}{.4pt}}}
\multiput(59.43,23.43)(.5,.55){6}{{\rule{.4pt}{.4pt}}}
\multiput(57.93,18.68)(.75,.55){6}{{\rule{.4pt}{.4pt}}}
\multiput(60.43,16.93)(.5,.5){4}{{\rule{.4pt}{.4pt}}}
\multiput(61.18,13.68)(.4,.45){6}{{\rule{.4pt}{.4pt}}}
\multiput(63.43,12.18)(.25,.375){3}{{\rule{.4pt}{.4pt}}}
\multiput(71.68,26.68)(.64706,.61765){18}{{\rule{.4pt}{.4pt}}}
\multiput(71.68,22.68)(.65476,.61905){22}{{\rule{.4pt}{.4pt}}}
\multiput(73.18,20.18)(.65,.65){11}{{\rule{.4pt}{.4pt}}}
\multiput(82.93,29.43)(.5,.53125){9}{{\rule{.4pt}{.4pt}}}
\multiput(74.18,17.68)(.66304,.63043){24}{{\rule{.4pt}{.4pt}}}
\multiput(75.43,15.43)(.70652,.6413){24}{{\rule{.4pt}{.4pt}}}
\multiput(76.93,13.68)(.70833,.60417){25}{{\rule{.4pt}{.4pt}}}
\multiput(79.43,12.93)(.73913,.59783){24}{{\rule{.4pt}{.4pt}}}
\multiput(93.68,20.93)(.59375,.53125){9}{{\rule{.4pt}{.4pt}}}
\multiput(98.93,22.43)(.5,.5){4}{{\rule{.4pt}{.4pt}}}
\end{picture}

\medskip

(3) Insert $\Pi$ in the hole of $\Delta$ obtained after step (1).

(4) Cut up the obtained disc diagram along $\topp ({\cal C}_1)\bf r$, and
obtain two diagrams $\Delta_1$ and $\Delta_2$, where $\Delta_1$ is a
minimal diagram with the same boundary label as the union of $\Psi,
\pi$ and $\Gamma_1$, and $\Delta_2$ is a union of $\Psi'$ and
$\Gamma_2$.

(5) Let $H_0$ be the history of the maximal trapezium
bounded by ${\cal C}_1$ and  ${\cal C}_{2l-L-1}$ in $\Psi$ (it is the filling trapezium
$Tp({\cal C}_{2l-L-1}, {\cal C}_1)$ if every maximal $\theta$-band
crossing ${\cal C}_{2l-L-1}$ also crosses ${\cal C}_1$),
and so $H_0$ is a suffix of
both $H_1$ and $H_{2l-L-1}.$ Therefore $2h_0=2||H_0||$ letters can be
canceled in the product $\Lab(\topp ({\cal C}_1))\Lab({\bf r}).$
And so we shorten the corresponding part of the boundary
 of $\Delta_2$ by $2h_0$ edges
 and replace the obtained diagram by
 a minimal diagram $\Delta'$.
 \medskip

 Thus the boundary of $\Delta'$
 is ${\bf p}_3{\bf p}_2$, where ${\bf p}_3=\bf xv'$, $\bf v'$ is obtained by the $t_{i}$-reflection of $\bf v,$ and
 $|{\bf x}|=|{\bf x}|_{\theta}=h_1+h_{2l-L-1}-2h_0.$
 Since the path ${\bf p}_1$ has at least $(h_1-h_0)+(h_{2l-L+1}-h_0)$
 $\theta$-edges (the ends of maximal $\theta$-bands which cross ${\cal C}_1$ but do not cross ${\cal C}_{2l-L-1},$
 and vice versa) and also has $q$-edges,  we have

 \begin{equation}\label{yp}
 |{\bf x}|=|{\bf x}|_{\theta}\le |{\bf p}_1|_{\theta}\le |{\bf p}_1|-1
 \end{equation}
 Moreover using the maximal $\theta$-bands crossing ${\cal C}_1$ and ${\cal C}_l$ in  the crescent $\Psi,$
 one can for every $\theta$-edges of $\bf x,$ find a $\theta$-edge of ${\bf p}_1$ corresponding to the
 same rule $\theta^{\pm 1},$ and the obtaining mapping from the set of (non-oriented) edges
 of $\bf x$ to the set of edges of ${\bf p}_1$ is injective.

\medskip

\begin{lemma}\label{Delta'} With the preceding notation, we have (a) $|\partial\Delta'|\le
 |\partial\Delta|-1$; (b) $\kappa(\Delta')\le \kappa(\Delta);$
 (c) $\lambda(\Delta')\le \lambda(\Delta)+2h_1^2;$ (d) $\nu_J(\Delta')\le \nu_J(\Delta).$

\end{lemma}
 \proof (a)  Since $|{\bf v'}|=|{\bf v}|$, the statement (a) follows from Inequality (\ref{yp}) and Lemma \ref{ochev} (b).

 (b) Recall that $\bf v'$ is constructed as the $t_i$-reflection of $\bf v.$ Thus, when passing
 from the boundary label of $\Delta$ to the boundary label of $\Delta',$
we, in essence, just replace  $\Lab({\bf p}_1)$ by $\Lab(\bf x).$ But $\bf x$ has no
$q$-edges, and so it has no black beads (see the definition of the $\kappa$-mixture of a diagram), and the number of white beads of $\bf x$ is at most the number of white beads on ${\bf p}_1$ by (\ref{yp}). Therefore $\kappa(\Delta')\le \kappa(\Delta)$  by Lemma \ref{mixture}
 (Parts b,c).

(d) Similarly, using the fact that the path $\bf x$ has no $t$- or $t'$-edges
  one concludes  that $\nu_J(\Delta')\le \nu_J(\Delta).$

  (c) The remark made before the formulation of the lemma, allows us to obtain
  an injective mapping from the set of white beads of the $\lambda$-necklace $O'$ for $\partial\Delta'$
  to the set of white beads of the $\lambda$-necklace $O$ on $\partial\Delta,$ so that the beads from $\bf x$ map to
  the beads on ${\bf p}_1.$  It follows that if $o,o'$ are two white beads from $O',$ but not both on $\bf x,$ and
  they are separated  by a black bead in $O',$ then the corresponding white beads of $O$ are also separated
  by a black bead. (We take into account that ${\bf p}_1$ starts and ends with $q$-edges having black beads by
  the definition of the necklace $O$.)
  Therefore to estimate the difference $\lambda(\Delta)-\lambda(\Delta')$ from below, we may consider
  only the pairs of white beads of $O',$ where both $o$ and $o'$ belong to $\bf x,$ whence
  $\lambda(\Delta)-\lambda(\Delta')\ge -\lambda(\bf x).$
    By Lemma \ref{mixturec} (a), $\lambda({\bf x})< |{\bf x}|^2/2 \le (2h_1)^2/2$
  and claim (c) is proved.
\endproof

\begin{rk}\label{anyclove} (1) The surgery described before the
formulation of Lemma \ref{Delta'}, can be also done for the original
clove $cl(\pi,{\cal B}_1, {\cal B}_{L-3})$ even if one does not assume
that $\Delta$ is a solid diagram. In this case again, exactly as in the
proof of Lemma \ref{Delta'}, we  obtain the inequality
$|\partial\Delta'|\le
 |\partial\Delta|-1.$

 (2) Assume that $\Gamma$ is a subcomb of a diagram $\Delta,$ and the  handle of $\Gamma$
 is a $t^{\pm 1}$
 or $(t')^{\pm 1}$-band, ${\bf y}={\bf y}^{\Gamma},$ and $\Delta'=\Delta\backslash\Gamma.$ Then by Lemma \ref{NoAnnul},
 we have a preserving the order bijective mapping from the set of $\theta$-edges of ${\bf y}^{-1}$ to
 the set of the $\theta$-edges of ${\bf z}={\bf z}^{\Gamma}.$ Then arguing exactly
 as in the proof of Part (c) of Lemma \ref{Delta'}, we get
 $\lambda(\Delta)-\lambda(\Delta')\ge -\lambda({\bf y})> -|{\bf y}|^2/2.$ Hence
 $$\mu(\Delta)-\mu(\Delta')> -|{\bf y}|^2/2,$$ by the definition of $\mu(.)$ and Lemmas
 \ref{positive} (a) and \ref{mu} (a).
 \end{rk}

 \begin{lemma}\label{case2} Assume that $n=|\partial\Delta|,$  $h_1$ does not
belong to any interval $(T_i,9T_i)$ ($i=1,2,\dots$,) and $h_2>
(1-\frac{1}{30N})h_1$.
 Then, with the notation of Lemma \ref{Delta'}, we have
 $$\area(\Delta_1) \le c_4(\mu(\Delta)- \mu(\Delta')) +
 c_5(\nu_J(\Delta)- \nu_J(\Delta')) +c_6n (n-|\partial\Delta'|)$$

 \end{lemma}

 \proof Since $h_1=\max_{i=1}^l h_i,$ the condition (2) from the definition of crescent implies
that every maximal $\theta$-band crossing the $t^{\pm 1}$-band ${\cal C}_2$ in the crescent $\Psi,$
has to cross ${\cal C}_1$ as well. Therefore we can consider the trapezium $\Lambda_{12}=Tp({\cal C}_2, {\cal C}_1)$ of height $h_2$ between ${\cal C}_1$ and ${\cal C}_2.$
The bottom path ${\bf z}_{12}$ of $\Lambda_{12}$, must be of $a$-length at
least $h_2/6$ by \ref{xv}, since $h_2$ does not belong to any
interval $(T_i,9T_i).$

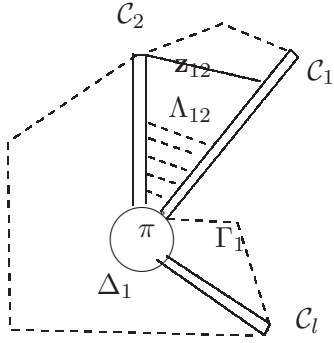
\unitlength 1mm 
\linethickness{0.4pt}
\ifx\plotpoint\undefined\newsavebox{\plotpoint}\fi 
\begin{picture}(50.75,53)(-25,0)
\put(28.75,16.25){\circle{8.246}}
\put(31.5,19.75){\circle{.5}}
\multiput(31.25,20)(.0336914063,.0419921875){512}{\line(0,1){.0419921875}}
\multiput(48.5,41.5)(.0333333,-.0333333){30}{\line(0,-1){.0333333}}
\multiput(49.5,40.5)(-.0337186898,-.0414258189){519}{\line(0,-1){.0414258189}}
\put(29.25,40.75){\line(0,-1){19.75}}
\put(27.75,40.75){\line(1,0){1.25}}
\put(27.5,41){\line(0,-1){20.25}}
\multiput(29.25,40.75)(.146634615,-.033653846){104}{\line(1,0){.146634615}}
\multiput(47.93,41.43)(-.077273,.031818){11}{\line(-1,0){.077273}}
\multiput(46.23,42.13)(-.077273,.031818){11}{\line(-1,0){.077273}}
\multiput(44.53,42.83)(-.077273,.031818){11}{\line(-1,0){.077273}}
\multiput(42.83,43.53)(-.077273,.031818){11}{\line(-1,0){.077273}}
\multiput(41.13,44.23)(-.077273,.031818){11}{\line(-1,0){.077273}}
\multiput(39.43,44.93)(-.0875,-.033333){10}{\line(-1,0){.0875}}
\multiput(37.68,44.263)(-.0875,-.033333){10}{\line(-1,0){.0875}}
\multiput(35.93,43.596)(-.0875,-.033333){10}{\line(-1,0){.0875}}
\multiput(34.18,42.93)(-.0875,-.033333){10}{\line(-1,0){.0875}}
\multiput(32.43,42.263)(-.0875,-.033333){10}{\line(-1,0){.0875}}
\multiput(30.68,41.596)(-.0875,-.033333){10}{\line(-1,0){.0875}}
\multiput(27.93,40.68)(-.0482955,-.0326705){16}{\line(-1,0){.0482955}}
\multiput(26.384,39.634)(-.0482955,-.0326705){16}{\line(-1,0){.0482955}}
\multiput(24.839,38.589)(-.0482955,-.0326705){16}{\line(-1,0){.0482955}}
\multiput(23.293,37.543)(-.0482955,-.0326705){16}{\line(-1,0){.0482955}}
\multiput(21.748,36.498)(-.0482955,-.0326705){16}{\line(-1,0){.0482955}}
\multiput(20.202,35.452)(-.0482955,-.0326705){16}{\line(-1,0){.0482955}}
\multiput(18.657,34.407)(-.0482955,-.0326705){16}{\line(-1,0){.0482955}}
\multiput(17.112,33.362)(-.0482955,-.0326705){16}{\line(-1,0){.0482955}}
\multiput(15.566,32.316)(-.0482955,-.0326705){16}{\line(-1,0){.0482955}}
\multiput(14.021,31.271)(-.0482955,-.0326705){16}{\line(-1,0){.0482955}}
\multiput(12.475,30.225)(-.0482955,-.0326705){16}{\line(-1,0){.0482955}}
\put(10.93,29.18){\line(0,-1){.97}}
\put(10.95,27.24){\line(0,-1){.97}}
\put(10.97,25.3){\line(0,-1){.97}}
\put(10.99,23.36){\line(0,-1){.97}}
\put(11.01,21.42){\line(0,-1){.97}}
\put(11.03,19.48){\line(0,-1){.97}}
\put(11.05,17.54){\line(0,-1){.97}}
\put(11.07,15.6){\line(0,-1){.97}}
\put(11.09,13.66){\line(0,-1){.97}}
\put(11.11,11.72){\line(0,-1){.97}}
\put(11.13,9.78){\line(0,-1){.97}}
\put(11.15,7.84){\line(0,-1){.97}}
\put(11.17,5.9){\line(0,-1){.97}}
\multiput(32,14)(.0514981273,-.0337078652){267}{\line(1,0){.0514981273}}
\multiput(45.75,5)(-.0326087,-.0652174){23}{\line(0,-1){.0652174}}
\multiput(45,3.5)(-.047979798,.0336700337){297}{\line(-1,0){.047979798}}
\put(32.18,18.93){\line(1,0){.925}}
\put(34.03,18.83){\line(1,0){.925}}
\put(35.88,18.73){\line(1,0){.925}}
\put(37.73,18.63){\line(1,0){.925}}
\put(39.58,18.53){\line(1,0){.925}}
\multiput(41.43,18.43)(.031481,-.1){9}{\line(0,-1){.1}}
\multiput(41.996,16.63)(.031481,-.1){9}{\line(0,-1){.1}}
\multiput(42.563,14.83)(.031481,-.1){9}{\line(0,-1){.1}}
\multiput(43.13,13.03)(.031481,-.1){9}{\line(0,-1){.1}}
\multiput(43.696,11.23)(.031481,-.1){9}{\line(0,-1){.1}}
\multiput(44.263,9.43)(.031481,-.1){9}{\line(0,-1){.1}}
\multiput(44.83,7.63)(.031481,-.1){9}{\line(0,-1){.1}}
\multiput(45.396,5.83)(.031481,-.1){9}{\line(0,-1){.1}}
\put(11.43,4.18){\line(1,0){.9706}}
\put(13.371,4.136){\line(1,0){.9706}}
\put(15.312,4.091){\line(1,0){.9706}}
\put(17.253,4.047){\line(1,0){.9706}}
\put(19.194,4.003){\line(1,0){.9706}}
\put(21.136,3.959){\line(1,0){.9706}}
\put(23.077,3.915){\line(1,0){.9706}}
\put(25.018,3.871){\line(1,0){.9706}}
\put(26.959,3.827){\line(1,0){.9706}}
\put(28.9,3.783){\line(1,0){.9706}}
\put(30.841,3.739){\line(1,0){.9706}}
\put(32.783,3.694){\line(1,0){.9706}}
\put(34.724,3.65){\line(1,0){.9706}}
\put(36.665,3.606){\line(1,0){.9706}}
\put(38.606,3.562){\line(1,0){.9706}}
\put(40.547,3.518){\line(1,0){.9706}}
\put(42.489,3.474){\line(1,0){.9706}}
\put(38.5,15.25){$\Gamma_1$}
\put(32.25,32.5){$\Lambda_{12}$}
\put(33,38.5){${\bf z}_{12}$}
\multiput(29.43,22.68)(.078125,-.03125){8}{\line(1,0){.078125}}
\multiput(30.68,22.18)(.078125,-.03125){8}{\line(1,0){.078125}}
\multiput(29.43,24.93)(.09375,-.03125){8}{\line(1,0){.09375}}
\multiput(30.93,24.43)(.09375,-.03125){8}{\line(1,0){.09375}}
\multiput(32.43,23.93)(.09375,-.03125){8}{\line(1,0){.09375}}
\multiput(29.68,27.18)(.089286,-.03125){8}{\line(1,0){.089286}}
\multiput(31.108,26.68)(.089286,-.03125){8}{\line(1,0){.089286}}
\multiput(32.537,26.18)(.089286,-.03125){8}{\line(1,0){.089286}}
\multiput(33.965,25.68)(.089286,-.03125){8}{\line(1,0){.089286}}
\multiput(29.68,29.68)(.090278,-.03125){9}{\line(1,0){.090278}}
\multiput(31.305,29.117)(.090278,-.03125){9}{\line(1,0){.090278}}
\multiput(32.93,28.555)(.090278,-.03125){9}{\line(1,0){.090278}}
\multiput(34.555,27.992)(.090278,-.03125){9}{\line(1,0){.090278}}
\multiput(29.68,31.68)(.083333,-.030556){10}{\line(1,0){.083333}}
\multiput(31.346,31.069)(.083333,-.030556){10}{\line(1,0){.083333}}
\multiput(33.013,30.457)(.083333,-.030556){10}{\line(1,0){.083333}}
\multiput(34.68,29.846)(.083333,-.030556){10}{\line(1,0){.083333}}
\multiput(36.346,29.235)(.083333,-.030556){10}{\line(1,0){.083333}}
\put(50.75,37.5){${\cal C}_1$}
\put(25.5,45){${\cal C}_2$}
\put(49.25,4.5){${\cal C}_l$}
\put(28.25,16.5){$\pi$}
\put(22.75,9){$\Delta_1$}
\end{picture}
\bigskip

Recall that the diagram $\Delta$ is solid, and therefore the clove $cl(\pi,{\cal C}_1, {\cal C}_2)$
has $h_1-h_2 < \frac {h_1}{30N}$ maximal $\theta$-bands outside $\Lambda_{12}.$ Hence the maximal $a$-bands
starting on ${\bf z}_{12}$ can end outside of $\Lambda_{1,2}$ on at most $N$
$(\theta, q)$-cells of each of the $<\frac {h_1}{30N}\;$ $\theta$-bands.
Hence by \ref{cell} (a), at least $|{\bf z}_{12}|_a-h_1/15$ $a$-bands starting
on ${\bf z}_{12}$ end on ${\bf p}_1$, and so $|{\bf p}_1|_a> h_1(1/6 - 1/15)=h_1/10$.

 Therefore by Lemma \ref{ochev} (a) and Inequality (\ref{yp}), we obtain
 $$|{\bf p}|=|{\bf p}_1|+|{\bf v}|> |{\bf p}_1|_{\theta}+\delta'h_1/10+|{\bf v}|  \ge
 |{\bf x}|+|{\bf v}|+\delta'h_1/10\ge |{\bf p}_3|+\delta'h_1/10,$$ and so $|{\bf p}|-|{\bf p}_3|>\delta'h_1/10.$
 Thus
 \begin{equation}\label{h1}
 h_1<10(\delta')\iv(|{\bf p}|-|{\bf p}_3|)
 \end{equation}

 We have by  (\ref{h1}) and Lemma \ref{psi1}:
 $$\area (\Delta_1)\le 2\area(\Psi)+1\le
4(2LN(h_1+h_l)+  \delta\iv|{\bf p}|)h_1+1\le
16LNh_1^2+5\delta\iv|{\bf p}|h_1$$
 \begin{equation} \label{hvost}<16LN\times
100(\delta')^{-2}(|{\bf p}|-|{\bf p}_3|)^2+
 5\delta\iv|{\bf p}| \times 10(\delta')\iv(|{\bf p}|-|{\bf p}_3|) <
 c_6|{\bf p}|(|{\bf p}|-|{\bf p}_3|)/2
 \end{equation} 
 since $|{\bf p}_3|\le |{\bf p}|.$
 By Lemma \ref{Delta'}, $\lambda(\Delta)-\lambda(\Delta') \ge - 2h_1^2,$
 and therefore by (\ref{h1}),
  \begin{equation}\label{c4c6}
  c_4(\lambda(\Delta)-\lambda(\Delta'))\ge -2c_4h_1^2
 \ge -2c_4(10)^2(\delta')^{-2}(|{\bf p}|-|{\bf p}_3|)^2\ge -c_6|{\bf p}|(|{\bf p}|-|{\bf p}_3|)/2
 \end{equation}
 since $c_6>400(\delta')^{-2}c_4.$ Hence by (\ref{hvost}) and (\ref{c4c6}),
 $$\area(\Delta_1)\le c_6|{\bf p}|(|{\bf p}|-|{\bf p}_3|)
 +c_4(\lambda(\Delta)-\lambda(\Delta'))\le c_6n(n-|\partial\Delta'|)+c_4(\lambda(\Delta)-\lambda(\Delta'))$$
 as
 $|{\bf p}|-|{\bf p}_3|\le n-|\partial\Delta'|.$ Now the statement follows from Lemma \ref{Delta'} (b,d).

\endproof

\bigskip

  Let now
  $\Psi_{2,l}$ be the part of the crescent $\Psi$ between ${\cal C}_2$ and ${\cal C}_l.$
  By Lemma \ref{narrow}, $\Psi_{2,l}$  is a crescent too.
  For the crescent $\Psi_{2,l}$, one can define the analogs of ${\bf p}, {\bf p}_1,
{\bf p}_3,$ $\bf v,$ $\bf v',$ $\Delta_1$ and $\Delta'$ introduced earlier for the crescent
$\Psi.$ We denote them by ${\bf p}(0), {\bf p}_1(0),$ ${\bf p}_3(0),$ ${\bf v}(0),$ ${\bf v'}(0),$ $\Delta_1(0)$ and
$\Delta'(0),$ respectively.

The substitution of $\Psi$ by $\Psi_{2,l}$ in Lemma \ref{psi1},
gives us

\begin{equation}\label{Phi}
    \area(\Psi_{2,l})\le (h_2+h_l)(2LN(h_2+h_l)+ \delta^{-1}|{\bf p}(0)|)
   \end{equation}

\begin{lemma}\label{areaD1}
Assume that $|{\bf p}(0)|\le 2LN\max(h_2,h_l),$ $h_2< (1-\frac{1}{30N})h_1,$
and $\max(h_2,h_l)$ does not belong to any interval $(T_i,9T_i).$
Then the following inequality holds:
\begin{equation}\label{arD1}
\area(\Delta_1(0))
\le c_4(\mu(\Delta)- \mu (\Delta'(0))) +
 c_5(\nu_J(\Delta)- \nu_J(\Delta'(0))) +c_6|\partial\Delta|(|\partial\Delta|-|\partial\Delta'(0)|).
 \end{equation}
\end{lemma}

\proof Assume that $h_0\ge (1-\frac{1}{30})h_{2,l},$ where
$h_{2,l}=\max(h_2,h_l).$ Since $h_i\ge h_0$ for any $i=1,..,l$, to
complete the proof, it suffices to apply Lemma \ref{case2} to
$\Delta_1(0).$ Then we assume that $h_0<  (1-\frac{1}{30N})h_{2,l}.$

By Lemma \ref{psi1} and the restriction on $|p(0)|$, $$\area (\Delta_1(0))\le 2\area(\Psi_{2,l})+1\le
4(2LN(h_2+h_l)+  \delta\iv|{\bf p}(0)|)h_{2,l}+1$$
\begin{equation}\label{1(0)}\le
16LNh_{2,l}^2+4\delta\iv(2LN)^2h^2_{2,l}\le(\delta')^{-1}h^2_{2,l}
\end{equation}
since $(\delta')\iv>20\delta^{-1}L^2N^2.$

Now we want to estimate $\nu_J(\Delta)-\nu_J(\Delta'(0)).$ For this
aid, we observe that the common $t$-edge $f_2$ of the spoke
${\cal C}_2$ and $\partial\Delta$
 separates at least
$h_1-h_2=m_1$ $\theta$-edges placed on $\bf p$ between ${\cal C}_1$ and
${\cal C}_2$ and $m_2$ ones placed between ${\cal C}_2$ and ${\cal
C}_l,$ where $m_2\ge \max(h_2-h_0, h_l-h_0)\ge
\frac{1}{30N}h_{2,l}.$  Lemmas \ref{malot} and \ref{mixture} (d) imply that one
decreases $\nu_J(\Delta)$ at least by $m_1m_2$ when erasing the black
bead on $f_2$ in the $\nu$-necklace on $\partial\Delta.$ 

\bigskip

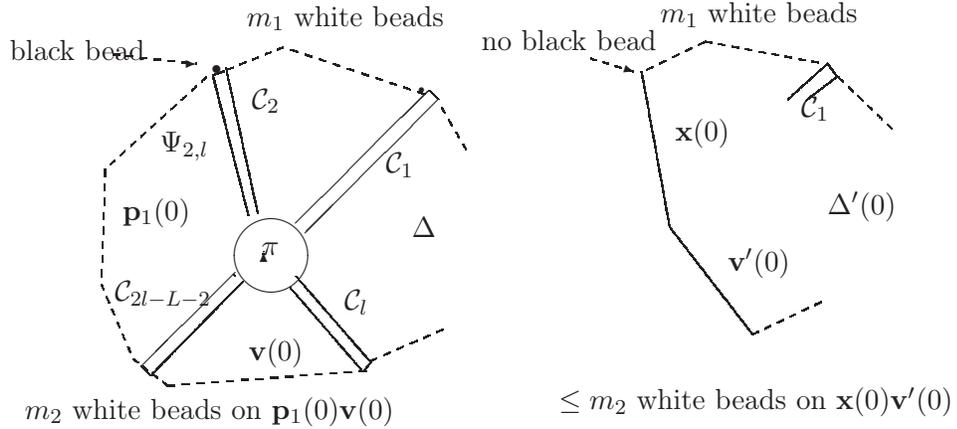
\unitlength 1mm 
\linethickness{0.4pt}
\ifx\plotpoint\undefined\newsavebox{\plotpoint}\fi 
\begin{picture}(115.25,63.25)(-5,0)
\put(32,22.25){\circle{10.062}}
\put(35.25,26.5){\line(1,1){17.75}}
\multiput(53,44.25)(.03289474,-.03289474){38}{\line(0,-1){.03289474}}
\put(54.25,43){\line(-1,-1){17.75}}
\multiput(26,47.25)(.033730159,-.154761905){126}{\line(0,-1){.154761905}}
\multiput(24.25,46.25)(.033687943,-.132978723){141}{\line(0,-1){.132978723}}
\multiput(24.25,46.25)(.1,.0333333){15}{\line(1,0){.1}}
\put(25,46.75){\circle*{.5}}
\put(52,44.25){\circle*{.707}}
\multiput(52.43,43.68)(-.09375,.03125){10}{\line(-1,0){.09375}}
\multiput(50.555,44.305)(-.09375,.03125){10}{\line(-1,0){.09375}}
\multiput(48.68,44.93)(-.09375,.03125){10}{\line(-1,0){.09375}}
\multiput(46.805,45.555)(-.09375,.03125){10}{\line(-1,0){.09375}}
\multiput(44.93,46.18)(-.09375,.03125){10}{\line(-1,0){.09375}}
\multiput(43.055,46.805)(-.09375,.03125){10}{\line(-1,0){.09375}}
\multiput(41.18,47.43)(-.09375,.03125){10}{\line(-1,0){.09375}}
\multiput(39.305,48.055)(-.09375,.03125){10}{\line(-1,0){.09375}}
\multiput(37.43,48.68)(-.09375,.03125){10}{\line(-1,0){.09375}}
\multiput(35.555,49.305)(-.09375,.03125){10}{\line(-1,0){.09375}}
\multiput(33.68,49.93)(-.083333,-.033333){10}{\line(-1,0){.083333}}
\multiput(32.013,49.263)(-.083333,-.033333){10}{\line(-1,0){.083333}}
\multiput(30.346,48.596)(-.083333,-.033333){10}{\line(-1,0){.083333}}
\multiput(28.68,47.93)(-.083333,-.033333){10}{\line(-1,0){.083333}}
\multiput(27.013,47.263)(-.083333,-.033333){10}{\line(-1,0){.083333}}
\multiput(27.18,19.43)(-.03125,.0625){4}{\line(0,1){.0625}}
\put(27.25,20){\line(-1,-1){12.5}}
\multiput(14.75,7.5)(.03289474,-.03289474){38}{\line(0,-1){.03289474}}
\multiput(16,6.25)(.0336538462,.0343406593){364}{\line(0,1){.0343406593}}
\multiput(34.5,18)(.0337370242,-.0389273356){289}{\line(0,-1){.0389273356}}
\multiput(44.25,6.75)(.0333333,.0416667){30}{\line(0,1){.0416667}}
\multiput(45.25,8)(-.0336700337,.0387205387){297}{\line(0,1){.0387205387}}
\multiput(23.43,45.93)(-.0355263,-.0322368){19}{\line(-1,0){.0355263}}
\multiput(22.08,44.705)(-.0355263,-.0322368){19}{\line(-1,0){.0355263}}
\multiput(20.73,43.48)(-.0355263,-.0322368){19}{\line(-1,0){.0355263}}
\multiput(19.38,42.255)(-.0355263,-.0322368){19}{\line(-1,0){.0355263}}
\multiput(18.03,41.03)(-.0355263,-.0322368){19}{\line(-1,0){.0355263}}
\multiput(16.68,39.805)(-.0355263,-.0322368){19}{\line(-1,0){.0355263}}
\multiput(15.33,38.58)(-.0355263,-.0322368){19}{\line(-1,0){.0355263}}
\multiput(13.98,37.355)(-.0355263,-.0322368){19}{\line(-1,0){.0355263}}
\multiput(12.63,36.13)(-.0355263,-.0322368){19}{\line(-1,0){.0355263}}
\multiput(11.28,34.905)(-.0355263,-.0322368){19}{\line(-1,0){.0355263}}
\put(9.93,33.68){\line(0,-1){.9531}}
\put(9.867,31.773){\line(0,-1){.9531}}
\put(9.805,29.867){\line(0,-1){.9531}}
\put(9.742,27.961){\line(0,-1){.9531}}
\put(9.68,26.055){\line(0,-1){.9531}}
\put(9.617,24.148){\line(0,-1){.9531}}
\put(9.555,22.242){\line(0,-1){.9531}}
\put(9.492,20.336){\line(0,-1){.9531}}
\multiput(9.43,18.43)(.032197,-.075758){11}{\line(0,-1){.075758}}
\multiput(10.138,16.763)(.032197,-.075758){11}{\line(0,-1){.075758}}
\multiput(10.846,15.096)(.032197,-.075758){11}{\line(0,-1){.075758}}
\multiput(11.555,13.43)(.032197,-.075758){11}{\line(0,-1){.075758}}
\multiput(12.263,11.763)(.032197,-.075758){11}{\line(0,-1){.075758}}
\multiput(12.971,10.096)(.032197,-.075758){11}{\line(0,-1){.075758}}
\multiput(13.68,8.43)(.0428571,-.0333333){15}{\line(1,0){.0428571}}
\multiput(14.965,7.43)(.0428571,-.0333333){15}{\line(1,0){.0428571}}
\multiput(16.251,6.43)(.0428571,-.0333333){15}{\line(1,0){.0428571}}
\multiput(17.537,5.43)(.0428571,-.0333333){15}{\line(1,0){.0428571}}
\put(18.18,4.93){\line(1,0){.9815}}
\put(20.143,5.078){\line(1,0){.9815}}
\put(22.106,5.226){\line(1,0){.9815}}
\put(24.069,5.374){\line(1,0){.9815}}
\put(26.032,5.522){\line(1,0){.9815}}
\put(27.995,5.67){\line(1,0){.9815}}
\put(29.957,5.819){\line(1,0){.9815}}
\put(31.92,5.967){\line(1,0){.9815}}
\put(33.883,6.115){\line(1,0){.9815}}
\put(35.846,6.263){\line(1,0){.9815}}
\put(37.809,6.411){\line(1,0){.9815}}
\put(39.772,6.559){\line(1,0){.9815}}
\put(41.735,6.707){\line(1,0){.9815}}
\put(43.698,6.856){\line(1,0){.9815}}
\put(12.25,27){${\bf p}_1(0)$}
\put(29,8.25){${\bf v}(0)$}
\put(-.75,.5){$m_2$ white beads on ${\bf p}_1(0){\bf v}(0)$}
\put(28.75,52.75){$m_1$ white beads}
\put(24.75,47){\circle*{1.118}}
\put(-2.75,48.25){black bead}
\put(22,47.75){\vector(1,0){.07}}\put(10.93,48.93){\line(1,0){.9167}}
\put(12.763,48.721){\line(1,0){.9167}}
\put(14.596,48.513){\line(1,0){.9167}}
\put(16.43,48.305){\line(1,0){.9167}}
\put(18.263,48.096){\line(1,0){.9167}}
\put(20.096,47.888){\line(1,0){.9167}}
\put(31,22.75){\vector(0,1){.07}}\put(30.93,22.43){\line(0,1){.125}}
\put(30.75,22.25){$\pi$}
\put(47.25,33.25){${\cal C}_1$}
\put(29.5,41.75){${\cal C}_2$}
\put(41.75,15){${\cal C}_l$}
\put(17.25,36){$\Psi_{2,l}$}
\multiput(54.68,41.93)(.0334821,-.0580357){14}{\line(0,-1){.0580357}}
\multiput(55.617,40.305)(.0334821,-.0580357){14}{\line(0,-1){.0580357}}
\multiput(56.555,38.68)(.0334821,-.0580357){14}{\line(0,-1){.0580357}}
\multiput(57.492,37.055)(.0334821,-.0580357){14}{\line(0,-1){.0580357}}
\multiput(45.18,8.18)(.073864,.032197){12}{\line(1,0){.073864}}
\multiput(46.952,8.952)(.073864,.032197){12}{\line(1,0){.073864}}
\multiput(48.725,9.725)(.073864,.032197){12}{\line(1,0){.073864}}
\multiput(50.498,10.498)(.073864,.032197){12}{\line(1,0){.073864}}
\multiput(52.271,11.271)(.073864,.032197){12}{\line(1,0){.073864}}
\multiput(54.043,12.043)(.073864,.032197){12}{\line(1,0){.073864}}
\put(50.75,24.75){$\Delta$}
\multiput(81.25,46.75)(.033482143,-.185267857){112}{\line(0,-1){.185267857}}
\multiput(85,26)(.0336391437,-.0435779817){327}{\line(0,-1){.0435779817}}
\multiput(106,47.75)(-.037234043,-.033687943){141}{\line(-1,0){.037234043}}
\multiput(106,47.75)(.03289474,-.04605263){38}{\line(0,-1){.04605263}}
\multiput(107.25,46)(-.043814433,-.033505155){97}{\line(-1,0){.043814433}}
\multiput(81.93,46.93)(.0662393,.0320513){13}{\line(1,0){.0662393}}
\multiput(83.652,47.763)(.0662393,.0320513){13}{\line(1,0){.0662393}}
\multiput(85.374,48.596)(.0662393,.0320513){13}{\line(1,0){.0662393}}
\multiput(87.096,49.43)(.0662393,.0320513){13}{\line(1,0){.0662393}}
\multiput(88.819,50.263)(.0662393,.0320513){13}{\line(1,0){.0662393}}
\multiput(89.68,50.68)(.177778,-.033333){5}{\line(1,0){.177778}}
\multiput(91.457,50.346)(.177778,-.033333){5}{\line(1,0){.177778}}
\multiput(93.235,50.013)(.177778,-.033333){5}{\line(1,0){.177778}}
\multiput(95.013,49.68)(.177778,-.033333){5}{\line(1,0){.177778}}
\multiput(96.791,49.346)(.177778,-.033333){5}{\line(1,0){.177778}}
\multiput(98.569,49.013)(.177778,-.033333){5}{\line(1,0){.177778}}
\multiput(100.346,48.68)(.177778,-.033333){5}{\line(1,0){.177778}}
\multiput(102.124,48.346)(.177778,-.033333){5}{\line(1,0){.177778}}
\multiput(103.902,48.013)(.177778,-.033333){5}{\line(1,0){.177778}}
\multiput(107.43,45.93)(.0339912,-.0328947){19}{\line(1,0){.0339912}}
\multiput(108.721,44.68)(.0339912,-.0328947){19}{\line(1,0){.0339912}}
\multiput(110.013,43.43)(.0339912,-.0328947){19}{\line(1,0){.0339912}}
\multiput(111.305,42.18)(.0339912,-.0328947){19}{\line(1,0){.0339912}}
\multiput(112.596,40.93)(.0339912,-.0328947){19}{\line(1,0){.0339912}}
\multiput(113.888,39.68)(.0339912,-.0328947){19}{\line(1,0){.0339912}}
\multiput(95.93,11.68)(.069444,.032986){12}{\line(1,0){.069444}}
\multiput(97.596,12.471)(.069444,.032986){12}{\line(1,0){.069444}}
\multiput(99.263,13.263)(.069444,.032986){12}{\line(1,0){.069444}}
\multiput(100.93,14.055)(.069444,.032986){12}{\line(1,0){.069444}}
\multiput(102.596,14.846)(.069444,.032986){12}{\line(1,0){.069444}}
\multiput(104.263,15.638)(.069444,.032986){12}{\line(1,0){.069444}}
\put(85.75,37.5){${\bf x}(0)$}
\put(92.75,20.25){${\bf v'}(0)$}
\put(106,27.75){$\Delta'(0)$}
\put(83.75,53.25){$m_1$ white beads}
\put(60,49.5){no black bead}
\put(80.25,46.75){\vector(3,-1){.07}}\multiput(74.43,48.43)(.102679,-.03125){8}{\line(1,0){.102679}}
\multiput(76.073,47.93)(.102679,-.03125){8}{\line(1,0){.102679}}
\multiput(77.715,47.43)(.102679,-.03125){8}{\line(1,0){.102679}}
\multiput(79.358,46.93)(.102679,-.03125){8}{\line(1,0){.102679}}
\put(70.25,1.75){$\le m_2$ white beads on ${\bf x}(0){\bf v'}(0)$}
\put(102.5,40.75){${\cal C}_1$}
\put(11.25,16.25){${\cal C}_{2l-L-2}$}
\end{picture}

\bigskip

\bigskip

\bigskip

Nevertheless we do such erasing
while passing from $\Delta$
to $\Delta'(0)$ since the path ${\bf x}(0)$ (replacing the path ${\bf p}_1(0)$ with edge $f_2$)
has no $t$- or $t'$-edges and ${\bf v'}(0)$ is a copy of ${\bf v}(0).$ (We might erase some other
black and white beads). Hence
\begin{equation}\label{m1m2}
\nu_J(\Delta)-\nu_J(\Delta'(0))\ge m_1m_2\ge
\frac{1}{30N}h_1(\frac{1}{30N})h_{2,l}\ge\frac{1}{(30N)^2}h^2_{2,l}
\end{equation}

The Inequalities (\ref{m1m2}) and (\ref{1(0)}) imply
\begin{equation}\label{ardel10}
\area(\Delta_1(0))\le c_5(\nu_J(\Delta)-\nu_J(\Delta'(0)))/2
\end{equation}
  since
$c_5> 2000N^2(\delta')^{-1}.$ 

By Lemma \ref{Delta'} (b,c) for the diagrams $\Delta$ and $\Delta'(0),$
we have $\mu(\Delta)-\mu(\Delta'(0))\ge\lambda(\Delta)-\lambda(\Delta'(0)) \ge - 2h_{2,l}^2,$ whence by (\ref{m1m2}), $0\le c_5(\nu_J(\Delta)-\nu_J(\Delta'(0)))/2+
c_4(\mu(\Delta)-\mu(\Delta'(0))$ 
since $c_5>2000N^2 c_4$. Adding this inequality with(\ref{ardel10}), we have
$$\area(\Delta_1(0))
\le c_5(\nu_J(\Delta)- \nu_J(\Delta'(0))) +c_4(\mu(\Delta)- \mu (\Delta'(0)))$$
This implies Inequality (\ref{arD1}) since the third summand at the right-hand
side of (\ref{arD1}) is positive
by Lemma \ref{Delta'} (a) applied to the diagrams $\Delta$ and $\Delta'(0).$

The notation of this subsection will also be used in the next section.

\section{Almost quadratic upper bound} \label{aqub}

We denote by $g(n)$  the minimal function such that
the height of any $M_4$-accepting trapezium is at most $g(n)$
if the $a$-length of its bottom label does not exceed $n.$
(Such upper bounds exist for every $n$  by Properties \ref{xiii} and \ref{xiv}.)
Since $g(n)$ is non-decreasing, the auxiliary function
$f(n)=n(8g(\delta^{-1} n^2)^2+\delta^{-1}n^3)$ used in this
section is also non-decreasing. For the beginning of this section, we need
 a crude upper bound for areas of diagrams. 

\begin{lemma} \label{grubo}  Let $\Delta$ be a minimal diagram of perimeter $\le n$.
Then $\area(\Delta)\le f(n)$.
\end{lemma}

\proof {\bf Step 1.} 
Assume that $\Delta$ has no hubs. Then we can use Lemma \ref{NoAnnul}, and 
the total number of maximal $q$-band and maximal $\theta$-bands of
$\Delta$ is at most $n/2$. Hence the number of $(q,\theta)$-cells is
at most $n^2/16.$ Since every maximal $a$-band ends either on the
boundary $\partial\Delta$ or on a $(q,\theta)$-cell, the number such
bands is at most $(2\times n^2/16+\delta^{-1}n)/2$ by Lemma \ref{ochev} (d) and \ref{cell} (a). Each of these
$a$-bands crosses at most $n/2$ $\theta$-bands, and so their
total area is at most $n^3/16 +\delta^{-1}n^2/4$. Therefore $\area
(\Delta)\le n^3/16+\delta^{-1}n^2/4 +n^2/16\le \delta^{-1}n^3/3.$

{\bf Step 2.} In any case, the number of hubs $n_{hub}$ in $\Delta$
is at most $2n/LN$ by Lemma \ref{mnogospits}. 
To complete the proof of the lemma,
it suffices to assume that $n_{hub}\ge 1$ and to prove by induction
on $n_{hub}$ that $$\area(\Delta)\le
n_{hub}(4LNg(\delta^{-1}n^2)^2+6\delta^{-1}n^3)+\delta^{-1}n^3/3$$

  There are a hub $\pi$   and a
clove $\Psi=cl(\pi,\cal B, \cal B')$ given by Lemma \ref{extdisc}. 
  Let $\Lambda$ be
the subdiagram of $\Psi$ formed by the $\theta$-bands of $\Psi$
crossing both $t$-bands ${\cal B}={\cal B}_1$ and ${\cal B'}={\cal
B}_{L-3}.$ Let the remaining part $\Lambda'=\Psi\backslash \Lambda $
be separated from $\Lambda$ by a path ${\bf p}(\Lambda).$

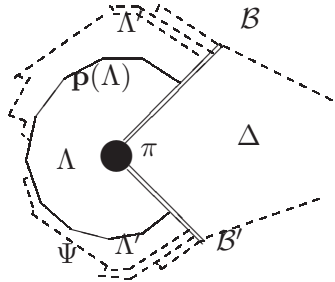
\unitlength 1mm 
\linethickness{0.4pt}
\ifx\plotpoint\undefined\newsavebox{\plotpoint}\fi 
\begin{picture}(51.25,46.5)(-35,0)
\put(22.5,19.25){\circle*{4.123}}
\put(23.75,21){\line(1,1){13}}
\put(23.5,21.5){\line(1,1){12.75}}
\multiput(36.25,34.25)(.03125,-.0625){8}{\line(0,-1){.0625}}
\put(24.25,18){\line(1,-1){10}}
\multiput(34.25,8)(-.125,-.03125){8}{\line(-1,0){.125}}
\put(24,17.5){\line(1,-1){10}}
\multiput(31,29)(-.051546392,.033505155){97}{\line(-1,0){.051546392}}
\put(26,32.25){\line(-1,0){5.25}}
\multiput(20.75,32.25)(-.06402439,-.03353659){82}{\line(-1,0){.06402439}}
\put(15.5,29.5){\line(-3,-4){3.75}}
\multiput(11.75,24.5)(-.03289474,-.15131579){38}{\line(0,-1){.15131579}}
\multiput(10.5,18.75)(.03333333,-.09583333){60}{\line(0,-1){.09583333}}
\multiput(12.5,13)(.038461538,-.033653846){104}{\line(1,0){.038461538}}
\put(16.5,9.5){\line(4,-1){5}}
\multiput(21.5,8.25)(.1956522,.0326087){23}{\line(1,0){.1956522}}
\multiput(26,9)(.04268293,.03353659){82}{\line(1,0){.04268293}}
\multiput(35.93,34.18)(-.0484375,.0328125){16}{\line(-1,0){.0484375}}
\multiput(34.38,35.23)(-.0484375,.0328125){16}{\line(-1,0){.0484375}}
\multiput(32.83,36.28)(-.0484375,.0328125){16}{\line(-1,0){.0484375}}
\multiput(31.28,37.33)(-.0484375,.0328125){16}{\line(-1,0){.0484375}}
\multiput(29.73,38.38)(-.0484375,.0328125){16}{\line(-1,0){.0484375}}
\multiput(28.18,39.43)(-.03125,-.072917){12}{\line(0,-1){.072917}}
\multiput(35.43,32.93)(-.0512821,.0320513){15}{\line(-1,0){.0512821}}
\multiput(33.891,33.891)(-.0512821,.0320513){15}{\line(-1,0){.0512821}}
\multiput(32.353,34.853)(-.0512821,.0320513){15}{\line(-1,0){.0512821}}
\multiput(30.814,35.814)(-.0512821,.0320513){15}{\line(-1,0){.0512821}}
\multiput(29.276,36.776)(-.0512821,.0320513){15}{\line(-1,0){.0512821}}
\multiput(27.737,37.737)(-.0512821,.0320513){15}{\line(-1,0){.0512821}}
\multiput(26.199,38.699)(-.0512821,.0320513){15}{\line(-1,0){.0512821}}
\put(25.43,39.18){\line(-1,0){.9375}}
\put(23.555,39.18){\line(-1,0){.9375}}
\put(21.68,39.18){\line(1,0){.125}}
\multiput(21.93,39.18)(.03125,-.125){4}{\line(0,-1){.125}}
\put(22.18,38.18){\line(1,0){.75}}
\put(23.68,38.18){\line(1,0){.75}}
\multiput(25.18,38.18)(.05,-.0333333){15}{\line(1,0){.05}}
\multiput(26.68,37.18)(.05,-.0333333){15}{\line(1,0){.05}}
\multiput(28.18,36.18)(.05,-.0333333){15}{\line(1,0){.05}}
\multiput(29.68,35.18)(.05,-.0333333){15}{\line(1,0){.05}}
\multiput(31.18,34.18)(.05,-.0333333){15}{\line(1,0){.05}}
\multiput(32.68,33.18)(.05,-.0333333){15}{\line(1,0){.05}}
\multiput(33.68,7.93)(-.046875,-.0334821){16}{\line(-1,0){.046875}}
\multiput(32.18,6.858)(-.046875,-.0334821){16}{\line(-1,0){.046875}}
\multiput(30.68,5.787)(-.046875,-.0334821){16}{\line(-1,0){.046875}}
\multiput(29.18,4.715)(-.046875,-.0334821){16}{\line(-1,0){.046875}}
\multiput(28.43,4.18)(-.03125,.125){4}{\line(0,1){.125}}
\put(28.18,5.18){\line(0,1){0}}
\multiput(32.68,8.43)(-.05,-.0333333){15}{\line(-1,0){.05}}
\multiput(31.18,7.43)(-.05,-.0333333){15}{\line(-1,0){.05}}
\multiput(29.68,6.43)(-.05,-.0333333){15}{\line(-1,0){.05}}
\multiput(28.18,5.43)(-.05,-.0333333){15}{\line(-1,0){.05}}
\multiput(26.68,4.43)(-.05,-.0333333){15}{\line(-1,0){.05}}
\multiput(25.18,3.43)(-.05,-.0333333){15}{\line(-1,0){.05}}
\put(24.43,2.93){\line(-1,0){.85}}
\put(22.73,3.13){\line(-1,0){.85}}
\put(21.03,3.33){\line(-1,0){.85}}
\multiput(32.18,9.43)(-.0484848,-.0318182){15}{\line(-1,0){.0484848}}
\multiput(30.725,8.475)(-.0484848,-.0318182){15}{\line(-1,0){.0484848}}
\multiput(29.271,7.521)(-.0484848,-.0318182){15}{\line(-1,0){.0484848}}
\multiput(27.816,6.566)(-.0484848,-.0318182){15}{\line(-1,0){.0484848}}
\multiput(26.362,5.612)(-.0484848,-.0318182){15}{\line(-1,0){.0484848}}
\multiput(24.907,4.657)(-.0484848,-.0318182){15}{\line(-1,0){.0484848}}
\put(24.18,4.18){\line(-1,0){.75}}
\put(22.68,4.055){\line(-1,0){.75}}
\multiput(21.18,4.68)(-.03125,-.078125){8}{\line(0,-1){.078125}}
\multiput(27.68,5.18)(.03125,-.03125){12}{\line(0,-1){.03125}}
\multiput(20.68,37.93)(-.0460784,-.0323529){17}{\line(-1,0){.0460784}}
\multiput(19.113,36.83)(-.0460784,-.0323529){17}{\line(-1,0){.0460784}}
\multiput(17.546,35.73)(-.0460784,-.0323529){17}{\line(-1,0){.0460784}}
\multiput(15.98,34.63)(-.0460784,-.0323529){17}{\line(-1,0){.0460784}}
\multiput(14.413,33.53)(-.0460784,-.0323529){17}{\line(-1,0){.0460784}}
\multiput(12.846,32.43)(-.0460784,-.0323529){17}{\line(-1,0){.0460784}}
\multiput(11.28,31.33)(-.0460784,-.0323529){17}{\line(-1,0){.0460784}}
\multiput(9.713,30.23)(-.0460784,-.0323529){17}{\line(-1,0){.0460784}}
\multiput(21.18,4.68)(-.0479911,.0334821){16}{\line(-1,0){.0479911}}
\multiput(19.644,5.751)(-.0479911,.0334821){16}{\line(-1,0){.0479911}}
\multiput(18.108,6.823)(-.0479911,.0334821){16}{\line(-1,0){.0479911}}
\multiput(16.573,7.894)(-.0479911,.0334821){16}{\line(-1,0){.0479911}}
\multiput(15.037,8.965)(-.0479911,.0334821){16}{\line(-1,0){.0479911}}
\multiput(13.501,10.037)(-.0479911,.0334821){16}{\line(-1,0){.0479911}}
\multiput(11.965,11.108)(-.0479911,.0334821){16}{\line(-1,0){.0479911}}
\multiput(9.43,29.93)(.078125,-.03125){8}{\line(1,0){.078125}}
\multiput(10.93,29.68)(-.032738,-.059524){12}{\line(0,-1){.059524}}
\multiput(10.144,28.251)(-.032738,-.059524){12}{\line(0,-1){.059524}}
\multiput(9.358,26.823)(-.032738,-.059524){12}{\line(0,-1){.059524}}
\multiput(8.573,25.394)(-.032738,-.059524){12}{\line(0,-1){.059524}}
\multiput(8.18,24.68)(.09375,-.03125){8}{\line(1,0){.09375}}
\multiput(9.68,24.18)(.09375,-.03125){8}{\line(1,0){.09375}}
\multiput(10.43,23.93)(-.03125,-.109375){8}{\line(0,-1){.109375}}
\multiput(9.93,22.18)(.041667,-.033333){10}{\line(1,0){.041667}}
\multiput(10.763,21.513)(.041667,-.033333){10}{\line(1,0){.041667}}
\put(11.43,11.68){\line(0,1){.375}}
\multiput(11.43,12.43)(-.03125,.09375){8}{\line(0,1){.09375}}
\multiput(10.93,13.93)(-.03125,.09375){8}{\line(0,1){.09375}}
\multiput(10.43,15.43)(-.03125,.09375){8}{\line(0,1){.09375}}
\multiput(10.18,16.18)(.0625,.03125){8}{\line(1,0){.0625}}
\multiput(36.68,33.93)(.0656109,-.0328054){13}{\line(1,0){.0656109}}
\multiput(38.386,33.077)(.0656109,-.0328054){13}{\line(1,0){.0656109}}
\multiput(40.091,32.224)(.0656109,-.0328054){13}{\line(1,0){.0656109}}
\multiput(41.797,31.371)(.0656109,-.0328054){13}{\line(1,0){.0656109}}
\multiput(43.503,30.518)(.0656109,-.0328054){13}{\line(1,0){.0656109}}
\multiput(45.209,29.665)(.0656109,-.0328054){13}{\line(1,0){.0656109}}
\multiput(46.915,28.812)(.0656109,-.0328054){13}{\line(1,0){.0656109}}
\multiput(48.621,27.959)(.0656109,-.0328054){13}{\line(1,0){.0656109}}
\multiput(50.327,27.106)(.0656109,-.0328054){13}{\line(1,0){.0656109}}
\multiput(34.43,8.43)(.088889,.030556){10}{\line(1,0){.088889}}
\multiput(36.207,9.041)(.088889,.030556){10}{\line(1,0){.088889}}
\multiput(37.985,9.652)(.088889,.030556){10}{\line(1,0){.088889}}
\multiput(39.763,10.263)(.088889,.030556){10}{\line(1,0){.088889}}
\multiput(41.541,10.874)(.088889,.030556){10}{\line(1,0){.088889}}
\multiput(43.319,11.485)(.088889,.030556){10}{\line(1,0){.088889}}
\multiput(45.096,12.096)(.088889,.030556){10}{\line(1,0){.088889}}
\multiput(46.874,12.707)(.088889,.030556){10}{\line(1,0){.088889}}
\multiput(48.652,13.319)(.088889,.030556){10}{\line(1,0){.088889}}
\put(39,36.25){$\cal B$}
\put(35.75,6.75){$\cal B'$}
\put(14.5,5){$\Psi$}
\put(14.5,17.5){$\Lambda$}
\put(22.5,35.75){$\Lambda'$}
\put(22.25,6.25){$\Lambda'$}
\put(25.75,19.25){$\pi$}
\put(38.5,20.75){$\Delta$}
\put(16.5,28.75){${\bf p}(\Lambda)$}
\end{picture}

It follows from the choice of $\Lambda$ that every maximal
$\theta$-band $\cal T$ of $\Lambda'$ starts  or ends on
$\partial\Delta$. Hence the number of such $\theta$-bands is at most
$n$. In the diagram $\Lambda'$, a $\theta$-band of $\Delta$ and a
$q$-band have at most one common $(q,\theta)$ -cells by Lemma
\ref{NoAnnul}. Since the number of maximal $q$-bands of $\Lambda'$ is
at most $|{\bf p}(\Psi)|\le n$, the number of $(q,\theta)$-cells in
$\Lambda'$ is not greater than $n^2.$ Since every maximal $a$-band
of $\Lambda'$ starting on the path ${\bf p}(\Lambda)$ must end on one of these
$(q,\theta)$-cells or on $\partial\Delta$, the number of $a$-edges
in ${\bf p}(\Lambda)$ is at most $2n^2+\delta^{-1}n\le 2\delta\iv n^2$ by Lemma \ref{ochev} (d). Since by Lemma \ref{NoAnnul}, a maximal $a$-band
intersects a $\theta$-band of the diagram $\Lambda'$ at most once,
there are at most $2\delta\iv n^3$ $(\theta,a)$-cells in $\Lambda'.$
Thus, $\area(\Lambda')\le 2\delta\iv n^3+n^2\le 2.5\delta^{-1}n^3.$

Since $|{\bf p}(\Lambda)|_a \le 2\delta^{-1}n^2$ and $\Lambda$ has $2(L-4)$
$M_4$-accepting
 trapezia whose top labels are just copies of  one of them (see Remark \ref{diasym} and Lemma \ref{simul})),
 the number of maximal $\theta$-bands in
$\Lambda$ is at most $g((2\delta^{-1}n^2)(2L-8)\iv)\le g(\delta\iv n^2).$  By \ref{vi} applied to $2L-8$ $M_4$-accepting trapezia, the number of cells in any maximal
$\theta$-band of $\Lambda$ does not exceed $LN+(2L-8)4 g(\delta\iv n^2).$ Multiplying
this number by the height of $\Lambda$, we obtain
$$\area(\Lambda)\le
((8L-32)g(\delta\iv n^2)+LN)g(\delta\iv n^2)\le 2LNg(\delta^{-1}n^2)^2.$$
Therefore $\area(\Psi)=\area(\Lambda)+\area(\Lambda')\le
2LNg(\delta^{-1}n^2)^2 +3\delta^{-1}n^3-1.$ 

Now we use the surgery from Remark \ref{anyclove} (1) and have $\area(\Delta_1)\le
2\area(\Psi)+1\le 4LNg(\delta^{-1}n^2)^2 +6\delta^{-1}n^3,$ and
$|\partial\Delta'|\le|\partial\Delta'|-1.$ Since the number of hubs
of $\Delta'$ is strictly less then this number for $\Delta$, we have
by the inductive hypothesis, $\area(\Delta') \le
(n_{hub}-1)(4LNg(\delta^{-1}n^2)^2+6\delta^{-1}n^3)+(\delta')\iv n^3/3$,
and therefore, by Lemma \ref{2diagr},  as required, $$\area(\Delta)\le \area(\Delta_1)+ \area(\Delta')\le
n_{hub}(4LNg(\delta^{-1}n^2)^2+6\delta^{-1}n^3)+(\delta')\iv n^3/3$$
\endproof

The following lemma summarizes our efforts and ensures the main result.

\begin{lemma} \label{main} Let the perimeter $n=|\partial\Delta|$ of a minimal diagram $\Delta$
satisfy inequality $f(n)\le T_i $ for some $i$, where $f$ is the function from Lemma \ref{grubo}.  Then \label{FD}
$\area(\Delta)\le F(\Delta)$ for $F(\Delta)=F(\partial\Delta)= 
c_4\mu(\Delta)+c_5\nu_J(\Delta)+c_6n^2 +c_7n_Qf(T_{i-1})),$ where $n_Q$ is
the number of $q$-edges in $\partial\Delta.$
\end{lemma}

\proof

(1) If $\partial\Delta$ has no $q$-edges, then $\Delta$ has
no $q$-edges by Lemmas \ref{extdisc} 
and \ref{NoAnnul}. Then $\area(\Delta)<\delta^{-1}n^2\le
c_6n^2$ because (1) a maximal $\theta$-band and a maximal $a$-band have at least one common
$(\theta,a)$-cell, (2) $\Delta$ has no $\theta$- and $a$-annuli and so every maximal $\theta$- and $a$-band starts and ends
on $\partial \Delta$ by Lemma \ref{NoAnnul}, and (3) $|\partial\Delta|_{a}\le||\partial\Delta|| \le\delta^{-1}n$ by Lemma \ref{ochev} (a,d). 

Thus we suppose $n_Q\ge 1.$
By Lemma \ref{grubo}, $\area(\Delta)\le f(T_{i-1})$ if $n\le
T_{i-1}$. Then we {\bf may suppose $n>T_{i-1}$, $n\ge 1$ and prove the lemma by
contradiction assuming further that $\Delta$ is a counter-example
with minimal perimeter $n.$}
\medskip

(2) If $\Delta$ is a union of a subdiagram $\Delta'$ and a rim $\theta$-band
$\cal T$ of base width $\le 2LN$, then there are at most $4LN$
$a$-edges on the boundaries of $(\theta,q)$-cells of $\cal T$ by \ref{cell} (a), and so $|\;|\topp({\cal
T})|-|\bott({\cal T})|\;|\le 4LN\delta$ by Lemma \ref{ochev}(a). Therefore $ |\partial
\Delta'|_a\le 4LN\delta +|\partial \Delta|_a$ but $|\partial
\Delta'|_{\theta}= |\partial \Delta|_{\theta}-2$. Hence
$|\partial \Delta'|\le |\partial \Delta|-1$ by Lemma \ref{ochev} (b) since $5LN<
\delta^{-1}.$ 

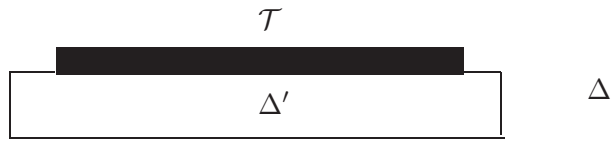
\unitlength 1mm 
\linethickness{0.4pt}
\ifx\plotpoint\undefined\newsavebox{\plotpoint}\fi 
\begin{picture}(89.75,30.5)(-12,-5)
\put(19.25,11.5){\rule{54\unitlength}{3.5\unitlength}}
\put(19,11.75){\line(-1,0){6}}
\put(13,11.75){\line(0,-1){8.75}}
\put(13,3){\line(1,0){65.75}}
\put(72.75,11.75){\line(1,0){5.5}}
\put(78.25,11.75){\line(0,-1){8.25}}
\put(89.75,8.25){$\Delta$}
\put(46,17.5){$\cal T$}
\put(46,6.25){$\Delta'$}
\end{picture}

If $\cal T$ has $m$ $q$-cells, then $n\ge m$, and so, by Lemma \ref{ochev} (d),
the number of cells in $\cal T$ is at most
$\delta^{-1}n+2m \le 2\delta\iv n$ by Lemma \ref{ochev} (d). Also we have $\mu(\Delta')\le \mu(\Delta)$ and
$\nu_J(\Delta')\le \nu_J(\Delta)$ by Lemma \ref{mixture} (b), and
$n'_Q\le n_Q$ by Lemma \ref{NoAnnul}. Therefore by the inductive
hypothesis for $\Delta'$, $$\area(\Delta)\le F(\Delta')+2\delta^{-1}n
\le F(\Delta)-c_6(2n-1)+2\delta^{-1}n \le F(\Delta)$$ since
$c_6>2\delta ^{-1};$ a contradiction. {\bf Therefore $\Delta$ has no rim
$\theta$-bands of base width at most $2LN.$}

\medskip
(3) Assume that $\Delta$ has a subcomb    $\bar\Delta$ of base width $15N.$
Hence we can apply Lemma \ref{itog}
to the comb $\bar\Delta$ and consider two arising cases.

(a) $\Bar\Delta$ admits a long quasicomb  $\Gamma$  such that
$$\area(\Gamma)\le c_3[\Gamma]+c_2\mu^c(\Gamma)+c_3(\nu_J(\bar\Delta)-\nu^c_J(\bar\Delta\backslash\Gamma))$$
We multiply the right hand side by the number $c_4c_2^{-1}>1$ and then replace the two
coefficients $c_3c_4c_2^{-1}$  by bigger coefficients $c_6$ and $c_5,$ resp.; this
is legal since $\nu_J(\bar\Delta)-\nu^c_J(\bar\Delta\backslash\Gamma)\ge 0$ by Lemma
\ref{mu}(e) and Remark \ref{quasiotrez}, and $[\Gamma]\ge 0$ since $\Gamma$ is a long 
subcomb. Hence 
\begin{equation}\label{0406}
\area(\Gamma)\le c_6[\Gamma]+c_4\mu^c(\Gamma)+c_5(\nu_J(\bar\Delta)-\nu^c_J(\bar\Delta\backslash\Gamma))
\end{equation}
Let ${\bf y}={\bf y}^{\Gamma}$ and ${\bf z}={\bf z}^{\Gamma}.$ Since $\Gamma$ is long, the
compliment diagram $\Delta'=\Delta\backslash\Gamma$ 
satisfies
$|\partial \Delta'|\le|\partial \Delta|-|{\bf z}|+|{\bf y}| < |\partial \Delta|$. By Lemma \ref{mu} (a,e) and Remark \ref{quasiotrez}, we also
have $\mu(\Delta)\ge\mu(\Delta')+\mu^c(\Gamma)$ and
$ \nu_J(\bar\Delta)-\nu_J(\bar\Delta\backslash\Gamma)\le\nu_J(\Delta)-\nu_J(\Delta').$ Since $\area(\Delta')\le F(\Delta')$
by the inductive hypothesis, it follows from (\ref{0406}) that $$\area(\Delta)\le
\area(\Delta')+
c_6[\Gamma]+c_4\mu^c(\Gamma)+c_5(\nu_J(\Delta)-\nu_J(\Delta'))\le
c_6(n-(|{\bf z}|-|{\bf y}|))^2+$$ $$c_4\mu(\Delta') +c_5\nu(\Delta') +
c_7n_Qf(T_{i-1})) +
c_6[\Gamma]+c_4\mu^c(\Gamma)+c_5(\nu_J(\Delta)-\nu_J(\Delta')) \le$$ $$
c_6n^2 -c_6n(|{\bf z}|-|{\bf y}|)+ c_6n(|{\bf z}|-|{\bf y}|)+c_4\mu(\Delta)+c_5\nu(\Delta)+c_7n_Qf(T_{i-1})\le
F(\Delta)$$ since $|{\bf y}|<|{\bf z}|\le n, c_3<c_6,$ $c_3<c_5,$ and $c_2<c_4.$
Therefore $\Delta$ is not a counter-example, a contradiction.

(b) $\Delta$ has subcomb whose handle $\cal C$ is a $t$- or
$t'$-band  with length $l$ satisfying $T_j\le l<200T_j$ for some $j$,
and $\cal C$ separates a  subcomb $\Gamma$ of base width at
most $14N$ from $\Delta.$ By Remark \ref{yzh} applied to $\Gamma$, we have $n> l.$ 
Now since $T_j\le l< n\le f(n)\le T_i$, we have $j\le
i-1$. Again let $\Delta'$ be the diagram $\Delta\backslash\Gamma.$ Let
${\bf y}={\bf y}^{\Gamma}$ and ${\bf z}={\bf z}^{\Gamma}.$ Then we have
\begin{equation}\label{perim}
|\partial\Delta'|\le|\partial\Delta|-(|{\bf z}|-|{\bf y}|)\le |\partial\Delta|-2
\end{equation}
since the handle $\cal C$ of
$\Gamma$ is passive, and so $\Gamma$ is a long subcomb. Since $\Gamma$ has a $q$-band
$\cal C,$ we immediately obtain
\begin{equation}\label{nQ}
n_Q-n'_Q\ge 2
\end{equation}
for the numbers of $q$-edges in $\partial\Delta$ and in $\partial\Delta',$ and
\begin{equation}\label{nud}
\nu_J(\Delta)-\nu_J(\Delta')\ge 0
\end{equation}
by Lemmas \ref{mu}(b) and \ref{positive}(b).

By Remark \ref{anyclove} (2),
\begin{equation}\label{3b}
\mu(\Delta')\le
\mu(\Delta)+l^2/2<\mu(\Delta)+(200T_{i-1})^2/2 \le
\mu(\Delta)+f(T_{i-1})
\end{equation}
 since $(\delta')\iv>2\times10^4.$ Now, from the definition of the function $F$ and
Inequalities (\ref{perim}), (\ref{nud}), and (\ref{nQ}), we get
$$F(\Delta)\ge F(\Delta')+ c_6(n^2-(n-(|{\bf z}|-|{\bf y}|))^2) +c_5(\nu(\Delta)-\nu(\Delta'))+c_4(\mu(\Delta)- \mu(\Delta'))$$ $$
+c_7(n_Q-n'_Q)f(T_{i-1})\ge F(\Delta')+ c_6n(|{\bf z}|-|{\bf y}|)+c_4(\mu(\Delta)- \mu(\Delta'))+2c_7f(T_{i-1})$$
which together with Inequality (\ref{3b}) implies
\begin{equation}\label{F}
F(\Delta)\ge F(\Delta')+c_6n(|{\bf z}|-|{\bf y}|)- c_4f(T_{i-1})+2c_7f(T_{i-1})
 \ge \area(\Delta')+ c_6[\Gamma]+ c_7f(T_{i-1})
\end{equation}
because $\area(\Delta')\le F(\Delta')$ by (\ref{perim}) and the minimality of the counter-example $\Delta.$ 

On the
other hand, by Lemmas \ref{comb}, \ref{ochev} (d), and inequality $|{\bf y}|=l<200T_{i-1},$ we get
$$\area(\Gamma)\le 60N|{\bf y}|^2+2\delta^{-1}|{\bf z}||{\bf y}|=
(60N+2\delta^{-1})|{\bf y}|^2+ 2\delta^{-1}|{\bf y}|(|{\bf z}|-|{\bf y}|)\le $$ \begin{equation}\label{contr}(60N+2\delta^{-1})|{\bf y}|^2+2\delta^{-1} [\Gamma]\le
200^2\times (60N+2\delta^{-1})T_{i-1}^2+ 2\delta^{-1}[\Gamma]\le c_6[\Gamma]+c_7f(T_{i-1})
\end{equation}
 because $c_6>2\delta^{-1},$ $c_7\ge 10^6,$ and $f(T_{i-1})\ge \delta^{-1}T_{i-1}^2.$ Now Inequalities
(\ref{F}, \ref{contr}) yield $$\area(\Delta) = \area(\Delta')+\area(\Gamma)\le
F(\Delta)- c_6[\Gamma]- c_7f(T_{i-1})+\area(\Gamma) \le F(\Delta),$$  a contradiction.
{\bf Hence $\Delta$ has no subcombs of base width $15N$.}
\medskip

(4) Assume that $\Delta$ has a one-Step subcomb $\Gamma$ whose handle is $t$ or $t'$-band.
By (3), we may assume that it base
width is less $15N$. Then we can use
Lemma \ref{t}(2) and come to a contradiction as in (3)(a) above.

By (2)-(4), {\bf the diagram $\Delta$ is solid. By Lemma
\ref{rimexists}(b), it has a hub.}

\medskip

(5) Suppose we have a hub $\pi$ and a crescent
$\Psi=cl(\pi,{\cal C}_1,\dots,{\cal C}_l)$ given by Lemma \ref{narrow}. Assume that
$|{\bf p}(\Psi)|>2LN\max_{i=1}^l h_i,$ where $h_i$ is the length of ${\cal
C}_i.$ Then by Lemma \ref{ppsi}, $|\partial\Psi'|<|\partial\Delta|$ for the
subdiagram $\partial\Psi'= \Delta\backslash (\pi\cup\Psi).$
 Besides it follows from the definition
of crescent that $n'_Q< n_Q,$ where $n'_Q$ is the number of
$q$-edges in $\partial\Psi'$. Since by the inductive hypothesis
$\area(\Psi')\le F(\Psi'),$ we obtain by Lemma \ref{ppsi} that
$$\area(\Delta)\le F(\Psi')+c_4(\mu(\Delta)- \mu(\Psi')) +
 c_5(\nu_J(\Delta)- \nu_J(\Psi')) +c_6n (n-|\partial\Psi'|)\le F(\Delta),$$ and so $\Delta$ is not a counter-example.
\medskip

(6) Now we assume that $\Delta$ has  a crescent $\Psi$ and a hub as
in (5), but now $|{\bf p}(\Psi)|\le 2LN\max_{i=1}^l h_i.$ If the
conditions of Lemma \ref{case2} are satisfied, then that lemma leads
to a contradiction as in case (5) above since
$|\partial\Delta'|<|\partial\Delta|$  by Lemma \ref{Delta'}(a), 
$n'_Q<n_Q,$ and the diagram $\Delta_1\cup\Delta'$ (with notation of Section \ref{separ})
has the same boundary label as $\Delta$ (see Lemma \ref{2diagr}). Similarly we obtain a contradiction under assumptions of
Lemma \ref{areaD1}, if we cut off the subdiagram $\Delta_1(0)$ with
the spokes ${\cal C}_2,\dots,{\cal C}_l$ since these spokes also
bound a crescent $\Psi_{2,l}$ by Lemma \ref{narrow}.
\medskip

(7) Thus it remains to assume that the maximal $h_i$ for the
crescent, say $h_1 $ (since the case with $\Psi_{2,l}$ from Lemma \ref{areaD1} is absolutely
similar) {\bf satisfies inequalities $T_j\le h_1<9T_j$ for some $j$ and
$|{\bf p}|=|{\bf p}(\Psi)|\le 2LNh_1.$} Notice that $j\le i-1$ because, by Lemma
\ref{grubo} and the assumption $f(n)\le T_i,$ we have
$T_j\le h_1<\area(\Delta)\le f(n)\le T_i.$

Now we will use the notation of Lemma \ref{Delta'}. By   Lemma \ref{psi1}, we have
 $$\area (\Delta_1)\le 2\area(\Psi)+1\le
4(2LN(2h_1)+  \delta\iv(2LNh_1))h_1+1\le$$
\begin{equation}\label{f1}
(16LN+8\delta\iv LN)h_1^2+1< 9\delta\iv LN
h_1^2<800\delta\iv LNT_{i-1}^2 \le 800LNf(T_{i-1})
\end{equation}

By Lemma \ref{Delta'} (c),
 $$\mu(\Delta')\le
\mu(\Delta)+2h_1^2<\mu(\Delta)+200T_{i-1}^2 \le
\mu(\Delta)+f(T_{i-1}),$$
and so by Lemma \ref{Delta'} (b) and the definition of $\mu(*),$
\begin{equation}\label{mumu}
\mu(\Delta)-
\mu(\Delta')=c_0(\kappa(\Delta)-
\kappa(\Delta')) + (\lambda(\Delta)-
\lambda(\Delta'))\ge
-f(T_{i-1})
\end{equation}

  Now using Lemma \ref{Delta'} (a, d), inequalities $n_Q\ge n'_Q+2$ and (\ref{mumu})
we have
$$F(\Delta)-F(\Delta')> c_6\times 0-c_4f(T_{i-1})+c_5\times 0 +2c_7f(T_{i-1})\ge c_7f(T_{i-1})$$
This inequality, (\ref{f1}), the inductive hypothesis (valid by Lemma \ref{Delta'} (a)),
and Lemma \ref{2diagr} imply
$$\area(\Delta)\le \area(\Delta')+\area(\Delta_1)< F(\Delta') +800LNf(T_{i-1})
\le F(\Delta')+c_7f(T_{i-1})<F(\Delta),$$ and so $\Delta$ is not a
counter-example in this case too.

The proof is complete.

\endproof

Now we go back to the combinatorial length $||.||$ and 
make use of the obvious quadratic upper bounds for the mixtures.

\begin{lemma}\label{final} There is a constant $C$ such that
for every $i=2,3,\dots$ and arbitrary minimal diagram $\Delta$ with
 $||\partial\Delta||=r$, we have $\area(\Delta)\le C
(r^2+rT_{i-1}g(CT_{i-1}^2)^2 + T_{i-1}^4)$ provided $Cr
(g(Cr^2)^2+Cr^3)<T_i.$
\end{lemma}

 \proof By lemma \ref{ochev}, $|\partial\Delta|=n\le
r\le\delta^{-1}n.$ Recall also that by Lemma \ref{mixture} (a) and
the definition of $\mu$- and $\nu_J$-mixtures, $\mu(\Delta)\le (c_0+1)r^2,$
$\nu_J(\Delta)\le Jr^2$, and also $n_Q\le n$.
Now the statement of the lemma follows with a constant $C\ge 2\delta\iv c_7$
from Lemma \ref{main} and from the definitions of the function $f.$
\endproof

Finally we apply Theorem \ref{re} converted into Property \ref{xiv} of trapezia.

\begin{lemma}\label{almostquad} The Dehn functions of the groups $G$ and $M$
 are almost quadratic.
\end{lemma}

\proof 
We consider only the Dehn function $d(r)$ of the group $G$ since a simpler proof
works for $M.$ (One considers only diagrams having no hubs in the later case.)

Assume that an integer $m$ satisfies the hypothesis of \ref{xiii} and
$m$ is large enough, say $m>\max(T_1, C^{20}(\log m)^{40}),$ where $C$
 is provided by Lemma \ref{final}.
There is a  maximal $i$ such that $T_{i-1}\le m.$ Since $T_{i-1}$
is the height of a standard trapezium with some bottom $W$ and
any rule corresponding to this trapezium can decrease the length of the input sector
at most by $1,$
we have $|W|_a\le T_{i-1}\le m,$ 
and so by \ref{xiii} and the choice of $i,$ we get

\begin{equation}\label{good}
\exp T_{i-1}<m <T_i,
\end{equation}
 It follows from (\ref{good}) and the choice of $m$ that
 \begin{equation}\label{120}
 CT_{i-1}^2< m^{1/20}.
 \end{equation}

 Now, on the one hand, inequality $n<m/7$ implies
 \begin{equation}\label{67}
 g(n)\le \frac47 m+3\log m<m<T_i
 \end{equation}
by \ref{xiii}.
On the other hand, by Property \ref{xv}, any value $g(n)$ of the function $g$ either
belongs to some interval $(T_j,9T_j)$ or $g(n)\le6n.$
By (\ref{67}), we have $j<i$ in the former case. Therefore if $n<m/7,$ then in any case
\begin{equation}\label{gn}
g(n)\le \max(9T_{i-1}, 6n)\le \max(9\log m, 6n)
\end{equation}
Hence there is a constant $D$ such that for every integer $r$ such that  $r>9\log m$ and $Dr^5\le m,$ we have  
\begin{equation}\label{kfinal}
Cr(g(Cr^2)^2+Cr^3) <Dr^5\le m<T_i
\end{equation}
Now by Inequalities (\ref{kfinal}), (\ref{good}), (\ref{gn}), (\ref{120}), and by Lemma \ref{final},
we have $$d(r)\le C
(r^2+rT_{i-1} g(CT_{i-1}^2)^2 + T_{i-1}^4)\le C
(r^2+r(\log m) g(CT_{i-1}^2)^2 + T_{i-1}^4)\le $$ $$ C(r^2 +r(\log
m) (\max (9\log m, 6m^{1/20}))^2 +(\log m)^4) \le 2Cr^2 $$ if 
$Dr^5\le m,$ and $r>m^{1/6}$.

Since the set of integers $m$ satisfying the hypothesis of \ref{xiii}, is infinite by \ref{xiv}, we can
find for almost every such $m,$ an integer $r$ satisfying inequalities  $Dr^5 \le
m<r^6,$ and so the inequality $d(r)\le 2Cr^2$ holds on an infinite
set of integers $r.$
\endproof

{\bf End of proofs of Theorems \ref{mainth} and \ref{thmain1}.} Using the notation 
of Lemma \ref{abc}, we consider a word $V\equiv W(M)$ 
for an arbitrary admissible input word $W$ of the machine $M_4$. Assume that $V=1$ in 
the group $G.$ Then there is a minimal diagram $\Delta$ whose boundary path
is labeled by $V.$ Since every state letter from the  vector of start states of $M$ occurs in $V$  exactly once, every maximal $q$-band of $\Delta$ must end on a hub, and $\Delta$
has $m\ge 1$ hubs. On the other hand, $m\le 1$ by Lemma \ref{mnogospits}, since
$|V|_q = LN$ by the definition of the standard base for the machine $M.$ Thus 
$\Delta$ has exactly one hub $\Pi$, and so every maximal $q$-band of $\Delta$ connects
the boundaries of $\Pi$ and $\Delta.$ 

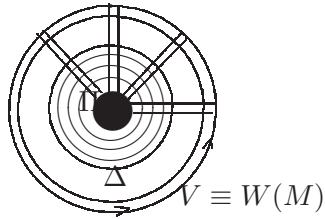
\unitlength 1mm 
\linethickness{0.4pt}
\ifx\plotpoint\undefined\newsavebox{\plotpoint}\fi 


Since $V$ has no $\theta$-edges, by Lemma \ref{NoAnnul}, every non-hub cell of $\Delta$ belongs
to a $\theta$-annulus surrounding the hub $\Pi.$ (The set of these annuli is not empty since $V$ has no state letters of the hub relation.) Hence one can remove $\Pi,$ make a radial
cut, and construct a trapezium with base (\ref{baza}). By Lemma \ref{simul} (1), the computation
of $M$ corresponding to this trapezium accepts the word $V,$ and therefore $V\in X_5$ by
Lemma \ref{abc}.

Conversely, assume that $V\equiv W(M)\in X_5.$ Then by Lemma \ref{simul} (2), there is
a trapezium with base (\ref{baza}) corresponding to an accepting computation $W(M)\to\dots$
of $M.$ Now one may identify the left-most and the right-most maximal $t$-bands of this
trapezium and paste up the hole of the obtained annular diagram by a hub. Hence $V$
is a boundary label of a disc van Kampen diagram, and therefore $V=1$ in $G.$

The obtained criterion shows that the word problem is undecidable for $G$ since the set $X_5$ is not recursive by Lemma \ref{abc}. By Lemma \ref{almostquad}, the proof of Theorem \ref{mainth}
is complete.

Relations (\ref{rel1}) of the group $M$ define the structure of a (multiple)
HNN-extension on the group $M$ whose base it the free subgroup generated
by all $a$- and $q$-letters, and for every rule, one has a stable $\theta$-letter.
(See the presentation of every S-machine as an HNN-extension in \cite{OS}.)
The statements 2 and 3 of Theorem \ref{thmain1} hold for $M$ by 
Lemma \ref{almostquad} and by Step 1 of the proof of Lemma \ref{grubo}. Finally,
a word $V\equiv W(M)$ is conjugate to the hub in $M$ iff $V\in X_5$. (The proof is
similar to the criterion obtained above for the equality $W(M)=1$ in $G$, but now
one considers annular diagrams over $M$ instead of disc diagrams over $G$.)
Now the statement 1 of Theorem \ref{thmain1} follows from Lemma \ref{abc},
and the proof is complete.

      \begin{theorem}\label{exp} There exists a finitely presented group $G$ with
almost quadratic Dehn function $d(n)$ such that $d(n)\ge \exp n$ for infinitely many $n$-s,
and  $d(n)$ is bounded from above on the entire $\mathbb N$  by an exponential function.
      \end{theorem} 
      
      \proof We will make a few alternations in the proof of Theorem \ref{mainth}.
      
       Given a word $a^{n},$ it is easy to check in linear time whether $n=2^m$
      for some natural $m$ or not and to compute $m=\log_2 n$ if $m\in \mathbb N.$
      Therefore there is a deterministic Turing machine $M_0$ with linear time complexity
      which accept a word $a^n$ iff  $n$ belongs to the sequence $n_1=1,$ 
      $n_i=2^{2^{2^{n_{i-1}}}}$ for $i>1.$ Clearly almost every $n_i$ is an $h_{\alpha}$-good number
      for any function $h_{\alpha}(x)=2^{2^{\alpha x}},$ and we can use this property
      instead of Theorem \ref{re}.  
      
      Starting with $M_0,$ we construct the machines $M_1,\dots, M_4, M$ and define
      the group $G$ as in the paper. Then we obtain, as in Theorem \ref{mainth}, that $d(n)$ is almost quadratic
      (since the non-recursiveness from Theorem \ref{re} has never been uses for this goal).
      
      For some positive constants $c'$ and $c'',$ Lemmas \ref{X2}(b) and \ref{M31}(b)
      give the estimates $\exp(c'n_i)<T_i<\exp(c''n_i)$ for the time of acceptance $T_i$
      of the word $a^{n_i}$ by the machine $M_3.$ As in the above ``End of proofs'',
      it follows that the length of the corresponding to $a^{n_i}$ accepted input word $V_i$ of the
      machine $M$ is $O(n_i)$ while the area is at least $T_i>\exp(c'n_i).$ Thus
      $d(n)$ is bounded from below on the infinite sequence of $n_i$-s by an exponent. It remains
      to obtain an exponential upper bound $d(n)< \exp(Cn)$ on the entire $\mathbb N.$
      (We do not need any mixtures for this goal.)
      
      Assume that $C$ is large enough, and a minimal diagram $\Delta$ over $G$ has area 
      $\ge \exp(Cn)$. Then $\Delta$ has no rim $\theta$-bands $\cal T$ of base width $\le 2LN$
      because $|\partial\Delta'|< n $ for the subdiagram $\Delta'=\Delta\backslash\cal T$ in
      Case (2) of the proof of Lemma \ref{main}. Similarly $\Delta$ has no long
      subcombs (or subquasicombs) since Lemmas \ref{grubo} (Step 1 of the proof)  provides us with a cubic upper bound of the area of any subcomb (as function of $n$.)
      Therefore the diagram $\Delta$ is solid, and therefore Lemma \ref{Delta'} (a) reduces
      our task to diagrams having at most one hub. Indeed, by Lemma \ref{mnogospits},
      The number of hubs in $\Delta$ does not exceed $n,$ and the functions $\exp(Cn)$
      and $n\exp(Cn)$ are equivalent.
      
      We may assume that $\Delta$ has exactly one hub since otherwise  its area is
      bounded by a cubic function of the perimeter. Now by Lemma \ref{psi1} applied to the whole $\Delta,$ we conclude that every
      maximal
      $\theta$-band of $\Delta$ is an annulus, and so, as at the ``End of proofs'' above,
      the boundary label of $\Delta$ is of the form $V\equiv W(M)$ for some admissible
      input word of the machine $M_4.$ Therefore it suffices (by Lemma \ref{simul}) to find an exponential
      upper bound for the accepting computations of $M_4$ with respect to the length $||W||$
      of an input admissible word $W$. Such an upper bound (even a linear bound) is given
      by Lemma \ref{*XVII} if the length of the reduced computation of $W$ does not belong
      to any interval $(T_i,9T_i).$ The argument of that lemma works in other cases
      if the computation does not contain the standard computation of length $T_i,$ (i.e.,
      the computation of $n_i$ in our situation). However the proof of Lemma \ref{*XVII}
      also shows that $||W||\ge n_i$ in the remaining cases. Therefore the length of the
      computation has the exponential upper bound $9T_i\le 9\exp(c''||W||),$ and the
      proof is complete since $C>>c''.$
      
      \begin{rk}\label{other} One can replace the exponential function by a multiexponential one
      or by many other functions with at least exponential growth in the formulation and in the proof of Theorem
      \ref{exp}.
      \end{rk}
     
      \endproof

\section {Appendix:  A very sparse immune set\\  {\rm By M.V.Sapir}}

\label{vsis}

Let $X$ be a recursively enumerable (r.e.) language in the binary
alphabet recognized by a 
Turing machine $M$. If
$x\in X$ then the {\em time} of $x$ \label{cost} (denoted $\cost(x)$ or $\cost_M(x)$) is, by
definition, the minimal time of an accepting computation of $M$ with
input $x$. For any increasing function $h\colon \N\to \N$, a real number $m$ is called {\em $h$-good} for $M$ if for
every $w\in X$, $||w||<m$ implies $h(\cost(w))<m$.

For every number $n$, the number of digits in $n$ is denoted by \label{||2}
$||n||_2$. This number is roughly $\log_2 n$. Since we are not using
any other logarithms in Appendix, we shall omit $2$ in $\log_2$.
Similarly, we shall write $\exp x$ for $2^x$.

The proof of the following theorem uses an idea communicated to the author by S.Yu. Podzorov. For every $\alpha>0$, let $h_\alpha(n)=\exp\exp(\alpha n).$

\begin{theorem} \label{re}
There exists a Turing machine $M_0$ recognizing a r.e. non-recursive
set $X$ such that the set of all $h_\alpha$-good numbers for $M_0$ is infinite for all $\alpha>0$.
\end{theorem}

\proof We use a recursive enumeration of all Turing machines from
\cite{Mal}. By Matiyasevich's solution of the 10th Hilbert problem
\cite{Mat}, there exists a polynomial $F(a,b,x_1,...,x_s)$ with
integer coefficients such that $a$ is recognized by the Turing machine number $b$ if and only if $F(a,b,x_1,...,x_s)=0$ for some
natural numbers $x_1,...,x_s$.
 We are going to use G\"odel
numeration of $s$-tuples of natural numbers. For every natural $m$
let $\mathfrak{g}(m)$ be the $s$-tuple having G\"odel number $m$. Note
that all coordinates of this tuple do not exceed $m$ and the time to
compute $\mathfrak{g}(m)$ is linear in $||m||_2$.

Note also that if $||a||_2,||b||_2,||x_i||_2\le n$ ($i=1,...,s$) then the
time needed to compute $F(a,b,x_1,...,x_s)$ in bounded by a
polynomial in $n$ depending only on $F$. Also the time to compute
binary value of the exponent $\exp n$ (given $n$ written in binary)
is linear in $n$ (and exponential in $||n||_2$).

The algorithm of enumerating elements in $X$ involves auxiliary
formulas for functions $f_m\colon \{1,2,...\}\to \{1,2,...\}$, and two sequences of numbers $b(m)$,
$x(m)$, $m=0,1,2,...$.

\bigskip

{\bf Step 0.} Set $f_0(i)=i$ (that is by definition $f_0$ is the identity function),
$b(0)=0$, $x(0)=0$,
$X=\emptyset$.

\medskip

{\bf Step $m\ge 1$.} Search for the minimal $i=i(m)\le m$ such that
for some $m'\le m$ $F(f_{m-1}(i),i,\mathfrak{g}(m'))=0$ and $f_{m-1}(i)$ is not already in $X$. If such an $i$ exists,
 add $f_{m-1}(i)$ to $X$,
compute the new numbers $b(m)=\exp\exp\exp\exp(m+||b(m-1)||_2+x(m-1))$ ({four
exponents}), $x(m)=\max(x(m-1), f_{m-1}(i))$, and define the function
$f_{m}$ by adding in the definition of $f_m$ that
$f_{m}(j)=j+b(m)$ for every $j>i$. (Note that $f_m(j)=f_{m-1}(j)$ for every $j\le i$.) In that case we say that the step $m$ was \label{succs} {\em successful}, and $i$ is \label{respon}
{\em responsible} for counting $f_{m-1}(i)$ into $X$. Otherwise (if either $i(m)$ does not exist or $f_{m-1}(i)$ is
already in $X$) let $f_{m}=f_{m-1}, b(m)=b(m-1), x(m)=x(m-1)$.
Then go to the next step. Note that for every $i,m$ we have $f_m(i)\le i+b(m)$. Therefore if step $m$ is successful, we have $x(m)\le x(m-1)+m+b(m-1)$. Hence \begin{equation}\label{xm}x(m)< b(m).\end{equation} Inequality (\ref{xm}) holds also for unsuccessful $m$ if $m$ is larger than the number of the first successful step  by induction, because in that case $x(m)=x(m-1), b(m)=b(m-1)$.

\medskip

We claim that every number $i\ge 1$ is responsible for at most $2^i$
members of $X$. Indeed, every $i$ can be responsible only for
numbers of the form $f_m(i)$. The value $f_m(i)$ can differ from $f_{m-1}(i)$
only if some number $i'< i$ is responsible for counting some number $f_{m'-1}(i')$ into $X$ at some step
$m'> m\ge i$.  Therefore we have $f_1(1)=f_2(1)=...$, so $1$ can be responsible only for at most one number in $X$.
This implies that $2$ can be responsible for at most two numbers only: the
value $f_0(2)$ is $2$, and the value $f_m(2)$ can differ from $f_{m-1}(2)$ only if $1$ is responsible for some number in $X$. Similarly, the value $f_m(i+1)$ can
differ from $f_{m-1}(i+1)$ only when a number $j\le i$ becomes responsible for counting a
number into $X$. By induction it can happen at most
$1+2+2^2+...+2^i=2^{i+1}-1$ times. Therefore $i+1$ can be
responsible for at most $2^{i+1}$ numbers in $X$ as claimed.

\medskip

Let us prove now that the set $X$ is what we need. It is clear that $X$
is recursively enumerable: the machine $M$ enumerating this set is
described in the definition of $X$. 
(Recall, that a Turing machine \label{enumeratingtm} enumerating a set of words $X$ in a finite alphabet differs from a Turing machine recognizing it: it does not have input sector and the accept configuration. It starts working with all tapes empty, and writes words from $X$ in the first tape one by one, separated by a special symbol. After a new word is written in tape $1$ (i.e. when the machine {\em counts a new word into $X$}), the machine puts the separating symbol next to that word and continues working. If $X$ is infinite, the machine works infinitely long. For every $u\in X$, we can talk about the {\em time to count it into} $X$, i.e. the shortest length of the computation after which the word first appears in the first tape.)

\medskip
Let us prove that $X$ is not recursive. Suppose the contrary - that $X$ is recursive. Then its complement is
recursively
enumerable. Therefore there exists a natural number $b$ such that
\bigskip

(*) $F(a,b,x_1,...,x_s)=0$ for some $x_1,...,x_s$ if and only if $a$ is
{\bf\em not} in $X$.

\bigskip

Let  $b$ be the number from (*).  There exists  $m\ge 1$ which is bigger than the number of the first successful step, and such that for every $m'\ge m$ either
$i(m') > b$ or the step number $m'$ is not successful (this follows from the fact that each $i$ is responsible for finitely many members of $X$ only).
Let $m$ be one of the numbers with this property.

By definition,

\begin{equation} \label{y}
{f_{m'}(j)=f_m(j)}
\end{equation}
for every $m'>m, j\le b$.

{\bf Claim.} No $j\ne b$ can be responsible for counting {$r=f_{m-1}(b)$} into $X$.

Suppose that $j<b$ is responsible for counting $r$ into $X$. That cannot happen at step $m'>m$ because $i(m')>b$ since $r=f_{m-1}(j)\le f_{m'-1}(j)< f_{m'-1}(i(m'))$ by
definition of $m$. If that happens at step $m$, then $f_{m-1}(j)=r=f_{m-1}(b)$ which is impossible since $f_{m-1}$ is
strictly increasing. If that happens at step number $m'<m$, then, since $j<b$, $f_{m'}(b)\ge b+b(m')\ge b(m')$, and
we would have
$$f_{m-1}(b)\ge f_{m'}(b)\ge b(m')>x(m')\ge r$$
by (\ref{xm}) and because $x(m')$ is the maximum of all numbers counted into $X$ at steps $\le m'$, including $r$, a contradiction.

Suppose that $j>b$ is responsible for counting $r$ into $X$ at some step $m'$. Suppose that $m'<m$. We have
$f_{m'-1}(j)=r$. Since $j\ne b$, $f_{m'-1}(b)\ne r$. Therefore $f_{k}(b)$ has changed at some step $k$ such that
$m'\le k\le m-1$. Hence there exists a successful step number $k$, $m'\le k\le m-1$ such that $i(k)<b$. But then
$$f_{m-1}(b)\ge f_{k}(b)\ge b(k)>x(k)\ge x(m')\ge r,$$
a contradiction.

It remains to consider the case when $j>b, m'\ge m$. But in that case (since $f_{m'-1}$ is strictly increasing)
$$f_{m'-1}(j)>f_{m'-1}(b)\ge f_{m-1}(b)=r, $$ a contradiction. This completes the proof of our claim.

Now if $r$ is in $X$ then for some $m'$, $F(r,b,\mathfrak{g}(m'))=0$ (since by the Claim only $b$ can be responsible for counting $r$ into $X$).
But this would mean, by the choice of $b$ (see (*)), that $r$ is {\bf not} in
$X$, a contradiction. On the other hand if $r$ is not in $X$ then
$F(r,b,\mathfrak{g}(j))=0$ for some $j$, therefore at some step $m'$,
$b$ would be responsible for counting $r$ into $X$, so $r\in X$, a
contradiction. This shows that $X$ is not recursive. In particular, $X$ is infinite.

Note that  there exists a
deterministic Turing machine $M_0$ which recognizes $X$ and such that
for every $m\in X$, the time to recognize it by $M_0$ is linearly bounded in terms of the time to count it into $X$ by $M$.
Indeed let us add the input tape to the tapes of $M.$ The machine $M_0$ will
execute $M$ on its tapes. Every time there is a new word
counted into $X$, the machine $M_0$ checks whether this word coincides with the input word. After the match is found, $M_0$ erases all
tapes and stops.

Now let us determine the $h_\alpha$-good numbers of the machine $M_0$.
We say that a number $n$ is \label{appropr}{\em appropriate} if $i(n)$ exists, $f_{n-1}(i(n))$ is counted into $X$ at step $n$
and none
of $i(n')$ with $n'>n$ is smaller than $i(n)$. Clearly the set of
appropriate numbers is infinite (since every number is responsible
only for finitely many members of $X$, see above). Let $B$ be the
set of numbers $[\log b(n)]$ for appropriate $n$. Let us show that
almost all numbers in $B$ are $h_\alpha$-good.

Indeed, let us estimate the time of a number $r=f_{n-1}(i_n)$
counted into $X$ at an appropriate step number $n$
by $M_0.$
The total number of evaluations of $F$ needed for this is
at most $n^3$ {($\le n$ steps, at most $n\times n$
evaluations of $F(f_{n'-1}(i),i,\mathfrak{g}(t))$ at each step
where $1\le i\le n, 1\le t\le n$)}.
We can estimate the time of each evaluation of $F$ as a polynomial
in $||n||_2+||b(n-1)||_2$. In addition of computing values of $F$, we
also have to compute the numbers $b(n')$ and the formulas for
$f_{n'}$ (at most $n$ times). The time of computing $b(n')$ and $f_{n'}$
does not exceed the time of computing $b(n-1)$ and $f_{n-1}$. And
those times can be bounded by a polynomial in $||b(n-1)||_2$.
Recall that the time of recognizing $r$ by $M_0$ is bounded by a constant times the time of counting $r$ into
$X$ by $M$. Thus the
total time of accepting $r$ by the machine $M_0$ is bounded by
$cn^c||b(n-1)||_2^c+c$ for some constant $c$. Note that

\begin{equation}\label{2s} \exp\exp(cn^c||b(n-1)||_2^c+c)<\exp\exp\exp(||b(n-1)||_2)<\log b(n)
\end{equation}
for almost all $n$. Also notice that since $n$ is appropriate,  by the definition of $f_n$,
there are no numbers $r'\in X$ between $r+1$ and $b(n)$. Hence if
$r\in X$ and $||r||_2<\log b(n)$ then {$\exp\exp(\alpha\cost(r))<\log b(n)$} by (\ref{2s})
(for all but finitely many $n$ and some $\alpha$). Hence $\log b(n)$ is an $h_\alpha$-good number
of $M_0$ for almost all $n$. Therefore the set of $h_\alpha$-good numbers for $M_0$ is
infinite.
\endproof

\twocolumn

\noindent {\bf Subject index} \label{sind}
\bigskip

\noindent $a$-band \pr{aband}\\ 
Accept (stop) configuration of a  machine \pr{acceptc}\\
Accepted input word \pr{acceptediw}\\
Accepted word by an S-machine \pr{acceptedws}\\
Accept vector $\overrightarrow s_0$ \pr{acceptv}\\
Active from the left/right band \pr{activelrb}\\
Active from the left/right cell \pr{alrcell}\\
Active from the left/right letter \pr{activel}\\
Admissible word \pr{admissiblew}\\
$a$-edge \pr{aedge}\\
$a$-length $|.|_a$ \pr {wa}, \pr{alength}\\
$a$-letter \pr{aletter}\\
Aligned base \pr{alignedb}\\
Almost quadratic function \pr{almostqf}\\
Annulus \pr{annulus}\\
Application of a rewriting rule \pr{applicationr}\\
Appropriate number \pr{appropr}\\
Area of a diagram \pr{aread}\\
Area of a word \pr{areaw}\\
Band \pr{band}\\
Base of a band \pr{baseb}\\
Base of a trapezium \pr{baset}\\
Base width of a comb \pr{basewc}\\
Base of a word \pr{basew}\\
Bead black or white \pr{bead}\\
Bottom of a band \pr{bottomb}\\
Bottom of a trapezium \pr{bottomt}\\
Cell \pr{cell}\\
Chain \pr{chain}\\
Chain-annulus \pr{chain-ann}\\
Close $q$-band (to) \pr{close} \\
Clove \pr{clove}\\
Comb \pr{comb}\\
Combinatorial length of a path \pr{clength}\\
Comb mixture \pr{combmix}\\
Command of a Turing machine \pr{command}\\
Computation \pr{computation}\\
Configuration of a Turing machine \pr{config}\\
Copy of a word \pr{copyw}\\
Control state letter \pr{controlsl}\\
Crescent \pr{crescent}\\
Crossing bands \pr{cross}\\
Dehn function \pr{Dehnf}\\
Derivative subcomb \pr{dersc}\\
Derivative $q$-band \pr{derivqb}\\
Diagram \pr{diagram}\\
Diagramm admitting a quasicomb \pr{admit}\\ 
Domain of a rule \pr{domain}\\
Enumerating Turing machine \pr{enumeratingtm}\\
Equivalence of functions \pr{equivf}\\
$f$-good number \pr{good}, \pr{fgoodn}\\
Filling subtrapezium \pr{fillingst}\\
Firm factorization \pr{firm}\\
Group $G$ \pr{groupG}\\
Group $M$ \pr{groupM}\\
Handle of a comb \pr{handlec}\\
Height of a trapezium \pr{heightt}\\
History of a band \pr{historyb}\\
History/ step history of a comb \pr{historyc}\\
History of a computation \pr{history}\\
History of a trapezium \pr{historyt}\\
History of type ... \pr{typeh}\\
Hub \pr{hub}, \pr{hubs}\\
$H'$-part of a comb \pr{H'partc}\\
$H'$-part of a trapezium \pr{H'partt}\\
Input configuration of a Turing machine \pr{inputc}\\
Input sector of $M_4$ \pr{inputs}\\
Inverse command \pr{inversec}\\
$(i)-(i+1)$ transition band \pr{(i)-b}\\
$K$-mixture on a necklace\\
Language accepted by a machine \pr{languageam}\\
Large base \pr{largeb}\\
Left and right rules \pr{leftrr}\\
Left/right sides of a $q$-band in a comb \pr{rightlsqbc}\\
Length (time) of a computation \pr{lengthtime}\\
Length $|.|$ of a word \pr{lengthw|}\\
Length $||.||$ of a word (combinatorial) \pr{lengthw||}\\
Length $|.|$ of a path \pr{lengthp|}\\
Link of a chain \pr{link}
Long (quasi)comb \pr{longqc}\\
$k$-band \pr{kband}\\
$k'$-band \pr{k'band}\\
$k$-letter \pr{kl}\\
$k'$-letter \pr{k'l}\\
Maximal band \pr{maxb}\\
Minimal diagram \pr{minimald}\\
$M$-index \pr{Mind}\\
$M_2$-index \pr{M2ind}\\
$M_4$-accepting trapezium \pr{M4acct}\\
Necklace \pr{neckl}\\
Normal base \pr{normalb}\\
Odd cell \pr{oddc}\\
Odd $\theta$-band \pr{oddb}\\
One Step comb \pr{onestepc}\\
Parameters \pr{param}\\
Part of a rule \pr{partr}\\
Passive band (from the left/right) \pr{passiveb}\\
Passive cell \pr{passivec}\\
Passive $q$-letter \pr{passive}\\
$p$-band \pr{pband}\\
Perimeter $|.|$ of a diagram \pr{perimd|}\\
$p$-letter \pr{pletter}\\
$p_i$-letter, edge, band \pr{piletter}\\
Positive and negative commands \pr{positivec}\\
Positive and negative rules \pr{positiver}\\
Projection of a word \pr{projectw}\\
Proper subcomb \pr{propersc}\\
$q$-band \pr{qband}\\
$q$-edge \pr{qedge}\\
$q$-length $|.|_q$ \pr{qlength}\\
$q$-letter \pr{qletter}\\
Quasicomb \pr{quasic}\\
Reduced computation \pr{reducedc}\\
Reduced diagram \pr{reducedd}\\
Regular comb \pr{regcomb}\\
Regular extension of a comb \pr{rexec}\\
Relations of the group $M$ \pr{relations}\\
Responsible for counting number \pr{respon}\\
Rim band \pr{rimb}\\
Rule \pr{rule}\\
Rule locks a sector \pr{locks}\\
Sector of a word \pr{sectorw}\\
Short derivative \pr{shortder}\\
Side of a band \pr{sideb}\\
Sides of a trapezium \pr{lrsidest}\\
$s$-band \pr{sband}\\
$s_i$-band \pr{siband}\\
Simple $\theta$-band \pr{simpletb}\\
$s$-letter \pr{sl}\\
S-machine \pr{Smachine}\\
S-machine $M$ \pr{MS}\\
S-machine $M_2$ \pr{M2}\\
S-machine $\tilde M_2$ \pr{tildeM2}\\
S-machine $M_3$ \pr{M3}\\
S-machine $M_4$ \pr{M4S}\\
S-machine $M_4(i)$ \pr{M4iS}\\
S-machine $S_1$ \pr{M1S}\\
S-machine $Z(A)$ \pr{ZA}\\
S-machines $\overrightarrow Z(A)$ and $\overleftarrow Z(A)$ \pr{overZ}\\
Solid diagram \pr{solidd}\\
Special $\theta$-edge \pr{specialte}\\
Spoke \pr{spoke}\\
Standard base \pr{standardb}\\
Standard computation of $M_4$ \pr{standardcM4}\\
Standard trapezium \pr{standardM}\\
Start and end edges of a band \pr{seedgesb}\\
Start vector $\overrightarrow s_1$ \pr{startv}\\
State letter \pr{statel}\\
Step history of a band \pr{stephb}\\
Step history of a computation \pr{steph}\\
Step history of a trapezium \pr{stepht}\\
Steps of $M_4$ \pr{Steps}\\
String of beads \pr{stringb}\\
Strongly active from the left/right band \pr{sactivelrb}\\
Subcomb of a diagram/comb \pr{subcombd}\\
Successful step \pr{succs}\\
Symmetric Turing machine \pr{symmetricm}\\
Tape letter \pr{tapel}\\
$t$-band \pr{tband}\\
$t'$-band \pr{t'band}\\
$t$-letter \pr{tl}\\
$t'$-letter \pr{t'l}\\
$t_i$-reflection \pr{tlrefl}\\
Top of a band \pr{topb}\\
Top of a trapezium \pr{topt}\\
Transition rule $(12)$ \pr{Tr12}\\
Transition rule $(23)$ \pr{Tr23}\\
Trapezium \pr{trapez}\\
Trimmed sides of a trapezium \pr{trim}
Turing machine \pr{Turingm}\\
Turing machine $M_0$ \pr{M0}\\
Turing machine $M_1$ \pr{M1}\\
$B$ \pr{Bst}\\
$B(i)$ \pr{Bi}\\
$\bott(.)$ \pr{topp}\\
${\cal C}_1,\dots {\cal C}_s$ \pr{Ci}\\
$cl(\pi,{\cal B},{\cal B'})$ \pr{cl.}\\
$\partial$ \pr{partial}\\
$F(\Delta)$\pr{FD}\\
$h$ \pr{hc}, \pr{hq}\\
$h'$\pr{h'}\\
$h_i$ \pr{hi}\\
$h'_1,\dots$ \pr{h'1...}\\
$h_-$ \pr{hminus}, \pr{hminusq}\\
$H_1,\dots,H_s$ \pr{derhist}\\
$l_-$ \pr{lminus}\\
$\Lab(.)$ \pr{Lab}\\
${\bf P}_k$ \pr{Pk}\\
${\bf p}(\Psi)$ \pr{pPsi}\\
$Sym(M_1)$ \pr{symM1}\\
$T_i$ \pr{Tis}\\
$\cost (x)$ \pr{cost}\\
$\tbott(.)$ \pr{tbott}\\
$\ttopp(.)$ \pr{ttopp}\\
$\topp(.)$\pr{topp}\\
$Tp(.\; ,\; .)$ \pr{Tp.}\\
$W_0$ \pr{W0}\\
$W(i)$ \pr{Wi}\\
$W(M)$ \pr{WM}\\
$X_1$ \pr{X1}\\
$X_2$ \pr{X2}\\
$X_3$ \pr{X3}\\
$X_4$ \pr{X4}\\
$X_5$ \pr{X5}\\
$\bf y$ \pr{y}, \pr{yq}\\
$\bf y'$\pr{y'}\\
$Y(3)$ and $Y'(3)$ \pr{Y3pr}\\
$Y_i(\theta)$ \pr{Yit}\\
$\bf z $ \pr{z}, \pr{zq}\\
$\bf z'$ \pr{z'}\\
$Z^{(\theta,i)}$ \pr{Zti}\\
$\overleftarrow Z^{(\theta,i)}$ and $\overrightarrow Z^{(\theta,i)}$\pr{Ztilr}\\ 
$\Gamma_i$\pr{Gi}\\
$[\Gamma]=h(\bf z - y)$ \pr{[Gamma]}, \pr{[q]}\\
$\zeta(\theta)$ \pr{zetat}\\
$\zeta_-(\theta)$ and $\zeta_+(\theta)$ \pr{zmt}\\
$\kappa(.)$\pr{mu11.}\\
$\kappa$-mixture \pr{muK1mix}, \pr{munustring}\\
$\kappa$-necklace \pr{mu1neckl}\\
$\kappa^c(.)$ \pr{muKc}, \pr{mucjK}\\
$\lambda(.)$\pr{mu21.}\\
$\lambda$-mixture \pr{mu2Kmix}, \pr{munustring}\\
$\lambda$-necklace \pr{mu2neckl}\\
$\lambda^c(.)$ \pr{muKc}, \pr{mucjK}\\
$\mu(.)$ \pr{mu.}, \pr{munustring}\\
$\mu_K(.)$ \pr{muKO}\\
$\mu^c(.)$ \pr{muc.}, \pr{muq}\\
$\nu_K$-mixture \pr{nuKmix}\\
$\nu$-necklace \pr{nuneckl}\\
$\nu_K(.)$ \pr{nuK.}\\
$\nu_K^c(.)$ \pr{nucK}, \pr{nuq}\\
$\pi_{12}(.)$ \pr{pi12}\\
$\pi_{21}(.)$ \pr{pi21}\\
$\pi_{23}(.)$ \pr{pi23w}\\
$\pi_{32}(.)$ \pr{pi32}\\
$\pi_{34}(.)$ \pr{pi34}\\
$\Pi_{12}(.)$ \pr{Pi12}\\
$\Pi_{21}(.)$ \pr{Pi21}\\
$\Pi_{23}(.)$ \pr{Pi23}\\
$\Pi_{32}(.)$ \pr{Pi32}\\
$\Pi_{34}(.)$ \pr{Pi34}\\
$\theta_{accept}$\pr{taccept}\\
$\theta_{start}$ \pr{tstart}\\
$\theta(M_4)$ \pr{tM4}\\
$\theta$-band \pr{thband}\\
$\theta$-edge \pr{thedge}\\
$\theta$-length $|.|_{\theta}$ \pr{thlength}\\
$\theta$-letter \pr{thetal}\\
$(\theta,a)$-cell \pr{ta}\\
$(\theta,a)$-relation \pr{thetaar}\\
$(\theta,q)$-cell \pr{tq}\\
$(\theta,q)$-relation \pr{thetaqr}\\
$\Theta(M_4)$ \pr{TM4}\\
$\Theta(3)$ \pr{T3}\\
$\Theta^{\pm }$ \pr{T+-}\\
$\bar\theta$ \pr{bart}\\
$(12)$- and $(23)$-bands \pr{12band}\\
$(12)$- and $(23)$-cells \pr{12cell}\\
$(12)$- and $(23)$-letters and edges \pr{12letter}\\
$\equiv$ \pr{equivgraph}\\
$|.|$ \pr{|.|}\\
$||.||_2$ \pr{||2}\\
$\tool$ \pr{tool}


\addtocontents{toc}{\contentsline {section}{\numberline { }Subject
index \hbox {}}{\pageref{sind}}}

 \end{document}